\documentclass[10pt, reqno]{amsart}
\usepackage{amssymb, latexsym}
\usepackage{amsmath}
\usepackage{amsthm}
\usepackage{graphicx}
\usepackage{hyperref}
\usepackage{titletoc}
\usepackage{pdfsync}
\usepackage{color, xcolor}
\usepackage{amsthm, amsmath, amssymb}
\usepackage{amsthm, amsmath, amssymb}
\usepackage{mathrsfs}

\hypersetup{
  colorlinks=true,
  linkcolor=blue,
  citecolor=red,
  urlcolor=blue
}

\usepackage{mathrsfs}

\numberwithin{equation}{section}
\newtheorem{theorem}{Theorem}[section]
\newtheorem{proposition}[theorem]{Proposition}
\newtheorem{lemma}[theorem]{Lemma}
\newtheorem{corollary}[theorem]{Corollary}

\usepackage{multirow}
\usepackage[font=bf, aboveskip=15pt]{caption}
\usepackage[toc, page]{appendix}

\theoremstyle{definition}

\newcommand{\R}{\mathbb R}

\newcommand{\cuad}{{\sqcap\kern-.68em\sqcup}}

\newcommand{\e}{\epsilon}

\newcommand{\be}{\begin{equation}}
\newcommand{\ee}{\end{equation}}

\newtheorem{remark}[theorem]{Remark}

\def\G{{\mathfrak G}}

\def\R{{\mathfrak R}}

\def\begeq{\begin{equation}}
\def\endeq{\end{equation}}
\def\p{\partial}

\def\lf{\left}
\def\ri{\right}

\def\R{\Bbb R}

\def\G{\Gamma}

\allowdisplaybreaks

\begin{document}

\title
[Gross-Pitaevskii equation with non-symmetric potential]
{Solutions with one dimensional concentration for a two dimensional Gross-Pitaevskii model with general potential}

\author{Lipeng Duan}
\address{Lipeng Duan,
\newline\indent School of Mathematics and Information Science, Guangzhou University,
\newline\indent Guangzhou 510006, P. R. China.
}
\email{lpduan777@sina.com}

\author{Suting Wei}
\address{Suting Wei,
\newline\indent Department of Mathematics, South China
Agricultural University,
\newline\indent Guangzhou, 510642, P. R. China.
}
\email{stwei@scau.edu.cn}

\author{Jun Yang$^\S$}
\address{School of Mathematics and Information Science,
Guangzhou University,
\newline\indent Guangzhou 510006, P. R. China.
}
\email{jyang2019@gzhu.edu.cn}

\begin{abstract}
We concern standing wave solutions with frequency $\lambda$ to a two dimensional Gross-Pitaevskii equation with a trap potential under the unit mass constraint,
which is used to describe Bose-Einstein condensates with attractive interaction.
First, we investigate the necessary conditions for existence of the solutions with concentration phenomena directed along closed smooth curves.
Next, not only imposing stationary and non-degeneracy conditions on the curves with respect to an auxiliary weighted length involving the trap potential,
but also adding some other technical assumptions,
we select a sequence $\{\lambda_j\}$ of the frequency $\lambda$ with $-\lambda_j\rightarrow +\infty$ and construct solutions with concentration directed along the curves.
Our result partially answers the conjecture raised in [A. Ambrosetti, A. Malchiodi, W.-M. Ni, Comm. Math. Phys. 2003] about necessary condition for solution concentrating at submanifolds.
The solutions constructed in this paper are concentrating on curves whose length are non-uniformly bounded,
and hence the situation is quite different from that in [M. del Pino, M. Kowalczyk, J. Wei, Comm. Pure Appl. Math. 2007].


\vspace{2mm}
\vspace{2mm}


\end{abstract}

\date{\today}

\thanks{$^\S$Corresponding author: Jun Yang, jyang2019@gzhu.edu.cn}

\maketitle

\tableofcontents

\section{Introduction}\label{Section1}

\subsection{ The problem with a constraint}\label{Section1.1}\

Consider the time-dependent Gross-Pitaevskii equation
\begin{align} \label{000original}
i \frac{\partial \Psi}{\partial {\mathfrak t}} = - \Delta
\Psi + W({\bf y} )\Psi - a |\Psi|^2\Psi, \quad \forall\, ({\mathfrak t}, {\bf y}) \in {\mathbb R}\times{\mathbb R}^N, \quad N\geq 2
\end{align}
with a mass constraint
\begin{align} \label{constraint00}
\int_{{\mathbb R}^N} |\Psi|^2 =1,
\end{align}
where $i$ is the imaginary unit, $a>0$ represents the attractive interaction strength and $W( {\bf y} )$ is the external trap potential.
This kind of equation \eqref{000original} with the mass constraint \eqref{constraint00} is a
well-known mathematical model to describe Bose-Einstein condensates
\cite{aftalionriviere, aftalionbook, Cornell, Davis, Gross, LS, Lieb2, Pethick-Smith, Pitaevskii}.

\smallskip
If we look for a standing wave solution of \eqref{000original}-\eqref{constraint00} with the form
$$
\Psi({\mathfrak t}, {\bf y}) = \tilde u({\bf y}) e^{-i \lambda {\mathfrak t}}
$$
where $\lambda \in {\mathbb R}$ is a frequency parameter to be determined and $\tilde u$ is an unknown function,
then the pair $(\lambda, \tilde u)$ must satisfy the following time-independent equation
\begin{align} \label{eq1}
 - \Delta \tilde u + W({\bf y}) \tilde u -a |\tilde u|^2 \tilde u = \lambda \tilde u,
\quad \forall\,{\bf y} \in \mathbb{R}^N,
\end{align}
with the mass constraint
\begin{align} \label{constraint0}
 \int_{{\mathbb R}^N} |\tilde u|^2 {\mathrm d} {\bf y} =1.
\end{align}
We also know that a ground state solution of \eqref{eq1}-\eqref{constraint0} is a minimizer of the following problem
\begin{align}\label{minimzation}
 E_a : = \min_{ u \in \mathscr M} J_a (u),
\end{align}
where
\[ J_a (u) = \int_{{\mathbb R}^N} \Big[\, \frac{1}{2}| \nabla u|^2 + \frac{1}{2} W({\bf y}) u^2 -\frac a 4 |u|^4 \,\Big] {\mathrm d} {\bf y}, \]
and
\[\mathscr M = \Big\{ u \in H^1({\mathbb R}^N)\, :\, \int_{{\mathbb R}^N} |u|^2 {\mathrm d} {\bf y} =1\,\text{and} \,\int_{{\mathbb R}^N} W({\bf y}) |u|^2 {\mathrm d} {\bf y} <+\infty\Big\}. \]
 Readers can refer to \cite{numerics, baoREVIEW, GS, Guo, Guo1, KG} on the ground state solutions of Gross-Pitaevskii equation.
 On the other hand, any solution of \eqref{eq1} whose energy is greater than the ground state is called excited state in the physical literature \cite{Bao}.

\smallskip
 Recently, P. Luo, S. Peng, J. Wei and S. Yan \cite{Luo} studied problem \eqref{eq1}-\eqref{constraint0} in ${\mathbb R}^N$ $ (N=2, 3)$,
and showed that, if $\tilde u_a$ is a solution to problem \eqref{eq1}-\eqref{constraint0} concentrating at some points as $a\rightarrow {\bf a}$, then there must have been
\[ \text{${\bf a} \geq 0$ \quad and \quad $\lambda\to -\infty, $ as $a\to {\bf a}$, } \]
where ${\bf a} = k a_* $ for some integer $k$ if $N=2$, or ${\bf a} =0$ if $N=3$. Here
$$
a_*= \int_{{\mathbb R}^2} |Q|^2 {\mathrm d} {\bf y},
$$
and $Q$ is the unique positive solution to the equation
$
- \Delta Q + Q - Q^3=0
$
in $ \mathbb{R}^2. $

\smallskip
In the present paper,
we will consider the asymptotic behavior of excited state solutions $\tilde u$ to problem \eqref{eq1}-\eqref{constraint0}, as $a \to {\bf a}, $ with some ${\bf a}$ to be determined later, which exhibit a higher dimensional concentration phenomenon. According to the analysis in \cite{Zhouyang} and \cite{Luo}, throughout this paper, we make the assumption that
\begin{align} \label{lambda a to bf a}
-\lambda \rightarrow + \infty, \quad \text{as $a \to {\bf a} $}.
\end{align}
Letting
\begin{align*}
\tilde u = \sqrt{\frac{-\lambda} {a} } u,
\end{align*}
then we get that $ u $ satisfies
\begin{align*}
- \frac 1 \lambda \Delta u - \Big[1- \frac 1 \lambda W({\bf y}) \Big] u + | u|^2 u = 0,
\quad \forall\,{\bf y} \in \mathbb{R}^N.
\end{align*}
Taking
\begin{align}
\epsilon^2 = \frac 1{-\lambda},
\label{epsilon}
\end{align}
 we will transform \eqref{eq1}-\eqref{constraint0} to
\begin{align} \label{eq2}
- \epsilon^2 \Delta u +\big[1+ \epsilon^2 W({\bf y}) \big]u - | u|^2 u = 0,
\quad \forall\,{\bf y} \in \mathbb{R}^N,
\end{align}
 and
\begin{align} \label{constraint1}
\int_{{\mathbb R}^N} | u|^2 {\mathrm d} {\bf y}=a \epsilon^2.
\end{align}
The assumption \eqref{lambda a to bf a} is equivalent to the following fact
\begin{align}
\epsilon \to 0, \quad \text{as\,$a \to {\bf a} $}.
\label{epsilonto0}
\end{align}
 The studies on problem \eqref{eq1}-\eqref{constraint0} (or \eqref{eq2}-\eqref{constraint1}), such as concentration, uniqueness and so on, have been the object of investigation of several researchers. We refer the readers to the papers \cite{Guoqing-chunhua, Zhouyang, Guo JMPA, Guo-luo-peng, Luo, Pistoia} and the references therein.

\smallskip
\subsection{ The singular perturbation problem}\label{Section1.2}\

 Here we mention that elliptic problem \eqref{eq2} is closely related to the following singular perturbation problem
\begin{align} \label{schrodinger}
 - \epsilon^2 \Delta u + V({\bf y}) u - | u|^{p-1} u = 0, \quad u>0, \quad u \in H^1(\Omega), \quad \Omega \subseteq {\mathbb R}^N,
\end{align}
where $1<p<2^*-1, $ $\epsilon >0$ is a small parameter, and $V$ (independent of $\epsilon$) is the potential.
Problem \eqref{schrodinger} appears in many applications, for instance in chemotaxis \cite{LinNiTak1988} and biological pattern formation \cite{Gierer-Meinhardt}.
Problem \eqref{schrodinger} with Dirichlet or Neumann boundary conditions were studied
extensively in recent years, involving the necessary and sufficient for the concentration.
See \cite{AMNconjecture1, AMNconjecture2, Ambrosetti, AmbMalSecc2001,
BP, BP1, BP2, CPY2019, DY10,
delPFelmer1996, delPFelmer1997, delPKowWei1,
FloerWein1986,
GuoYang,
WangWeiYang, WeiYang, Wei-xu-yang, SWei-Yang}
for backgrounds and references about problem \eqref{schrodinger}.
More generally, for the complex value cases, the concentration of solution for the singular perturbation problem
\begin{align} \label{schrodingercomplexvalue}
 - \epsilon^2 \Delta u + V({\bf y}) u - | u|^{p-1} u = 0, \quad u\in \mathbb{C}, \quad u \in H^1(\Omega; \mathbb{C}), \quad \Omega \subseteq {\mathbb R}^N,
\end{align}
 was also studied.
 In \cite{machiodi}, F. Mahmoudi, A. Malchiodi and M. Montenegro exhibited the complex value solutions which concentrate along curves and whose phases are highly oscillatory. L. Wang and C. Zhao in \cite{Wangzhao} studied the nonlinear Schr\"{o}dinger problem with electromagnetic potential and showed the existence of solution which concentrates at a closed curve.


\medskip
\section{Main results}\label{Section2}

\smallskip
Inspired by the previous discussions, the main objective of the present paper is to study problem \eqref{eq2}-\eqref{constraint1} in ${\mathbb R}^2$ with the assumption \eqref{lambda a to bf a} $\,(i.e.\, \eqref{epsilonto0}\,)$ from two perspectives:

\smallskip
\noindent$\clubsuit$
{\emph{
We want to investigate some necessary conditions for the existence of a solution with concentration phenomenon directing along a closed curve.
}}

\smallskip
\noindent$\clubsuit$
{\emph{
For problem \eqref{eq2}-\eqref{constraint1} with non-symmetric potential, we will impose some reasonable conditions on the potential and then construct a solution concentrating at a curve by choosing $\lambda$ $($i.e. the parameter $\epsilon$$)$.
}}

\subsection{ Necessary conditions for solutions with concentration}\label{Section2.2}\

In this part, we present some results about the necessary conditions of the concentration
 phenomenon happening for problem \eqref{eq2}-\eqref{constraint1} and the singular perturbation problem \eqref{schrodinger}.

\subsubsection{Necessary conditions for concentration phenomena to the singular perturbation problem}\

For problem \eqref{schrodinger}, if the concentrations take place at some points $x_{\epsilon, j}$, then $\nabla V (x_{\epsilon, j})=0, $ see \cite{Ambrosetti, AmbMalSecc2001, CPY2019, delPFelmer1996, delPFelmer1997, FloerWein1986}. On the other hand, the necessary conditions for the higher dimensional concentration were also studied.
We now recall the following proposition concerning the radially symmetric case.
\\[2mm]
{\bf Proposition A}. (\cite{AMNconjecture1})
{\em Let $\Omega = {\mathbb R}^N$ in \eqref{schrodinger}, and the function $V$ is radial.
Suppose that problem \eqref{schrodinger} has a radial solution which concentrates at the sphere $|{\bf y}|=\bar r$,
then $\bar r$ is a critical point of $M(r)= r^{N-1} V^\sigma (r)$ with $\sigma = \frac{p+1} {p-1} - \frac 12$.
\qed
}


\bigskip
A rather complicated but interesting problem is to determine the necessary conditions for the non-radial concentration phenomena of higher dimension for problem \eqref{schrodinger}.
The necessary condition for the existence of a non-radial solution with concentration on $k$-dimensional submanifold $\Sigma$ was derived formally in Section 8 (Open Problems and Perspectives) in \cite{AMNconjecture1}, i.e.,
\begin{align}
\bigg[\frac{p+1}{p-1}-\frac{N-k}{2}\bigg] \nabla^\perp V=V \mathbf{H}\quad \mbox{on } \Sigma,
\label{stationarygeneral}
\end{align}
where $\mathbf{H}$ is the mean curvature of $\Sigma$
and $\nabla^\perp$ denotes the component of $V$ normal to $\Sigma$.
This was exemplified in \cite{delPKowWei1}, \cite{SWei-Yang} and also \cite{WangWeiYang}.
However, to the best of our knowledge, there is no rigorous proof to \eqref{stationarygeneral} in the available literature.
This will be addressed in Corollary \ref{corollary2.2} for the two dimensional case.

\subsubsection{Necessary conditions for solutions with concentration to problem \eqref{eq2}-\eqref{constraint1}}\

Let's go back to problem \eqref{eq2}-\eqref{constraint1}.
 In the recent work \cite{Zhouyang}, Q. Guo, S. Tian and Y. Zhou considered problem \eqref{eq2}-\eqref{constraint1} in ${\mathbb R}^2$ with $W({\bf y}) = W(|{\bf y}|)$
and proved the necessary conditions for radial solutions concentrating at spheres.
More precisely,
{\em
 if a radial solution $u_\epsilon$ of the problem \eqref{eq2}-\eqref{constraint1}
concentrated at a sphere $r=r_\epsilon $ with $r=|{\bf y}|$, then there must have been
\[
r_\epsilon \to +\infty, \quad \text{and}\quad M_\epsilon' (r_\epsilon) \to 0, \quad \text{as}\,\epsilon \to 0,
\]
where $M_\epsilon (r) = r\big[1+ \epsilon^2 W(r) \big]^\frac 32$.
}
We note, in \cite{Zhouyang} the symmetry reduces the technical difficulties and allows the authors to handle the problem by ODE methods.

\smallskip
A strong motivation of the current work is to give some necessary conditions for solutions with concentration to \eqref{eq2}-\eqref{constraint1}, without imposing any symmetry assumption on the trapping potential $W$.
To avoid complicated calculations, we here only consider the two dimensional case, i.e., $N=2$,
with concentration directed along a smooth simple closed curve $\Gamma_\epsilon\subseteq{\mathbb R}^2$.
More precisely, the following problems under the assumption \eqref{lambda a to bf a} will be addressed:

\smallskip
{\em
$\bullet$\ If $u_\epsilon$ is a non-radial solution to \eqref{eq2}-\eqref{constraint1} with higher dimensional concentration $($see the form in \eqref{u-decomposition}-\eqref{smallperturbation1}$)$ as $a \to {\bf a}$, then what's the value of ${\bf a}$?
What's the limit of $\G_\epsilon$ when $a\to {\bf a}$ $($or $\epsilon \to 0$ $)$?
What necessary conditions shall the potential $W$ and $\G_\epsilon$ satisfy?
}

\medskip
Note that in the present work, the curves $\G_\epsilon$ depend on the parameter $\epsilon$.
For the above purpose, the following are the settings to avoid some tedious computations and make mathematical analysis rigorous.
Here are the assumptions in the first one for the curves, which are the concentration sets of the solutions.

\smallskip\noindent {\bf{(A1).}}
{\it
Assume that $\Gamma_\epsilon$ are simple smooth closed curves in ${\mathbb R}^2$.
Consider the natural parameterization $\gamma_\epsilon(\theta)$ of $\Gamma_\epsilon$ with positive orientation, where $\theta$ denotes the arclength parameter measured from a fixed point of $\Gamma_\epsilon$.
We denote
\begin{align}
\ell = |\Gamma_\epsilon|,
\label{ell-epsilon}
\end{align}
and the curvature of $\G_\epsilon$ by $\kappa(\theta)$.
We also assume that there exists a positive constant $C_0$ independent of $\epsilon$ such that
\begin{align} \label{curvature bounded}
|\kappa(\theta)|, \, |\kappa'(\theta)|\leq C_0,
\quad\forall\, \theta\in [0, \ell).
\end{align}
}

\smallskip
The assumption $-C_0\leq\kappa(\theta)\leq C_0$ in \eqref{curvature bounded} implies that
there exists a small positive number $M_0$ independent of $\epsilon$ (which can be chosen such that $C_0M_0<1$) in such a way that any point ${\bf y} $ near $\Gamma_\epsilon$ in ${\mathbb R}^2 $ can be represented in {\bf{Fermi coordinates}}
\begin{align}\label{fermi}
{\bf y}= \gamma_\epsilon(\theta) +t \nu_\epsilon(\theta), \quad \forall\,  t\in(-M_0, M_0),\ \forall\,\theta\in [0, \ell),
\end{align}
where $ \nu_\epsilon(\theta) $ denotes the unit outer normal to $\Gamma_\epsilon$.
From \eqref{fermi}, we know that the map
${\bf y} \mapsto (t, \theta)$ is a local diffeomorphism (see \cite{Gilbarg}, Sec. 14.6). By a slight abuse use of notation, for ${\bf y}$ in \eqref{fermi} and any function $W$, we denote $W(t, \theta)$ to actually mean $W({\bf y})$.
We now can propose the assumptions for the setting of the potential function $W$.

\smallskip
\noindent {\bf{(A2).}}
{\it
The smooth potential function $W$ is positive on ${\mathbb R}^2$.
Moreover, there exist two positive constants $C$ and $\varrho<1$ independent of $\epsilon$, such that
\begin{align}
\int_0^{\ell}\big[1+ \epsilon^2 W(0, \theta) \big]\,{\mathrm d}\theta \,\leq\, C\epsilon^{\varrho-1},
\label{length-constraint}
\end{align}
\begin{align}
\max\big\{\, \epsilon^2|W_t (t, \theta)|, \ \epsilon^2|W_{tt} (t, \theta)| \, \big\}
\leq
C \big[1+\epsilon^2 W(0, \theta)\big], \quad \forall\, \theta \in [0, \ell),\ \forall\,t\in(-M_0, M_0),
\label{W_t-control}
\end{align}
and
\begin{align} \label{W_theta-control}
& \max\Big\{\,
\epsilon^2|W_{\theta}(0, \theta)|,\
\epsilon^2|W_{\theta\theta}(0, \theta)|,\
\,\epsilon^2|W_{\theta\theta\theta}(0, \theta)|,\
\, \epsilon^2|W_{t\theta}(0, \theta)|,
\nonumber\\[2mm]
& \quad \quad \quad
\,\epsilon^2|W_{tt\theta}(0, \theta)|,\
\,\epsilon^2|W_{t\theta\theta}(0, \theta)|
\,\Big\}
\,\leq\,
C\big[1+\epsilon^2 W(0, \theta)\big],
\quad \forall\, \theta \in [0, \ell).
\end{align}
}

\smallskip
\begin{remark} \label{remark2-1}
{\emph{
Note that we need the assumptions in \eqref{W_t-control} and \eqref{W_theta-control} such that some estimates can be carried out.
The readers can refer to \eqref{beta'}-\eqref{beta-bound}, the proof of Lemma \ref{Lemma4.2} and also the last paragraph in Section \ref{Section3.2}.
On the other hand, we will use \eqref{length-constraint} to deal with the complicated resonance involved in the higher dimensional non-radial concentration for problem \eqref{eq2}-\eqref{constraint1}, which will be clarified in Remark \ref{remark2-7}.
\qed
}}
\end{remark}

\medskip
Our first goal is to find necessary conditions for the existence of a concentrated solution $u_\epsilon$ to problem \eqref{eq2}-\eqref{constraint1},
whose properties can be described in the following way.

\smallskip
{\it
We denote by $U$ the unique positive solution to the problem
\begin{align}
U'' - U + U^3= 0 \quad \text{for }\, x\in {\mathbb R},
\qquad U'(0) =0, \qquad U(\pm \infty) =0,
\label{block}
\end{align}
and also define a positive function $\beta>1$ in the form
\begin{align}
\beta(\theta) \,=\, \big[1+ \epsilon^2 W(0, \theta) \big]^{\frac 12},
\label{beta}
\end{align}
as well as a smooth cut-off function $\chi$ by
\begin{align}
\chi(t)=1\quad\mbox{if}\ |t|<M_1,
\qquad
\chi(t)=0\quad\mbox{if}\ |t|>2M_1:=M_2,
\label{cut-off000000}
\end{align}
where $M_2<M_0$ with $M_0$ given in \eqref{fermi}.
Suppose that
\begin{align}
u_\epsilon({\bf y})={\mathbf u}({\bf y})+\phi({\bf y})
\label{u-decomposition}
\end{align}
is a solution with concentration on $\Gamma_{\epsilon}$ to problem \eqref{eq2}.
In other words, the function ${\mathbf u}$ has the following local form in the Fermi coordinates $(t, \theta)$
\begin{align}
{\mathbf u} (t, \theta)
\,=\,
\chi(t)\beta(\theta)\, U\big( \beta(\theta)t/\epsilon\big).
\label{mathbfu}
\end{align}
Moreover, the requirements for $\phi$ can be listed in the forms
\begin{align}
\Big\| \frac{\partial^J\phi}{\partial t^{J_1}\partial \theta^{J_2}}\Big\|_{L^2\big(\{ |t|< M_2\}\big)}
=
 \frac{O\big(\epsilon^\varrho\big)}{\displaystyle\int_0^{\ell}|\beta(\theta)|^5\,{\mathrm d}\theta}
\Big\| \frac{\partial^J{\mathbf u}}{\partial t^{J_1}\partial \theta^{J_2}}\Big\|_{L^2\big(\{ |t|< M_2\}\big)},
\label{smallperturbation0}
\end{align}
\begin{align}
\Big\|\frac{\partial^j\phi}{\partial \theta^j}\Big\|_{L^2\big(\{ |t|< M_2\}\big)}
=
 \frac{O\big(\epsilon^\varrho\big)}{\displaystyle\int_0^{\ell}|\beta(\theta)|^5\,{\mathrm d}\theta}
\Big\|\frac{\partial^j{\mathbf u}}{\partial \theta^j} \Big\|_{L^2\big(\{ |t|< M_2\}\big)},
\label{smallperturbation1}
\end{align}
where $\varrho<1$ is the positive constant in \eqref{length-constraint}, $j$ is an index with $2\leq j\leq 3$ and $J=(J_1, J_2)$ is an index of two components with $|J|\leq 1$ or $J=(1, 1)$.
}

\medskip
 Now we give our first result of the present paper.

\begin{theorem} \label{necessary condition theorem}
Suppose that: problem \eqref{eq2}-\eqref{constraint1} has a solution $u_\epsilon$ which concentrates at the smooth simple closed curve $\Gamma_\epsilon$,
i.e. the solution $u_\epsilon$ enjoys the decomposition \eqref{u-decomposition} plus the requirements in \eqref{mathbfu}-\eqref{smallperturbation1}.
If \eqref{lambda a to bf a} and the assumptions in {\bf{(A1)}}-{\bf{(A2)}} hold the validity, then we have the following results

\smallskip
\noindent{\bf{(1).}}
There exists a positive constant $\ell_0$ independent of $\epsilon$ such that
\begin{align}\label{ellzero}
 \ell\,\geq\, \ell_0,\qquad \forall\,\epsilon >0.
\end{align}

\smallskip
\noindent{\bf{(2).}}
For the $\bf a$ given in \eqref{lambda a to bf a}, there holds
\begin{align}
 {\bf a} = + \infty.
\end{align}

\smallskip
\noindent{\bf{(3).}}
The following relation holds:
\begin{align}\label{condition1}
 \frac{ 3 \epsilon^2 W_t(0, \theta)}{\, 2\big[1+\epsilon^2W(0, \theta)\big]\,}
 \,+\,
 \kappa (\theta) \,=\, o(1)
 \qquad \mbox{as}\ \epsilon\rightarrow 0,
\end{align}
\qquad\ for all $\theta \in [0, \ell)$.
\qed
\end{theorem}

\smallskip
Note that, by using a similar method as the proof of Theorem \ref{necessary condition theorem},
we can prove the necessary condition in \eqref{stationarygeneral} for the solution concentrating at a closed smooth curve $\G\subset {\mathbb R}^2$ to singular perturbation problem \eqref{schrodinger} in ${\mathbb R}^2$.
This case is much more easier than Theorem \ref{necessary condition theorem}, since the limit set $\Gamma$ is independent of $\epsilon$.
The result will be provided in the following corollary, which is clearly an extension of {\bf Proposition A} given by A. Ambrosetti, A. Malchiodi and W.-M. Ni in \cite{AMNconjecture1}.

\begin{corollary}\label{corollary2.2}
Suppose that problem \eqref{schrodinger} in ${\mathbb R}^2$ has a solution which concentrates at a closed smooth curve $\G$,
then the identity holds
\begin{align*}
 \Big(\frac{p+1} {p-1} - \frac 12\Big) V_t(0, \theta) = - \kappa (\theta)V(0, \theta), \quad \forall\, \theta \in [0, \ell),
\end{align*}
where $\kappa (\theta)$ is the curvature of $\G$, and $(t, \theta)$ are the adapted version of Fermi coordinates for ${\bf y}$ near $\G$.
\qed
\end{corollary}

\smallskip
\subsection{ Existence of solutions with concentration phenomena}\

We now consider the converse of Theorem \ref{necessary condition theorem}: consider problem \eqref{eq2}-\eqref{constraint1} with \eqref{epsilonto0} and construct solutions concentrating at a curve $\G_\epsilon$ which satisfies the necessary conditions in \eqref{condition1}.

\smallskip
 Since the pioneering work by A. Floer and A. Weinstein \cite{FloerWein1986}, to construct concentrated solutions for problem \eqref{eq2}-\eqref{constraint1} or singular perturbation problem \eqref{schrodinger} has been a topic of general interest in the study of singular perturbation theory,
 including cases concentrating on points or $k$-dimensional submanifolds, with $ 1\leq k\leq N-1$.

\smallskip
For the construction of a solution to problem \eqref{schrodinger} concentrating at higher dimensional submanifold of ${\mathbb R}^N$,
we mention that in \cite{AMNconjecture1}, A. Ambrosetti, A. Malchiodi, and W.-M. Ni considered \eqref{schrodinger} in ${\mathbb R}^N$ with radial potential $V$ and showed the existence of higher-dimensional layer solutions.
Furthermore, M. del Pino, M. Kowalczyk, and J. Wei \cite{delPKowWei1} studied \eqref{schrodinger} in ${\mathbb R}^2$ without any symmetry hypothesis on the potential $V$.
By using infinite Lyapunov-Schmidt reduction, a non-radial solution $u_\epsilon$ was constructed, which concentrates at a closed stationary and non-degenerate curve for the weighted functional $\int_\Gamma V^\sigma$ with $\sigma = \frac{p+1} {p-1}- \frac 12$,
provided that $\epsilon$ is small and satisfies a gap condition due to the resonance character of the problem.
See also \cite{WangWeiYang} for solutions concentrating on hypersurfaces.

\smallskip
Here we will mention some works to the constructions of solutions to problem \eqref{eq2}-\eqref{constraint1}.
In \cite{Luo}, P. Luo, S. Peng, J. Wei and S. Yan considered \eqref{eq2}-\eqref{constraint1} with a class of degenerate trapping potential with non-isolated critical points,
and showed that the existence and local uniqueness of excited states by accurately analyzing the location of the concentrated points and the Lagrange multiplier $\lambda$.
Recently, in \cite{Zhouyang}, Q. Guo, S. Tian and Y. Zhou considered the case that $W$ is radially symmetric and studied concentration phenomena of solutions to problem \eqref{eq2}-\eqref{constraint1}.
By using a finite dimensional Lyapunov-Schmidt reduction method and blow-up analysis based on Pohozaev identity, they showed the existence of radial solutions concentrating on a circle with radius $r_\epsilon \to +\infty$ as $\epsilon \to 0$.
Here we point out that the symmetry condition of $W$ in \cite{Zhouyang} reduces the problem to just adjusting one parameter $r_\epsilon$ representing the location of the concentration curve and a finite-dimensional Lyapunov-Schmidt reduction was used.

\smallskip
It's natural to ask whether we could remove the symmetry condition of $W$ imposed in \cite{Zhouyang}.
When $W$ is non-radial, we are unable to transform the problem into a problem of ODE, and whence the current problem is fundamentally different from that in \cite{Zhouyang} involving radial potential.
On the other hand, we know that, for the semi-linear problem with non-symmetric potential, the gluing methods provided in \cite{delPKowWei1} make it possible to construct a solution concentrating at a curve. In general, for semi-linear elliptic problem like \eqref{eq2}-\eqref{constraint1} or \eqref{schrodinger}, the situation of adjusting a finite number of points
({\em finite dimensional Lyapunov-Schmidt reduction}), and that of adjusting a higher dimensional object such as a geodesic in suitable metrics as limiting concentrating sets ({\em infinite Lyapunov-Schmidt reduction}) will be treated.
However, in contrast to the situations of bounded concentration sets in \cite{delPKowWei1} (as well as the corresponding references \cite{machiodi, WangWeiYang, Wangzhao, WeiYang, Wei-xu-yang, SWei-Yang}), the concentration directed along curves $\G_\epsilon$ with non-uniformly bounded length will lead to some atypical challenges, see Remark \ref{remark2-7}.

\smallskip
In the present paper, inspired by \cite{delPKowWei1}, we will consider problem \eqref{eq2}-\eqref{constraint1} with a non-symmetric potential $W$ and construct a solution that concentrates at a curve $\Gamma_\epsilon$, to make an extension of the results in \cite{Zhouyang}.
In addition to the assumptions in {\bf{(A1)}}-{\bf{(A2)}}, by recalling the Fermi coordinates in \eqref{fermi},
the following assumptions in {\bf{(B1)}}-{\bf{(B3)}} will be imposed.
We note that these conditions will play a crucial role in the reduction process of constructing solutions concentrated on the curve $\G_\epsilon$.

\smallskip
The following definitions are adapted from \cite{delPKowWei1}.
Recall the Fermi coordinates defined as in \eqref{fermi}. For $\epsilon$ small, any curve $\G^g_\epsilon$ close to $\G_\epsilon$ can be parameterized as
\begin{align*}
\gamma_\epsilon^g (\theta) = \gamma_\epsilon(\theta) + g (\theta) \nu_\epsilon (\theta), \quad \theta \in [0, \ell).
\end{align*}
Following the analysis and computations as in \cite{delPKowWei1}, we can easily get
\begin{align*}
J_\epsilon (g) & = \int_{\G_\epsilon^g} \big[1+ \epsilon^2 W (g, \theta) \big]^{\frac 32} {\mathrm d} \sigma
\\[2mm]
&= \int_0^{\ell} \Big[1+ \epsilon^2 W (g, \theta) \Big]^{3/2}\, \Big[ (1+\kappa g)^2 + |g'|^2\Big]^{1/2} {\mathrm d} \theta.
\end{align*}
In the present paper, for any fixed small $\epsilon$, we just make the {\emph {stationary assumption} for the curve $\Gamma_\epsilon$, which means that the first variation of the functional $J_\epsilon (g)$ at $g=0$ equals to $0$.
In other words, for any smooth, $\ell$-periodic function $h(\theta)$, there must be
\[
\displaystyle\int_0^{\ell} \Bigg\{ \frac 3 2\big[1+ \epsilon^2 W(0, \theta) \big]^{1/2}\epsilon^2 W_t(0, \theta)
\ +\
\big[1+ \epsilon^2 W(0, \theta) \big]^{\frac 3 2} \kappa \Bigg\} h{\mathrm d} \theta =0,
\]
which is equivalent to the following condition
\begin{align} \label{stationary assumption}
\kappa(\theta) \big[1+ \epsilon^2 W(0, \theta) \big] + \frac 3 2 \epsilon^2 W_t(0, \theta) = 0,
\quad\forall\, \theta\in [0, \ell).
\end{align}
By setting
\begin{align}
\mathbb{F}_\epsilon(h)
=
\Big(\big[1+ \epsilon^2 W(0, \theta) \big]^{\frac32}h' \Big)'
\ -\
\Bigg\{ \Big( \big[1+ \epsilon^2 W(t, \theta) \big]^{\frac32} \Big)_{tt}\Big|_{t=0}
\,-\, 2 \big[1+ \epsilon^2 W(0, \theta) \big]^{\frac32} \kappa^2 \Bigg\} h,
\label{mathbbF}
\end{align}
(see \eqref{weighted-reduced tilde f} for the expression of $\mathbb{F}_\epsilon$ in other coordinates),
we say that a stationary curve $\Gamma_\epsilon$ is {\emph {non-degenerate}} if
the following problem
\begin{align}
\mathbb{F}_\epsilon(h)=0,
\quad
h(0) = h(\ell), \quad h'(0) = h'(\ell),
\label{nondeneracy}
\end{align}
only has the trivial solution $h\equiv 0$.
Inspired by this, to obtain a precise {\em a priori} estimate for the corresponding in-homogeneous problem in the reduction procedure, we propose the following assumptions

\smallskip\noindent {\bf{(B1).}}
{\it{
For any small $\epsilon$, $\Gamma_\epsilon$ is a stationary curve.
Moreover, we assume that
\smallskip
\begin{align}\label{resonace-f}
{\mathrm d}_\epsilon \geq \hat{C}_2 \epsilon^{\alpha}
\quad \text{with}\ {\mathrm d}_\epsilon = \min_{j=1,2,\cdots} |\Lambda_{\epsilon, j}|,
\end{align}
where $\alpha\in(0, 1)$ is a constant with constraint in \eqref{constraint-alpha-varrho}, and $\{\Lambda_{\epsilon, j}\} _{j=1}^\infty$ are the real eigenvalues of the eigenvalue problem
\begin{align}
\mathbb{F}_\epsilon(h) =\beta\Lambda_{\epsilon, j}h,
\qquad
h(0) = h(\ell), \qquad h'(0) = h'(\ell),
\label{mathbbF-boundary}
\end{align}
where $\beta$ has been given in \eqref{beta}.
}}

\medskip
On the other hand, we will also encounter delicate resonance phenomena in the construction of solutions, which were analyzed in \cite{delPKowWei1}.
For the similar purpose, the following gap condition for $\epsilon$ small will be useful.

\smallskip\noindent {\bf{(B2).}}
{\it{
For given $c>0$ small, there holds
\begin{align} \label{gap condition1}
\Big| j^2 \frac{\epsilon^2}{{\tilde\ell}^2} - \lambda_* \Big| \geq \frac{c \epsilon}{{\tilde\ell}}, \quad \forall\,j \in \mathbb N,
\end{align}
where by recalling the unique positive eigenvalue $\lambda_0$ of the eigenvalue problem \eqref{eigenvalue}, we have used the conventions
\begin{align}
\lambda_* = \frac{\lambda_0} {4\pi^2}
\qquad \text{and}\qquad
{\tilde\ell} = \int_0^{\ell} \big[1+ \epsilon^2 W(0, \theta)\big]^{\frac{1}{2}} {\mathrm d} \theta.
\label{lambda*}
\end{align}
}}

Furthermore, in order to guarantee the successful execution of the reduction process, we require the following condition (The readers can refer to Remark \ref{remark2-7} for explanations).

\smallskip \noindent {\bf{(B3).}}
{\it{
For $\varrho, \alpha\in (0, 1)$ given in \eqref{length-constraint} and \eqref{resonace-f} respectively, they will satisfy
\begin{align}
 \varrho> \frac 13+ \frac43 \alpha.
\label{constraint-alpha-varrho}
\end{align}
}
}

\begin{remark}\label{remark2-4}
{\it
Note that ${\tilde\ell}$ has a uniformly positive lower bound due to \eqref{ellzero}, i.e.
\begin{align}
{\tilde\ell}\geq C_0>0
\label{tildeell-lowerbound}
\end{align}
where $C_0$ is a universal constant independent of $\epsilon$.
On the other hand, the assumption in \eqref{length-constraint} implies that
\begin{align}
{\tilde\ell}
\leq C\epsilon^{\varrho-1},
\label{boundneess of length}
\end{align}
where $0<\varrho<1$ has been given in \eqref{length-constraint}.
By combining \eqref{boundneess of length} and \eqref{gap condition1}, we obtain
\begin{align} \label{gap condition222222}
\Big| j^2 \frac{\epsilon^2}{{\tilde\ell}^2} - \lambda_* \Big| \geq C \epsilon^{2-\varrho}, \quad \forall\,j \in \mathbb N.
\end{align}

We shall mention that the validity of \eqref{gap condition1} will be the case whenever ${\epsilon}/{{\tilde\ell}}$ is small and away from the critical numbers ${\sqrt{\lambda_*}}/j$ with $j \in {\mathbb N}$, i.e.,
for arbitrarily small fixed $2c<\sqrt{\lambda_*}$,
\begin{align}
\label{gap condition1.22}
\frac{\epsilon}{{\tilde\ell}} \notin \Bigg[ \frac{\sqrt{\lambda_*}} {j}-\frac{c} {j^2}, \ \frac{\sqrt{\lambda_*}} {j}+\frac{c} {j^2} \Bigg],
\quad \forall\,j \in {\mathbb N}.
\end{align}
From the mass constraint \eqref{constraint1}, there is a relationship between
$\epsilon$ and $a$
given by \eqref{realtion for a eps}, i.e.
\begin{align} \label{realtionship between epsilon and a}
\frac\epsilon{\tilde \ell} = \frac{\varrho_0}{a} \left(1+ O\left(\epsilon^{2-\alpha}\right) \right),
\end{align}
where $\alpha\in (0,1)$ is given in \eqref{resonace-f} and $\varrho_0$ is the constant defined in \eqref{varrho0}.
Under the relation \eqref{realtionship between epsilon and a}, we make the gap condition for $a$
\begin{align}
\label{condition for a}
\frac{\varrho_0}{a}\notin \Bigg[ \frac{\sqrt{\lambda_*}}{j}-\frac{2c}{j^{2}}, \ \frac{\sqrt{\lambda_*}}{j}+\frac{2c} {j^2} \Bigg],
\quad \forall\,j \in \mathbb N.
\end{align}
where $c$ is given by \eqref{gap condition1.22}.
In fact, \eqref{condition for a} is a sufficient condition for the validity of \eqref{gap condition1}, see the analysis in Section \ref{Section9}.
\qed
}
\end{remark}

\smallskip
The following is the theorem for the existence of solutions with concentration.

\begin{theorem}
\label{theorem 1.1}
Assume that $N=2$, the smooth simple closed curves $\Gamma_\epsilon$ and the potential $W$ satisfy the assumptions in
{\bf{(A1)}}-{\bf{(A2)}} and {\bf{(B1)}}-{\bf{(B3)}}.
For problem \eqref{eq2}-\eqref{constraint1} with relation between $a$ and $\epsilon $ in \eqref{realtionship between epsilon and a}, there exists a positive solution $u_{\epsilon}$ with the following properties.

\smallskip
\noindent $\clubsuit$
For ${\bf y}$ given by \eqref{fermi} in the neighbourhood of $\Gamma_\epsilon$,
$u_{\epsilon}$ has the local profile
\begin{equation}
\label{taketheform}
 u_{\epsilon}({\bf y})\,=\, \big[1+ \epsilon^2 W(0, \theta) \big]^{\frac 12}\, U\Bigg(\big[1+ \epsilon^2 W(0, \theta) \big]^{\frac 12}\, \frac{t} \epsilon\Bigg)\big(1+o(1)\big),
\end{equation}
where $U$ is the function given in \eqref{block}.

\smallskip
\noindent$\clubsuit$
The solution $u_{\epsilon}$ satisfies globally
$$u_{\epsilon}({\bf y})\,\leq\, C\exp \Big(-\frac{c_0}{\epsilon} {\mathrm{dist}} ({\bf y}, \Gamma_\epsilon)\Big),$$
where $c_0$ and $C$ are two positive constants independent of $\epsilon$.
\qed
\end{theorem}

For the convenience of readers, we go back to problem (\ref{eq1})-\eqref{constraint0} and provide another version of Theorem \ref{theorem 1.1}.
By slight abuse of notation, we use $\G^\lambda$ to denote the curve $\G_\epsilon$ by the relation \eqref{epsilon}.

\begin{theorem}
\label{theorem 1.11}
Assume that $N=2$, the smooth simple closed curves $\Gamma^\lambda$ and the potential function $W$ satisfy {\bf{(A1)}}-{\bf{(A2)}} and {\bf{(B1)}}-{\bf{(B3)}}.
There exists a large constant $a^*>0$ such that for all $a > a^*$ satisfying gap condition \eqref{condition for a},
the parameter $\lambda$ in \eqref{eq1}-\eqref{constraint0} can be chosen by the relations \eqref{epsilon} and \eqref{realtionship between epsilon and a}.
Moreover, problem \eqref{eq1}-\eqref{constraint0} has a positive solution $\tilde{u}_\lambda$ with the following properties

\smallskip
\noindent $\clubsuit$
For ${\bf y}$ given by \eqref{fermi} in the neighbourhood of $\Gamma^\lambda $,
$\tilde{u}_\lambda$ takes the form
\begin{equation}
\label{taketheform1}
\tilde{u}_{\lambda}({\bf y})
\,=\,
\sqrt{\frac{-\lambda} {a} }\, \Bigg[1 - \frac 1{\lambda} W(0, \theta) \Bigg]^{1/2}
\, U\Bigg(\Bigg[1 - \frac 1{\lambda} W(0, \theta) \Bigg]^{1/2}\, \frac{t}{\sqrt{-1/\lambda}} \Bigg)\big(1+o(1)\big),
\end{equation}
where $U$ is the function given in \eqref{block}.

\smallskip
\noindent $\clubsuit$
For a number $c_0>0, \,\tilde{u}_{\lambda}$\,satisfies globally
$$
\tilde{u}_{\lambda}({\bf y})\,\leq\, C\sqrt{\frac{-\lambda} a}\, \exp \Big(- \frac{c_0} {\sqrt{-\lambda}}{\mathrm{dist}} ({\bf y}, \Gamma_\lambda)\Big),
$$
where $c_0$ and $C$ are two positive constants independent of $\lambda$.
\qed
\end{theorem}

\medskip
\subsection{ Outline of the strategy and further comments}\label{Section2.3}\

The substitution
\begin{equation}\label{definitionofy}
 u ({\bf y}) \,=\, v(y)\quad\mbox{with}\quad y= {\bf y}/\epsilon,
\end{equation}
in problem \eqref{eq2}-\eqref{constraint1} will imply that
\begin{align} \label{eq3}
 -\Delta v + \big[1+ \epsilon^2 W (\epsilon y) \big] v - |v|^2 v \,=\, 0,
\quad \forall\, y\in {\mathbb R}^2,
\end{align}
and
\begin{align} \label{constraint2}
 \int_{{\mathbb R}^2} |v|^2 {\mathrm d} y \,=\, a.
\end{align}

\smallskip
\noindent$\clubsuit$\
In Sections \ref{Section3.1}-\ref{Section3.2}, we will introduce the Fermi coordinates and write down equation \eqref{eq3} in local forms,
see \eqref{eq4} as well as \eqref{problem-local4}.
After that, some preliminary facts on the profile function $U$, and some conventions will be given in Section \ref{Section3.3}.

\smallskip
\noindent$\clubsuit$\
Section \ref{Section4} is devoted to the proof of Theorem \ref{necessary condition theorem},
which is constituted by a Pohozaev type identity in Lemma \ref{lemmaPohozaev} and delicate analysis to estimate all terms of the identity.

\smallskip
\noindent$\clubsuit$\
On the other hand, the basic philosophy of the construction procedure of a solution to problem \eqref{eq3}-\eqref{constraint2} and the proof of
Theorem \ref{theorem 1.1} will be organized in the following way: we first construct a solution with concentration to equation \eqref{eq3} by refining the approach presented in \cite{delPKowWei1},
and then choose suitable $\epsilon$ to fulfill the constraint \eqref{constraint2}.
More words for the explanations of details will be provided in the following.
\begin{itemize}

\item[(1).]
In Section \ref{Section5}, we first make a try to construct a local approximation solution to \eqref{eq3}.
In order to improve the approximation, suitable correction terms will be added and a good approximate solution $w_2$ will be given in \eqref{newthirdapproximatesolution},
which is involving two free parameters $f$ and $e$.

\smallskip
 \item[(2).]
The inner-outer gluing scheme from \cite{delPKowWei1} will be set up in Section \ref{Section6}.
By means of this useful method, we can deduce the projected version of \eqref{eq3}, see problem \eqref{project 6.13}-\eqref{projectedproblem3}, and then give the existence together with Lipschitz-continuity of its solutions in Proposition \ref{proposition6.5}.

\smallskip
\item[(3).]
In order to get a real solution to \eqref{eq3}, the well-known infinite-dimensional reduction method will be needed in Section \ref{Section7} to derive a nonlinear coupled system of second order differential equations of $f$ and $e$, see \eqref{last reduced equation on U'}-\eqref{last reduced equation on mathcal Z}.
The key step is to find the parameter pair $(f, e)$ by solving that system, which will be done in Section \ref{Section8}.

\smallskip
\item[(4).]
In order to fulfill the mass constraint \eqref{constraint2}, in Section \ref{Section9} we shall substitute the solution constructed in the above steps into it and then find the relation between the parameters $a$ and $\epsilon$, in such a way that $\epsilon$ can be settle down for fixed $a$.
\end{itemize}
We shall mention that we may use the assumptions in {\bf{(A1)}}-{\bf{(A2)}} and {\bf{(B1)}}-{\bf{(B3)}}
without any further notification in Sections \ref{Section5} through Section \ref{Section9} as well as Appendix \ref{appendixa1}.

\medskip
To conclude this section, here are more words to explain Theorem \ref{theorem 1.1} and Theorem \ref{theorem 1.11} and discuss the difficulties encountered in the procedure of constructing solutions with higher dimensional concentration.

\begin{remark} \label{remark2-7}
{\it
In local coordinates $(s, z)$ given in \eqref{eqszfinal},  problem \eqref{eq3} can be expressed as in \eqref{eq4}, and its linearization at the profile $U(s)$ is
\begin{align}
\phi_{ss} + \phi_{zz} - \big[1+ \epsilon^2 W (\epsilon s, \epsilon z) \big] \phi + 3U^2(s)\phi
+\big(\epsilon \kappa -\epsilon^2 \kappa^2 s \big) \phi_{s}+ B_0(\phi)=0,
\label{linearproblem}
\end{align}
where $s\in (-M_0/\epsilon,\, M_0/\epsilon),\ z\in \big[0, \ell/\epsilon\big)$.
As done in \cite{delPKowWei1}, if we neglect the role of $\epsilon^2 W$ and the terms of higher orders in $\epsilon$, the linearization of problem \eqref{eq3} at the profile $U(s)$ has the local form
$$
\phi_{zz}+\phi_{ss}-\phi+3U^2(s)\phi=0, \quad s\in {\mathbb R},\ z\in \big[0, \ell/\epsilon\big).
$$
There exist two types of eigenvalues
\begin{align}
\Big(\frac{2j\pi\epsilon}{\ell}\Big)^2,
\qquad
\lambda_0\,-\,\Big(\frac{2j\pi\epsilon}{\ell}\Big)^2,
\qquad
j\in{\mathbb N}^+,
\label{eigenvalue-problem}
\end{align}
correspond respectively to the eigenfunctions
\begin{align}
\Upsilon_{1,j}=U'(s)\Bigg[\mathbf{c}_{1j}\sin\Big(\frac{2j\pi\epsilon z}{\ell}\Big)  \,+\, \mathbf{c}_{2j}\cos\Big(\frac{2j\pi\epsilon z}{\ell}\Big)\Bigg],
\\[2mm]
\Upsilon_{2,j}={\mathcal Z}(s)\Bigg[\mathbf{c}_{3j}\sin\Big(\frac{2j\pi\epsilon z}{\ell}\Big)  \,+\, \mathbf{c}_{4j}\cos\Big(\frac{2j\pi\epsilon z}{\ell}\Big)\Bigg],
\label{eigenfunction-problem}
\end{align}
where $\lambda_0$ is the unique positive eigenvalue with corresponding eigenfunction ${\mathcal Z}$ to the eigenvalue problem \eqref{eigenvalue},
 and $\mathbf{c}_{1j},\, \mathbf{c}_{2j},\, \mathbf{c}_{3j},\, \mathbf{c}_{4j}$ are constants independent of $\epsilon$.
The effects of $\Upsilon_{1,j}$ and $\Upsilon_{2,j}$ will lead to the resonance phenomena.
In \cite{delPKowWei1}, in order to deal with the resonance, they introduced two free parameters $f$ and $e$ in the Lyapunov-Schmidt reduction procedure.
Of course, the reduced system of nonlinear differential equations involving the parameters $f$ and $e$ will inherit the resonance (The readers can refer to \eqref{last reduced equation on U'}-\eqref{last reduced equation on mathcal Z}, as well as \eqref{problem-auxil}-\eqref{problem-auxil-boundary} and \eqref{reduced for e2 new}),
which is solvable under non-degeneracy condition and the gap condition similar as \eqref{nondeneracy} and \eqref{gap condition222222}.

\smallskip
However, the length $\ell$ of the curve may depend on $\epsilon$ and the term $\epsilon^2 W(0,\theta)$ appears in the operator ${\mathbb F}$ in \eqref{mathbbF}.
These will lead to more complicated situation: the eigenvalues $\Lambda_{\epsilon, j}$ of problem \eqref{mathbbF-boundary} depend on $\epsilon$ and may approach zero
(As a result, the linear problem \eqref{problem-auxil}-\eqref{problem-auxil-boundary} may have resonance).
In order to get a good {\em a priori} estimate for $f$ as in \eqref{reduced mathfrak f boundarycondition} so that the reduction procedure
can be carried out in Section \ref{Section8}, we impose further the assumption \eqref{resonace-f} in {\bf{(B1)}} to handle the resonance caused by the effect $\Upsilon_{1,j}$.

\smallskip
Moreover, we recall that the gap condition in \eqref{gap condition1} relies heavily on the weighted length ${\tilde\ell}$, which is also depends on $\epsilon$.
Whence, the gap condition \eqref{gap condition1} in {\bf{(B2)}} plus the constraint \eqref{length-constraint} in {\bf{(A2)}} will be used to deal with the resonance given by $\Upsilon_{2,j}$.
The readers can refer to Proposition \ref{linear theory for e}.

\smallskip
Therefore, we will encounter more delicate resonance phenomena than that in \cite{delPKowWei1},
and we also impose the constraint {\bf{(B3)}} to balance the competition between the above two resonance phenomena
in such a way that we can save the procedure of construction of approximate solution, which is enough for the resolution theory involving gluing procedure and the reduction method.

\smallskip
Based on the above reasons, we will introduce new norms $\|\cdot\|_*$ and $\|\cdot\|_{**}$ in \eqref{constraint-f}-\eqref{constraint-e} to do estimates for $f$ and $e$.
Besides, in Section \ref{Section6}, we also define a weighted Sobolev norm $\|\cdot\|_{H^2_*(\mathfrak S)} $ in \eqref{norm h2 mathfrak S} to obtain good estimates for the perturbation function.
We shall mention that the adoption of these new norms significantly contributes to the success of the reduction process.
\qed
}
\end{remark}

\medskip
\begin{remark}\label{remark2-8}
{\it
Since the length $\ell$ of curve $\G_\epsilon$ is not uniformly bounded, it is difficult for us to make a prior estimates to functions involving $\theta$ with $\theta\in[0, \ell)$. To overcome this difficulty, we perform the following Liouville transformation
\begin{align*}
 \tau \,=\, \frac{2\pi}{\tilde\ell} \int_0^{\theta} \beta (\xi) \,{\mathrm d} \xi
\qquad\mbox{with}\quad
 {\tilde\ell} \,=\, \int_0^{\ell} \beta (\xi) \,{\mathrm d} \xi,
\end{align*}
and consequently get the bijective map
\begin{align*}
\theta \in [0, \ell)\mapsto \tau \in [0, 2\pi),
\end{align*}
see \eqref{ztau}.
We obtain the expression of the equation \eqref{eq3} in the new coordinates $(x, \tau)$ with $x$ in \eqref{definition of x 2.14}, and further derive the expression of the error function $S(w_2)$, seeing \eqref{sw2new}, where $w_2$ is given in \eqref{newthirdapproximatesolution}.
By projecting the error function onto $\mathcal Z$ and $U$ respectively, we obtain the reduced equations \eqref{last reduced equation on U'}-\eqref{last reduced equation on mathcal Z} on the coordinate $\tau$. To ensure the solvability of reduced equations and successfully implement the reduction process, we need the conditions {\bf{(B1)}}-{\bf{(B3)}}.
 For a detailed analysis of how conditions {\bf{(B1)}}-{\bf{(B3)}} affect the solvability of the reduced equations, please refer to Remark \ref{remark2-7}.
\qed
}
\end{remark}

\begin{remark}\label{remark2-9}
{\it
 In Theorem \ref{theorem 1.1}, we construct non-radial high-dimensional concentrated solutions to problem \eqref{eq2}-\eqref{constraint1}, which pioneers the application of infinite-dimensional Lyapunov-Schmidt reduction method to investigate nonlinear elliptic problems with mass constraint.
 Moreover, we clarify the effect of the length of the concentrating curve in the higher dimensional concentration phenomenon of problem \eqref{eq2}-\eqref{constraint1}, and construct solutions concentrating at $\G_\epsilon$. The results provided in Theorem \ref{theorem 1.1} and Theorem \ref{theorem 1.11} differ significantly from those in \cite{delPKowWei1} and its references \cite{machiodi, WangWeiYang, Wangzhao, WeiYang, Wei-xu-yang, SWei-Yang}. Furthermore, the results provided in this paper significantly complement the application scope of the infinite-dimensional Lyapunov-Schmidt reduction theory introduced in \cite{delPKowWei1}.
\qed
}
\end{remark}


\medskip
\section{Some preliminaries and preparing works} \label{Section3}

\subsection{ The local forms of the equations in Fermi coordinates}\label{Section3.1}\

In the Fermi coordinates $(t, \theta)$ given in \eqref{fermi}, the Riemannian metric of ${\mathbb R}^2$ can be locally written as
\begin{equation*}
\mathfrak{g} \,=\,{\mathrm d}t\otimes{\mathrm d}t \ +\ \big(1 + \kappa(\theta) t\big)^2 {\mathrm d}\theta\otimes{\mathrm d}\theta.
\end{equation*}
The Laplace-Beltrami operator has the form
\begin{align}
\Delta_\mathfrak{g}
& \,=\, \frac{1}{1+\kappa(\theta) t} \frac{\partial}{\partial t} \Bigg[ \big(1+\kappa(\theta) t\big) \frac{\partial}{\partial t} \Bigg]
\ +\
\frac{1}{1+\kappa(\theta) t}\frac{\partial}{\partial \theta} \Bigg[ \frac{1}{1+\kappa(\theta) t} \frac{\partial}{\partial \theta} \Bigg]
\nonumber\\[2mm]
& \,=\, \frac{\partial^2}{\partial t^2}+ \frac{\kappa(\theta)}{1+\kappa(\theta) t} \frac{\partial }{\partial t}
\ +\
\frac{1}{\big(1+\kappa(\theta) t\big)^2 }\frac{\partial^2}{\partial \theta^2}
\ -\
\frac{ t \kappa'(\theta)}{\big(1+ \kappa(\theta) t\big)^3 }\frac{\partial}{\partial \theta}.
\label{Laplacian}
\end{align}
It is trivial that equation \eqref{eq2} takes the local form
\begin{align}
& \frac{\epsilon^2}{1+\kappa(\theta) t} \frac{\partial}{\partial t} \Bigg[ \big(1+\kappa(\theta) t\big) \frac{\partial}{\partial t} u\Bigg]
\ +\
\frac{\epsilon^2}{1+\kappa(\theta) t}\frac{\partial}{\partial \theta} \Bigg[ \frac{1}{1+\kappa(\theta) t} \frac{\partial}{\partial \theta} u\Bigg]
\nonumber\\[2mm]
&\ -\
\big[1+ \epsilon^2 W (t, \theta) \big] u
\ +\
|u|^2 u
 =0.
\label{eq3.1-0}
\end{align}

\medskip
 In the sequel, we will denote
\begin{equation}\label{eqDepsilon}
\tilde{\Gamma}_\epsilon \,=\, \epsilon^{-1}\Gamma_\epsilon,
\end{equation}
and set the natural stretched coordinates in the neighbourhood of the curve $\tilde{\Gamma}_\epsilon$,
\begin{equation}\label{eqszfinal}
(s, z) \,=\,
\epsilon^{-1} \big(t,\, \theta\big),
\end{equation}
which are defined for
\begin{equation}\label{eqFermi-stretched}
s\in \Big(- \frac{ M_0}{\epsilon}, \frac{ M_0}{\epsilon}\Big), \ \ \ z\in
\Big[0, \frac{\ell}{\epsilon} \Big).
\end{equation}
Hence, equation \eqref{eq3} has the local form
\begin{align} \label{eq3.1}
& \frac{1}{1+\epsilon \kappa s} \frac{\partial}{\partial s} \Big[ \big(1+\epsilon \kappa s\big) \frac{\partial}{\partial s} v \Big]
 +
 \frac{1}{1+\epsilon \kappa s}\frac{\partial}{\partial z} \Big[ \frac{1}{1+\epsilon \kappa s} \frac{\partial}{\partial z} v \Big]
 -
 \big[1+ \epsilon^2 W (\epsilon s, \epsilon z) \big] v + |v|^2 v
 =0,
\end{align}
which can also be written in the following way
\begin{align}
v_{ss} + v_{zz} + B_1 (v) - \big[1+ \epsilon^2 W (\epsilon s, \epsilon z) \big] v + |v|^2 v \,=\, 0,
\label{equationlocal}
\end{align}
 where
\begin{align}
B_1(v)
& \,=\,
\frac{\epsilon \kappa (\epsilon z) }{1+\epsilon \kappa (\epsilon z) s} v_{s}
+\Big[ \frac{1}{(1+\epsilon \kappa(\epsilon z) s)^2 } -1 \Big]v_{zz}
+\frac{\epsilon^2 s \kappa' (\epsilon z) }{(1+\epsilon \kappa (\epsilon z) s)^3 }v_z.
\label{B1}
\end{align}

\medskip
In the neighborhood of $\tilde \Gamma_\epsilon$, by taking the Taylor expansion,
we have
\begin{align} \label{expansion of W}
1+ \epsilon^2 W (\epsilon s, \epsilon z)
 \,=\, & 1+ \epsilon^2 W (0, \epsilon z)+ \epsilon^3s W_t(0, \epsilon z) + \frac 12\epsilon^4s^2 W_{tt}(0, \epsilon z)
\nonumber\\[2mm]
&+ \frac 1 6 \epsilon^5s^3 W_{ttt}(0, \epsilon z)
+\frac{1}{24} \epsilon^6s^4 W_{tttt}\big(\hbar(\epsilon z)\epsilon s, \, \epsilon z\big),
\end{align}
where $\hbar(\theta)$ is a smooth function with $|\hbar(\theta)|\leq 1, \, \forall \, \theta\in [0, \ell)$.
On the other hand, it is trivial that
\begin{align*}
B_1(v) & \,=\, \big(\epsilon \kappa -\epsilon^2 \kappa^2 s \big) v_{s} + B_0(v),
\end{align*}
where
\begin{align}
\label{B0v}
B_0 (v) \,=\, \epsilon^3 s^2 \kappa^3(\epsilon z) a_3(\epsilon s, \epsilon z)v_{s}
\,+\,
\epsilon s \kappa a_2 v_{zz}
\,+\,
\epsilon^2s \kappa'(\epsilon z) a_1(\epsilon s, \epsilon z) v_z,
\end{align}
with
\begin{align}
a_3(t, \theta)=\frac{1}{1+t \kappa (\theta) },
\label{a1a2a3-1}
\qquad
a_2(t, \theta)=\frac{-2-t \kappa (\theta)}{(1+t \kappa (\theta) )^2 },
\qquad
a_1(t, \theta)&=\frac{1}{(1+t \kappa (\theta) )^3 }.
\end{align}
The assumptions in \eqref{curvature bounded} imply that
$a_1, a_2$ and $a_3$ are three smooth bounded functions on the following domain
$$
\Big\{\, (t, \theta)\, :\, |t|<M_0, \quad \theta\in [0, \ell)\, \Big\}
$$
provided that the constant $M_0$ in \eqref{fermi} is chosen sufficiently small.

By using the above calculations, we derive from \eqref{equationlocal} to get
\begin{align} \label{eq4}
{\mathcal D}(v):=\,&v_{ss} + v_{zz} - \big[1+ \epsilon^2 W (0, \epsilon z) \big] v + v^3
+\big(\epsilon \kappa -\epsilon^2 \kappa^2 s \big) v_{s}
+ B_0(v)
- \epsilon^3s W_t(0, \epsilon z) v
\nonumber\\[2mm]
 & - \frac 12 \epsilon^4s^2 W_{tt}(0, \epsilon z) v- \frac 16 \epsilon^5s^3 W_{ttt}(0, \epsilon z)v
 -\frac{1}{24} \epsilon^6s^4 W_{tttt}\big(\hbar(\epsilon z)\epsilon s, \, \epsilon z\big)v \,=\, 0,
\end{align}
where $\epsilon s\in(-M_0, M_0),\, \epsilon z\in [0, \ell)$.

\subsection{ Setting up near the curves}\label{Section3.2}\

 To prove Theorem \ref{theorem 1.1}, we need to reformulate problem \eqref{eq3}  in suitable coordinates. In this section, we will  do this work  step by step.

\medskip
\noindent
{\bf{Step 1.}}
By recalling  the function $\beta$ in \eqref{beta},
we define a function ${\tilde v}({\tilde s}, z) $ by the relation
\begin{align} \label{coordinate-tildes}
v(s, z) \,=\, \beta(\epsilon z) {\tilde v}({\tilde s}, z)
\qquad \text{with}\quad
{\tilde s} \,=\, \beta(\epsilon z)s.
\end{align}

We will express equation (\ref{eq4}) in terms of new coordinates $({\tilde s}, z)$ by the following preparations.
It's easy to get that
\begin{align*}
\p_z {\tilde s}
 \,=\,
 \epsilon\beta's \,=\, \epsilon\frac{\beta'}{\beta}{\tilde s},
\qquad
\p_{zz} {\tilde s} \,=\, \epsilon^2\beta''s
\,=\,
\epsilon^2\frac{\beta''}{\beta}{\tilde s}.
\end{align*}
After that, there hold
\begin{align*}
v_{s}& \,=\, \beta^2{\tilde v}_{\tilde s},
\qquad\qquad
v_{ss} \,=\, \beta^3{\tilde v}_{{\tilde s}{\tilde s}},
\qquad\qquad
v_z \,=\, \epsilon\beta' {\tilde v} +\epsilon \beta' {\tilde s}{\tilde v}_{\tilde s} +\beta {\tilde v}_z,
\qquad\qquad
\\[2mm]
v_{zz}&
\,=\,
\epsilon^2\beta'' {\tilde v}
+ 2 \epsilon^2\, \frac{|\beta'|^2}{\beta} {\tilde s} {\tilde v}_{\tilde s}
+2 \epsilon \beta' {\tilde v}_z
+2\epsilon \beta' {\tilde s} {\tilde v}_{z{\tilde s}}
+ \epsilon^2 \frac{|\beta'|^2}{\beta} {\tilde s}^2 {\tilde v}_{{\tilde s}{\tilde s}}
+ \epsilon^2 \beta'' {\tilde s} {\tilde v}_{\tilde s}
+ \beta {\tilde v}_{zz}.
\end{align*}
Furthermore, for later use, we also compute
\begin{align}
\label{beta'}
\beta' (\theta)
\,=\,
\frac{ \epsilon^2 W_\theta(0, \theta)}{ 2 \beta(\theta)}
\,=\,
\frac{ \epsilon^2 W_\theta(0, \theta)}{ 2 \big[1+ \epsilon^2 W(0, \theta) \big]^{\frac 12}},
\end{align}
\begin{align}
\label{beta''}
\beta'' (\theta)
\,=\,
\frac{ \, \epsilon^2 W_{\theta \theta}(0, \theta)\, }{2\beta(\theta)}
\,-\,
\frac{\, \epsilon^4|W_{\theta}(0, \theta)|^2\, }{4\beta^3(\theta)}
\,=\,
\frac{ \, \epsilon^2 W_{\theta \theta}(0, \theta)\, }{2\big[1+ \epsilon^2 W(0, \theta) \big]^{\frac{1}{2}}}
\,-\,
\frac{\, \epsilon^4|W_{\theta}(0, \theta)|^2\, }{4\big[1+ \epsilon^2 W(0, \theta) \big]^{\frac 3 2}},
\end{align}
and
\begin{align}
\label{beta'''}
\beta'''(\theta)
\,=\,
\frac{ \, \epsilon^2 W_{\theta \theta \theta}(0, \theta)\, }{2\big[1+ \epsilon^2 W(0, \theta) \big]^{\frac{1}{2}}}
\,-\,\frac{ 3\epsilon^4 W_{\theta \theta}(0, \theta)\,W_{\theta }(0, \theta) }{4\big[1+ \epsilon^2 W(0, \theta) \big]^{\frac 3 2}}
\,+\,
\frac{\,3 \epsilon^6|W_{\theta}(0, \theta)|^3\, }{8\big[1+ \epsilon^2 W(0, \theta) \big]^{\frac 5 2}}.
\end{align}
The combination of the above formulas and the assumptions in \eqref{W_theta-control} will imply that
\begin{align}
|\beta'|\,\leq\, C\beta,
\qquad
|\beta''|\,\leq\, C\beta,
\qquad
|\beta'''|\,\leq\, C\beta.
\label{beta-bound}
\end{align}

\medskip
Finally, dividing ${\mathcal D}(v)$ by $\beta^3$, we conclude that problem \eqref{eq4} would be transformed to
\begin{align} \label{eq5}
{\tilde S}({\tilde v}) :\,=\, {\tilde v}_{{\tilde s}{\tilde s}} - {\tilde v}+{\tilde v}^3 +\frac{1}{\beta^2} {\tilde v}_{zz} + B_3({\tilde v}) \,=\, 0,
\end{align}
with
\begin{align*}
B_3({\tilde v})\,=\,
& \Bigg[\epsilon \frac{\kappa}{\, \beta\, }
\,-\, \epsilon^2\frac{\kappa^2}{\beta^2} {\tilde s}\Bigg] {\tilde v}_{\tilde s}
+\epsilon^2\frac{\beta''} {\beta^3}{\tilde v}
+ \epsilon^2\, \frac{2|\beta'|^2}{\beta^4} {\tilde s} {\tilde v}_{\tilde s}
+\epsilon \frac{2\beta'} {\beta^3}{\tilde v}_z
+ \epsilon \frac{2\beta'}{\beta^3} {\tilde s} {\tilde v}_{z{\tilde s}}
+ \epsilon^2 \frac{|\beta'|^2}{\beta^4}{\tilde s}^2 {\tilde v}_{{\tilde s}{\tilde s}}
\nonumber\\[2mm]
&
+ \epsilon^2 \frac{\beta''}{\beta^3} {\tilde s} {\tilde v}_{\tilde s}
- \epsilon^3 \frac{W_t(0, \epsilon z)}{\beta^3} {\tilde s} {\tilde v}
-\frac{1}{2}\, \epsilon^4\frac{W_{tt}(0, \epsilon z)}{\beta^4} {\tilde s}^2 {\tilde v}
+ B_2 ({\tilde v}),
\end{align*}
where $\epsilon {\tilde s}\in(-M_0, M_0),\, \epsilon z\in [0, \ell)$.
In the above formula, we have denoted
\begin{align} \label{B2w}
B_2 ({\tilde v}) \,=\,&
 - \frac{\epsilon^5}{6} \frac{\, \, W_{ttt}(0, \epsilon z)}{\beta^5} {\tilde s}^3{\tilde v}
 -\frac{\epsilon^6}{24} \frac{\, \, W_{tttt}\big(\hbar(\epsilon z)\epsilon s, \, \epsilon z\big)}{\beta^6} {\tilde s}^4{\tilde v}
 + \frac{1}{\, \beta^3\, }B_0 (v)
\nonumber\\[2mm]
 \,=\, &
 - \frac{\epsilon^5}{6} \frac{\, \, W_{ttt}(0, \epsilon z)}{\beta^5} {\tilde s}^3{\tilde v}
 -\frac{\epsilon^6}{24} \frac{\, \, W_{tttt}\big(\hbar(\epsilon z)\epsilon s, \, \epsilon z\big)}{\beta^6} {\tilde s}^4{\tilde v}
+\epsilon^3 \kappa^3(\epsilon z) a_3(\epsilon s, \epsilon z) \frac{1}{\, \beta^3\, }{\tilde s}^2{\tilde v}_{\tilde s}
\nonumber\\[2mm]
&\,+\,
\epsilon \kappa(\epsilon z) a_2(\epsilon s, \epsilon z)
\Bigg[
\epsilon^2\frac{\beta''} {\beta^4} {\tilde s} {\tilde v}
+ \epsilon^2\,\frac{2|\beta'|^2}{\beta^5} {\tilde s}^2 {\tilde v}_{\tilde s}
+\epsilon \frac{2\beta'} {\beta^4} {\tilde s}{\tilde v}_z
+ \epsilon \frac{2\beta'}{\beta^4} {\tilde s}^2 {\tilde v}_{z{\tilde s}}
\nonumber\\[2mm]
&\qquad\qquad\qquad\qquad\qquad
+ \epsilon^2 \frac{|\beta'|^2}{\beta^5} {\tilde s}^3{\tilde v}_{{\tilde s}{\tilde s}}
+ \epsilon^2 \frac{\beta''}{\beta^4} {\tilde s}^2 {\tilde v}_{\tilde s}
+\frac{1}{\, \beta^3\, } {\tilde s}{\tilde v}_{zz}
\Bigg]
\nonumber\\[2mm]
&\,+\,
\epsilon^2a_1(\epsilon s, \epsilon z) \kappa'(\epsilon z)
\Bigg[
\epsilon\frac{\beta'}{\beta^4} {\tilde s}{\tilde v}
+\epsilon\frac{\beta'}{\beta^4} {\tilde s}^2 {\tilde v}_{\tilde s}
+ \frac{1}{\, \beta^3\, } {\tilde s}{\tilde v}_z
\Bigg],
\end{align}
where the linear operator $B_0 (v)$ and the function $\hbar$ are defined in \eqref{B0v}, \eqref{expansion of W} respectively.
In the above, the terms $\beta'$ and $\beta''$ have been calculated in \eqref{beta'} and \eqref{beta''}.

\medskip
\noindent{\bf{Step 2.}}
By recalling the definition of ${\tilde\ell}$ in \eqref{lambda*}, we take the transformation
\begin{align}
 \tau (z) \,=\, \frac{2\pi}{\tilde\ell} \int_0^{\epsilon z} \beta (\xi) \,{\mathrm d} \xi,
\label{ztau}
\end{align}
 and
\begin{align}\label{definition of hat v}
{\tilde v}({\tilde s}, z)
\,=\, {\hat v}({\tilde s}, \tau).
\end{align}
The above setting will imply that
\begin{align*}
 \frac{{\mathrm d}\tau} {{\mathrm d}z}
 \,=\, \epsilon\frac{ 2\pi} {\tilde\ell} \beta (\epsilon z),
 \qquad
 \frac{{\mathrm d}^2 \tau} {{\mathrm d} z^2}
 \,=\, \epsilon^2\frac{ 2\pi} { {\tilde\ell}} \beta' (\epsilon z),
\end{align*}
\begin{align*}
{\tilde v}_{\tilde s}({\tilde s}, z) \,=\,{\hat v}_{\tilde s} ({\tilde s}, \tau),\qquad
{\tilde v}_{{\tilde s}{\tilde s}}({\tilde s}, z) \,=\,{\hat v}_{{\tilde s}{\tilde s}} ({\tilde s}, \tau),
\end{align*}
\begin{align*}
{\tilde v}_z ({\tilde s}, z)
\,=\,\epsilon\frac{2\pi}{{\tilde\ell}}\beta(\epsilon z){\hat v}_\tau ({\tilde s}, \tau),
\qquad
 {\tilde v}_{z{\tilde s}} ({\tilde s}, z)
\,=\, \epsilon\frac{2\pi}{{\tilde\ell}}\beta(\epsilon z){\hat v}_{\tau {\tilde s}} ({\tilde s}, \tau),
\end{align*}
and
\begin{align*}
{\tilde v}_{zz}({\tilde s}, z)
\,=\,
&\epsilon\partial_z\Big[\frac{2\pi}{{\tilde\ell}}\beta(\epsilon z){\hat v}_\tau ({\tilde s}, \tau)\Big]
\,=\,
\epsilon^2\frac{2\pi}{{\tilde\ell}}\beta'{\hat v}_\tau ({\tilde s}, \tau)
\,+\,\epsilon^2\frac{ 4\pi^2}{ {\tilde\ell^2}}\beta^2{\hat v}_{\tau\tau }({\tilde s}, \tau).
\end{align*}

\medskip
We substitute the above results into \eqref{eq5} and then obtain
\begin{align}
0= {\tilde S}({\tilde v})
\,=\,&
{\tilde v}_{{\tilde s}{\tilde s}} - {\tilde v}+{\tilde v}^3 +\frac{1}{\beta^2} {\tilde v}_{zz} + \Bigg[\epsilon \frac{\kappa}{\, \beta\, }
\,-\, \epsilon^2\frac{\kappa^2}{\beta^2} {\tilde s}\Bigg] {\tilde v}_{\tilde s}
+\epsilon^2\frac{\beta''} {\beta^3}{\tilde v}
+ \epsilon^2\, \frac{2|\beta'|^2}{\beta^4} {\tilde s} {\tilde v}_{\tilde s}
+\epsilon \frac{2\beta'} {\beta^3}{\tilde v}_z
\nonumber\\[2mm]
&
+ \epsilon \frac{2\beta'}{\beta^3} {\tilde s} {\tilde v}_{z{\tilde s}}
+ \epsilon^2 \frac{|\beta'|^2}{\beta^4}{\tilde s}^2 {\tilde v}_{{\tilde s}{\tilde s}}
+ \epsilon^2 \frac{\beta''}{\beta^3} {\tilde s} {\tilde v}_{\tilde s}
- \epsilon^3 \frac{W_t(0, \epsilon z)}{\beta^3} {\tilde s} {\tilde v}
-\frac{1}{2}\, \epsilon^4\frac{W_{tt}(0, \epsilon z)}{\beta^4} {\tilde s}^2 {\tilde v}
+ B_2 ({\tilde v})
\nonumber\\[2mm]
\,=\,
&{\hat v}_{{\tilde s}{\tilde s}} - {\hat v} + {\hat v}^3
+\epsilon^2\frac{1}{\beta^2}\Bigg[\frac{4\pi^2}{{\tilde\ell}^2} \beta^{2} {\hat v}_{\tau\tau}
+\frac{2\pi}{ {\tilde\ell}}\beta'{\hat v}_\tau
\Bigg]
+ \Big[\epsilon \frac{\kappa}{\, \beta\, }\,-\, \epsilon^2\frac{\kappa^2}{\beta^2} {\tilde s}\Big] {\hat v}_{\tilde s}
+\epsilon^2\frac{\beta''} {\beta^3}{\hat v}
+ \epsilon^2\, \frac{2|\beta'|^2}{\beta^4} {\tilde s} {\hat v}_{\tilde s}
\nonumber\\[2mm]
&
+\epsilon^2 \frac{\beta'} {\beta^2}\frac{4\pi}{{\tilde\ell}}{\hat v}_\tau
+ \epsilon^2 \frac{\beta'}{\beta^2} \frac{4\pi}{{\tilde\ell}} {\tilde s}{\hat v}_{\tau {\tilde s}}
+ \epsilon^2 \frac{|\beta'|^2}{\beta^4}{\tilde s}^2 {\hat v}_{{\tilde s}{\tilde s}}
+ \epsilon^2 \frac{\beta''}{\beta^3} {\tilde s} {\hat v}_{\tilde s}
- \epsilon \frac{\epsilon^2W_t(0, \epsilon z)}{\beta^3} {\tilde s} {\hat v}
\nonumber\\[2mm]
&-\frac{1}{2}\, \epsilon^2\frac{\epsilon^2W_{tt}(0, \epsilon z)}{\beta^4} {\tilde s}^2 {\hat v}
+ B_2 ({\tilde v})
\nonumber\\[2mm]
\,=\,&{\hat v}_{{\tilde s}{\tilde s}} - {\hat v} + {\hat v}^3
+\epsilon^2\frac{4\pi^2}{{\tilde\ell}^2} {\hat v}_{\tau\tau}
\,+\,\epsilon \frac{\kappa}{\beta} \Bigg[ {\hat v}_{\tilde s} +\frac 2 3 {\tilde s}{\hat v}\Bigg]
+\epsilon^2\frac{6\pi}{{\tilde\ell}}\frac{\beta'} {\beta}\frac{1}{\beta}{\hat v}_\tau
+\epsilon^2 \frac{4\pi}{{\tilde\ell}} \frac{\beta'}{\beta^2} {\tilde s}{\hat v}_{\tau {\tilde s}}
\nonumber\\[2mm]
&
+\epsilon^2\Big[ \frac{\beta''}{\beta^3}\,+\,\frac{2|\beta'|^2}{\beta^4} \,-\,\frac{\kappa^2}{\beta^2}\Big]{\tilde s} {\hat v}_{\tilde s}
+
\epsilon^2\frac{\beta''}{\beta^3}{\hat v}
+\epsilon^2 \frac{|\beta'|^2}{\beta^4}{\tilde s}^2 {\hat v}_{{\tilde s}{\tilde s}}
-\frac{1}{2}\, \epsilon^2\frac{\epsilon^2W_{tt}(0, \epsilon z)}{\beta^4} {\tilde s}^2 {\hat v}
+B_2 ({\tilde v}),
\label{eqn-hatv}
\end{align}
where $\epsilon {\tilde s}\in(-M_0, M_0),\, \tau\in [0, 2\pi)$.
In \eqref{eqn-hatv}, we have used the stationary condition in \eqref{stationary assumption}.
In the coordinates $(\tilde{s}, \tau)$, the expression of $B_2 ({\tilde v})$ in \eqref{eqn-hatv} has the form
\begin{align}
B_2 ({\tilde v})
\,=\, &
 - \frac{\epsilon^5}{6} \frac{\, \, W_{ttt}(0, \epsilon z)}{\beta^5} {\tilde s}^3{\hat v}
 -\frac{\epsilon^6}{24} \frac{\, \, W_{tttt}\big(\hbar(\epsilon z)\epsilon s, \, \epsilon z\big)}{\beta^6} {\tilde s}^4{\hat v}
 +\epsilon^3 \kappa^3(\epsilon z) a_3(\epsilon s, \epsilon z) \frac{1}{\, \beta^3\, }{\tilde s}^2{\hat v}_{\tilde s}
\nonumber\\[2mm]
&\,+\,
\epsilon \kappa(\epsilon z) a_2(\epsilon s, \epsilon z)
\Bigg\{
\epsilon^2\frac{\beta''} {\beta^4} {\tilde s} {\hat v}
+ \epsilon^2\,\frac{2|\beta'|^2}{\beta^5} {\tilde s}^2 {\hat v}_{\tilde s}
+\epsilon^2 \frac{2\beta'} {\beta^4} {\tilde s}\frac{ 2\pi}{ {\tilde\ell}}\beta{\hat v}_\tau
+ \epsilon^2 \frac{2\beta'}{\beta^4} {\tilde s}^2 \frac{ 2\pi}{ {\tilde\ell}}\beta{\hat v}_{\tau {\tilde s}}
\nonumber\\[2mm]
&\qquad\qquad\qquad\qquad\qquad
+ \epsilon^2 \frac{|\beta'|^2}{\beta^5} {\tilde s}^3{\hat v}_{{\tilde s}{\tilde s}}
+ \epsilon^2 \frac{\beta''}{\beta^4} {\tilde s}^2 {\hat v}_{\tilde s}
 +\epsilon^2 \frac{1}{\, \beta^3\, } {\tilde s}\Bigg[
\frac{ 4\pi^2} {{\tilde\ell}^2} \beta^{2} {\hat v}_{\tau\tau}
+\frac{ 2\pi}{ {\tilde\ell}}\beta' {\hat v}_\tau
\Bigg]
\Bigg\}
\nonumber\\[2mm]
&\,+\,
\epsilon^2a_1(\epsilon s, \epsilon z) \kappa'(\epsilon z)
\Bigg\{
\epsilon\frac{\beta'}{\beta^4} {\tilde s}{\hat v}
+\epsilon\frac{\beta'}{\beta^4} {\tilde s}^2 {\hat v}_{\tilde s}
+\epsilon\frac{1}{\, \beta^3\, } {\tilde s}\frac{ 2\pi}{ {\tilde\ell}}\beta{\hat v}_\tau
\Bigg\}
\nonumber\\[2mm]
\,=\, &
 - \frac{\epsilon^5}{6} \frac{\, \, W_{ttt}(0, \epsilon z)}{\beta^5} {\tilde s}^3{\hat v}
 -\frac{\epsilon^6}{24} \frac{\, \, W_{tttt}\big(\hbar(\epsilon z)\epsilon s, \, \epsilon z\big)}{\beta^6} {\tilde s}^4{\hat v}
 +\epsilon^3 \kappa^3(\epsilon z) a_3(\epsilon s, \epsilon z) \frac{1}{\, \beta^3\, }{\tilde s}^2{\hat v}_{\tilde s}
\nonumber\\[2mm]
&\,+\,
\epsilon^3 \kappa(\epsilon z) a_2(\epsilon s, \epsilon z)
\Bigg\{
\frac{\beta''} {\beta^4} {\tilde s} {\hat v}
+\,\frac{2|\beta'|^2}{\beta^5} {\tilde s}^2 {\hat v}_{\tilde s}
+\frac{ 6\pi}{ {\tilde\ell}}\frac{\beta'} {\beta^3} {\tilde s}{\hat v}_\tau
+\frac{ 4\pi}{ {\tilde\ell}}\frac{\beta'} {\beta^3} {\tilde s}^2{\hat v}_{\tau \tilde s }
\nonumber\\[2mm]
&\qquad\qquad\qquad\qquad\qquad
+ \frac{|\beta'|^2}{\beta^5} {\tilde s}^3{\hat v}_{{\tilde s}{\tilde s}}
+ \frac{\beta''}{\beta^4} {\tilde s}^2 {\hat v}_{\tilde s}
+\frac{ 4\pi^2} {{\tilde\ell}^2}\frac{1}{\beta} {\tilde s}{\hat v}_{\tau\tau}\Bigg\}
\nonumber\\[2mm]
&\,+\,
\epsilon^3a_1(\epsilon s, \epsilon z) \kappa'(\epsilon z)
\Bigg\{
\epsilon\frac{\beta'}{\beta^4} {\tilde s}{\hat v}
+\frac{\beta'}{\beta^4} {\tilde s}^2 {\hat v}_{\tilde s}
+\frac{ 2\pi}{ {\tilde\ell}}\frac{1}{ \beta^2} {\tilde s}{\hat v}_\tau
\Bigg\}.
\label{B2tildev}
\end{align}

\noindent{\bf{Step 3.}}
We define a function $w (x, \tau) $ by the relation
\begin{align} \label{definition of x 2.14}
{\hat v}({\tilde s}, \tau) \,=\, w(x, \tau)
\qquad \text{with}\quad
x \,=\, {\tilde s}-f(\tau) = \beta s- f(\tau),
\end{align}
where $f :[0, 2\pi) \mapsto{\mathbb R}$ with the constraint \eqref{constraint-f} is a function to be chosen later (see Sections \ref{Section7} and \ref{Section8}) and satisfies the boundary conditions
$$
\quad f(0) \,=\, f(2\pi), \quad f'(0) \,=\, f'(2\pi).
$$

It's easy to get that
\begin{align*}
\p_\tau x & \,=\, f',
\qquad
\p_{\tau\tau} x \,=\, f''.
\end{align*}
After that, there hold
\begin{align*}
&{\hat v}_{{\tilde s}} \,=\, w_x,
\qquad\qquad
{\hat v}_{ss} \,=\, w_{xx},
\\[2mm]
&{\hat v}_\tau \,=\, - f' w_x +w_\tau,
\qquad\qquad
{\hat v}_{{\tilde s}\tau} \,=\, - f' w_{xx} +w_{x\tau},
\\[2mm]
&{\hat v}_{\tau\tau} \,=\, |f'|^2 w_{xx} -f''w_x +w_{\tau\tau} -2f' w_{x\tau}.
\end{align*}
Therefore, we write \eqref{eqn-hatv} in the form
\begin{align}
&\, w_{xx} - w + w^3
+ \frac{\epsilon^2 4\pi^2}{{\tilde\ell}^2} w_{\tau\tau}
+\epsilon^2\frac{4\pi^2}{{\tilde\ell}^2}\Big[|f'|^2 w_{xx} -f''w_x -2f' w_{x\tau}\Big]
\nonumber\\[2mm]
&
+\epsilon \frac{\kappa}{\, \beta\, } \Big[ w_x +\frac{2}{3} (x+f)w\Big]
+\epsilon^2\frac{6\pi}{{\tilde\ell}}\frac{\beta'} {\beta}\frac{1}{\beta}\big[- f' w_{x} +w_{\tau}\big]
\nonumber\\[2mm]
&
+ \epsilon^2 \frac{4\pi}{{\tilde\ell}} \frac{\beta'}{\beta^2} (x+f)\big[- f' w_{xx} +w_{x\tau}\big]
+
\epsilon^2\Big[\frac{\beta''}{\beta^3}\,+\,\frac{2|\beta'|^2}{\beta^4} \,-\,\frac{\kappa^2}{\beta^2} \Big](x+f) w_{x}
\nonumber\\[2mm]
&+\epsilon^2\frac{\beta''}{\beta^3}w
+ \epsilon^2 \frac{|\beta'|^2}{\beta^4}(x+f)^2 w_{xx}
-\frac{1}{2}\, \epsilon^2\frac{\epsilon^2W_{tt}(0, \epsilon z)}{\beta^4} (x+f)^2 w
+B_2 ({\tilde v})
=0,
\label{S(w)}
\end{align}
where $(x, \tau)\in{\mathbb R}\times[0, 2\pi)$.

\medskip
For the simplification of notation in \eqref{S(w)}, we denote
\begin{align} \label{new opreator B4}
 B_4(w)
 =\epsilon \frac{\kappa}{\, \beta\, } \Big[ w_x +\frac{2}{3} (x+f)w\Big],
\end{align}
and
\begin{align} \label{new opreator B5}
 B_5(w)
 =&\epsilon^2\frac{4\pi^2}{{\tilde\ell}^2}\Big[|f'|^2 w_{xx} -f''w_x -2f' w_{x\tau}\Big]
+
\epsilon^2\frac{6\pi}{{\tilde\ell}}\frac{\beta'} {\beta^2}\Big[- f' w_x +w_\tau\Big]
\nonumber\\[2mm]
&
+ \epsilon^2 \frac{4\pi}{{\tilde\ell}} \frac{\beta'}{\beta^2} (x+f)\big[- f' w_{xx} +w_{x\tau}\big]
+
\epsilon^2\Big[
\frac{\beta''}{\beta^3}\,+\,\frac{2|\beta'|^2}{\beta^4} \,-\,\frac{\kappa^2}{\beta^2} \Big] (x+f) w_{x}
\nonumber\\[2mm]
&+\epsilon^2\frac{\beta''}{\beta^3}w
+ \epsilon^2 \frac{|\beta'|^2}{\beta^4}(x+f)^2 w_{xx}
-\frac{1}{2}\, \epsilon^2\frac{\epsilon^2W_{tt}(0, \epsilon z)}{\beta^4} (x+f)^2 w.
\end{align}
For the last term in \eqref{S(w)}, by \eqref{B2tildev} and the relation $\hat v (\tilde s, \tau)=w(x, \tau)$, we derive
\begin{align}
 B_2 ({\tilde v})
 \,=\,& - \frac{\epsilon^3}{6} \frac{\, \, \epsilon^2 W_{ttt}(0, \epsilon z)}{\beta^5} (x+f)^3 w
 -\frac{\epsilon^4}{24} \frac{\, \, \epsilon^2 W_{tttt}\big(\hbar(\epsilon z)\epsilon s, \, \epsilon z\big)}{\beta^6} (x+f)^4w
\nonumber\\[2mm]
&
 +\epsilon^3 \kappa^3(\epsilon z) a_3(\epsilon s, \epsilon z) \frac{1}{\, \beta^3\, }(x+f)^2w_{x}
\nonumber\\[2mm]
&\,+\,
\epsilon^3 \kappa(\epsilon z) a_2(\epsilon s, \epsilon z)
\Bigg\{
\frac{\beta''} {\beta^4} (x+f) w
+\frac{2|\beta'|^2}{\beta^5} (x+f)^2 w_x
+\frac{ 6\pi}{ {\tilde\ell}}\frac{\beta'} {\beta^3} (x+f)\big[-f'w_x +w_\tau\big]
\nonumber\\[2mm]
&\qquad\qquad\qquad\qquad\qquad
+\frac{ 4\pi}{ {\tilde\ell}}\frac{\beta'} {\beta^3} (x+f)^2[- f' w_{xx} +w_{x\tau}]
+ \frac{|\beta'|^2}{\beta^5} (x+f)^3w_{xx}
+ \frac{\beta''}{\beta^4} (x+f)^2 w_{x}
\nonumber\\[2mm]
&\qquad\qquad\qquad\qquad\qquad
+\frac{ 4\pi^2} {{\tilde\ell}^2}\frac{1}{\beta} (x+f)\big[ |f'|^2 w_{xx} -f''w_x +w_{\tau\tau} -2f' w_{x\tau}\big]
\Bigg\}
\nonumber\\[2mm]
&\,+\,
\epsilon^3 \kappa'(\epsilon z)a_1(\epsilon s, \epsilon z)
\Bigg\{
\frac{\beta'}{\beta^4} (x+f)w
+\frac{\beta'}{\beta^4} (x+f)^2 w_{x}
+\frac{ 2\pi}{ {\tilde\ell}}\frac{1}{ \beta^2} (x+f)\big[-f'w_x +w_\tau\big]
\Bigg\}
\nonumber\\[2mm]
: \,=\,& B_6(w).
\label{B2w-2222new}
\end{align}
Note that the coefficients in $B_4$, $B_5$ and $B_6$ are smooth bounded functions on the following domain
$$
\Big\{\, (s, z)\, :\, |\epsilon s|<M_0, \quad \epsilon z\in [0, \ell)\, \Big\},
$$
due to the assumptions in \eqref{curvature bounded} and \eqref{W_t-control}-\eqref{W_theta-control}.
We finally use \eqref{S(w)} to obtain the equation
\begin{align}
S(w):=&\, \frac{\epsilon^2 4\pi^2}{{\tilde\ell}^2} w_{\tau\tau} + w_{xx} - w +w^3 + B_4(w)+ B_5(w)+ B_6(w)=0,
\label{problem-local4}
\end{align}
where $x\in(-M_0/\epsilon,\, M_0/\epsilon),\, \tau\in [0, 2\pi)$.

\subsection{ The block solution $U$ and conventions}\label{Section3.3}\

Recall that $U$ is the unique solution to \eqref{block}, which is even in $x$, and also enjoys the asymptotic behaviours
\begin{align}
 \lim_{|x| \rightarrow +\infty} U(x) e^{|x|} = C<+\infty,
\qquad\text{and} \qquad
\lim_{|x| \rightarrow +\infty} \frac{U(x)}{U'(x)}=-1.
\label{U-asymp}
\end{align}
Moreover, $U$ satisfies the non-degeneracy condition, in the sense that the kernel of the linear operator $-\Delta + 1 -3U^2 $ in $H^1(\mathbb{R}) $ is spanned by the function $U'(x)$.
For the linear eigenvalue problem on $H^1({\mathbb R})$
\begin{align}
\label{eigenvalue}
h{''}-h+3U^{2}h\,=\, \lambda h\quad \mbox{in }\,\Bbb R, \,\quad h(\pm \infty)\,=\, 0,
\end{align}
it is also well-known that it possesses a unique positive eigenvalue $\lambda_0$,
with associated eigenfunction ${\mathcal Z}$ (even and positive) which can be normalized so that
$$
\int_{{\mathbb R}}{\mathcal Z}^2(x) \,{\mathrm d}x=1.
$$
In fact, a simple computation shows that
\begin{equation}
\label{lambda0}
\lambda_0\,=\, 4,
\qquad
\mathcal Z(x)\,=\, \frac{U^{2}(x)}{\sqrt{\displaystyle\int_{\mathbb R} U^{4} \,{\mathrm d}x}\, }.
\end{equation}
The reader can refer to \cite{delPKowWei1}.

It is easy to derive the following relations
\begin{align} \label{identities for U}
 3 \int_{{\mathbb R}} U^4 \,{\mathrm d}x = 4\int_{{\mathbb R}} U'^3 \,{\mathrm d}x= 12 \int_{{\mathbb R}} |U'|^2 \,{\mathrm d}x := 12 \varrho_1,
\end{align}
and
\begin{align} \label{identities for U2}
 \int_{{\mathbb R}} xU U' \,{\mathrm d}x = - \frac 12 \int_{{\mathbb R}} U^2 \,{\mathrm d}x= - \frac 32 \int_{{\mathbb R}} |U'|^2 \,{\mathrm d}x = - \frac 32 \varrho_1,
\end{align}
where
\begin{align}
\varrho_1
\,:=\,
\int_{\Bbb R}|U'(x)|^2\,{\mathrm d}x
=
-\,2\int_{\Bbb R}xU'(x)U''(x)\,{\mathrm d}x.
\label{varrho1}
\end{align}
For later use, we denote the following constants
\begin{align}
\varrho_0:=\int_{\Bbb R}|U(x)|^2\,{\mathrm d}x =-\,2\int_{\Bbb R}xU(x)U'(x)\,{\mathrm d}x,
\label{varrho0}
\end{align}
\begin{align}
\varrho_2
\,:=\,
\int_{\Bbb R}|xU'(x)|^2\,{\mathrm d}x
=-\,\frac{2}{3}\int_{\Bbb R}x^3U'(x)U''(x)\,{\mathrm d}x,
\label{varrho2}
\end{align}
\begin{align}
\varrho_3
\,:=\,
\int_{\Bbb R}|xU''(x)|^2\,{\mathrm d}x,
\label{varrho3}
\end{align}
\begin{align}
\varrho_4
\,:=\,
\int_{\Bbb R}|x^2U''(x)|^2\,{\mathrm d}x,
\label{varrho4}
\end{align}
\begin{align}
\varrho_p
\,:=\,
\int_{\Bbb R}|U(x)|^{p+1}\,{\mathrm d}x.
\label{varrhop}
\end{align}
Moreover, it is trivial that
\begin{align}\label{varrho-additional}
\int_{\Bbb R}x^2U(x)U''(x)\,{\mathrm d}x=\varrho_0-\varrho_2.
\end{align}

\medskip
We end this section by providing a collection of notation.
{Throughout the paper, ${\mathfrak S}$ represents the strip in $\mathbb{R}^2$ of the form
\begin{equation}
\label{huaxies}
{\mathfrak S}=\Big\{ (x, \tau): \quad x\in {\mathbb R}, \quad 0<\tau< 2\pi\Big\}.
\end{equation}
For the introduction of the parameters $f$ and $e$, the reader can refer to \eqref{definition of x 2.14}
and \eqref{first-approximation}.
In all what follows, we will make the following assumptions on the parameters $f$ and $e$
\begin{align} \label{constraint-f}
\|f\|_* :\,=\,
\left\|\beta f\right\|_{L^2(0, 2\pi)}
+\left\|\frac{\beta}{\tilde\ell} f_\tau\right\|_{L^2(0, 2\pi)}
+\left\|\frac{\beta}{{\tilde\ell}^2} f_{\tau\tau}\right\|_{L^2(0, 2\pi)}
\leq \epsilon^{1-\alpha},
\end{align}
\begin{align} \label{constraint-e}
\|e\|_{**} : \,=\,
\|e\|_{L^{\infty}(0, 2\pi)}
 + \frac{\epsilon}{{\tilde\ell}} \|e'\|_{L^2 (0, 2\pi)}
 + \frac{\epsilon^2} {{\tilde\ell}^2}\|e''\|_{L^2(0, 2\pi)}
 \leq \epsilon^{\frac 32 \varrho + \frac 12- 2\alpha},
\end{align}
where $\beta$ has been given in \eqref{beta}, and the parameters $\varrho$ and $\alpha$ have been given in \eqref{length-constraint} and \eqref{resonace-f}
with the requirement in \eqref{constraint-alpha-varrho}.
By \eqref{constraint-alpha-varrho}, we conclude that
\begin{align*}
\|e\|_{**}  \le \epsilon^{\frac 32 \varrho + \frac 12- 2\alpha} \ll \epsilon^{1-\alpha},
\end{align*}
as $\epsilon\rightarrow 0^+$.

Finally, by the norm defined in \eqref{constraint-f}, we claim that
\begin{align} \label{Linfty of f}
\left\|\beta f\right\|_{L^\infty(0, 2\pi)} \le C\,{\tilde\ell}^{\frac 12} \|f\|_*,
\qquad
\left\|\frac{\beta}{\tilde\ell} f_\tau\right\|_{L^\infty(0, 2\pi)} \le C\,{\tilde\ell}^{\frac 12} \|f\|_*.
\end{align}
In fact, we consider the following two cases.

\smallskip
\noindent $\clubsuit$\
If
$$
\sup_{\tau\in (0, 2\pi)}(\beta f)^2\leq 2\inf_{\tau\in (0, 2\pi)}(\beta f)^2,
$$
then
$$
\sup_{\tau\in (0, 2\pi)}(\beta f)^2
\leq
2\inf_{\tau\in (0, 2\pi)}(\beta f)^2
\leq
2\frac{1}{2\pi}\int_0^{2\pi}(\beta f)^2{\mathrm d}\tau
\leq C\|f\|_*^2.
$$
Note that there exists a positive constant ${\tilde\ell}_0$ independent of $\epsilon$ such that
\begin{align*}
 {\tilde\ell}\geq {\tilde\ell}_0,
\end{align*}
due to the fact in \eqref{ellzero}.
Therefore, there holds
$$
\|(\beta f)^2\|_{L^{\infty}(0, 2\pi)}
\leq C{\tilde\ell}\|f\|_*^2.
$$

\smallskip
\noindent $\clubsuit$\
We consider the case
$$
\sup_{\tau\in (0, 2\pi)}(\beta f)^2> 2\inf_{\tau\in (0, 2\pi)}(\beta f)^2.
$$
By denoting $\tau_0$ the  point  such that
\begin{equation*}
\big(\beta f\big)^2 \Big|_{\tau=\tau_0}\leq \frac{5}{4}\inf_{\tau\in (0, 2\pi)}(\beta f)^2,
\end{equation*}
and we obtain
\begin{align*}
\int^\tau_{\tau_0} \beta f (\beta f)_\tau
=\frac{1}{2}\big(\beta f\big)^2\Big|^\tau_{\tau_0}
=\frac{1}{2}\big(\beta f\big)^2-\frac{1}{2}\big(\beta f\big)^2\Big|_{\tau=\tau_0},
\end{align*}
and then
\begin{align*}
(\beta f)^2
\leq
\big(\beta f\big)^2\Big|_{\tau=\tau_0}
\,+\,2\int^\tau_{\tau_0} |\beta f (\beta f)_\tau|{\mathrm d}\tau.
\end{align*}
We get
\begin{align*}
\sup_{\tau\in (0, 2\pi)} (\beta f)^2
\,\leq\,&
\big(\beta f\big)^2\Big|_{\tau=\tau_0}
\,+\,2\int^{2\pi}_0 |\beta f (\beta f)_\tau|{\mathrm d}\tau
\\[2mm]
\,\leq\,&
\frac{5}{4}\inf_{\tau\in (0, 2\pi)}(\beta f)^2
\,+\,2\int^{2\pi}_0 |\beta f (\beta f)_\tau|{\mathrm d}\tau
\\[2mm]
\,\leq\,&
\frac{5}{8}\sup_{\tau\in (0, 2\pi)}(\beta f)^2
\,+\,2\int^{2\pi}_0 |\beta f (\beta f)_\tau|{\mathrm d}\tau.
\end{align*}
Hence,
\begin{align*}
\|(\beta f)^2\|_{L^{\infty}(0, 2\pi)}
\leq
C\int^{2\pi}_0 |\beta f (\beta f)_\tau|{\mathrm d}\tau.
\end{align*}
The term in the right hand side can be estimated as follows
\begin{align*}
\int_0^{2\pi} |\beta f| |(\beta f)_\tau|{\mathrm d}\tau
&\le C \left(\tilde\ell \int_0^{2\pi} |\beta f|^2{\mathrm d}\tau
\ +\
\frac{{\tilde \ell}^2}{\tilde \ell}\int_0^{2\pi} \Big|\frac{(\beta f)_\tau}{\tilde \ell}\Big|^2{\mathrm d}\tau \right)
\nonumber\\[2mm]
 &\le C \left(\tilde\ell \int_0^{2\pi} |\beta f|^2{\mathrm d}\tau
\ +\ \frac{{\tilde \ell}^2}{\tilde \ell}\int_0^{2\pi} \Big|\frac{\beta f_\tau}{\tilde \ell}\Big|^2{\mathrm d}\tau
\ +\ \frac1{\tilde\ell} \int_0^{2\pi} \Big|\frac{\beta'} {\beta}\Big|^2\Big|\frac{{\mathrm d}\theta} {{\mathrm d}\tau}\Big|^2\,|\beta f|^2{\mathrm d}\tau \right)
\nonumber\\[2mm]
&\le C \left(\tilde\ell \int_0^{2\pi} |\beta f|^2{\mathrm d}\tau
\ +\ \frac{{\tilde \ell}^2}{\tilde \ell}\int_0^{2\pi} \Big|\frac{\beta f_\tau}{\tilde \ell}\Big|^2{\mathrm d}\tau
\ +\ \frac1{\tilde\ell} \int_0^{2\pi} \Big|\frac{\beta'} {\beta}\Big|^2\Big|\frac{\tilde \ell} {2 \pi \beta}\Big|^2 \, |\beta f|^2{\mathrm d}\tau \right)
\nonumber\\[2mm]
&\le C \left(\tilde\ell \int_0^{2\pi} |\beta f|^2{\mathrm d}\tau
\ +\ \tilde \ell \int_0^{2\pi} \Big|\frac{\beta f_\tau}{\tilde \ell}\Big|^2{\mathrm d}\tau \right)
\nonumber\\[2mm]
 &\le C\tilde \ell \|f\|_*^2,
\end{align*}
where we have used \eqref{beta-bound}.
 Therefore,
\[
 \|\beta f\|_{L^\infty(0, 2\pi)} \le C {\tilde \ell}^{\frac12} \|f\|_*.
 \]
We can prove the second inequality in \eqref{Linfty of f} in a similar way.


\medskip
\section{Proof of Theorem \ref{necessary condition theorem}}\label{Section4}

 As stated in the assumptions of Theorem \ref{necessary condition theorem}, the solution $u_\epsilon$ of problem \eqref{eq2}-\eqref{constraint1} has
 a concentration on the curve $\Gamma_\epsilon$ in the sense of the decomposition in \eqref{u-decomposition}-\eqref{smallperturbation1}.

\subsection{ The identities of Pohozaev type}\label{Section4.1}\

 We here provide some identities by mimicking the proof of Pohozaev identity.

\begin{lemma}\label{lemmaPohozaev}
 Suppose that $u=u_\epsilon$ is a solution of \eqref{eq2}. 
 Then the following identities are valid in the local coordinates given in \eqref{fermi}
\begin{align}
 &\epsilon^2\int_{-M}^{M} \big({1+\kappa t}\big) u_t^2 \,{\mathrm d}t
 \,+\,
 \int_{-M}^{M} \big({1+\kappa t}\big) \big[1+ \epsilon^2 W (t, \theta) \big] u^2 \,{\mathrm d}t
 \,-\,
 \int_{-M}^{M} \big({1+\kappa t}\big) u^4 \,{\mathrm d}t
\nonumber\\[2mm]
 &\,=\,
 \epsilon^2\int_{-M}^{M} u
 \frac{\partial}{\partial\theta} \Big[ \frac{1}{1+\kappa t} \frac{\partial u}{\partial\theta} \Big] \,{\mathrm d}t
 \ +\
 \epsilon^2\big(1+\kappa t\big) u u_t \Bigg|_{t=-M}^{M},
 \label{first identity}
\end{align}
 and
\begin{align}
 & \epsilon^2 \int_{-M}^{M} \big(1+\kappa t\big)^2 W_t(t, \theta) u^2 \,{\mathrm d}t
 \ +\
 2\kappa \int_{-M}^{M} \big(1+\kappa t\big) \big[1+ \epsilon^2 W(t, \theta)\big] u^2 \,{\mathrm d}t
 \ - \ \kappa \int_{-M}^{M} \big(1+\kappa t\big) u^4 \,{\mathrm d}t
\nonumber\\[2mm]
 &=\,
 -2 \epsilon^2\int_{-M}^{M} \big(1+\kappa t\big) u_t \frac{\partial}{\partial\theta}
 \Big[ \frac{1}{1+\kappa t} \frac{\partial u}{\partial\theta} \Big]\,{\mathrm d}t
 \,-\,
 \epsilon^2\Big[ \big(1+\kappa t\big) u_t\Big]^2\, \Bigg|_{t=-M}^{M}
\nonumber\\[2mm]
 &\quad
 \,+\,
 \big(1+\kappa t\big)^2 \big[1+ \epsilon^2 W(t, \theta)\big] u^2\, \Bigg|_{t=-M}^{M}
 \ -\
 \frac{1}{2} \big(1+\kappa t\big)^2 |u|^4\, \Bigg|_{t=-M}^{M},
 \label{second identity}
\end{align}
 where $0<M<M_0$ is any universal constant.
 Consequently,
\begin{align}
 &2\epsilon^2 \kappa\int_{-M}^{M} \big({1+\kappa t}\big) u_t^2 \,{\mathrm d}t
 \,-\,
 \epsilon^2 \int_{-M}^{M} \big(1+\kappa t\big)^2 W_t(t, \theta) u^2 \,{\mathrm d}t
 \,-\,
 \kappa\int_{-M}^{M} \big({1+\kappa t}\big) u^4 \,{\mathrm d}t
\nonumber\\[2mm]
 &\,=\,
 2 \epsilon^2\int_{-M}^{M}
 \frac{\partial}{\partial t} \Big[\big(1+\kappa t\big) u \Big]
 \frac{\partial}{\partial\theta} \Big[ \frac{1}{1+\kappa t} \frac{\partial u}{\partial\theta} \Big]\,{\mathrm d}t
 \ +\
 2\epsilon^2 \kappa\big(1+\kappa t\big) u u_t \Bigg|_{t=-M}^{M}
 \ +\
 \epsilon^2\Big[ \big(1+\kappa t\big) u_t \Big]^2\, \Bigg|_{t=-M}^{M}
\nonumber\\[2mm]
 &\quad
 \ -\
 \big(1+\kappa t\big)^2 \big[1+ \epsilon^2 W(t, \theta)\big] u^2\, \Bigg|_{t=-M}^{M}
 \ +\
 \frac{1}{2} \big(1+\kappa t\big)^2 |u|^4\, \Bigg|_{t=-M}^{M}.
 \label{third identity}
\end{align}
\end{lemma}

\begin{proof}
 For the proof, we will use the local form of \eqref{eq2}, which has been given in \eqref{eq3.1-0}.

\smallskip\noindent{\bf{Step 1.}}
 Multiplying $(1+\kappa t)\, u$ on the both sides of \eqref{eq3.1-0} and integrating, we get
\begin{align} \label{identity 3.15}
 & -\ \epsilon^2\int_{-M}^{M}
 u
 \frac{\partial}{\partial t}
 \Bigg[ \big(1+\kappa t\big) \frac{\partial u}{\partial t} \Bigg] \,{\mathrm d}t
 \ -\
 \epsilon^2\int_{-M}^{M} u
 \frac{\partial}{\partial\theta}
 \Bigg[ \frac{1}{1+\kappa t} \frac{\partial u}{\partial\theta} \Bigg]\,{\mathrm d}t
\nonumber\\[2mm]
 & +\ \int_{-M}^{M} \big({1+\kappa t}\big) \big[1+\epsilon^2 W (t, \theta) \big] u^2\,{\mathrm d}t
 \ -\ \int_{-M}^{M} \big({1+\kappa t}\big) u^4 \,{\mathrm d}t
 \,=\, 0.
\end{align}
 By some trivial calculations, we can get
\begin{align} \label{identity 3.15-22}
 \epsilon^2\int_{-M}^{M}
 -u\frac{\partial}{\partial t} \Big[ \big(1+\kappa t\big) \frac{\partial u}{\partial t} \Big] \,{\mathrm d}t
 =
 \epsilon^2\int_{-M}^{M} \big({1+\kappa t}\big) u_t^2 \,{\mathrm d}t
 \ -\ \epsilon^2\big(1+\kappa t\big)u u_t \Bigg|_{t=-M}^{M}.
\end{align}
 By \eqref{identity 3.15}-\eqref{identity 3.15-22}, we have proved \eqref{first identity}.

\medskip
\noindent{\bf{Step 2.}}
 Multiplying equation \eqref{eq3.1-0} by $(1+\kappa t)^2 u_t$ and integrating in the interval $(-M, M)$, we have
\begin{align*}
 & \epsilon^2\int_{-M}^{M} \big(1+\kappa t\big)u_t\,
 \frac{\partial}{\partial t} \Big[ \big(1+\kappa t\big) \frac{\partial u}{\partial t} \Big]\,{\mathrm d}t
 \ +\
 \epsilon^2\int_{-M}^{M} \big(1+\kappa t\big)u_t\,
 \frac{\partial}{\partial\theta} \Big[ \frac{1}{1+\kappa t} \frac{\partial u}{\partial\theta} \Big]
 \,{\mathrm d}t
\nonumber\\[2mm]
 \ +\ &\int_{-M}^{M}\Big\{ -\big[1+ \epsilon^2 W(t, \theta)\big] u \,+\, |u|^2 u \Big\}
 \big(1+\kappa t\big)^2 u_t \,{\mathrm d}t
 \,=\,0.
\end{align*}
 It is trivial to derive
\begin{align*}
 \epsilon^2\int_{-M}^{M} \big(1+\kappa t\big) u_t \, \frac{\partial}{\partial t}
 \Big[ \big(1+\kappa t\big) u_t\Big]\,{\mathrm d}t
 \,=\,
 \epsilon^2\frac{1}{2}\Big[ \big(1+\kappa t\big) u_t\Big]^2\, \Bigg|_{t=-M}^{M}.
\end{align*}
 Similarly, we also have
\begin{align*}
 &-\int_{-M}^{M} \big(1+\kappa t\big)^2 \big[1+ \epsilon^2 W(t, \theta)\big] u u_t \,{\mathrm d}t
\nonumber\\[2mm]
 &\,=\,
 \frac{1}{2} \epsilon^2 \int_{-M}^{M} \big(1+\kappa t\big)^2 W_t(t, \theta) u^2 {\mathrm d}t
 \,+\,
 \kappa \int_{-M}^{M} \big(1+\kappa t\big) \big[1+ \epsilon^2 W(t, \theta)\big] u^2 {\mathrm d}t
\nonumber\\[2mm]
 &\quad\,-\, \frac{1}{2}\big(1+\kappa t\big)^2 \big[1+ \epsilon^2 W(t, \theta)\big] u^2\, \Bigg|_{t=-M}^{M},
\end{align*}
 and
\begin{align*}
 \int_{-M}^{M} \big(1+ \kappa t\big)^2 |u|^2 u u_t \,{\mathrm d}t
 &\,=\,
 -\frac{1}{2} \kappa \int_{-M}^{M} \big(1+\kappa t\big) |u|^4\,{\mathrm d}t
 \ +\
 \frac 1 4 \big(1+\kappa t\big)^2 |u|^4\, \Bigg|_{t=-M}^{M}.
\end{align*}
Whence, we obtain the validity of \eqref{second identity} by combining the above calculations.
Moreover, \eqref{third identity} is a direct consequence of \eqref{first identity} and \eqref{second identity}.
\end{proof}

\medskip
\subsection{ The estimates of the profile with concentration}\label{Section4.2}\

 For later convenience of notation, we here recall the descriptions of the solutions with concentration,
 see \eqref{u-decomposition}-\eqref{smallperturbation1}.

\medskip
 In order to give explicit forms of further estimates, some analysis will be done here.
 The expression of ${\mathbf u}$ in \eqref{mathbfu} gives that
\begin{align}
 {\mathbf u}_t =\chi'(t)\beta U\big(\beta t/\epsilon\big)
 \,+\,
 \epsilon^{-1}\chi(t)\beta\beta U'\big(\beta t/\epsilon\big),
 \label{mathbfu-t}
\end{align}
\begin{align}
 {\mathbf u}_{\theta} = \chi(t)\beta' U\big(\beta t/\epsilon\big)
 \,+\,
 \epsilon^{-1} \chi(t)\beta \beta' U'\big(\beta t/\epsilon\big)t,
 \label{mathbfu-theta}
\end{align}
 and
\begin{align}
 {\mathbf u}_{tt} = \chi''(t) \beta U\big(\beta t/\epsilon\big)
 \,+\,
 2\epsilon^{-1}\chi'(t)\beta\beta U'\big(\beta t/\epsilon\big)
 \,+\,
 \epsilon^{-2}\chi(t)\beta\beta^2 U''\big(\beta t/\epsilon\big),
 \label{mathbfu-tt}
\end{align}
\begin{align}
 {\mathbf u}_{t\theta} =&\,\chi'(t)\beta' U\big(\beta t/\epsilon\big)
 \,+\,\epsilon^{-1}\chi(t)\beta'\beta\, U'\big(\beta t/\epsilon\big)
 \,+\,\epsilon^{-1}\chi'(t)\beta \beta' U'\big(\beta t/\epsilon\big)t
\nonumber\\[2mm]
 &
 \,+\,\epsilon^{-1} \chi(t)\beta \beta' U'\big(\beta t/\epsilon\big)
 \,+\,\epsilon^{-2}\chi(t)\beta\beta \beta' U''\big(\beta t/\epsilon\big)t,
 \label{mathbfu-ttheta}
\end{align}
\begin{align}
 {\mathbf u}_{\theta\theta}
 \,=\,
 &\chi(t)\beta''U\big(\beta t/\epsilon\big)
 \,+\, 2 \epsilon^{-1}\chi(t)\beta'\beta' U'\big(\beta t/\epsilon\big)t
 \,+\, \epsilon^{-1}\chi(t)\beta\beta'' U'\big(\beta t/\epsilon\big)t
\nonumber\\[2mm]
 & \,+\, \epsilon^{-2}\chi(t)\beta|\beta'|^2 U''\big(\beta t/\epsilon\big)t^2,
 \label{mathbfu-thetatheta}
\end{align}
 where the expressions of $\beta'$, $\beta''$ have been given in \eqref{beta'}-\eqref{beta''}.
 Specially, for $|t|\leq M_1$, there hold
\begin{align}
 {\mathbf u} = \beta U\big(\beta t/\epsilon\big),
 \qquad
 {\mathbf u}_t = \epsilon^{-1}\beta\beta U'\big(\beta t/\epsilon\big),
 \qquad
  {\mathbf u}_{\theta} = \beta' U\big(\beta t/\epsilon\big)
 \,+\, \epsilon^{-1}\beta \beta' U'\big(\beta t/\epsilon\big)t,
 \label{mathbfu-t1}
\end{align}
 and
\begin{align}
 {\mathbf u}_{tt} =
 \epsilon^{-2}\beta\beta^2 U''\big(\beta t/\epsilon\big),
 \label{mathbfu-tt1}
\end{align}
\begin{align}
 {\mathbf u}_{t\theta} =&\epsilon^{-1}\beta'\beta\, U'\big(\beta t/\epsilon\big)
 \,+\,\epsilon^{-1}\beta \beta' U'\big(\beta t/\epsilon\big)
 \,+\,\epsilon^{-2}\beta\beta \beta' U''\big(\beta t/\epsilon\big)t,
 \label{mathbfu-ttheta1}
\end{align}
\begin{align}
 {\mathbf u}_{\theta\theta}
 \,=\, &\beta''U\big(\beta t/\epsilon\big)
 \,+\, \frac{2}{\epsilon}\beta'\beta' U'\big(\beta t/\epsilon\big)t
 \,+\, \frac{1}{\epsilon^2}\beta|\beta' |^2 U''\big(\beta t/\epsilon\big)t^2
 \,+\, \frac{1}{\epsilon}\beta\beta'' U'\big(\beta t/\epsilon\big)t.
 \label{mathbfu-thetatheta1}
\end{align}

\medskip
 By setting
$$
{\mathbf t}=\beta t/\epsilon,
 \quad\mbox{ i.e. }
 t=\epsilon {\mathbf t}/\beta,
$$
 we have that for $|t|\leq M_1$, there hold
\begin{align}\label{v1-t999}
 {\mathbf u} = \beta U({\mathbf t}),
 \qquad
 {\mathbf u}_t = \epsilon^{-1}\beta\beta U'({\mathbf t}),
  \qquad
   {\mathbf u}_{\theta} = \beta' U({\mathbf t})
 \,+\, \beta^{-1}\beta \beta' U'({\mathbf t}){\mathbf t},
\end{align}
 and
\begin{align}
 {\mathbf u}_{tt} =
 \epsilon^{-2}\beta\beta^2 U''({\mathbf t}),
   \qquad
 {\mathbf u}_{t\theta} =&\epsilon^{-1}\beta'\beta\, U'({\mathbf t})
 \,+\,\epsilon^{-1}\beta \beta' U'({\mathbf t})
 \,+\,\epsilon^{-1}\beta \beta' U''({\mathbf t}){\mathbf t},
 \label{v1-ttheta999}
\end{align}
\begin{align}
 {\mathbf u}_{\theta\theta}
 \,=\, &\beta''U({\mathbf t})
 \,+\, 2\frac{\beta'}{\beta}\beta' U'({\mathbf t}){\mathbf t}
 \,+\, \frac{\beta}{\beta^2} |\beta' |^2 U''({\mathbf t}){\mathbf t}^2
 \,+\, \beta^{-1}\beta\beta'' U'({\mathbf t}){\mathbf t}.
 \label{v1-thetatheta999}
\end{align}

\medskip
 Recall the formulas in \eqref{mathbfu-t1}-\eqref{mathbfu-thetatheta1} and the constants in \eqref{varrho0}-\eqref{varrho-additional}, we first compute
\begin{align}
 \int_{{\mathbb R}^2} |{\mathbf u}|^2 \,{\mathrm d} y
 &\,=\,
 \int_0^{\ell} \int_{-M_2}^{M_2}
 \Big[\,
 \chi(t)\beta\, U\big(\beta t/\epsilon\big)
 \, \Big]^2 (1+\kappa t )\,{\mathrm d}t{\mathrm d}\theta
\nonumber\\[2mm]
 &\,=\,
 \int_0^{\ell} \int_{-\frac{M_2\beta}{\epsilon}}^{\frac{M_2\beta}{\epsilon}}
 \Big[\, \chi\Big(\frac{\epsilon}{\beta}{\mathbf t}\Big)\beta\, U({\mathbf t})\, \Big]^2
 \big(1+\epsilon\kappa\beta^{-1}{\mathbf t}\big) \frac{\epsilon}{\beta} \,{\mathrm d}{\mathbf t}{\mathrm d}\theta
\nonumber\\[2mm]
 & \,=\, \big(1+o(1)\big)\epsilon\int_0^{\ell} \frac{\beta^2}{\beta} {\mathrm d}\theta \int_{\mathbb R} |U({\mathbf t})|^2 \,{\mathrm d}{\mathbf t}
\nonumber\\[2mm]
 & \,=\,
 \epsilon\, \varrho_0\big(1+o(1)\big)\int_0^{\ell} \beta {\mathrm d}\theta,
 \label{v1-L2norm-0}
\end{align}
 and similarly
\begin{align}
 \int_{{\mathbb R}^2} |\beta {\mathbf u}|^2 \,{\mathrm d} y
 & \,=\,
 \epsilon\, \varrho_0\big(1+o(1)\big)\int_0^{\ell} \beta^3 {\mathrm d}\theta,
 \label{v1-L2norm-1}
\end{align}
\begin{align}
 \int_{{\mathbb R}^2} |\beta^2 {\mathbf u}|^2 \,{\mathrm d} y
 & \,=\,
 \epsilon\, \varrho_0\big(1+o(1)\big)\int_0^{\ell} \beta^5 {\mathrm d}\theta.
 \label{v1-L2norm-2}
\end{align}
 On the other hand, there hold
\begin{align}
 \int_{{\mathbb R}^2} | \beta{\mathbf u}_t|^2 \,{\mathrm d} y
 &\,=\,
 \int_0^{\ell} \int_{-M_2}^{M_2} \beta^2
 \Big[\,
 \chi'(t)\beta U\big(\beta t/\epsilon\big)
 \,+\,
 \epsilon^{-1}\chi(t)\beta\beta U'\big(\beta t/\epsilon\big)
 \,\Big]^2 (1+\kappa t )\,{\mathrm d}t{\mathrm d}\theta
\nonumber\\[2mm]
 &\,=\,
 \big(1+o(1)\big)\epsilon^{-2}\int_0^{\ell} \int_{-\frac{M_2\beta}{\epsilon}}^{\frac{M_2\beta}{\epsilon}}
 \beta^2
 \Big[\, \beta\beta\, U'({\mathbf t})\, \Big]^2
 \big(1+\epsilon\kappa\beta^{-1}{\mathbf t}\big)\frac{\epsilon}{\beta}{\mathrm d}{\mathbf t}{\mathrm d}\theta
\nonumber\\[2mm]
 & \,=\, \epsilon^{-1}\big(1+o(1)\big)\int_0^{\ell} \beta^4\beta {\mathrm d}\theta \int_{\mathbb R} |U'({\mathbf t})|^2 \,{\mathrm d}{\mathbf t}
\nonumber\\[2mm]
 & \,=\,
 \epsilon^{-1}\varrho_1\big(1+o(1)\big) \int_0^{\ell} \beta^5 {\mathrm d}\theta,
 \label{v1-s-L2norm}
\end{align}
 and
\begin{align}
 \int_{{\mathbb R}^2} |{\mathbf u}_{\theta}|^2 \,{\mathrm d} y
 &\,=\,
 \int_0^{\ell} \int_{-M_2}^{M_2}
 \Big[\,
 \chi(t)\beta' U\big(\beta t/\epsilon\big)
 \,+\,
 \epsilon^{-1} \chi(t)\beta \beta' U'\big(\beta t/\epsilon\big)t
 \, \Big]^2 (1+\kappa t )\,{\mathrm d}t{\mathrm d}\theta
\nonumber\\[2mm]
 &\,=\,
 \int_0^{\ell} \int_{-\frac{M_2\beta}{\epsilon}}^{\frac{M_2\beta}{\epsilon}}
 \Big[\,\chi\Big(\frac{\epsilon}{\beta}{\mathbf t}\Big)\beta'U({\mathbf t})\Big]^2\,
 \big(1+\epsilon\kappa\beta^{-1}{\mathbf t}\big)\frac{\epsilon}{\beta} \,{\mathrm d}{\mathbf t}{\mathrm d}\theta
\nonumber\\[2mm]
 &\quad\,+\,\int_0^{\ell} \int_{-\frac{M_2\beta}{\epsilon}}^{\frac{M_2\beta}{\epsilon}}
 \Big[\,\chi\Big(\frac{\epsilon}{\beta}{\mathbf t}\Big)\beta\beta^{-1} \beta' U'({\mathbf t}){\mathbf t}\Big]^2\,
 \big(1+\epsilon\kappa\beta^{-1}{\mathbf t}\big)\frac{\epsilon}{\beta} \,{\mathrm d}{\mathbf t}{\mathrm d}\theta
\nonumber\\[2mm]
 &\quad\,+\,\int_0^{\ell} \int_{-\frac{M_2\beta}{\epsilon}}^{\frac{M_2\beta}{\epsilon}}
 2\Big[\,\chi\Big(\frac{\epsilon}{\beta}{\mathbf t}\Big)\,\Big]^2 \beta\beta^{-1} \beta'\beta' U({\mathbf t})U'({\mathbf t}){\mathbf t}\,
 \big(1+\epsilon\kappa\beta^{-1}{\mathbf t}\big)\frac{\epsilon}{\beta} \,{\mathrm d}{\mathbf t}{\mathrm d}\theta
\nonumber\\[2mm]
 & \,=\, \epsilon \big(1+o(1)\big)\int_0^{\ell}
 \frac{\, |\beta'|^2\, }{\beta} {\mathrm d}\theta
 \int_{\mathbb R} |U({\mathbf t})|^2 \,{\mathrm d}{\mathbf t}
 \,+\,
 \epsilon \big(1+o(1)\big)\int_0^{\ell}
 \frac{\, |\beta'|^2\, }{\beta} {\mathrm d}\theta
 \int_{\mathbb R} |U'({\mathbf t})|^2 \,{\mathrm d}{\mathbf t}
\nonumber\\[2mm]
 &\quad\,+\,
 \epsilon \big(1+o(1)\big)\int_0^{\ell}
 \frac{\, \beta\beta'\beta'\, }{\beta^2} {\mathrm d}\theta
 \int_{\mathbb R} 2{\mathbf t}U({\mathbf t})U'({\mathbf t}) \,{\mathrm d}{\mathbf t}
\nonumber\\[2mm]
 & \,=\,
 \epsilon \big(1+o(1)\big)
 \varrho_1\int_0^{\ell} \frac{\, |\beta'|^2\, }{\beta} {\mathrm d}\theta.
 \label{v1-z-L2norm-0}
\end{align}
 The estimates in \eqref{v1-z-L2norm-0} and \eqref{beta-bound} will give
\begin{align}
 \int_{{\mathbb R}^2} |\beta {\mathbf u}_{\theta}|^2 \,{\mathrm d} y
 \,\leq\, \epsilon\, C \int_0^{\ell} \beta^3 {\mathrm d}\theta.
 \label{v1-z-L2norm-1}
\end{align}
 Note that the validity of \eqref{beta-bound} relies on the assumptions in \eqref{W_theta-control}.
 Similarly, we have
\begin{align}
 \int_{{\mathbb R}^2} |{\mathbf u}_{t\theta}|^2 \,{\mathrm d} y
 \,\leq\, \epsilon\, C \int_0^{\ell} \beta^3 {\mathrm d}\theta.
 \label{v1-zt-L2norm-0}
\end{align}

\medskip
 Furthermore, we have
\begin{align}
 \int_{{\mathbb R}^2} |{\mathbf u}_{\theta\theta}|^2 \,{\mathrm d} y
 \,=\, &
 \int_0^{\ell} \int_{-M_2}^{M_2}
 \Big[\,
 \chi(t)\beta''U\big(\beta t/\epsilon\big)
 \,+\, 2 \epsilon^{-1}\chi(t)\beta'\beta' U'\big(\beta t/\epsilon\big)t
\nonumber\\[2mm]
 &\qquad
 \,+\, \epsilon^{-1}\chi(t)\beta\beta'' U'\big(\beta t/\epsilon\big)t
 \,+\, \epsilon^{-2}\chi(t)\beta|\beta'|^2 U''\big(\beta t/\epsilon\big)t^2
 \, \Big]^2 (1+\kappa t )\,{\mathrm d}t{\mathrm d}\theta
\nonumber\\[2mm]
 \,=\,& \big(1+o(1)\big)
 \int_0^{\ell} \int_{-\frac{M_2\beta}{\epsilon}}^{\frac{M_2\beta}{\epsilon}}
 \Big[\,\beta''U({\mathbf t})+\beta^{-1}(2\beta'\beta'+\beta\beta'')U'({\mathbf t}){\mathbf t}
\nonumber\\[2mm]
 &\qquad\qquad\qquad\qquad\qquad+\frac{\beta}{\beta^2} |\beta'|^2 U''({\mathbf t}){\mathbf t}^2\,\Big]^2
 \,\big(1+\epsilon\kappa\beta^{-1}{\mathbf t}\big)\frac{\epsilon}{\beta}{\mathrm d}{\mathbf t}{\mathrm d}\theta
\nonumber\\[2mm]
 \,=\, &\epsilon \big(1+o(1)\big)\int_0^{\ell}
 \frac{\, |\beta''|^2\, }{\beta} {\mathrm d}\theta\int_{\mathbb R} |U({\mathbf t})|^2 \,{\mathrm d}{\mathbf t}
\nonumber\\[2mm]
 &\,+\,
 \epsilon \big(1+o(1)\big)\int_0^{\ell}
 \frac{\, (2\beta'\beta'+\beta\beta'')^2\, }{\beta^3} {\mathrm d}\theta
 \int_{\mathbb R} |{\mathbf t}U'({\mathbf t})|^2 \,{\mathrm d}{\mathbf t}
\nonumber\\[2mm]
 &\,+\,
 \epsilon \big(1+o(1)\big)\int_0^{\ell}
 \frac{\, \beta^2|\beta'|^4\, }{\beta^5} {\mathrm d}\theta
 \int_{\mathbb R} |{\mathbf t}^2U''({\mathbf t})|^2 \,{\mathrm d}{\mathbf t}
\nonumber\\[2mm]
 &\,+\,
 \epsilon \big(1+o(1)\big)\int_0^{\ell}
 \frac{\, \beta''(2\beta'\beta'+\beta\beta'')\, }{\beta^2} {\mathrm d}\theta
 \int_{\mathbb R} 2{\mathbf t}U({\mathbf t})U'({\mathbf t}) \,{\mathrm d}{\mathbf t}
\nonumber\\[2mm]
 &\,+\,
 \epsilon \big(1+o(1)\big)\int_0^{\ell}
 \frac{\, 2\beta''\beta|\beta'|^2\, }{\beta^3} {\mathrm d}\theta
 \int_{\mathbb R} {\mathbf t}^2U({\mathbf t})U''({\mathbf t}) \,{\mathrm d}{\mathbf t}
\nonumber\\[2mm]
 &\,+\,
 \epsilon \big(1+o(1)\big)\int_0^{\ell}
 \frac{\, (2\beta'\beta'+\beta\beta'')\beta|\beta'|^2\, }{\beta^4} {\mathrm d}\theta\int_{\mathbb R} 2{\mathbf t}^3U'({\mathbf t})U''({\mathbf t}) \,{\mathrm d}{\mathbf t}
\nonumber\\[2mm]
 \,=\,
 &\epsilon \big(1+o(1)\big)
 \Bigg[
 \varrho_0\int_0^{\ell}\frac{\, |\beta''|^2\, }{\beta} {\mathrm d}\theta
 \,+\,
 \varrho_2\int_0^{\ell}\frac{\, (2\beta'\beta'+\beta\beta'')^2\, }{\beta^3} {\mathrm d}\theta
 \,+\,
 \varrho_4\int_0^{\ell}\frac{\,|\beta'|^4\, }{\beta^3} {\mathrm d}\theta
\nonumber\\[2mm]
 &\qquad\qquad\quad\
 \,-\,
 \varrho_0\int_0^{\ell}\frac{\, 2\beta''(2\beta'\beta'+\beta\beta'')\, }{\beta^2} {\mathrm d}\theta
 \,+\,
 (\varrho_0-\varrho_2)\int_0^{\ell}\frac{\, 2\beta''|\beta'|^2\, }{\beta^2} {\mathrm d}\theta
\nonumber\\[2mm]
 &\qquad\qquad\quad\
 \,-\,
 3\varrho_2\int_0^{\ell}\frac{\, 2(2\beta'\beta'+\beta\beta'')|\beta'|^2\, }{\beta^3} {\mathrm d}\theta
 \Bigg].
 \label{v1-zz-L2norm}
\end{align}
 The estimates in \eqref{v1-zz-L2norm} and \eqref{beta-bound} will give
\begin{align}
 \int_{{\mathbb R}^2} |\beta {\mathbf u}_{\theta\theta}|^2 \,{\mathrm d} y
 \,\leq\,
 \epsilon\,C \int_0^{\ell}\beta^3 {\mathrm d}\theta.
 \label{v1-zz-L2norm-1}
\end{align}


\subsection{ The estimates  of the perturbation term} \label{Section4.3}\


By recall the decomposition of $u_\epsilon={\mathbf u}+\phi$ in \eqref{u-decomposition} and the assumptions in \eqref{mathbfu}-\eqref{smallperturbation1},
we first provide the following lemma.
\begin{lemma}\label{Lemma4-3-2026}
We have
\begin{align}\label{linftyofphittheta}
\sup_{\theta \in [0, \ell],\, t\in [-M, M]}  |\phi (t, \theta)|^2
=  \frac{O\big(\epsilon^{2\varrho}\big)}{\left|\displaystyle\int_0^{\ell}\beta^5\,{\mathrm d}\theta \right|^2}
\Bigg[ \int_0^{\ell} \beta^3 {\mathrm d}\theta
\ +\
\int_0^{\ell}\beta {\mathrm d}\theta \Bigg]
=
\frac{O\big(\epsilon^{2\varrho}\big)}{\displaystyle\int_0^{\ell}\beta^5\,{\mathrm d}\theta},
\end{align}
where $M<M_2$ with $M_2$ given in \eqref{cut-off000000}, and $\varrho\in (0, 1)$ is given in \eqref{length-constraint}.
\end{lemma}

\begin{proof}
We  first claim:
{\it for any $t\in [-M, M],  $ there holds
\begin{align}   \label{4-37-1-23}
       \sup_{\theta \in [0, \ell] }  |\phi (t, \theta)|^2 & \le   C \int_0^{\ell} |\phi (t, \theta)|^2 {\mathrm d} \theta + C  \int_0^{\ell} |\p_\theta\phi (t, \theta)|^2 {\mathrm d} \theta,
\end{align}
    where $C$ is independent of $t$.
}

\smallskip
\noindent In fact,    if
\begin{align*}
\sup_{\theta \in [0, \ell] }  |\phi (t, \theta)|^2 \le  2 \inf_{\theta \in [0, \ell] } |\phi (t, \theta)|^2,
\end{align*}
  then
\begin{align}  \label{1-22-2026}
  \sup_{\theta \in [0, \ell] }  |\phi (t, \theta)|^2  \le 2 \frac 1{\ell}  \int_0^{\ell} |\phi (t, \theta)|^2 {\mathrm d} \theta \le C\int_0^{\ell} |\phi (t, \theta)|^2 {\mathrm d} \theta.
\end{align}
  The last inequality  in \eqref{1-22-2026} holds due to Lemma \ref{Lemma4.1}.
We then consider the case
\begin{align*}
\sup_{\theta \in [0, \ell] }  |\phi (t, \theta)|^2 >  2 \inf_{\theta \in [0, \ell] } |\phi (t, \theta)|^2.
\end{align*}
  By denoting $\theta_0 \in [0, \ell]$ the point such that
\begin{align*}
    |\phi (t, \theta_0)|^2   \le \frac 54   \inf_{\theta \in [0, \ell] } |\phi (t, \theta)|^2,
\end{align*}
we obtain
\begin{align}  \label{4-39-1-23}
   \sup_{\theta \in [0, \ell] }  |\phi (t, \theta)|^2 & \le     |\phi (t, \theta_0)|^2   + 2 \int_0^{\ell} |\phi (t, \theta)|\, |\p_\theta\phi (t, \theta)| {\mathrm d} \theta\nonumber
    \\[2mm]
 \leq&   \frac 58   \sup_{\theta \in [0, \ell] } |\phi (t, \theta)|^2+     C \int_0^{\ell} |\phi (t, \theta)|^2 {\mathrm d} \theta +  C\int_0^{\ell} |\p_\theta\phi (t, \theta)|^2 {\mathrm d} \theta.
\end{align}
   Combining \eqref{1-22-2026} and \eqref{4-39-1-23},  we get  \eqref{4-37-1-23}.

   Next, we will prove
\begin{align} \label{4-4026}
\sup_{t\in [-M, M] }     \int_0^{\ell}|\phi(t, \theta)|^2 {\mathrm d} \theta
  \leq
C\e^{-1} \int_{-M}^{M}  \int_0^{\ell}|\phi(t, \theta)|^2 {\mathrm d} \theta{\mathrm d} t
\ +\
C\e \int_{-M}^{M}  \int_0^{\ell} |\p_t\phi(t, \theta)|^2 {\mathrm d} \theta{\mathrm d} t,
\end{align}
   and
\begin{align} \label{4-4126}
 \sup_{t\in [-M, M] }     \int_0^{\ell} |\p_\theta\phi(t, \theta)|^2 {\mathrm d} \theta
\leq
C\e^{-1}  \int_{-M}^{M} \int_0^{\ell} |\p_\theta\phi(t, \theta)|^2 {\mathrm d} \theta{\mathrm d} t
\ +\
C\e  \int_{-M}^{M}  \int_0^{\ell} |\p_{t\theta}\phi(t, \theta)|^2 {\mathrm d} \theta{\mathrm d} t.
\end{align}

\smallskip
\noindent
In fact,  if
\begin{align*}
\sup_{t\in [-M, M] }        \int_0^{\ell} |\phi(t, \theta)|^2 {\mathrm d} \theta  \le 2  \inf_{t \in [-M, M] }     \int_0^{\ell} |\phi(t, \theta)|^2 {\mathrm d} \theta,
\end{align*}
  then
\begin{align*}
\sup_{t\in [-M, M] }        \int_0^{\ell} |\phi(t, \theta)|^2 {\mathrm d} \theta  \le  C \int_{-M}^{M} \int_0^{\ell} |\phi(t, \theta)|^2 {\mathrm d} \theta\, {\mathrm d} t.
\end{align*}
  On the other hand,  we consider the case
\begin{align*}
\sup_{t\in [-M, M] }        \int_0^{\ell} |\phi(t, \theta)|^2 {\mathrm d} \theta  > 2  \inf_{t\in (- M,  M) }     \int_0^{\ell} |\phi(t, \theta)|^2 {\mathrm d} \theta.
\end{align*}
   By denoting $t_0 \in  [-M, M]$ the point such that
\begin{align*}
 \int_0^{\ell} |\phi(t_0, \theta)|^2 {\mathrm d} \theta   \le  \frac 54\inf_{t \in [-M, M] }     \int_0^{\ell} |\phi(t, \theta)|^2 {\mathrm d} \theta,
\end{align*}
we derive
\begin{align*}
  \sup_{t\in [-M, M] }     \int_0^{\ell} |\phi(t, \theta)|^2 {\mathrm d} \theta & \le     \int_0^{\ell} |\phi(t_0, \theta)|^2 {\mathrm d} z + 2 \int_{-M}^{M}   \int_0^{\ell}|\phi(t, \theta)|\, |\p_t\phi(t, \theta)| {\mathrm d} \theta{\mathrm d} t \nonumber
    \\[2mm]
 & \leq  \frac 58  \sup_{t\in [-M, M] }     \int_0^{\ell} |\phi(t, \theta)|^2 {\mathrm d} \theta
 \nonumber
    \\[2mm] & \quad
 +  C \e^{-1} \int_{-M}^{M}  \int_0^{\ell} |\phi(t, \theta)|^2 {\mathrm d} \theta{\mathrm d} t+   C\e \int_{-M}^{M}    \int_0^{\ell} |\p_t\phi(t, \theta)|^2 {\mathrm d} \theta{\mathrm d} t.
\end{align*}
   Then we get
\begin{align*}
 \sup_{t\in [-M, M] }     \int_0^{\ell} |\phi(t, \theta)|^2 {\mathrm d} \theta & \le   C \e^{-1}  \int_{-M}^{M}   \int_0^{\ell} |\phi(t, \theta)|^2 {\mathrm d} \theta {\mathrm d} t+ C\e  \int_{-M}^{M}  \int_0^{\ell} |\p_t\phi(t, \theta)|^2 {\mathrm d} \theta{\mathrm d} t
\end{align*}
 and thus \eqref{4-4026} holds.
 Similarly, we can prove \eqref{4-4126}.

Combining \eqref{4-37-1-23}, \eqref{4-4026} and \eqref{4-4126}, we finally get
\begin{align}  \label{4-33-2026}
\sup_{\theta \in [0, \ell],\, t\in [-M, M]}  |\phi (t, \theta)|^2
 \leq&  C \e^{-1}  \int_{-M}^{M}   \int_0^{\ell} |\phi(t, \theta)|^2 {\mathrm d} \theta {\mathrm d} t+ C\e  \int_{-M}^{M}  \int_0^{\ell} |\p_t\phi(t, \theta)|^2 {\mathrm d} \theta{\mathrm d} t
  \nonumber
    \\[2mm]
  & \quad +   C \e^{-1}  \int_{-M}^{M}   \int_0^{\ell} |\p_{\theta} \phi(t, \theta)|^2 {\mathrm d} \theta {\mathrm d} t+ C\e  \int_{-M}^{M}  \int_0^{\ell} |\p_{t\theta} \phi(t, \theta)|^2 {\mathrm d} \theta{\mathrm d} t.
\end{align}
By substituting \eqref{smallperturbation0}-\eqref{smallperturbation1}
as well as \eqref{v1-L2norm-1}, \eqref{v1-s-L2norm}, \eqref{v1-z-L2norm-1} and \eqref{v1-zt-L2norm-0} into \eqref{4-33-2026}, we can get
\begin{align*}
 \sup_{\theta \in [0, \ell],\, t\in [-M, M]}  |\phi (t, \theta)|^2
&   = \,
\frac{O\big(\epsilon^{2\varrho}\big)}{\left[\displaystyle\int_0^{\ell}\beta^5\,{\mathrm d}\theta \right]^2}
\Bigg[ \e^{-1}  \int_{{\mathbb R}^2} |{\mathbf u}|^2 \,{\mathrm d} y
\ +\
\e^{-1}   \int_{{\mathbb R}^2} |{\mathbf u}_{\theta}|^2 \,{\mathrm d} y
      \nonumber
    \\[2mm]
    & \qquad\qquad\qquad\qquad
\ +\  \e  \int_{{\mathbb R}^2} |{\mathbf u}_t|^2 \,{\mathrm d} y
\ +\ \e \int_{{\mathbb R}^2} |{\mathbf u}_{t\theta}|^2 \,{\mathrm d} y\Bigg]
           \nonumber
    \\[2mm]
    & =\,  \frac{O\big(\epsilon^{2\varrho}\big)}{\left[\displaystyle\int_0^{\ell}\beta^5\,{\mathrm d}\theta \right]^2} \Bigg[ \int_0^{\ell} \beta^3 {\mathrm d}\theta +   \int_0^{\ell}\beta {\mathrm d}\theta \Bigg]
    = \frac{O\big(\epsilon^{2\varrho}\big)}{\displaystyle\int_0^{\ell}\beta^5\,{\mathrm d}\theta}.
\end{align*}
Therefore, \eqref{linftyofphittheta}  holds true.
\end{proof}

We here also provide the following lemma on the decay property of the perturbation term $\phi$.

\begin{lemma}\label{Lemma4.2}
 Then there exists two positive constants $M_3$ and $M_4$ ($\e M_3<M_4<M_1$ with $M_1$ in \eqref{cut-off000000}) independent of $\epsilon$ in such a way that
 for all ${\mathbf y}$ satisfying $\e M_3<{\mathrm{dist}} ({\bf y}, \Gamma_\epsilon) <M_4$,
 the following estimate is valid
\begin{align}
 &| \phi({\bf y}) | + | \p_t \phi({\bf y})| + | \p_{tt} \phi({\bf y}) |
 \leq
 C\exp {\Big(-\frac{\sigma_0\beta}{\epsilon}{\mathrm{dist}} ({\bf y}, \, \Gamma_\epsilon)\Big)},
 \label{estimateforphi}
\end{align}
where $\sigma_0$ is a positive constant independent of $\epsilon$.
\end{lemma}

\begin{proof}
To begin with, by the elliptic regularity theory \cite{Gilbarg}, we know that $u_\epsilon={\mathbf u}+\phi$ is a classic solution.
We substitute $u_\epsilon={\mathbf u}+\phi$ into \eqref{eq2} and then obtain
\begin{align}
\epsilon^2\Delta_{\bf y}\phi - \big[1+ \epsilon^2 W ({\bf y}) \big] \phi + 3 {\mathbf u}^2 \phi \,=\, - \tilde N_0 (\phi ) - \tilde E_0,
 \label{eqnphi}
\end{align}
 where
\begin{align}
 \label{eqnN0E0}
 \tilde N_0(\phi) \,=\, 3 {\mathbf u} \phi^2 + \phi^3,
 \qquad
 \tilde E_0\,=\, &\epsilon^2\Delta_{\bf y} {\mathbf u} - \big[1+ \epsilon^2 W ({\bf y}) \big] {\mathbf u} + {\mathbf u}^3.
\end{align}
Note that $\Delta_{\bf y}$ has local form in \eqref{Laplacian}, which can be written as
\begin{align}
\Delta_{\bf y}=\frac{\partial^2}{\partial t^2}+\frac{\partial^2}{\partial \theta^2}+{\mathbb B}
\label{laplacian1}
\end{align}
where
\begin{align}
{\mathbb B}=
 \frac{\kappa}{1+\kappa t} \frac{\partial }{\partial t}
\ -\
\frac{ t \kappa'}{\big(1+ \kappa t\big)^3 }\frac{\partial}{\partial \theta}
\ -\
\frac{\,-2\kappa t-\kappa^2t^2\,}{\big(1+\kappa t\big)^2 }\frac{\partial^2}{\partial \theta^2}.
\end{align}
The assumptions in \eqref{curvature bounded} imply that the coefficients in ${\mathbb B}$ are three smooth bounded functions on the following domain
$$
\Big\{\, (t, \theta)\, :\, |t|<M_0, \quad \theta\in [0, \ell)\, \Big\}
$$
provided that the constant $M_0$ in \eqref{fermi} is chosen sufficiently small.
In order to derive the estimate,
we shall use the local forms of \eqref{eqnphi} and \eqref{eqnN0E0},
 which will be formulated in the sequel.

Note that the function $\mathbf u$ has the expression in \eqref{mathbfu}.
We obtain that the term $\tilde E_0$ has the local form
\begin{align*}
 \tilde E_0
 \,=\, &- \epsilon^2\big[W (t, \theta) - W(0, \theta) \big] {\mathbf u}
 + \epsilon^2{\mathbf u}_{\theta\theta}
 + \epsilon^2{\mathbb B} ({\mathbf u}),
\quad\mbox{if } |t|<M_1,
\end{align*}
due to \eqref{mathbfu-tt1}.
More analysis will be provided to get the estimates of the error components.

\smallskip
\noindent$\clubsuit$
 The Taylor expansion will give
\begin{align}
 W (t, \theta) - W(0, \theta)=W_t\big(\hbar(\theta) t, \theta\big)\,t,
 \label{WTayloer}
\end{align}
 where $\hbar$ is a function with $|\hbar|\leq 1$.
 We then use \eqref{W_t-control} to derive that
\begin{align*}
 \Big|- \epsilon^2\big[W (t, \theta) - W(0, \theta) \big] {\mathbf u}\Big|
 &=\epsilon^2 \chi(t) \,\big|W_t\big(\hbar(\theta) t, \theta\big)\,t\big|\, \beta U\big(\beta t/\epsilon\big)
 \\[2mm]
 &\leq C \,\beta^3(\theta)\, e^{- |\beta(\theta) t/\epsilon|}, \quad \text{if }\, |t|\leq M_0,
\end{align*}
due to \eqref{U-asymp}.

\smallskip
\noindent$\clubsuit$
 On the other hand, we use \eqref{mathbfu-t1}-\eqref{mathbfu-thetatheta1} to derive that, for  $|t|<M_1$
\begin{align*}
\epsilon^2{\mathbf u}_{\theta\theta}
 + \epsilon^2{\mathbb B} ({\mathbf u})
&=\epsilon^2\Bigg\{\beta''U\big(\beta t/\epsilon\big)
 + \, \epsilon^{-1} 2|\beta'|^2 t U'\big(\beta t/\epsilon\big)
 + \, \epsilon^{-2}\beta|\beta'|^2t^2 U''\big(\beta t/\epsilon\big)
 + \, \epsilon^{-1}\beta\beta'' t U'\big(\beta t/\epsilon\big)
 \Bigg\}
 \nonumber\\[2mm]
 &\quad
\,+\,\epsilon \frac{\kappa\beta^2}{1+\kappa t}  U'\big(\beta t/\epsilon\big)
\,-\,\epsilon^2
\frac{ t \kappa'}{\big(1+ \kappa t\big)^3 }
 \Bigg[
\beta' U\big(\beta t/\epsilon\big)
 +\epsilon^{-1}\beta\beta' t U'\big(\beta t/\epsilon\big)
 \Bigg]
\nonumber\\[2mm]
 &\quad\,-\,
\epsilon^2\frac{\,2\kappa t+\kappa^2t^2\,}{\big(1+\kappa t\big)^2 }
 \Bigg\{\beta''U\big(\beta t/\epsilon\big)
 + \, \epsilon^{-1} 2|\beta'|^2 t U'\big(\beta t/\epsilon\big)
 + \, \epsilon^{-2}\beta|\beta'|^2t^2 U''\big(\beta t/\epsilon\big)
\nonumber\\[2mm]
 &\qquad\qquad\qquad \qquad \qquad
 + \, \epsilon^{-1}\beta\beta'' t U'\big(\beta t/\epsilon\big)
 \Bigg\},
\end{align*}
where $\kappa$ and $\kappa'$ are uniformly bounded due to \eqref{curvature bounded}.
By the results in \eqref{beta-bound} (which is a direct consequence of \eqref{W_theta-control}) and also \eqref{U-asymp},  we get
\begin{align*}
|\tilde E_0(t, \theta)|
\,\leq\,
 C \beta^3(\theta) e^{-|\beta(\theta)t/\epsilon|}, \quad
\frac {\e M_3} {\beta} \le  |t|\leq M_1.
\end{align*}

\smallskip
 Using \eqref{laplacian1}, we can derive from \eqref{eqnphi} that
\begin{align} \label{odefortildephi}
\epsilon^2\phi_{tt} + \epsilon^2\phi_{\theta\theta} - \Big\{\big[1+ \epsilon^2 W (t, \theta) \big] - 3 {\mathbf u}^2\Big\}\phi
\,=\, - \tilde N_0(\phi) - \tilde E_0
 - \epsilon^2{\mathbb B}(\phi),
\end{align}
where $(t, \theta) \in (-M_1, M_1)\times [0, \ell)$.
Recalling \eqref{WTayloer},
we know that there exist positive constants $\sigma_2$, $M_3$ and $M_4$ independent of $\epsilon$ with the property $\e M_3< M_4<M_1$ for $M_1$ in \eqref{cut-off000000}, in such a way that
\begin{align*}
 \big[1+ \epsilon^2 W (t, \theta) \big] - 3 {\mathbf u}^2
 &=\big[1+ \epsilon^2 W(0, \theta) \big]
 \Bigg\{1 - 3 \big|U\big(\beta(\theta)t/\epsilon\big)\big|^2 \,+\, t \frac{\ \epsilon^2W_t\big(\hbar(\theta) t, \, \theta\big)\,}{1+ \epsilon^2 W(0, \theta)}
 \Bigg\}
 \\
 &\geq \sigma_2\big[1+ \epsilon^2 W(0, \theta) \big],
 \quad \text{if }\,  \frac {\e M_3} {\beta}\leq |t|\leq M_4,
\end{align*}
due to the assumptions in \eqref{W_t-control}.
\medskip
 By using  the result provided in Lemma \ref{Lemma4-3-2026}, we have that
\begin{align} \label{linftyepsilonphi}
 \|\phi \|_{L^\infty (\{\frac {\e M_3} {\beta}\leq |t|\leq M_4,\, \theta\in [0, \ell)\}) } \le C \e^{\varrho}, \quad \text{with}~\varrho\in (0, 1).
\end{align}
 Then, a standard barrier argument with a barrier function as
\begin{align*}
 \bar \phi \,=\, \mathbf{M}  \|\phi \|_{L^\infty (\{\frac {\e M_3} {\beta}\leq |t|\leq M_4,\, \theta\in [0, \ell)\}) } \beta e^{-\sigma_3 |\beta t/\epsilon| },
\end{align*}
 where $\mathbf{M}$ is chosen large and $\sigma_3=\min\{\sigma_1, \sigma_2\}$, yields that
\begin{align}\label{55}
 |\phi |\leq \bar \phi\leq&C  \|\phi \|_{L^\infty (\{\frac {\e M_3} {\beta}\leq |t|\leq M_4,\, \theta\in [0, \ell)\}) }\beta e^{-\sigma_3 |\beta t/\epsilon| }\nonumber
 \\[2mm]
 \leq&  C \epsilon^\varrho \beta e^{-\sigma_3 |\beta t/\epsilon| }, \quad \text{when}\, \frac {\e M_3} {\beta}\leq |t|\leq M_4,
\end{align}
 where the constant $C$ is independent of $\epsilon$.   From \eqref{55}, it follows immediately that
 \begin{align}
 |\phi |\leq \bar \phi\leq&  C \epsilon^\varrho e^{-\hat{\sigma}_3 |\beta t/\epsilon| },    \quad \text{when}\, \e M_3\leq |t|\leq M_4.
 \end{align}

  By the elliptic estimate, we conclude that there exists a constant $\sigma_4>0$ such that
\begin{align*}
 |\phi_t| + |\phi_{tt} |\leq  C \,\epsilon^{\varrho}  \,e^{-\sigma_4 |\beta t/\epsilon| },
 \quad \text{when}\, \e M_3\leq |t|\leq M_4.
\end{align*}
 Therefore, we complete the proof of this lemma.
\end{proof}

\subsection{ The estimates of all terms in the identities}\label{Section4.4}\

We can take $M>0$ such that $M_3<M<M_4$ with $M_3$ and $M_4$ given in Lemma \ref{Lemma4.2}.
 It is trivial to derive from \eqref{third identity} to obtain
\begin{align}
 &\,-\,\epsilon^2\int_0^\ell \varphi \kappa\int_{-M}^{M} \frac{u_{\theta\theta} u}{(1+t\,\kappa)}\,{\mathrm d}t{\mathrm d}\theta
 \,+\, \epsilon^2 \int_0^\ell \varphi kk'\int_{-M}^{M} \frac{t\,u_\theta u}{(1+t\,\kappa)^2}\,{\mathrm d}t{\mathrm d}\theta
\nonumber\\[2mm]
 &
 \,-\,\epsilon^2\int_0^\ell \varphi \int_{-M}^{M} u_{\theta\theta} u_t\,{\mathrm d}t{\mathrm d}\theta
 \,+\, \epsilon^2 \int_0^\ell \varphi \kappa'\int_{-M}^{M} \frac{t\,u_\theta u_t}{(1+t\,\kappa)} \,{\mathrm d}t{\mathrm d}\theta
 \,-\,
 \frac{1}{2}\epsilon^2\int_0^\ell \varphi \int_{-M}^{M} W_t|u|^2
 (1+t\,\kappa)^2\,{\mathrm d}t{\mathrm d}\theta
\nonumber\\[2mm]
 &\,+\,\int_0^\ell \varphi \kappa\frac{(1-p)}{p+1}\int_{-M}^{M} |u|^{p+1} (1+t\,\kappa)\,{\mathrm d}t{\mathrm d}\theta
 \,+\,
 \epsilon^2\int_0^\ell \varphi \kappa\int_{-M}^{M} |u_t|^2 (1+t\,\kappa)\,{\mathrm d}t{\mathrm d}\theta
\nonumber\\[2mm]
 &\,=\,
 \epsilon^2 \int_0^\ell\varphi\Bigg\{\kappa\big(1+\kappa t\big) u u_t \Big|_{t=-{M}}^{M}
 \ +\
 \frac{1}{2} \Big[ \big(1+\kappa t\big) u_t \Big]^2\, \Big|_{t=-{M}}^{M}\Bigg\} \,{\mathrm d}\theta
\nonumber\\[2mm]
 &\quad
 \ -\
 \int_0^\ell\varphi\Bigg\{\frac{1}{2}\big(1+\kappa t\big)^2 \big[1+ \epsilon^2 W\big] u^2\, \Big|_{t=-{M}}^{M}
 \ -\
 \frac 1 4 \big(1+\kappa t\big)^2 |u|^4\, \Big|_{t=-{M}}^{M}\Bigg\} \,{\mathrm d}\theta,
 \label{Identity51}
\end{align}
where $p=3$ and $\varphi(\theta)$ is any smooth function with $\varphi(0)=\varphi(\ell)$.
We will use the assumptions in {\bf{(A1)}}-{\bf{(A2)}} (see \eqref{curvature bounded}-\eqref{W_theta-control})
 and also the assumptions in \eqref{u-decomposition}-\eqref{smallperturbation1}
 together with the results in \eqref{v1-L2norm-0}-\eqref{v1-zz-L2norm-1} to estimate all the terms in the above identity.
In this subsection, we used $\|\cdot\|_2$ represent the norm $\|\cdot\|_{L^2(\{|t|\leq M\}\})}$.

\medskip
\noindent {\bf{Part 1.}}
 We first use \eqref{U-asymp} and \eqref{estimateforphi} to do the following analysis for the boundary terms in \eqref{Identity51}.

\smallskip
\noindent$\clubsuit$ Here is the first boundary term
\begin{align} \label{boundaryterm2}
 &\Bigg|\epsilon^2 \int_0^\ell\varphi\Bigg\{\kappa\big(1+\kappa t\big) u u_t \Big|_{t=-{M}}^{M}
 \ +\
 \frac{1}{2} \Big[ \big(1+\kappa t\big) u_t \Big]^2\, \Big|_{t=-{M}}^{M}\Bigg\} \,{\mathrm d}\theta\Bigg|
\nonumber\\[2mm]
 &\,\leq\, \epsilon^2\,C \|\varphi\kappa\|_{L^\infty(0,\ell)}\int_0^\ell\big(1+\kappa M\big)
 \Big|\beta \, U\big( \beta M/\epsilon\big) +\phi(M, \theta) \Big|
 \times\Big|\epsilon^{-1}\beta^2\, U'\big( \beta M/\epsilon\big) +\phi_t(M, \theta) \Big|
 \,{\mathrm d}\theta
\nonumber\\[2mm]
 &\quad\,+\, \epsilon^2\,C \|\varphi\kappa\|_{L^\infty(0,\ell)}\int_0^\ell\big(1-\kappa M\big)
 \Big|\beta \, U\big( -\beta M/\epsilon\big) +\phi(-M, \theta)\Big|
 \times\Big|\epsilon^{-1}\beta^2\, U'\big( -\beta M/\epsilon\big) +\phi_t(-M, \theta) \Big|
 \,{\mathrm d}\theta
\nonumber\\[2mm]
 &\quad\,+\, \epsilon^2\,C \|\varphi\kappa\|_{L^\infty(0,\ell)}\int_0^\ell\big(1+\kappa M\big)^2
 \Big[\epsilon^{-1}\beta^2\, U'\big( \beta M/\epsilon\big) +\phi_t(M, \theta) \Big]^2
 \,{\mathrm d}\theta
\nonumber\\[2mm]
 &\quad\,+\, \epsilon^2\,C \|\varphi\kappa\|_{L^\infty(0,\ell)}\int_0^\ell\big(1+\kappa M\big)^2
 \Big[\epsilon^{-1}\beta^2\, U'\big(-\beta M/\epsilon\big) +\phi_t(-M, \theta) \Big]^2
 \,{\mathrm d}\theta
\nonumber\\[2mm]
 &\,\leq\, \epsilon^2\,C \|\varphi\|_{L^\infty(0,\ell)}\int_0^\ell
 \Big|\frac{\epsilon}{M}e^{-M/(2\epsilon)} +e^{-\sigma_0M/(2\epsilon)} \Big|
 \times\Big|\frac{\beta}{M}e^{-M/(2\epsilon)} +e^{-\sigma_0M/(2\epsilon)} \Big|
 \,{\mathrm d}\theta
\nonumber\\[2mm]
 &\quad\,+\, \epsilon^2\,C \|\varphi\|_{L^\infty(0,\ell)}\int_0^\ell
 \Big[\frac{\beta}{M}e^{-M/(2\epsilon)} +e^{-\sigma_0M/(2\epsilon)} \Big]^2
 \,{\mathrm d}\theta
\nonumber\\[2mm]
 &\,\leq\, \epsilon^2 e^{-\sigma_0 M/\epsilon}\,C \|\varphi\|_{L^\infty(0,\ell)}\displaystyle\int_0^{\ell}\beta^2 {\mathrm d}\theta
 \,\leq\, \epsilon^{1+\varrho}\,C \|\varphi\|_{L^\infty(0,\ell)},
 \qquad\qquad \mbox{if}\ e^{-\sigma_0 M/\epsilon}\epsilon^{1-\varrho}\,\int_0^\ell \beta^2\,{\mathrm d}\theta\leq C,
\end{align}
 due to $\beta>1$. Note that the assumption in \eqref{length-constraint} has been used in the proof of \eqref{boundaryterm2}.

\smallskip
\noindent$\clubsuit$
 The second boundary term will be estimated in the following
\begin{align}
 &\Bigg|\ -\
 \int_0^\ell\varphi\Bigg\{\frac{1}{2}\big(1+\kappa t\big)^2 \big[1+ \epsilon^2 W(t, \theta)\big] u^2\, \Big|_{t=-{M}}^{M}
 \ -\
 \frac 1 4 \big(1+\kappa t\big)^2 |u|^4\, \Big|_{t=-{M}}^{M}\Bigg\} \,{\mathrm d}\theta\Bigg|
\nonumber\\[2mm]
 &\,\leq\, \epsilon^2\,C \|\varphi\|_{L^\infty(0,\ell)}\int_0^\ell\big(1+\kappa M\big)^2
 \big[1+ \epsilon^2 W(0, \theta)+\epsilon^2 MW_t\big(\hbar_1(\theta)M, \theta\big)\big]
 \Big[\beta \, U\big( \beta M/\epsilon\big) +\phi(M, \theta) \Big]^2
 \,{\mathrm d}\theta
\nonumber\\[2mm]
 &\quad\,+\, \epsilon^2\,C \|\varphi\|_{L^\infty(0,\ell)}\int_0^\ell\big(1-\kappa M\big)^2
 \big[1+ \epsilon^2 W(0, \theta)-\epsilon^2 MW_t\big(\hbar_2(\theta)M, \theta\big)\big]
 \Big[\beta \, U\big( -\beta M/\epsilon\big) +\phi(-M, \theta) \Big]^2
 \,{\mathrm d}\theta
\nonumber\\[2mm]
 &\quad\,+\, \epsilon^2\,C \|\varphi\|_{L^\infty(0,\ell)}\int_0^\ell\big(1+\kappa M\big)^2
 \Big[\beta \, U\big( \beta M/\epsilon\big) +\phi(M, \theta) \Big]^4
 \,{\mathrm d}\theta
\nonumber\\[2mm]
 &\quad\,+\, \epsilon^2\,C \|\varphi\|_{L^\infty(0,\ell)}\int_0^\ell\big(1-\kappa M\big)^2
 \Big[\beta \, U\big( -\beta M/\epsilon\big) +\phi(-M, \theta) \Big]^4
 \,{\mathrm d}\theta
\nonumber\\[2mm]
 &\,\leq\, \epsilon^2\,C \|\varphi\|_{L^\infty(0,\ell)}\int_0^\ell \beta^2
 \Big[\frac{\epsilon}{M}e^{-M/(2\epsilon)} +e^{-\sigma_0M/(2\epsilon)} \Big]^2
 \,{\mathrm d}\theta
\nonumber\\[2mm]
 &\,\leq\,
 \epsilon^2 e^{-\sigma_0 M/\epsilon}\,C \|\varphi\|_{L^\infty(0,\ell)}\displaystyle\int_0^{\ell}\beta^2 {\mathrm d}\theta
 \,\leq\,
 \epsilon^{1+\varrho}\,C \|\varphi\|_{L^\infty(0,\ell)},
 \qquad\qquad \mbox{if}\ e^{-\sigma_0 M/\epsilon}\epsilon^{1-\varrho}\,\int_0^\ell \beta^2\,{\mathrm d}\theta\leq C,
 \label{boundaryterm1}
\end{align}
 where $\hbar_1(\theta)$ and $\hbar_2(\theta)$ are two smooth functions with infinity norm less than 1.
 Note that the assumptions in \eqref{W_t-control} have been used in the proof of \eqref{boundaryterm1}.

\bigskip
\noindent {\bf{Part 2.}}
 Next, some easy estimates can be given in the following way.

\smallskip
\noindent$\clubsuit$
 Using \eqref{mathbfu-t1}-\eqref{mathbfu-thetatheta1}, we can obtain that
\begin{align*}
 &\Bigg|-\,\epsilon^2\int_0^\ell \varphi \kappa\int_{-M}^{M} \frac{{\mathbf u}_{\theta\theta}
 {\mathbf u}}{(1+t\,\kappa)}
 \,{\mathrm d}t{\mathrm d}\theta\Bigg|
\nonumber\\[2mm]
 &\,=\,\epsilon^2\Bigg|\int_0^\ell \varphi \kappa \displaystyle\int_{-\frac{M\beta}{\epsilon}}^{\frac{M\beta}{\epsilon}}
 \frac{\big(\beta''U
 \,+\, 2\frac{\beta'}{\beta}\beta' U'{\mathbf t}
 \,+\, \frac{\beta}{\beta^2} |\beta' |^2 U''{\mathbf t}^2
 \,+\, \beta^{-1}\beta\beta'' U'{\mathbf t}\big)\beta U}
 {\big(1+\epsilon \kappa \beta^{-1}{\mathbf t}\big)}
 \,\frac{\epsilon}{\beta} \,{\mathrm d}{\mathbf t}{\mathrm d}\theta\Bigg|
\nonumber\\[2mm]
 &\,=\,\epsilon^2\Bigg|\int_0^\ell \varphi \kappa \displaystyle\int_{-\frac{M\beta}{\epsilon}}^{\frac{M\beta}{\epsilon}}
 \frac{\epsilon \beta\beta''\beta^{-1}U^2({\mathbf t})}{\big(1+\epsilon \kappa \beta^{-1}{\mathbf t}\big)}
 \,{\mathrm d}{\mathbf t}{\mathrm d}\theta
 \,+\,
 \int_0^\ell \varphi \kappa \displaystyle\int_{-\frac{M\beta}{\epsilon}}^{\frac{M\beta}{\epsilon}}
 \frac{ 2\epsilon\beta\beta'\beta^{-2}\beta' U({\mathbf t})U'({\mathbf t}){\mathbf t}}{\big(1+\epsilon \kappa \beta^{-1}{\mathbf t}\big)}
 \,{\mathrm d}{\mathbf t}{\mathrm d}\theta
\nonumber\\[2mm]
 &\qquad\quad\,+\,\int_0^\ell \varphi \kappa \displaystyle\int_{-\frac{M\beta}{\epsilon}}^{\frac{M\beta}{\epsilon}}
 \frac{\epsilon \beta^2\beta^{-3}|\beta'|^2{\mathbf t}^2U({\mathbf t})U''({\mathbf t})}{\big(1+\epsilon \kappa \beta^{-1}{\mathbf t}\big)}
 \,{\mathrm d}{\mathbf t}{\mathrm d}\theta
\nonumber\\[2mm]
 &\qquad\quad\,+\,\int_0^\ell \varphi \kappa \displaystyle\int_{-\frac{M\beta}{\epsilon}}^{\frac{M\beta}{\epsilon}}
 \frac{\epsilon\beta^2\beta^{-2}\beta'' {\mathbf t} U({\mathbf t})U'({\mathbf t})}{\big(1+\epsilon \kappa \beta^{-1}{\mathbf t}\big)}
 \,{\mathrm d}{\mathbf t}{\mathrm d}\theta
 \Bigg|
\nonumber\\[2mm]
 &\,=\,\epsilon^3\big(1+o(1)\big)
 \Bigg|
 -\varrho_0\int_0^\ell \varphi \kappa \beta\beta''\beta^{-1}{\mathrm d}\theta
 \,+\,
 \varrho_0\int_0^\ell \varphi \kappa \beta\beta'\beta^{-2}\beta'{\mathrm d}\theta
\nonumber\\[2mm]
 &\qquad\qquad\qquad\quad
 \,+\,(\varrho_2-\varrho_0)\int_0^\ell \varphi \kappa \beta^2\beta^{-3}|\beta'|^2{\mathrm d}\theta
 \,-\, \varrho_0\int_0^\ell \varphi \frac12 \kappa \beta^2\beta^{-2}\beta''{\mathrm d}\theta \Bigg|
\nonumber\\[2mm]
 &\,\leq\, \epsilon^3\,C \|\varphi\|_{L^\infty(0,\ell)}\int_0^\ell \beta\,{\mathrm d}\theta
 \,\leq\, \epsilon^{1+\varrho}\,C \|\varphi\|_{L^\infty(0,\ell)} ,\qquad\qquad \mbox{if}\ \epsilon^{2-\varrho}\,\int_0^\ell \beta\,{\mathrm d}\theta\leq C
\end{align*}
 due to \eqref{beta-bound}, and
\begin{align*}
 &\Bigg|-\,\epsilon^2\int_0^\ell \varphi \kappa\int_{-M}^{M}
 \frac{\big({\mathbf u}_{\theta\theta}\phi+{\mathbf u}\phi_{\theta\theta}+\phi_{\theta\theta}\phi\big)} {(1+t\,\kappa)}
 \,{\mathrm d}t{\mathrm d}\theta\Bigg|
\nonumber\\[2mm]
 &\leq
 C\epsilon^2\|\varphi\kappa\|_{L^\infty(0,\ell)}
 \Big[ \|{\mathbf u}_{\theta\theta}\|_2\|\phi\|_2 +\|{\mathbf u}\|_2\|\phi_{\theta\theta}\|_2 +\|\phi_{\theta\theta}\|_2\|\phi\|_2\,\Big]
\nonumber\\[2mm]
 &\leq
 C\epsilon^2\|\varphi\|_{L^\infty(0,\ell)}\Bigg[ \|\phi\|_2\sqrt{\epsilon\int_0^{\ell}\beta {\mathrm d}\theta\,}
 \,+\, \|\phi_{\theta\theta}\|_2 \sqrt{\epsilon\int_0^{\ell}\beta {\mathrm d}\theta\,}
 \,\Bigg]
\nonumber\\[2mm]
 &\leq
 C\epsilon^{3+\varrho}\|\varphi\|_{L^\infty(0,\ell)}\, \frac{1}{\displaystyle\int_0^{\ell}\beta^5 {\mathrm d}\theta}\sqrt{\int_0^{\ell} \beta {\mathrm d}\theta \int_0^{\ell}\beta {\mathrm d}\theta\,}
 \leq
 \epsilon^{3+\varrho}\,C \|\varphi\|_{L^\infty(0,\ell)},
\end{align*}
where we have used the assumptions in \eqref{curvature bounded}, \eqref{smallperturbation0} and \eqref{smallperturbation1}.
Therefore
\begin{align}\label{inedtitiesterm1}
 \Bigg|\,-\,\epsilon^2\int_0^\ell \varphi \kappa\int_{-M}^{M} \frac{u_{\theta\theta} u}{(1+t\,\kappa)}
 \,{\mathrm d}t{\mathrm d}\theta\Bigg|
 &\,=\,
\Bigg|-\epsilon^2\int_0^\ell \varphi \kappa\int_{-M}^{M}
 \frac{ \big({\mathbf u}_{\theta\theta}{\mathbf u} + {\mathbf u}_{\theta\theta}\phi+{\mathbf u}\phi_{\theta\theta}+\phi_{\theta\theta}\phi\big)}{(1+t\,\kappa)}
 \,{\mathrm d}t{\mathrm d}\theta\Bigg|
 \nonumber\\[2mm]
  &\,\leq\, \epsilon^{1+\varrho}\,C \|\varphi\|_{L^\infty(0,\ell)}.
\end{align}

\smallskip
\noindent$\clubsuit$
From \eqref{mathbfu-t1}-\eqref{mathbfu-thetatheta1}, we can derive that
\begin{align*}
 &\Bigg|-\, \epsilon^2 \int_0^\ell \varphi kk'\int_{-M}^{M} \frac{t\,{\mathbf u}_{\theta}
 {\mathbf u}}{(1+t\,\kappa)^2}
 \,{\mathrm d}t{\mathrm d}\theta\Bigg|
\nonumber\\[2mm]
 &\,=\,\Bigg|\epsilon^2\int_0^\ell \varphi kk' \displaystyle\int_{-\frac{M\beta}{\epsilon}}^{\frac{M\beta}{\epsilon}}
 \frac{\big[\beta' U({\mathbf t})
 \,+\, \beta^{-1}\beta \beta' U'({\mathbf t}){\mathbf t}\big]\beta U({\mathbf t})}{\big(1+\epsilon \kappa \beta^{-1}{\mathbf t}\big)^2}
 \,{\mathbf t}\epsilon^2 \beta^{-2} \,{\mathrm d}{\mathbf t}{\mathrm d}\theta\Bigg|
\nonumber\\[2mm]
 &\,=\,\Bigg|\epsilon^2\int_0^\ell \varphi kk' \displaystyle\int_{-\frac{M\beta}{\epsilon}}^{\frac{M\beta}{\epsilon}}
 \frac{\epsilon^2 \beta \beta'\beta^{-2}{\mathbf t}U^2({\mathbf t})}{\big(1+\epsilon \kappa \beta^{-1}{\mathbf t}\big)^2}
 \,{\mathrm d}{\mathbf t}{\mathrm d}\theta
 \,+\,
 \epsilon^2\int_0^\ell \varphi kk' \displaystyle\int_{-\frac{M\beta}{\epsilon}}^{\frac{M\beta}{\epsilon}}
 \frac{\epsilon^2 \beta^2 \beta^{-3}\beta'{\mathbf t}^2 U({\mathbf t})U'({\mathbf t})}{\big(1+\epsilon \kappa \beta^{-1}{\mathbf t}\big)^2}
 \,{\mathrm d}{\mathbf t}{\mathrm d}\theta\Bigg|
\nonumber\\[2mm]
 &\,=\,\epsilon^5\big(1+o(1)\big)\Bigg|\int_0^\ell \varphi 2\kappa^2\kappa'\beta \beta'\beta^{-3} {\mathrm d}\theta \int_{\mathbb R} {\mathbf t}^2U^2({\mathbf t}) \,{\mathrm d}{\mathbf t}
\nonumber\\[2mm]
 &\qquad\qquad\qquad\qquad
 -\int_0^\ell \varphi 2\kappa^2\kappa' \beta^2 \beta^{-4}\beta' {\mathrm d}\theta \int_{\mathbb R} {\mathbf t}^3 U({\mathbf t})U'({\mathbf t}) \,{\mathrm d}{\mathbf t}\Bigg|
\nonumber\\[2mm]
 &\,\leq\, \epsilon^5\,C \|\varphi\|_{L^\infty(0,\ell)} \int_0^\ell \beta\,{\mathrm d}\theta
 \,\leq\, \epsilon^{1+\varrho}\,C \|\varphi\|_{L^\infty(0,\ell)},
 \qquad\qquad \mbox{if}\ \epsilon^{4-\varrho}\,\int_0^\ell \beta\,{\mathrm d}\theta\leq C,
\end{align*}
 and
\begin{align*}
 &\Bigg|-\, \epsilon^2 \int_0^\ell \varphi kk'\int_{-M}^{M}
 \frac{t\,\big({\mathbf u}_{\theta}\phi+{\mathbf u}\phi_{\theta} +\phi_{\theta}\phi+\phi_{2,\theta}\phi_{2}\big)}
 {(1+t\,\kappa)^2}
 \,{\mathrm d}t{\mathrm d}\theta\Bigg|
\nonumber\\[2mm]
 &\leq
 C\epsilon^2\|\varphi \kappa\kappa'\|_\infty
 \Big[\|{\mathbf u}_{\theta}\|_2\|\phi\|_2+\|{\mathbf u}\|_2\|\phi_{\theta}\|_2 +\|\phi_{\theta}\|_2\|\phi\|_2\Big]
\nonumber\\[2mm]
 &\leq
 C\epsilon^{3+\sigma}\|\varphi\|_{L^\infty(0,\ell)}\, \frac{1}{\displaystyle\int_0^{\ell}\beta^5 {\mathrm d}\theta}
 \sqrt{\int_0^{\ell} \beta {\mathrm d}\theta \int_0^{\ell}\beta {\mathrm d}\theta\,}
 \,\leq\, \epsilon^{3+\varrho}\,C \|\varphi\|_{L^\infty(0,\ell)},
\end{align*}
where we have used the assumptions in \eqref{curvature bounded}, \eqref{smallperturbation0} and \eqref{smallperturbation1}.
Therefore,
\begin{align}\label{inedtitiesterm2}
\Bigg|\,-\, \epsilon^2 \int_0^\ell \varphi \kappa\kappa'\int_{-M}^{M} \frac{t\,u_\theta u}{(1+t\,\kappa)^2}
 \,{\mathrm d}t{\mathrm d}\theta\Bigg|
 &\,=\,
\Bigg| -\epsilon^2 \int_0^\ell \varphi kk'\int_{-M}^{M}
 \frac{t\,\big({\mathbf u}_{\theta}{\mathbf u}+{\mathbf u}_{\theta}\phi+{\mathbf u}\phi_{\theta}+\phi_{\theta}\phi\big) }{(1+t\,\kappa)^2}
 \,{\mathrm d}t{\mathrm d}\theta\Bigg|
  \nonumber\\[2mm]
  &\,\leq\, \epsilon^{1+\varrho}\,C \|\varphi\|_{L^\infty(0,\ell)}.
\end{align}

\smallskip
\noindent$\clubsuit$
From \eqref{mathbfu-t1}-\eqref{mathbfu-thetatheta1}, we can derive that
\begin{align*}
 \epsilon^2\int_0^\ell \varphi \int_{-M}^{M}
 {\mathbf u}_{\theta\theta}{\mathbf u}_t
 \,{\mathrm d}t{\mathrm d}\theta
 &\,=\,\epsilon^2\int_0^\ell \varphi \displaystyle\int_{-\frac{M\beta}{\epsilon}}^{\frac{M\beta}{\epsilon}}
 \Bigg[\beta''U({\mathbf t})
 \,+\, 2\frac{\beta'}{\beta}\beta' U'({\mathbf t}){\mathbf t}
 \,+\, \frac{\beta}{\beta^2} |\beta' |^2 U''({\mathbf t}){\mathbf t}^2
\nonumber\\[2mm]
 &\qquad\qquad\qquad\qquad+ \beta^{-1}\beta\beta'' U'({\mathbf t}){\mathbf t}
 \Bigg]\beta U'({\mathbf t})\,{\mathrm d}{\mathbf t}{\mathrm d}\theta
\nonumber\\[2mm]
 &\,=\,\,\epsilon^2\int_0^\ell \varphi \displaystyle\int_{-\frac{M\beta}{\epsilon}}^{\frac{M\beta}{\epsilon}}
 \beta\beta''U({\mathbf t})U'({\mathbf t})
 \,{\mathrm d}{\mathbf t}{\mathrm d}\theta+\,\epsilon^2\int_0^\ell \varphi \displaystyle\int_{-\frac{M\beta}{\epsilon}}^{\frac{M\beta}{\epsilon}}
 2\beta\beta'\beta' {\mathbf t}|U'({\mathbf t})|^2
 \,{\mathrm d}{\mathbf t}{\mathrm d}\theta
\nonumber\\[2mm]
 &\quad+\,\epsilon^2\int_0^\ell \varphi \displaystyle\int_{-\frac{M\beta}{\epsilon}}^{\frac{M\beta}{\epsilon}} \beta^2\beta^{-2}|\beta'|^2{\mathbf t}^2U'({\mathbf t})U''({\mathbf t})
 \,{\mathrm d}{\mathbf t}{\mathrm d}\theta
\nonumber\\[2mm]
 &\quad+\,\epsilon^2\int_0^\ell \varphi \displaystyle\int_{-\frac{M\beta}{\epsilon}}^{\frac{M\beta}{\epsilon}}
 \frac{\beta^2}{\beta}\beta'' {\mathbf t} |U'({\mathbf t})|^2
 \,{\mathrm d}{\mathbf t}{\mathrm d}\theta
\nonumber\\[2mm]
 &
 \,=\,0,
\end{align*}
 and
\begin{align*}
 &\Bigg|\epsilon^2\int_0^\ell \varphi \int_{-M}^{M}
 ({\mathbf u}_{\theta\theta}\phi_t+{\mathbf u}_t\phi_{\theta\theta}+\phi_{\theta\theta}\phi_t)
 \,{\mathrm d}t{\mathrm d}\theta
 \Bigg|
\nonumber\\[2mm]
 &\leq
 C\epsilon^2\|\varphi\|_{L^\infty(0,\ell)}
 \Big[\|{\mathbf u}_{\theta\theta}\|_2\|\phi_t\|_2+\|{\mathbf u}_t\|_2\|\phi_{\theta\theta}\|_2
 +\|\phi_{\theta\theta}\|_2\|\phi_t\|_2\Big]
\nonumber\\[2mm]
 &\leq
 C\epsilon^{2+\varrho}\|\varphi\|_{L^\infty(0,\ell)}
 \frac{1}{\displaystyle\int_0^{\ell}\beta^5 {\mathrm d}\theta}
 \Big[\|{\mathbf u}_{\theta\theta}\|_2\|{\mathbf u}_t\|_2+\|{\mathbf u}_t\|_2\|{\mathbf u}_{\theta\theta}\|_2
 +\epsilon^\varrho\|{\mathbf u}_{\theta\theta}\|_2\|{\mathbf u}_t\|_2\Big]
\nonumber\\[2mm]
 &\leq
 C\epsilon^{2+\sigma}\|\varphi\|_{L^\infty(0,\ell)}\, \frac{1}{\displaystyle\int_0^{\ell}\beta^5 {\mathrm d}\theta}
 \sqrt{\int_0^{\ell} \beta^3 {\mathrm d}\theta \int_0^{\ell}\beta {\mathrm d}\theta\,}\leq
 C\epsilon^{2+\varrho}\|\varphi\|_{L^\infty(0,\ell)},
\end{align*}
where we have used the assumptions in \eqref{smallperturbation0} and \eqref{smallperturbation1}.
Therefore, we have
\begin{align}\label{inedtitiesterm3}
\Bigg| \epsilon^2\int_0^\ell \varphi \int_{-M}^{M}
 u_{\theta\theta} u_t\,{\mathrm d}t{\mathrm d}\theta\Bigg|
 \,=\,&
\Bigg| \epsilon^2\int_0^\ell \varphi \int_{-M}^{M}
 \big({\mathbf u}_{\theta\theta}{\mathbf u}_t+{\mathbf u}_{\theta\theta}\phi_t+{\mathbf u}_t\phi_{\theta\theta}+\phi_{\theta\theta}\phi_t\big)
 \,{\mathrm d}t{\mathrm d}\theta\Bigg|
  \nonumber\\[2mm]
  &\,\leq\, \epsilon^{1+\varrho}\,C \|\varphi\|_{L^\infty(0,\ell)}.
\end{align}

\smallskip
\noindent$\clubsuit$
 Using \eqref{mathbfu-t1}-\eqref{mathbfu-thetatheta1}, we can obtain that
\begin{align*}
 &\Bigg|\epsilon^2\int_0^\ell \varphi\kappa' \int_{-M}^{M} \frac{t\,}{(1+t\,\kappa)}
 {\mathbf u}_{\theta}{\mathbf u}_t
 \,{\mathrm d}t{\mathrm d}\theta\Bigg|
\nonumber\\[2mm]
 &\,=\,\epsilon^2\Bigg|\int_0^\ell \varphi\kappa' \displaystyle\int_{-\frac{M\beta}{\epsilon}}^{\frac{M\beta}{\epsilon}}
 \frac{\epsilon  \beta \beta^{-1} {\mathbf t} U'({\mathbf t})}{\big(1+\epsilon \kappa \beta^{-1}{\mathbf t}\big)} \Big[\beta'U({\mathbf t})+\beta' {\mathbf t} U'({\mathbf t})
 \Big]
 \,{\mathrm d}{\mathbf t}{\mathrm d}\theta\Bigg|
\nonumber\\[2mm]
 &\,=\,\,\epsilon^2\Bigg|
 \int_0^\ell \varphi\kappa' \displaystyle\int_{-\frac{M\beta}{\epsilon}}^{\frac{M\beta}{\epsilon}}
 \epsilon  \beta \beta' \beta^{-1} \frac{{\mathbf t} U({\mathbf t})U'({\mathbf t})}{\big(1+\epsilon \kappa \beta^{-1}{\mathbf t}\big)}
 \,{\mathrm d}{\mathbf t}{\mathrm d}\theta
\nonumber\\[2mm]
 &\qquad\quad+\,\int_0^\ell \varphi \kappa' \displaystyle\int_{-\frac{M\beta}{\epsilon}}^{\frac{M\beta}{\epsilon}}
 \epsilon \beta \beta' \beta^{-1}
 \frac{|{\mathbf t} U'({\mathbf t})|^2}{\big(1+\epsilon \kappa \beta^{-1}{\mathbf t}\big)}
 \,{\mathrm d}{\mathbf t}{\mathrm d}\theta
\Bigg|
\nonumber\\[2mm]
 &\,=\,\epsilon^3\big(1+o(1)\big)
 \Bigg|
 -\varrho_0 \int_0^\ell \varphi\frac12 \kappa' \beta \beta' \beta^{-1}{\mathrm d}\theta
 \,+\,
 \varrho_2 \int_0^\ell \varphi 2\kappa' \beta \beta' \beta^{-1} {\mathrm d}\theta
\Bigg|
\nonumber\\[2mm]
 &\,\leq\, \epsilon^3\,C \|\varphi\|_{L^\infty(0,\ell)} \int_0^\ell \beta\,{\mathrm d}\theta
 \,\leq\, \epsilon^{1+\varrho}\,C \|\varphi\|_{L^\infty(0,\ell)}
 \qquad\qquad \text{if}\ \epsilon^{2-\varrho}\,\int_0^\ell \beta\,{\mathrm d}\theta\leq C,
\end{align*}
 and
\begin{align*}
 &\Bigg|\epsilon^2\int_0^\ell \varphi\kappa' \int_{-M}^{M}
 \frac{t\,}{(1+t\,\kappa)}
 ({\mathbf u}_{\theta}\phi_t +{\mathbf u}_t\phi_{\theta}+\phi_{\theta}\phi_t)
 \,{\mathrm d}t{\mathrm d}\theta
 \Bigg|
\nonumber\\[2mm]
 &\leq
 C\epsilon^2\|\varphi\kappa'\|_{L^\infty(0,\ell)}
 \Big[ \|{\mathbf u}_{\theta}\|_2\|\phi_t\|_2+\|{\mathbf u}_t\|_2\|\phi_{\theta}\|_2+\|\phi_{\theta}\|_2\|\phi_t\|_2\Big]
\nonumber\\[2mm]
 &\leq
 C\epsilon^{2+\varrho}\|\varphi\kappa'\|_{L^\infty(0,\ell)} \frac{1}{\displaystyle\int_0^{\ell}\beta^5 {\mathrm d}\theta}
 \Big[ \|{\mathbf u}_{\theta}\|_2\|{\mathbf u}_t\|_2+\|{\mathbf u}_t\|_2\|{\mathbf u}_{\theta}\|_2+\epsilon^{\varrho}\|{\mathbf u}_{\theta}\|_2\|{\mathbf u}_t\|_2\Big]
\nonumber\\[2mm]
 &\leq
 C\epsilon^{2+\sigma}\|\varphi\kappa'\|_{L^\infty(0,\ell)}\, \frac{1}{\displaystyle\int_0^{\ell}\beta^5 {\mathrm d}\theta}
 \sqrt{\int_0^{\ell} \beta^3 {\mathrm d}\theta \int_0^{\ell}\beta {\mathrm d}\theta\,}\leq
 C\epsilon^{2+\varrho}\|\varphi\|_{L^\infty(0,\ell)},
\end{align*}
where we have used the assumptions in \eqref{curvature bounded}, \eqref{smallperturbation0} and \eqref{smallperturbation1}.
Therefore
\begin{align}\label{inedtitiesterm4}
 \epsilon^2 \Bigg|\int_0^\ell \varphi\kappa' \int_{-M}^{M} \frac{t\,u_{\theta} u_t}{(1+t\,\kappa)}
 \,{\mathrm d}t{\mathrm d}\theta\Bigg|
 &\,=\,
 \epsilon^2\Bigg|\int_0^\ell \varphi\kappa' \int_{-M}^{M}
 \frac{t\, \big({\mathbf u}_{\theta}{\mathbf u}_t+{\mathbf u}_{\theta}\phi_t+{\mathbf u}_t\phi_{\theta}+\phi_{\theta}\phi_t\big) }{(1+t\,\kappa)}
 \,{\mathrm d}t{\mathrm d}\theta\Bigg|
  \nonumber\\[2mm]
  &\,\leq\, \epsilon^{1+\varrho}\,C \|\varphi\|_{L^\infty(0,\ell)}.
\end{align}

\medskip
\noindent {\bf{Part 3.}}
 We now do the following estimates for the key terms.

\smallskip
\noindent$\clubsuit$
Since $u={\mathbf u}+\phi$, then we have
\begin{align*}
 &\frac{1}{2}\epsilon^2\int_0^\ell \varphi \int_{-M}^{M} \partial_t W |u|^2
 (1+t\,\kappa)^2\,{\mathrm d}t{\mathrm d}\theta
\nonumber\\[2mm]
 &\,=\,
 \frac{1}{2}\epsilon^2\int_0^\ell \varphi \int_{-M}^{M}
 \partial_t W {\mathbf u}^2
 (1+t\,\kappa)^2\,{\mathrm d}t{\mathrm d}\theta
 \,+\,
 \frac{1}{2}\epsilon^2\int_0^\ell \varphi \int_{-M}^{M}
 \partial_t W\big(2{\mathbf u}\phi+\phi\phi\big)
 (1+t\,\kappa)^2\,{\mathrm d}t{\mathrm d}\theta
\nonumber\\[2mm]
 &\,=\,
 \frac{1}{2}\epsilon^2\int_0^\ell \varphi \int_{-M}^{M}
 \partial_t W(0, \theta) {\mathbf u}^2
 (1+t\,\kappa)^2\,{\mathrm d}t{\mathrm d}\theta
 \,+\,
 \frac{1}{2}\epsilon^2\int_0^\ell \varphi \int_{-M}^{M}
 t\partial_{tt} W\big(\hbar(\theta)t,\theta\big) {\mathbf u}^2
 (1+t\,\kappa)^2\,{\mathrm d}t{\mathrm d}\theta
\nonumber\\[2mm]
 &\quad
 \,+\,
 \frac{1}{2}\epsilon^2\int_0^\ell \varphi \int_{-M}^{M}
 \partial_t W(t, \theta)
 \big(2{\mathbf u}\phi+\phi\phi\big)
 (1+t\,\kappa)^2\,{\mathrm d}t{\mathrm d}\theta,
\end{align*}
 where $\hbar(\theta)$ is a smooth function with $|\hbar(\theta)|\leq 1$.
 It is easy to derive the following
\begin{align*}
\frac{\epsilon^2}{2}\int_0^\ell \varphi \int_{-M}^{M}
 \partial_t W(0, \theta){\mathbf u}^2
 (1+t\,\kappa)^2\,{\mathrm d}t{\mathrm d}\theta
 &\,=\,
 \frac{\epsilon^2}{2}\int_0^\ell \varphi \int_{-\frac{M\beta}{\epsilon}}^{ \frac{M\beta}{\epsilon}}
 \partial_t W(0, \theta)
 \big|\,\beta\, U({\mathbf t})\, \big|^2\Big(1+\epsilon \frac{\kappa {\mathbf t}}{\epsilon}\Big)^2
 \frac{\epsilon}{\beta}
 \,{\mathrm d}{\mathbf t}{\mathrm d}\theta
\nonumber\\[2mm]
 & \,=\,
 \epsilon^3\big(1+o(1)\big)\frac{\varrho_0}{2}
 \int_0^\ell \varphi \beta
 \partial_t W(0, \theta){\mathrm d}\theta.
\end{align*}
 On the other hand,  recall the inequality
\begin{align*}
 \sup_{\theta \in (0,  \ell) }  |\phi (t, \theta)|^2 \le C \int_0^{\ell} |\phi (t, \theta)|^2 {\mathrm d} \theta + C  \int_0^{\ell} |\p_\theta\phi (t, \theta)|^2 {\mathrm d} \theta,
\end{align*}
as in the proof of Lemma \ref{Lemma4-3-2026}.
By using \eqref{W_t-control} and \eqref{smallperturbation0}, we have
\begin{align}
 &\Bigg|\frac{1}{2}\epsilon^2\int_0^\ell \varphi \int_{-M}^{M}
 \partial_t W(t, \theta)\big(2{\mathbf u}\phi+\phi\phi\big)
 (1+t\,\kappa)^2\,{\mathrm d}t{\mathrm d}\theta
 \Bigg|
 \nonumber\\[2mm]
 &\leq C\int_0^\ell \int_{-M}^{M}
 \big|\beta^2\varphi\big|\Big| \big(2{\mathbf u}\phi+\phi\phi\big)\Big|
 (1+t\,\kappa)^2\,{\mathrm d}t{\mathrm d}\theta
 \nonumber\\[2mm]
 &\leq C \|\varphi\|_{L^\infty(0,\ell)}  \|\beta^2{\mathbf u}\|_2 \|\phi\|_2
 \ +\
 C \|\varphi\|_{L^\infty(0,\ell)}  \int_0^\ell \beta^2 {\mathrm d}\theta\cdot \int_{-M}^{M}
    \sup_{\theta \in (0,  \ell) }  |\phi (t, \theta)|^2
\,{\mathrm d}t
 \nonumber\\[2mm]
 &\leq  C  \epsilon^{1+ \varrho}   \|\varphi\|_{L^\infty(0,\ell)}\frac{1}{\displaystyle\int_0^{\ell}\beta^5 {\mathrm d}\theta} \sqrt{\int_0^{\ell} \beta^5 {\mathrm d}\theta \int_0^{\ell} \beta {\mathrm d}\theta}
 \ +\
 C \|\varphi\|_{L^\infty(0,\ell)}    \int_0^{\ell} \beta^2 {\mathrm d}\theta \Bigg[\|\phi\|_2^2+ \|\p_\theta \phi\|_2^2\Bigg]
 \nonumber\\[2mm]
 &\leq
 \epsilon^{1+\varrho}\, C \|\varphi\|_{L^\infty(0,\ell)} + C\e^{2\varrho} \|\varphi\|_{L^\infty(0,\ell)}  \frac{\displaystyle\int_0^{\ell} \beta^2 {\mathrm d}\theta}{\left| \displaystyle\int_0^{\ell}\beta^5 {\mathrm d}\theta\right|^2}  \Bigg[\|{\mathbf u}\|_2^2+ \| \p_\theta{\mathbf u}\|_2^2\Bigg]
 \nonumber\\[2mm]
 &\leq
 \epsilon^{1+\varrho}\, C \|\varphi\|_{L^\infty(0,\ell)},
 \label{FirstOrder2222}
\end{align}
and
\begin{align}
 &\Bigg|\frac{1}{2}\epsilon^2\int_0^\ell \varphi \int_{-M}^{M}
 t\partial_{tt} W\big(\hbar(\theta)t,\theta\big) {\mathbf u}^2
 (1+t\,\kappa)^2\,{\mathrm d}t{\mathrm d}\theta\Bigg|
\nonumber\\[2mm]
 &\,\leq\,
 C\Bigg|\int_0^\ell \int_{-\frac{M\beta}{\epsilon}}^{\frac{M\beta}{\epsilon}}
 \frac{\epsilon |{\mathbf t}|}{\beta}\big|\beta^2\varphi\big|
 \big[\,\beta\, U({\mathbf t})\, \big]^2\big(1+\epsilon \kappa \beta^{-1}{\mathbf t}\big)^2 \frac{\epsilon}{\beta}
 \,{\mathrm d}{\mathbf t}{\mathrm d}\theta\Bigg|
\nonumber\\[2mm]
 &\leq
 \epsilon^2\, C \|\varphi\|_{L^\infty(0,\ell)} \int_0^\ell\beta^2\,{\mathrm d}\theta
 \leq
 \epsilon^{1+\varrho}\, C \|\varphi\|_{L^\infty(0,\ell)}
 \qquad\qquad \mbox{if}\ \epsilon^{1-\varrho}\,\int_0^\ell \beta^2\,{\mathrm d}\theta\leq C.
 \label{FirstOrder3333}
\end{align}
Therefore, we can derive that
\begin{align}\label{inedtitiesterm5}
 &\frac{1}{2}\epsilon^2\int_0^\ell \varphi \int_{-M}^{M} \partial_t W |u|^2
 (1+t\,\kappa)^2\,{\mathrm d}t{\mathrm d}\theta
\,=\,
  \epsilon^3\big(1+o(1)\big)\frac{\varrho_0}{2}
 \int_0^\ell \varphi \beta
 \partial_t W(0, \theta){\mathrm d}\theta + O( \epsilon^{1+\varrho}) \|\varphi\|_{L^\infty(0,\ell)}.
\end{align}

\medskip
\noindent$\clubsuit$
 Here is a little bit complicated term
\begin{align*}
 &-\frac{1}{2}\int_0^\ell \varphi\kappa \int_{-M}^{M} |u|^4 (1+t\,\kappa)\,{\mathrm d}t{\mathrm d}\theta
\,=\,-\frac{1}{2}\int_0^\ell \varphi\kappa \int_{-M}^{M}
 \big({\mathbf u}^4 +4{\mathbf u}^3\phi +6{\mathbf u}^2\phi^2 +4{\mathbf u}\phi^3 +\phi^4\big)
 (1+t\,\kappa)\,{\mathrm d}t{\mathrm d}\theta.
\end{align*}
 We consider
\begin{align*}
 -\frac{1}{2}\int_0^\ell \varphi\kappa \int_{-M}^{M} {\mathbf u}^4 (1+t\,\kappa)\,{\mathrm d}t{\mathrm d}\theta
 &\,=\,
 -\frac{1}{2}\int_0^\ell \varphi\kappa \int_{-\frac{M\beta}{\epsilon}}^{\frac{M\beta}{\epsilon}}
 \big[\,\beta\, U({\mathbf t})\, \big]^4\big(1+\epsilon \kappa \beta^{-1}{\mathbf t}\big)
 \frac{\epsilon}{\beta}
 \,{\mathrm d}{\mathbf t}{\mathrm d}\theta
\nonumber\\[2mm]
 & \,=\,
 -\epsilon\big(1+o(1)\big)\frac{\varrho_p}{2} \int_0^\ell \varphi \kappa\beta^3 \,{\mathrm d}\theta,
\end{align*}
 where $\varrho_p$ is given in \eqref{varrho3} with $p=3$.
 On the other hand, by Lemma \ref{Lemma4-3-2026}, we have
\begin{align*}
 &\Bigg|-\frac{1}{2}\int_0^\ell \varphi\kappa \int_{-M}^{M}
 \big(4{\mathbf u}^3\phi +6{\mathbf u}^2\phi^2 +4{\mathbf u}\phi^3 +\phi^4\big)
 (1+t\,\kappa)\,{\mathrm d}t{\mathrm d}\theta
 \Bigg|
\nonumber\\[2mm]
 &\le  C\|\varphi\|_{L^\infty(0,\ell)}\int_0^\ell \int_{-M}^{M}
 \Big[\big|\beta^2{\mathbf u}\big| \big|\phi \big| +\e^{\varrho} \big|\beta{\mathbf u}\big| \big|\phi \big| + \e^{2\varrho} \big|{\mathbf u}\big| \big|\phi \big|+ \e^{2\varrho}  \big|\phi \big|^2\Big]
 (1+t\,\kappa)\,{\mathrm d}t{\mathrm d}\theta
\nonumber\\[2mm]
 &\leq
 C\|\varphi\|_{L^\infty(0,\ell)}
 \Big[\|\beta^2{\mathbf u}\|_2\|\phi\|_2+ \e^{\varrho}\|\beta{\mathbf u}\|_2\|\phi\|_2 + \e^{2\varrho}\|{\mathbf u}\|_2\|\phi\|_2 +   \e^{2\varrho} \|\phi\|_2^2 \Big]
 \nonumber\\[2mm]
  &
   \leq
 C\epsilon^{\varrho}\|\varphi\|_{L^\infty(0,\ell)} \frac{1}{\displaystyle\int_0^{\ell}\beta^5 {\mathrm d}\theta} \|\beta^2{\mathbf u}\|_2 \|{\mathbf u}\|_2 \leq
 \epsilon^{1+\varrho}\, C \|\varphi\|_{L^\infty(0,\ell)},
\end{align*}
where we have used the assumptions in \eqref{smallperturbation0}-\eqref{smallperturbation1}.
 Therefore, we can derive that
\begin{align}\label{inedtitiesterm6}
 &-\frac{1}{2}\int_0^\ell \varphi\kappa \int_{-M}^{M} |u|^4 (1+t\,\kappa)\,{\mathrm d}t{\mathrm d}\theta
\,=\, -\epsilon\big(1+o(1)\big)\frac{\varrho_p}{2} \int_0^\ell \varphi \kappa\beta^3 \,{\mathrm d}\theta  + O( \epsilon^{1+\varrho}) \|\varphi\|_\infty.
\end{align}

\medskip
\noindent$\clubsuit$
 We compute
\begin{align*}
 \epsilon^2\int_0^\ell \varphi \kappa\int_{-M}^{M} |u_t|^2 (1+t\,\kappa)\,{\mathrm d}t{\mathrm d}\theta
 &\,=\,\epsilon^2\int_0^\ell \varphi \kappa\int_{-M}^{M} {\mathbf u}_t{\mathbf u}_t (1+t\,\kappa)\,{\mathrm d}t{\mathrm d}\theta
\nonumber\\[2mm]
 &\quad+\epsilon^2\int_0^\ell \varphi \kappa\int_{-M}^{M} \big(2{\mathbf u}_t\phi_t+\phi_t\phi_t\big) (1+t\,\kappa)\,{\mathrm d}t{\mathrm d}\theta,
\end{align*}
 where
\begin{align*}
 \epsilon^2\int_0^\ell \varphi \kappa\int_{-M}^{M} {\mathbf u}_t{\mathbf u}_t (1+t\,\kappa)\,{\mathrm d}t{\mathrm d}\theta
 &\,=\,
 \epsilon^2\,\int_0^\ell \varphi \displaystyle\int_{-\frac{M\beta}{\epsilon}}^{\frac{M\beta}{\epsilon}}
 \kappa \epsilon^{-2}\beta^2\beta^2 \left|U'({\mathbf t})\right|^2\big(1+\epsilon \kappa \beta^{-1}{\mathbf t}\big)
 \,\frac{\epsilon}{\beta} \,{\mathrm d}{\mathbf t}{\mathrm d}\theta
\nonumber\\[2mm]
 &\,=\,\epsilon\,\int_0^\ell \varphi \displaystyle\int_{-\frac{M\beta}{\epsilon}}^{\frac{M\beta}{\epsilon}}
 \kappa \beta^2 \beta \left|U'({\mathbf t})\right|^2\big(1+\epsilon \kappa \beta^{-1}{\mathbf t}\big)
 \,{\mathrm d}{\mathbf t}{\mathrm d}\theta
\nonumber\\[2mm]
 &\,=\,\epsilon\big(1+o(1)\big)\varrho_1\int_0^\ell \varphi \kappa\beta^3\,{\mathrm d}\theta,
\end{align*}
 and
\begin{align*}
 \Bigg|\epsilon^2\int_0^\ell \varphi \kappa\int_{-M}^{M} \big(2{\mathbf u}_t\phi_t+\phi_t\phi_t\big) (1+t\,\kappa)\,{\mathrm d}t{\mathrm d}\theta\Bigg|
 \,\leq\,
 &C\epsilon^2\|\varphi\|_{L^\infty(0,\ell)}
 \Big[\|{\mathbf u}_t\|_2\|\phi_t\|_2+\|\phi_t\|_2^2\Big]
\nonumber\\[2mm]
 \,\leq\,
 &C\epsilon^{2+\varrho}\|\varphi\|_{L^\infty(0,\ell)}\frac{1}{\displaystyle\int_0^{\ell}\beta^5 {\mathrm d}\theta} \|{\mathbf u}_t\|_2^2
\nonumber\\[2mm]
 \,\leq\,
 &\epsilon^{1+\varrho}\, C \|\varphi\|_{L^\infty(0,\ell)}.
\end{align*}
Here we have used the assumptions in \eqref{smallperturbation0}-\eqref{smallperturbation1}.
Therefore, we have
\begin{align}\label{inedtitiesterm7}
 \epsilon^2\int_0^\ell \varphi \kappa\int_{-M}^{M} |u_t|^2 (1+t\,\kappa)\,{\mathrm d}t{\mathrm d}\theta
 &\,=\,\epsilon\big(1+o(1)\big)\varrho_1\int_0^\ell \varphi \kappa\beta^3\,{\mathrm d}\theta
 + O( \epsilon^{1+\varrho}) \|\varphi\|_{L^\infty(0,\ell)}.
\end{align}

\medskip
 Finally, by combining \eqref{Identity51}, \eqref{boundaryterm2}, \eqref{boundaryterm1}, \eqref{inedtitiesterm1}-\eqref{inedtitiesterm7} and careful rearrangements,
 a conclusion will be drawn from the above analysis
\begin{align}
 &\varrho_0
 \int_0^\ell \varphi\beta\frac{\epsilon^2}{2}\partial_t W(0, \theta) {\mathrm d}\theta
 \,+\,
 \frac{(1-p)}{p+1}\varrho_p
 \int_0^\ell \varphi\kappa\beta^p
 {\mathrm d}\theta
 \,+\,
 \varrho_1\int_0^\ell \varphi \kappa\beta^3\,{\mathrm d}\theta
 \,=\,
 O(\epsilon^\varrho)\|\varphi\|_{L^\infty(0,\ell)},
 \label{estimates-identiy3}
\end{align}
 where $p=3$ and $\varphi(\theta)$ is any smooth function with $\varphi(0)=\varphi(\ell)$.


\medskip
\subsection{ Finishing up the proof of Theorem \ref{necessary condition theorem}} \label{Section4.5}\

 Now we will use the results in Sections \ref{Section4.1}-\ref{Section4.4} to complete the proof of Theorem \ref{necessary condition theorem}.

\begin{lemma}\label{Lemma4.1}
 The assumption $-C_0\leq\kappa(\theta)\leq C_0$ in \eqref{curvature bounded} implies that
 there exists a constant $\ell_0$ independent of $\epsilon$ such that
\begin{align}
 \ell\geq\ell_0,\quad\forall\, \epsilon>0.
 \label{length-lower}
\end{align}
\end{lemma}

\begin{proof}
 By using the assumption in \eqref{curvature bounded} and the formula
\begin{align*} 
 2\pi=\int_0^{\ell} \kappa(\theta)\,{\mathrm d}\theta,
\end{align*}
 (see the Theorem of Turning Tangents in \cite{docarmo}),
 we have that $\ell$ has a positive lower bound independent of the parameter $\epsilon$.
\end{proof}

\medskip
 Other results in Theorem \ref{necessary condition theorem} will be shown in the sequel.

\noindent $\clubsuit$
 For problem \eqref{eq2}-\eqref{constraint1}, $u_\epsilon({\mathbf y})$ is a solution with concentration given in \eqref{u-decomposition}-\eqref{smallperturbation1},
 then from \eqref{constraint1} we have
\begin{align}\label{constraint4-000000000}
 a\epsilon^2
 \,=\,
 \int_{{\mathbb R}^2} |u_\epsilon|^2 \,{\mathrm d}{\mathbf y}
 \geq&\int_{{\mathrm{dist}} ({\mathbf y}, \Gamma_\epsilon) <M_2} \Big({\mathbf u}^2 + 2 {\mathbf u} \phi+ \phi^2 \Big) \,{\mathrm d}{\mathbf y}
 =
\left(1 + o(1)\right) \int_{{\mathrm{dist}} ({\mathbf y}, \Gamma_\epsilon) <M_2} {\mathbf u}^2\,{\mathrm d}{\mathbf y},
\end{align}
 due to the hypothesis \eqref{smallperturbation0}.
 Then, \eqref{v1-L2norm-0} and \eqref{constraint4-000000000} will give that
\begin{align} \label{constraint4}
 a\epsilon^2
 & \,\geq\, \epsilon {\varrho_0}\big(1+o(1)\big)\int_0^{\ell} \beta(\theta) {\mathrm d}\theta.
\end{align}
 We then conclude from \eqref{length-lower} and \eqref{constraint4} that $a \to +\infty $ as $\epsilon \to 0$ (or $-\lambda\to +\infty$)
 due to the fact $\beta\geq 1$.
 Therefore ${\bf a} \,=\, + \infty$.

\medskip
\noindent $\clubsuit$
 The estimate \eqref{estimates-identiy3} can be written in the form
\begin{align*}
 &
 \int_0^{\ell} \varphi\beta^3
 \Bigg[
 \frac{\epsilon^2\partial_t W(0, \theta)}{2\beta^2}
 \,+\,
 \frac{(1-p)}{p+1}\frac{\varrho_p}{\varrho_0}\kappa
 \,+\,
 \frac{\varrho_1}{\varrho_0}\kappa
 \Bigg] {\mathrm d}\theta
 \,=\,
 O(\epsilon^\varrho)\|\varphi\|_{L^\infty(0,\ell)},
\end{align*}
 where $\varphi(\theta)$ is any smooth function with $\varphi(0)=\varphi(\ell)$.
 Recalling the definition of $\beta$ in \eqref{beta} and $p=3$, we obtain that, as $\epsilon \rightarrow 0$
\begin{align*}
 \frac{3\epsilon^2 W_t(0, \theta)}{2\big[1+\epsilon^2W(0, \theta)\big]} +\kappa (\theta)
 &\,=\,
 o(1),\quad \forall\, \theta\in [0, \ell)
\end{align*}
 due to the assumptions in \eqref{curvature bounded} and \eqref{W_t-control}-\eqref{W_theta-control}.
 \qed

\medskip
\section{The approximate solutions}\label{Section5}

\subsection{ The first approximate solution}\label{Section5.1}
\subsubsection{}
We take
\begin{align}
w_0(x, \tau)=U(x)+\epsilon \frac {e(\tau)}{\beta(\epsilon z)} \mathcal{Z}(x),
\qquad (x, \tau)\in{\mathbb R}\times[0, 2\pi)
\label{first-approximation}
\end{align}
as {\bf{the first approximate solution}}.
The functions $U$ and $\mathcal{Z}$ are given in \eqref{block} and \eqref{eigenvalue}.
The parameter $e$ with the constraint \eqref{constraint-e} will be chosen by the reduction method under the boundary conditions
$$
\quad e(0) \,=\, e(2\pi), \quad e'(0) \,=\, e'(2\pi),
$$
see Sections \ref{Section7} and \ref{Section8}.
It is easy to derive that
\begin{align*}
\Big(\frac{e(\tau)}{\beta(\epsilon z)} {\mathcal Z}(x)\Big)_x= \frac{e}{\beta}  {\mathcal Z}',
\qquad\qquad
\Big(\frac{e(\tau)}{\beta(\epsilon z)} {\mathcal Z}(x)\Big)_{xx}=\frac{e}{\beta} {\mathcal Z}'',
\end{align*}
\begin{align*}
\left(\frac{e(\tau)}{\beta(\epsilon z)} {\mathcal Z}(x)\right)_{\tau}
= \left[\frac{e'}{\beta}- \frac{\beta'} {\beta^2} \frac{\tilde \ell} {2\pi \beta} e\right]{\mathcal Z},
\qquad\qquad
 \left(\frac{e(\tau)}{\beta(\epsilon z)} {\mathcal Z}(x)\right)_{x\tau}
= \left[\frac{e'}{\beta}- \frac{\beta'} {\beta^2} \frac{\tilde \ell} {2\pi \beta} e\right]{\mathcal Z}',
\end{align*}
\begin{align*}
 \left(\frac{e(\tau)}{\beta(\epsilon z)} {\mathcal Z}(x)\right)_{\tau\tau}
 = \left[\frac{e''}{\beta}- 2\frac{\beta'} {\beta^2} \frac{\tilde \ell} {2\pi \beta} e' +\left(- \frac{\beta''} {\beta^3} + 3\frac{|\beta'|^2} {\beta^4} \right) \frac{\tilde \ell^2} {4\pi^2 \beta}e\right] {\mathcal Z}.
\end{align*}
Therefore, there holds
\begin{align}
S(w_0)=&\, w_{0,xx} - w_{0} + w_{0}^3
+ \epsilon^2\frac{4\pi^2}{{\tilde\ell}^2} w_{0,\tau\tau}
+ B_4(w_0)+B_5(w_0)+ B_6\big(w_0\big)
\nonumber\\[2mm]
=&\,
\epsilon^3\frac{4\pi^2}{{\tilde\ell}^2} \big( \frac {e}{\beta} \big)''\mathcal Z
 +\frac{\epsilon\lambda_0 e{\mathcal Z}}{\beta}
 +\frac{3U(\epsilon e{\mathcal Z} )^2}{\beta^2}\,+\, \frac{(\epsilon e{\mathcal Z} )^3}{\beta^3}
 + B_4\big(w_0\big)
+ B_5\big(w_0\big)
+ B_6\big(w_0\big)
\label{S(w_0)}
\nonumber\\[2mm]
=&\, \epsilon^3\frac{4\pi^2}{{\tilde\ell}^2} \left[e''- 2\frac{\beta'} {\beta^2} \frac{\tilde \ell} {2\pi} e' +\left(- \frac{\beta''} {\beta^3} + 3\frac{|\beta'|^2} {\beta^4} \right) \frac{\tilde \ell^2} {4\pi^2}e\right]\frac{\mathcal Z}\beta+\frac{\epsilon\lambda_0 e{\mathcal Z}}{\beta}
 +\frac{3U(\epsilon e{\mathcal Z} )^2}{\beta^2}\,+\, \frac{(\epsilon e{\mathcal Z} )^3}{\beta^3}
\nonumber\\[2mm]
& + B_4\big(w_0\big)
+ B_5\big(U\big) + B_5\left(\epsilon\frac{e}{\beta}{\mathcal Z} \right)
+ B_6\big(U\big) + B_6\left(\epsilon\frac{e}{\beta}{\mathcal Z} \right).
\end{align}
The last five terms in \eqref{S(w_0)} will be calculated on details.

\smallskip
\noindent$\clubsuit$
From the expressions of $B_4$ and $B_5$ given in \eqref{new opreator B4}-\eqref{new opreator B5}, we have
\begin{align*}
B_4(w_0)\,=\,&
\epsilon\frac{\kappa}{\, \beta\, } \Big[ U'+\frac{2}{3} (x+f) U\Big]
+\epsilon^2 \frac{\kappa}{\, \beta^2\, } \left[ e {\mathcal Z}' +\frac{2}{3} (x+f) e{\mathcal Z}\right],
\nonumber\\[2mm]
B_5(U)
\,=\,&
-\epsilon^2\frac{\epsilon^2W_{tt}(0, \epsilon z)}{\beta^4}f xU
+\epsilon^2\Big[\frac{|\beta'|^2}{\beta^4} 2f
-\frac{4\pi}{{\tilde\ell}}\frac{\beta'}{\beta^2}f'
\Big]xU''
\nonumber\\[2mm]
&
+\epsilon^2\Big[
-\frac{4\pi^2}{{\tilde\ell}^2}f''
-\frac{6\pi}{\tilde\ell} \frac{\beta'}{\beta^2} f'
+\frac{2|\beta'|^2}{\beta^4}f
+\frac{\beta''} {\beta^3}f
\,-\,\frac{\kappa^2}{\beta^2}f
\Big]U'
\nonumber\\[2mm]
&
+ \epsilon^2\Big[
-\frac{1}{2}\, \frac{\epsilon^2 W_{tt}(0, \epsilon z)}{\beta^4}f^2 + \frac{\beta''} {\beta^3}
\Big]U
+\epsilon^2\Big[\frac{4\pi^2}{{\tilde\ell}^2}|f'|^2
-\frac{4\pi}{{\tilde\ell}}\frac{\beta'}{\beta^2}ff'
+\frac{|\beta'|^2}{\beta^4} f^2
\Big]U''
\nonumber\\[2mm]
&
-\epsilon^2\frac{1}{2}\, \frac{\epsilon^2W_{tt}(0, \epsilon z)}{\beta^4} x^2U
+\epsilon^2\Big[\frac{2|\beta'|^2}{\beta^4}
+\frac{\beta''} {\beta^3}
\,-\,\frac{\kappa^2}{\beta^2}
\Big]xU'
+\epsilon^2 \frac{|\beta'|^2}{\beta^4} x^2 U'',
\end{align*}
and
\begin{align*}
 B_5\left(\epsilon\frac{e}{\beta}{\mathcal Z} \right)
\,=\,&
\epsilon^3\frac{4\pi^2}{{\tilde\ell}^2}
\left[
|f'|^2 \frac{e}{\beta} {\mathcal Z}''
-f'' \frac{e}{\beta} {\mathcal Z}'
-2f' \Big(\frac{e'}{\beta}- \frac{\beta'} {\beta^2}\frac{\tilde \ell} {2\pi \beta} e\Big){\mathcal Z}'
\right]
\nonumber\\[2mm]
&
+\epsilon^3\frac{6\pi}{ {\tilde\ell}}\frac{\beta'} {\beta^2}
\left[- f' \frac{e}{\beta} {\mathcal Z}' +\Big(\frac{e'}{\beta}- \frac{\beta'} {\beta^2}\frac{\tilde \ell} {2\pi \beta} e\Big){\mathcal Z} \right]
\nonumber\\[2mm]
&
+\epsilon^3\frac{4\pi}{{\tilde\ell}} \frac{\beta'}{\beta^2} (x+f)
\left[- f' \frac{e}{\beta} {\mathcal Z}'' +\Big(\frac{e'}{\beta}- \frac{\beta'} {\beta^2}\frac{\tilde \ell} {2\pi \beta} e\Big){\mathcal Z}' \right]
\nonumber\\[2mm]
&+
\epsilon^3\Big[ \frac{\beta''}{\beta^3} \,+\,\frac{|\beta'|^2}{\beta^4} \,-\,\frac{\kappa^2}{\beta^2} \Big] (x+f) \frac{e}{\beta} {\mathcal Z}'
+\epsilon^3\frac{\beta''}{\beta^3}\frac{e}{\beta}{\mathcal Z}
\nonumber\\[2mm]
&+\epsilon^3 \frac{|\beta'|^2}{\beta^4}(x^2+2xf+f^2) \frac{e}{\beta} {\mathcal Z}''
-\epsilon^3\frac{1}{2}\,\frac{\epsilon^2W_{tt}(0, \epsilon z)}{\beta^4} (x^2+2xf+f^2) \frac{e}{\beta}{\mathcal Z}.
\end{align*}

\smallskip
\noindent$\clubsuit$
On the other hand, the expression of $B_6$ given in \eqref{B2w-2222new} will give
\begin{align*}
B_6(U)
\,=\,& - \frac{\epsilon^3}{6} \frac{\, \epsilon^2 W_{ttt}(0, \epsilon z)}{\beta^5} (x+f)^3 U
 -\frac{\epsilon^4}{24} \frac{\,\epsilon^2 W_{tttt}\big(\hbar(\epsilon z)\epsilon s, \, \epsilon z\big)}{\beta^6} (x+f)^4U
\nonumber\\[2mm]
&
\,+\,\epsilon^3 \kappa^3(\epsilon z) a_3(\epsilon s, \epsilon z)
\frac{1}{\, \beta^3\, }(x^2U'+2fxU'+f^2U')
\nonumber\\[2mm]
&\,+\,
\epsilon^3 \kappa(\epsilon z) a_2(\epsilon s, \epsilon z)
\Bigg\{
\frac{\beta''} {\beta^4} (xU+fU)
+\frac{2|\beta'|^2}{\beta^5} (x^2U'+2fxU'+f^2U')
\nonumber\\[2mm]
&\qquad\qquad\qquad\qquad\qquad
+\frac{6\pi}{{\tilde\ell}}\frac{\beta'} {\beta^3} \big[- f' xU' - ff' U' \big]
- \frac{4\pi}{{\tilde\ell}} \frac{\beta'} {\beta^3} f' x^2U''
-\frac{ 8\pi}{ {\tilde\ell}}\frac{\beta'} {\beta^3} f' fxU''
\nonumber\\[2mm]
&\qquad\qquad\qquad\qquad\qquad
-\frac{4\pi}{{\tilde\ell}} \frac{\beta'} {\beta^3} f' f^2U''
+ \frac{|\beta'|^2}{\beta^5} \big[x^3U''+3fx^2U''+3f^2xU''+f^3U''\big]
\nonumber\\[2mm]
&\qquad\qquad\qquad\qquad\qquad
+ \frac{\beta''}{\beta^4} \big[x^2U'+2fxU'+f^2U'\big]
+\frac{4\pi^2}{{\tilde\ell}^2}\frac{1}{\beta} \big[ |f'|^2 xU'' -f''xU'\big]
\nonumber\\[2mm]
&\qquad\qquad\qquad\qquad\qquad
+\frac{4\pi^2}{{\tilde\ell}^2}\frac{1}{\beta} \big[ f|f'|^2 U'' -ff''U' \big]
\Bigg\}
\nonumber\\[2mm]
&\,+\,
\epsilon^3 \kappa'(\epsilon z)a_1(\epsilon s, \epsilon z)
\Bigg\{
\frac{\beta'}{\beta^4} (xU+fU)
+\frac{\beta'}{\beta^4} (x^2U'+2fxU'+f^2U')
\nonumber\\[2mm]
&\qquad\qquad\qquad\qquad\qquad
-\frac{2\pi}{{\tilde\ell}}\frac{1}{ \beta^2}f' xU'
-\frac{2\pi}{{\tilde\ell}}\frac{1}{ \beta^2} ff' U'
\Bigg\}.
\end{align*}
By \eqref{a1a2a3-1}, we recall that
\begin{align*}
a_3(t, \theta)=\frac{1}{1+t \kappa (\theta) },
\qquad
a_2(t, \theta)=\frac{-2-t \kappa (\theta)}{(1+t \kappa (\theta) )^2 },
\qquad
a_1(t, \theta)&=\frac{1}{(1+t \kappa (\theta) )^3 }.
\end{align*}
Then we can rewrite $a_1, a_2, a_3$ as
\begin{align}\label{a1a2a3}
a_1(t, \theta) = 1- \epsilon \frac{3\kappa}{\beta}(x+f) + \epsilon^2 {\tilde a}_1,
\qquad
&{\tilde a}_1=\frac{6\beta^{-2}\kappa^2(x+f)^2 \,+\,8\epsilon\beta^{-3}\kappa^3(x+f)^3
\,+\,3\epsilon^2\beta^{-4}\kappa^4(x+f)^4}{\big[1+\epsilon\beta^{-1}\kappa(x+f)\big]^3},
\nonumber\\[2mm]
a_2(t, \theta) = -2 + \epsilon \frac{3\kappa}{\beta}(x+f) + \epsilon^2 {\tilde a}_2,
\qquad
&{\tilde a}_2=\frac{-4\beta^{-2}\kappa^2(x+f)^2 \,-\,3\epsilon\beta^{-3}\kappa^3(x+f)^3}{\big[1+\epsilon\beta^{-1}\kappa(x+f)\big]^2},
\nonumber\\[2mm]
a_3(t, \theta) = 1- \epsilon \frac{\kappa}{\beta}(x+f) + \epsilon^2 {\tilde a}_3,
\qquad
&{\tilde a}_3=\frac{\beta^{-2}\kappa^2(x+f)^2}{1+\epsilon\beta^{-1}\kappa(x+f)}.
\end{align}
Then we can rewritten $B_6(U)$ as
\begin{align*}
B_6(U)= & - \epsilon^3 \frac{\,\epsilon^2 W_{ttt}(0, \epsilon z)}{6\beta^5}\, x^3 U
\ -\
\epsilon^3\frac{\, \epsilon^2W_{ttt}(0, \epsilon z)}{2\beta^5}f\,x^2 U
\ +\ \epsilon^3\Big[- \frac{\, \, \epsilon^2W_{ttt}(0, \epsilon z)}{2\beta^5}f^2 - \frac{2\kappa\beta''}{\beta^4} +\frac{\kappa'\beta'} {\beta^4} \Big] x U
\nonumber\\[2mm]
& + \epsilon^3\Big[-\frac{\, \, \epsilon^2 W_{ttt}(0, \epsilon z)}{6\beta^5}f^3 -\frac{2\kappa\beta''}{\beta^4} f +\frac{\kappa'\beta'} {\beta^4} f\Big] U
\ +\ \epsilon^3\Big[ \frac{\kappa^3}{\beta^3}-\frac{4\kappa|\beta'|^2}{\beta^5} -\frac{2\kappa\beta''}{\beta^4} + \frac{\kappa'\beta'} {\beta^4} \Big]x^2U'
\nonumber\\[2mm]
 & + \epsilon^3\Big[
 \frac{2\kappa^3}{\beta^3}f
 -\frac{8\kappa|\beta'|^2}{\beta^5} f
 -\frac{4\kappa\beta''}{\beta^4}f
 -\frac{8\pi^2}{{\tilde\ell}^2}\frac{\kappa}{\beta} f''
+\frac{12\pi }{ {\tilde\ell}}\frac{\kappa\beta'}{\beta^3} f'
+\frac{2\kappa'\beta'}{\beta^4}f
- \frac{2\pi}{{\tilde\ell}} \frac{\kappa'}{\beta^2} f'\Big]xU'
\nonumber\\[2mm]
 & + \epsilon^3 \Big[
 \frac{\kappa^3}{\beta^3}f^2
 -\frac{4\kappa|\beta'|^2}{\beta^5} f^2
 - \frac{2\kappa\beta''}{\beta^4}f^2
 -\frac{8\pi^2}{{\tilde\ell}^2}\frac{\kappa}{\beta} f''f
 +\frac{12\pi }{ {\tilde\ell}}\frac{\kappa\beta'}{\beta^3} f' f
 + \frac{\kappa'\beta'} {\beta^4}f^2
 - \frac{2\pi}{{\tilde\ell}} \frac{\kappa'}{\beta^2} f'f\Big]U'
\nonumber\\[2mm]
 & -\epsilon^3 \frac{2\kappa|\beta'|^2}{\beta^5} x^3 U''
 \ +\
 \epsilon^3\Big[ -\frac{6\kappa|\beta'|^2}{\beta^5}f
 +\frac{8\pi}{ {\tilde\ell}}\frac{\kappa\beta'}{\beta^3} f'\Big] x^2 U''
\nonumber\\[2mm]
 & +\epsilon^3\Big[ -\frac{6\kappa|\beta'|^2}{\beta^5}f^2
 + \frac{16\pi}{ {\tilde\ell}} \frac{\kappa\beta'}{\beta^3} f' f
-\frac{8\pi^2} {{\tilde\ell}^2}\frac{\kappa}{\beta} |f'|^2\Big] x U''
\nonumber\\[2mm]
 & +\epsilon^3\Big[ -\frac{2\kappa|\beta'|^2}{\beta^5}f^3
 +\frac{8\pi}{ {\tilde\ell}} \frac{\kappa\beta'}{\beta^3} f' f^2
 -\frac{8\pi^2} {{\tilde\ell}^2}\frac{\kappa}{\beta} |f'|^2 f \Big] U''
\nonumber\\[2mm]
 & + \frac{\epsilon^4}{\beta^2} \mathcal F_1\left[\tau, x, \beta f, \frac{\beta f'}{\tilde\ell}\right]
 + \frac{\epsilon^4}{\beta^2} \hat{\mathcal F}_1\left[\tau, x, \beta f, \frac{\beta f''}{\tilde\ell^2} \right],
\end{align*}
 where
\begin{align} \label{mathcal F1}
 &\frac{\epsilon^4}{\beta^2} \mathcal F_1\left[\theta, x, \beta f, \frac{\beta f'}{\tilde\ell} \right]
 \nonumber\\[2mm]
 =&-\frac{\epsilon^4}{24} \frac{\,\epsilon^2 W_{tttt}\big(\hbar(\epsilon z)\epsilon s, \, \epsilon z\big)}{\beta^6} (x+f)^4U
\nonumber\\[2mm]
 &+
\epsilon^4 \kappa^3(\epsilon z) \Big( -\frac{\kappa}{\beta}(x+f) + \epsilon {\tilde a}_3 \Big)
\frac{1}{\, \beta^3\, }(x^2U'+2fxU'+f^2U')
\nonumber\\[2mm]
&\,+\,
\epsilon^4 \kappa(\epsilon z) \Big( \frac{3\kappa}{\beta}(x+f) + \epsilon {\tilde a}_2 \Big)
\Bigg\{
\frac{\beta''} {\beta^4} (xU+fU)
+\frac{2|\beta'|^2}{\beta^5} (x^2U'+2fxU'+f^2U')
\nonumber\\[2mm]
&\qquad\qquad\qquad\qquad\qquad\qquad\qquad
+\frac{6\pi}{{\tilde\ell}}\frac{\beta'} {\beta^3} \big[- f' xU' - ff' U' \big]
- \frac{4\pi}{{\tilde\ell}} \frac{\beta'} {\beta^3} f' x^2U''
-\frac{ 8\pi}{ {\tilde\ell}}\frac{\beta'} {\beta^3} f' fxU''
\nonumber\\[2mm]
&\qquad\qquad\qquad\qquad\qquad\qquad\qquad
-\frac{4\pi}{{\tilde\ell}} \frac{\beta'} {\beta^3} f' f^2U''
+ \frac{|\beta'|^2}{\beta^5} \big[x^3U''+3fx^2U''+3f^2xU''+f^3U''\big]
\nonumber\\[2mm]
&\qquad\qquad\qquad\qquad\qquad\qquad\qquad
+ \frac{\beta''}{\beta^4} \big[x^2U'+2fxU'+f^2U'\big]
+\frac{4\pi^2}{{\tilde\ell}^2}\frac{1}{\beta} \big[ |f'|^2 xU'' + f|f'|^2 U'' \big]
\Bigg\}
\nonumber\\[2mm]
& \,+\,
\epsilon^4 \kappa'(\epsilon z)\Big(- \frac{3\kappa}{\beta}(x+f) + \epsilon {\tilde a}_1\Big)
\Bigg\{
\frac{\beta'}{\beta^4} (xU+fU)
+\frac{\beta'}{\beta^4} (x^2U'+2fxU'+f^2U')
\nonumber\\[2mm]
&\qquad\qquad\qquad\qquad\qquad\qquad\qquad\qquad
-\frac{2\pi}{{\tilde\ell}}\frac{1}{ \beta^2}f' xU'
-\frac{2\pi}{{\tilde\ell}}\frac{1}{ \beta^2} ff' U'
\Bigg\},
\end{align}
and
\begin{align} \label{hat mathcal F1}
 \frac{\epsilon^4}{\beta^2} \hat{\mathcal F}_1\left[\theta, x, \beta f, \frac{\beta f''}{\tilde\ell^2} \right]
 = \epsilon^4 \kappa(\epsilon z) \Big( \frac{3\kappa}{\beta}(x+f) + \epsilon {\tilde a}_2 \Big) \left[\frac{4\pi^2}{{\tilde\ell}^2}\frac{1}{\beta} \big( ff''U' +f''xU'\big)
\right].
\end{align}

\smallskip
\noindent$\clubsuit$
From the expression of $B_6$ given in \eqref{B2w-2222new},
we can select from $B_6\left(\epsilon\frac{e}{\beta}{\mathcal Z} \right) $ the terms involving $e''$
\begin{align*}
\epsilon^3 \kappa a_2 \frac{ 4\pi^2} {{\tilde\ell}^2} \frac{1}{\, \beta\, }
 (x+f) \left(\epsilon\frac{e}{\beta}{\mathcal Z} \right) _{\tau\tau}
& =\epsilon^4\frac{ 4\pi^2} {{\tilde\ell}^2} \frac{\kappa a_2}{\, \beta\, } (x+f)
\left[e''- \frac{2\beta'} {\beta^2} \frac{\tilde \ell} {2\pi} e'
-\left(\frac{\beta''} {\beta^3} - 3\frac{|\beta'|^2} {\beta^4} \right) \frac{\tilde \ell^2} {4\pi^2}e
\right]\frac{\mathcal Z}\beta,
\end{align*}
and the term from $B_6\left(\epsilon\frac{e}{\beta}{\mathcal Z} \right) $ involving $f''$ is
\begin{equation*}
-\epsilon^4 \frac{ 4\pi^2} {{\tilde\ell}^2} \frac{\kappa a_2 }{\, \beta^2\, } (x+f) f'' e \mathcal Z' .
\end{equation*}
 Then we can rewrite $B_6\left(\epsilon\frac{e}{\beta}{\mathcal Z} \right) $ as
\begin{align}\label{expression of B6mathcal Z}
B_6\left(\epsilon\frac{e}{\beta}{\mathcal Z} \right) &
 =
\epsilon^4 \frac{4\pi^2} {{\tilde\ell}^2} \frac{\kappa a_2}{\, \beta^2\, } e'' (x+f)\mathcal Z
-\epsilon^4\frac{ 4\pi^2} {{\tilde\ell}^2}  \frac{ \kappa a_2 }{\, \beta^2\, } (x+f) f'' e \mathcal Z'
+\frac{\epsilon^4} {\beta^3} \mathcal G_1\left[\theta, x, \beta f, \frac{\beta f'}{\tilde\ell}, e, \frac{e'}{\tilde\ell}\right],
\end{align}
where
\begin{align} \label{mathcal G1}
& \frac{\epsilon^4} {\beta^3} \mathcal G_1\left[\theta, x, \beta f, \frac{\beta f'}{\tilde\ell}, e, \frac{e'}{\tilde\ell}\right]
\nonumber\\[2mm]
\,=\,& - \frac{\epsilon^3}{6} \frac{\, \, \epsilon^2 W_{ttt}(0, \epsilon z)}{\beta^5} (x+f)^3 \left(\epsilon\frac{e}{\beta}{\mathcal Z} \right)
 -\frac{\epsilon^4}{24} \frac{\, \, \epsilon^2 W_{tttt}\big(\hbar(\epsilon z)\epsilon s, \, \epsilon z\big)}{\beta^6} (x+f)^4\left(\epsilon\frac{e}{\beta}{\mathcal Z} \right)
\nonumber\\[2mm]
&
 +\epsilon^3 \kappa^3(\epsilon z) a_3(\epsilon s, \epsilon z) \frac{1}{\, \beta^3\, }(x+f)^2\left(\epsilon\frac{e}{\beta}{\mathcal Z} \right)_{x}
\nonumber\\[2mm]
&
 + \epsilon^3 \kappa(\epsilon z) a_2(\epsilon s, \epsilon z)\times
\Bigg\{
\frac{\beta''} {\beta^4} (x+f) \left(\epsilon\frac{e}{\beta}{\mathcal Z} \right)
+\frac{2|\beta'|^2}{\beta^5} (x+f)^2 \left(\epsilon\frac{e}{\beta}{\mathcal Z} \right)_x
\nonumber\\[2mm]
&\qquad\qquad\qquad\qquad\qquad
+\frac{ 6\pi}{ {\tilde\ell}}\frac{\beta'} {\beta^3} (x+f)\left[-f'\left(\epsilon\frac{e}{\beta}{\mathcal Z} \right)_x +\left(\epsilon\frac{e}{\beta}{\mathcal Z} \right)_\tau\right]
\nonumber\\[2mm]
&\qquad\qquad\qquad\qquad\qquad
+\frac{ 4\pi}{ {\tilde\ell}}\frac{\beta'} {\beta^3} (x+f)^2\Bigg[- f' \left(\epsilon\frac{e}{\beta}{\mathcal Z} \right)_{xx} +\left(\epsilon\frac{e}{\beta}{\mathcal Z} \right)_{x\tau}\Bigg]
\nonumber\\[2mm]
&\qquad\qquad\qquad\qquad\qquad
+ \frac{|\beta'|^2}{\beta^5} (x+f)^3\left(\epsilon\frac{e}{\beta}{\mathcal Z} \right)_{xx}
+ \frac{\beta''}{\beta^4} (x+f)^2 \left(\epsilon\frac{e}{\beta}{\mathcal Z} \right)_{x}
\nonumber\\[2mm]
&\qquad\qquad\qquad\qquad\qquad
+\frac{ 4\pi^2} {{\tilde\ell}^2}\frac{1}{\beta} (x+f)\left[ |f'|^2 \left(\epsilon\frac{e}{\beta}{\mathcal Z} \right)_{xx} -2f' \left(\epsilon\frac{e}{\beta}{\mathcal Z} \right)_{x\tau}\right]
\Bigg\}
\nonumber\\[2mm]
&+
\epsilon^3 \kappa'(\epsilon z)a_1(\epsilon s, \epsilon z)\times
\Bigg\{
\frac{\beta'}{\beta^4} (x+f)\left(\epsilon\frac{e}{\beta}{\mathcal Z} \right)
+\frac{\beta'}{\beta^4} (x+f)^2 \left(\epsilon\frac{e}{\beta}{\mathcal Z} \right)_{x}
\nonumber\\[2mm]
&\qquad\qquad\qquad\qquad\qquad
+\frac{ 2\pi}{ {\tilde\ell}}\frac{1}{ \beta^2} (x+f)\left[-f'\left(\epsilon\frac{e}{\beta}{\mathcal Z} \right)_x +\left(\epsilon\frac{e}{\beta}{\mathcal Z} \right)_\tau\right]
\Bigg\}
\nonumber\\[2mm]
&+
\epsilon^4\frac{ 4\pi^2} {{\tilde\ell}^2} \frac{\kappa a_2}{\, \beta\, } (x+f)
\left[- \frac{2\beta'} {\beta^2} \frac{\tilde \ell} {2\pi} e'
-\left(\frac{\beta''} {\beta^3} - 3\frac{|\beta'|^2} {\beta^4} \right) \frac{\tilde \ell^2} {4\pi^2}e
\right]\frac{\mathcal Z}\beta.
\end{align}

\medskip
\subsubsection{}
For  the later use, we will rearrange all terms in $S(w_0)$ based on oddness and evenness of the terms in the variable $x$.
We denote
\begin{align}\label{S1S2}
\epsilon S_1 = \epsilon \frac{ \kappa}{\, \beta\, } \Big[ U'+\frac{2}{3} x U\Big],
\qquad
\epsilon S_2 = \epsilon\frac{2}{3} \frac{ \kappa}{\, \beta\, } f U,
\end{align}
\begin{align}\label{expression of s3}
 \epsilon^2 S_3& = \epsilon^2
 \Bigg\{\frac{\kappa}{\, \beta^2\, }\left[ e {\mathcal Z}'+\frac 23 e x\mathcal Z \right]
 -\,\frac{\epsilon^2W_{tt}(0, \epsilon z)}{\beta^4}f xU
+\Big[\frac{|\beta'|^2}{\beta^4} 2f
-\frac{4\pi}{{\tilde\ell}}\frac{\beta'}{\beta^2}f'
\Big]xU''
\nonumber\\[2mm]
&\qquad\quad
+ \Big[
-\frac{4\pi^2}{{\tilde\ell}^2}f''
-\frac{6\pi}{\tilde\ell} \frac{\beta'}{\beta^2} f'
+\frac{|\beta'|^2}{\beta^4}f
+\frac{\beta''} {\beta^3}f
\,-\,\frac{\kappa^2}{\beta^2}f
\Big]U'
\Bigg\},
\end{align}
\begin{align}\label{expression of s4}
 \epsilon^2 S_4& =\epsilon^2
 \Bigg\{ \Big[
-\frac{1}{2}\, \frac{\epsilon^2W_{tt}(0, \epsilon z)}{\beta^4}f^2 + \frac{\beta''} {\beta^3} \Big]U
+\Big[\frac{4\pi^2}{{\tilde\ell}^2}|f'|^2
-\frac{4\pi}{{\tilde\ell}}\frac{\beta'}{\beta^2}ff'
+\frac{|\beta'|^2}{\beta^4} f^2
\Big]U''
\nonumber\\[2mm]
&\qquad -\frac{1}{2}\, \frac{\epsilon^2W_{tt}(0, \epsilon z)}{\beta^4} x^2U
+\Big[\frac{2|\beta'|^2}{\beta^4}
+\frac{\beta''} {\beta^3}
\,-\,\frac{\kappa^2}{\beta^2}
\Big]xU'
\nonumber\\[2mm]
&\qquad+\frac{|\beta'|^2}{\beta^4} x^2 U''
+\frac 23 \frac{\kappa}{\, \beta^2\, } fe{\mathcal Z}
+\frac3{\beta^2}e^2 U {\mathcal Z }^2
\Bigg\} ,
\end{align}
and
\begin{align}\label{S5}
 \epsilon^3 S_5& =
 \epsilon^3 \Bigg\{-\frac{1}{6} \frac{\, \epsilon^2 W_{ttt}(0, \epsilon z)}{\beta^5} x^3 U
 + \Big[- \frac{1}{2} \frac{\epsilon^2 W_{ttt}(0, \epsilon z)}{\beta^5}f^2-2 \frac{\kappa\beta''}{\beta^4} +\frac{\kappa'\beta'} {\beta^4} \Big] x U
\nonumber\\[2mm]
 & \qquad\quad
 + \Big[ \frac{\kappa^3}{\beta^3}-\frac{4\kappa|\beta'|^2}{\beta^5}-\frac{2\kappa\beta''}{\beta^4} + \frac{\kappa'\beta'} {\beta^4} \Big]x^2U'
\nonumber\\[2mm]
 & \qquad\quad+ \Big[
 \frac{\kappa^3}{\beta^3}f^2
 -\frac{4\kappa|\beta'|^2}{\beta^5} f^2
 - \frac{2\kappa\beta''}{\beta^4}f^2
 -\frac{8\pi^2}{{\tilde\ell}^2}\frac{\kappa}{\beta} f''f
 +\frac{12\pi }{ {\tilde\ell}}\frac{\kappa\beta'}{\beta^3} f' f
 + \frac{\kappa'\beta'} {\beta^4}f^2
 - \frac{2\pi}{{\tilde\ell}} \frac{\kappa'}{\beta^2} f'f\Big]U'
\nonumber\\[2mm]
 & \qquad\quad-\frac{2\kappa|\beta'|^2}{\beta^5} x^3 U''
 + \Big[ -\frac{6\kappa|\beta'|^2}{\beta^5}f^2
 + \frac{16\pi}{ {\tilde\ell}} \frac{\kappa\beta'}{\beta^3} f' f
-\frac{8\pi^2} {{\tilde\ell}^2}\frac{\kappa}{\beta} |f'|^2\Big]x U''
\Bigg\} ,
\end{align}
\begin{align}\label{S6}
 \epsilon^3 S_6& = \epsilon^3\Bigg\{
 - \frac{1}{2} \frac{\,\epsilon^2 W_{ttt}(0, \epsilon z)}{\beta^5}f x^2 U
\ +\
\Big[- \frac{1}{6} \frac{\epsilon^2W_{ttt}(0, \epsilon z)}{\beta^5}f^3
 -\frac{2\kappa\beta''}{\beta^4} f +\frac{\kappa'\beta'} {\beta^4} f\Big] U
\nonumber\\[2mm]
 & \qquad\quad + \Big[
 \frac{2\kappa^3}{\beta^3}f
 -\frac{8\kappa|\beta'|^2}{\beta^5} f
 -\frac{4\kappa\beta''}{\beta^4}f
 -\frac{8\pi^2}{{\tilde\ell}^2}\frac{\kappa}{\beta} f''
+\frac{12\pi }{ {\tilde\ell}}\frac{\kappa\beta'}{\beta^3} f'
+\frac{2\kappa'\beta'}{\beta^4}f
- \frac{2\pi}{{\tilde\ell}} \frac{\kappa'}{\beta^2} f'\Big]xU'
\nonumber\\[2mm]
 & \qquad\quad+ \Big[ -\frac{6\kappa|\beta'|^2}{\beta^5}f
 +\frac{8\pi}{ {\tilde\ell}}\frac{\kappa\beta'}{\beta^3} f'\Big] x^2 U''
 + \Big[ -\frac{2\kappa|\beta'|^2}{\beta^5}f^3
 +\frac{8\pi}{ {\tilde\ell}} \frac{\kappa\beta'}{\beta^3} f' f^2
 -\frac{8\pi^2} {{\tilde\ell}^2}\frac{\kappa}{\beta} |f'|^2 f \Big]U''
 \Bigg\},
\end{align}
\begin{align}\label{S7}
\epsilon^3S_7
& = \frac{\epsilon^3}\beta\Bigg\{ \Big[
-\frac{4\pi^2}{{\tilde\ell}^2}f''
-\frac{6\pi}{\tilde\ell} \frac{\beta'}{\beta^2} f'
+\frac{|\beta'|^2}{\beta^4}f
+\frac{\beta''} {\beta^3}f
\,-\,\frac{\kappa^2}{\beta^2}f
\Big]e {\mathcal Z}'
+\Big[\frac{2|\beta'|^2}{\beta^4} f
-\frac{4\pi}{{\tilde\ell}}\frac{\beta'}{\beta^2}f'
\Big]e x {\mathcal Z}''
\nonumber\\[2mm]
&\qquad\quad -\,\frac{\epsilon^2W_{tt}(0, \epsilon z)}{\beta^4}f e x{\mathcal Z}
+
\Big[ -\frac{8\pi^2}{{\tilde\ell}^2}f' +\frac{4\pi}{{\tilde\ell}}\frac{\beta'}{\beta^2}f\Big]e' {\mathcal Z}'
+
\Big[ \frac{4\pi}{{\tilde\ell}}f' -\frac{2\beta'}{\beta^2}f\Big] \frac{\beta'}{\beta^2}e {\mathcal Z}'
\Bigg\},
\end{align}
\begin{align}\label{S8}
 \epsilon^3S_8
& = \frac{\epsilon^3}\beta\Bigg\{
\frac{ 2\pi}{\tilde\ell} \frac{\beta'}{\beta^2} e'{\mathcal Z}
+\frac{4\pi}{{\tilde\ell}}\frac{\beta'}{\beta^2} e' x {\mathcal Z}'
-\frac{2|\beta'|^2} {\beta^4} e x\mathcal Z'
-\frac{1}{2}\, \frac{\epsilon^2W_{tt}(0, \epsilon z)}{\beta^4}f^2 e\mathcal Z
\nonumber\\[2mm]
&\qquad\quad
+\Big[\frac{4\pi^2}{{\tilde\ell}^2}|f'|^2
-\frac{4\pi}{{\tilde\ell}}\frac{\beta'}{\beta^2}ff'
+\frac{|\beta'|^2}{\beta^4} f^2
\Big]e {\mathcal Z}''
-\frac{1}{2}\, \frac{\epsilon^2W_{tt}(0, \epsilon z)}{\beta^4} e x^2{\mathcal Z}
\nonumber\\[2mm]
&\qquad\quad
+\Big[\frac{2|\beta'|^2}{\beta^4}
+\frac{\beta''} {\beta^3}
\,-\,\frac{\kappa^2}{\beta^2}
\Big]e x {\mathcal Z}'
+
\frac{|\beta'|^2}{\beta^4} e x^2{\mathcal Z}''
+ \frac1{\beta^2}e^3 {\mathcal Z}^3
\Bigg\}.
\end{align}

Finally, we can rewrite $S(w_0)$ in \eqref{S(w_0)} as
\begin{align}
 S(w_0) \,= \,&
 \epsilon S_1 + \epsilon S_2 + \epsilon^2 S_3 + \epsilon^2 S_4 + \epsilon^3 S_5
 + \epsilon^3 S_6 + \epsilon^3 S_7+\epsilon^3 S_8 +\epsilon^3\frac{4\pi^2}{{\tilde\ell}^2} \frac{e''}{\beta}\mathcal Z
 +\frac1{\beta} \epsilon\lambda_0 e{\mathcal Z}
\nonumber\\[2mm]
 &
 \underbrace{ +\epsilon^4 \frac{4\pi^2} {{\tilde\ell}^2} \frac{\kappa a_2}{\, \beta^2\, } e'' (x+f)\mathcal Z
  -\epsilon^4  \frac{ 4\pi^2} {{\tilde\ell}^2}  \frac{\kappa a_2}{\, \beta^2\, } (x+f) f'' e \mathcal Z' + \frac{\epsilon^4}{\beta^2} \hat{\mathcal F}_1\left[\tau, x, \beta f, \frac{\beta f''}{\tilde\ell^2} \right] }_{\text{terms involving}~f'', e''}
\nonumber\\[2mm]
 &
 + \frac{\epsilon^4}{\beta^2} \mathcal F_1\left[\tau, x, \beta f, \frac{\beta f'}{\tilde\ell}\right]
 +\frac{\epsilon^4} {\beta^2} \mathcal G_1\left[\tau, x, \beta f, \frac{\beta f'}{\tilde\ell}, e, \frac{e'}{\tilde\ell}\right].
\label{error1}
\end{align}
We note that the expressions of the following terms
$$
\frac{\epsilon^4}{\beta^2} \mathcal F_1\left[\tau, x, \beta f, \frac{\beta f'}{\tilde\ell}\right],
\quad
\frac{\epsilon^4}{\beta^2} \hat{\mathcal F}_1\left[\tau, x, \beta f, \frac{\beta f''}{\tilde\ell^2} \right],
\quad
\frac{\epsilon^4} {\beta^2} \mathcal G_1\left[\tau, x, \beta f, \frac{\beta f'}{\tilde\ell}, e, \frac{e'}{\tilde\ell}\right]
$$
are given in \eqref{mathcal F1}, \eqref{hat mathcal F1}, \eqref{mathcal G1} respectively.

\medskip
\subsection{ The second approximate solution via first correction}\label{Section5.2}\

\subsubsection{}
In this part, we want to construct a correction function to eliminate the terms $\epsilon S_1$ and $\epsilon S_2$
in the error $S(w_0)$.
By the same method in \cite{delPKowWei1}, we first choose correction terms $\phi_{11}, {\phi_{12}}$ to cancel the terms
$ S_1, S_2. $
That is to say, for fixed $\tau\in [0, 2\pi)$,
we need solutions of
\begin{align*}
\phi_{11, xx} - \phi_{11} + 3 U^2 \phi_{11} \,=\, - S_1, \quad \phi_{11}(\pm \infty, \tau) \,=\, 0,
\end{align*}
\begin{align*}
\phi_{12, xx} - {\phi_{12}} + 3 U^2 {\phi_{12}} \,=\, - S_2, \quad{\phi_{12}}(\pm \infty, \tau) \,=\, 0.
\end{align*}
\medskip
As it is well known, those problems are solvable provided that
\begin{equation*}
 \int_{{\mathbb R}} S_1\, U' \,{\mathrm d}x \,=\, 0,
\qquad
 \int_{{\mathbb R}} S_2\, U' \,{\mathrm d}x \,=\, 0.
\end{equation*}
Using the identities \eqref{identities for U2},
we can get
\begin{align*}
 \int_{{\mathbb R}} S_1\, U' \,{\mathrm d}x
\,=\, & \frac{\kappa}{\beta}\int_{{\mathbb R}}\Big[ U' +\frac{2}{3}xU \Big] \, U' \,{\mathrm d}x=0.
\end{align*}
Thus, for $\phi_{11}$, we have
\begin{align*}
\phi_{11} (x, \tau) \,=\, \frac{ \kappa}{\beta}\varpi_1(x),
\end{align*}
and $\varpi_1$ is the unique solution of
\begin{align}\label{varphi1equation new}
\varpi_{1, xx} - \varpi_1 + 3 U^2 \varpi_1 \,=\, -\Big( U' + \frac 2 3 x U\Big), \quad \int_{{\mathbb R}} \varpi_1 U' \,{\mathrm d}x \,=\, 0.
\end{align}
In fact, the function $\varpi_1$ is an odd function and enjoys the exponential decay
\[ |\varpi_1 |\leq C |x| e^{-c|x|}, \quad \text{when $|x|$ is large}. \]
On the other hand, we get
\begin{align*}
{\phi_{12}} (x, \tau)
 \,=\, \frac{2}{3} \frac{ \kappa}{\, \beta\, } f \varpi_2(x),
\end{align*}
with
$\varpi_2 (x) \,=\, -\frac 12(U+ x U'), $
and $\varpi_2$ satisfies
\begin{align} \label{varphi2equation new}
\varpi_{2, xx} - \varpi_2 + 3 U^2 \varpi_2 \,=\, -U.
\end{align}
Therefore, we define the correction as
\begin{align}\label{definition of phi1}
 \epsilon\phi_1(x, \tau)
 \,=\, \epsilon{\phi_{11}}(x, \tau) + \epsilon{\phi_{12}}(x, \tau)
 = \epsilon{\mathbf b}_1(\epsilon z)\varpi_1(x)
\ +\
\epsilon\frac{2}{3} {\mathbf b}_1(\epsilon z)f(\tau) \varpi_2(x),
\end{align}
where the relation between $z$ and $\tau$ is given in \eqref{ztau}
and
\begin{equation}\label{definitionofb1}
{\mathbf b}_1(\theta)=\frac{\kappa(\theta)}{\beta(\theta)}.
\end{equation}

Then we take
\begin{align}\label{secondapproximatesolution}
w_1(x, \tau) \,=\, w_0(x, \tau)\,+\, \epsilon\phi_1(x, \tau)
\end{align}
as the {\bf{second approximate solution}}.
From \eqref{error1}, we can derive that
\begin{align}\label{sw1new}
S (w_0+ \epsilon\phi_1)
\,=\, &S(w_0) + \epsilon \Big[\phi_{1, xx} - \phi_1 +3 U^2\phi_1 \Big]
+ \frac{\epsilon^2 4\pi^2}{{\tilde\ell}^2} \epsilon\phi_{1, \tau\tau}
\nonumber\\[2mm]
& +\left[ \left( w_0+ \epsilon\phi_1\right)^3 -w_0^3 - \epsilon 3U^2 \phi_1 \right]
+ B_4(\epsilon\phi_1)+ B_5(\epsilon\phi_1)+ B_6(\epsilon\phi_1)
\nonumber\\[2mm]
=&
\epsilon^2 S_3+\epsilon^2 S_4 + \epsilon^3 S_5 + \epsilon^3 S_6
+\epsilon^3 S_7+\epsilon^3 S_8
+\frac1\beta\epsilon^3\frac{4\pi^2}{{\tilde\ell}^2} e''\mathcal Z
+ \frac1\beta \epsilon\lambda_0 e{\mathcal Z}
\nonumber\\[2mm]
&
 \underbrace{ +\epsilon^4 \frac{4\pi^2} {{\tilde\ell}^2} \frac{\kappa a_2}{\, \beta^2\, } e'' (x+f)\mathcal Z
 -\epsilon^4 \frac{ 4\pi^2} {{\tilde\ell}^2}  \frac{\kappa a_2 }{ \beta^2} (x+f) f'' e \mathcal Z' + \frac{\epsilon^4}{\beta^2} \hat{\mathcal F}_1\left[\tau, x, \beta f, \frac{\beta f''}{\tilde\ell^2} \right] }_{\text{terms involving}~f'', e''}
\nonumber\\[2mm]
 &
 + \frac{\epsilon^4}{\beta^2} \mathcal F_1\left[\tau, x, \beta f, \frac{\beta f'}{\tilde\ell}\right]
 +\frac{\epsilon^4} {\beta^2} \mathcal G_1\left[\tau, x, \beta f, \frac{\beta f'}{\tilde\ell}, e, \frac{e'}{\tilde\ell}\right]
\nonumber\\[2mm]
&+ \frac{\epsilon^2 4\pi^2}{{\tilde\ell}^2} \epsilon\phi_{1, \tau\tau}
+ \mathbb{N}_0(\epsilon\phi_1) + B_4(\epsilon\phi_1)+ B_5(\epsilon\phi_1)+ B_6(\epsilon\phi_1),
\end{align}
where the linear operators $ B_4, B_5, B_6 $ are given in \eqref{new opreator B4}, \eqref{new opreator B5}, \eqref{B2w-2222new},
and the nonlinear term is
\begin{align}\label{nonlinerarterm}
\mathbb{N}_0(\epsilon\phi_1)
=
\left( w_0+ \epsilon\phi_1\right)^3 -w_0^3 - \epsilon 3U^2 \phi_1.
\end{align}

By some calculations, we provide the following facts
\begin{align*}
\phi_{1, x}
={\mathbf b}_1\varpi_{1,x} +\frac{2}{3}{\mathbf b}_1 f \varpi_{2,x},
\qquad
\phi_{1, xx}
={\mathbf b}_1\varpi_{1,xx} +\frac{2}{3}{\mathbf b}_1 f \varpi_{2,xx},
\end{align*}
\begin{align*}
\phi_{1, \tau}
= & {\mathbf b}_1'\frac{{\tilde\ell}}{2\pi \beta} \varpi_{1}
+
\frac{2}{3}{\mathbf b}_1'\frac{{\tilde\ell}}{2\pi\beta} f\varpi_{2}
+
 \frac{2}{3}{\mathbf b}_1 f'\varpi_{2}
\nonumber\\[2mm]
= & \frac{{\tilde\ell}}{2\pi} {\mathbf b}_2 \varpi_{1}
+
\frac{2}{3}\frac{{\tilde\ell}}{2\pi}{\mathbf b}_2f \varpi_{2}
+
\frac{2}{3}{\mathbf b}_1 f'\varpi_{2},
\end{align*}
where
\begin{align*}
{\mathbf b}_2(\theta)=\frac{{\mathbf b}_1'(\theta)}{\beta(\theta)},
\end{align*}
and
\begin{align*}
\phi_{1, x\tau}
= &
\frac{{\tilde\ell}}{2\pi} {\mathbf b}_2 \varpi_{1,x}
+
\frac{2}{3}\frac{{\tilde\ell}}{2\pi}{\mathbf b}_2 f\varpi_{2,x}
+
\frac{2}{3}{\mathbf b}_1 f'\varpi_{2,x},
\end{align*}
\begin{align*}
\phi_{1, \tau\tau}
= &
\frac{{\tilde\ell}}{2\pi} {\mathbf b}_2'\frac{{\tilde\ell}}{2\pi\beta} \varpi_{1}
+
\frac{2}{3}\frac{{\tilde\ell}}{2\pi}{\mathbf b}_2'\frac{{\tilde\ell}}{2\pi\beta}f \varpi_{2}
+
\frac{2}{3}\frac{{\tilde\ell}}{2\pi}{\mathbf b}_2f' \varpi_{2}
+
\frac{2}{3}{\mathbf b}_1'\frac{{\tilde\ell}}{2\pi\beta} f'\varpi_{2}
 +
\frac{2}{3}{\mathbf b}_1 f''\varpi_{2}
\nonumber\\
= &
\frac{{\tilde\ell}^2}{4\pi^2} {\mathbf b}_3 \varpi_{1}
+
\frac{2}{3}\frac{{\tilde\ell}^2}{4\pi^2}{\mathbf b}_3 f \varpi_{2}
+
\frac{4}{3}\frac{{\tilde\ell}}{2\pi}{\mathbf b}_2f' \varpi_{2}
+
\frac{2}{3}{\mathbf b}_1 f''\varpi_{2},
\label{expression of phi1tautau}
\end{align*}
where
\begin{align}
{\mathbf b}_3(\theta)=\frac{{\mathbf b}_2'(\theta)}{\beta(\theta)}.
\end{align}

\medskip
\subsubsection{More precise expression of $S(w_1)$}

We will do the following analysis.

\smallskip\noindent $\clubsuit$
From the expression of $\epsilon^2 S_4$ as in \eqref{expression of s4}, we can further decompose the term $\epsilon^2S_4$ as follows
\begin{align} \label{decompose S3+S4}
\epsilon^2 S_4 = \epsilon^2 S_{41} + \epsilon^2 S_{42},
\end{align}
where
\begin{align}\label{barS41}
 \epsilon^2 S_{41}& =\epsilon^2\Bigg\{
\frac{\beta''}{\beta^3}U
-
\frac{1}{2}\, \frac{\epsilon^2W_{tt}(0, \epsilon z)}{\beta^4} x^2U
+\Big[\frac{2|\beta'|^2}{\beta^4}
+\frac{\beta''} {\beta^3}
\,-\,\frac{\kappa^2}{\beta^2}
\Big]xU'
+
\frac{|\beta'|^2}{\beta^4} x^2 U''
\Bigg\},
\end{align}

\begin{align}\label{barS42}
 \epsilon^2 S_{42}& =\epsilon^2\Bigg\{
-\frac{1}{2}\, \frac{ \epsilon^2W_{tt}(0, \epsilon z)}{\beta^4}|f|^2 U
+
\Big[\frac{4\pi^2}{{\tilde\ell}^2}|f'|^2
-\frac{4\pi}{{\tilde\ell}}\frac{\beta'}{\beta^2}ff'
+\frac{|\beta'|^2}{\beta^4} |f|^2
\Big]U''
\nonumber\\[2mm]
 &\qquad\quad
+\frac 23 \frac{\kappa}{\, \beta^2\, } fe{\mathcal Z}
+\frac3{\beta^2}e^2U{\mathcal Z}^2
\Bigg\}.
\end{align}

\smallskip\noindent $\clubsuit$
Moreover, by the definitions of $B_4$, $\epsilon \phi_{1}$ as  in \eqref{new opreator B4}, \eqref{definition of phi1}, we can decompose $ B_4(\epsilon\phi_1)$ as
\begin{align}\label{B4phi1}
 B_4(\epsilon\phi_1)
\,=\,& \epsilon^2 \frac{\kappa}{\, \beta\, } \Big[ \phi_{1,x} + \frac{2}{3}x\phi_1+\frac{2}{3}f\phi_1\Big]
\nonumber\\[2mm]
\,=\,&\epsilon^2 {\mathbf b}_1^2\Big[
\varpi_{1,x}
+ \frac{2}{3}f\varpi_{2,x}
+\frac{2}{3}x\varpi_1
+ \frac{4}{9} f x\varpi_2
+\frac{2}{3}f\varpi_1
 + \frac{4}{9} f^2\varpi_2
 \Big]
\nonumber\\[2mm]
\,=\,& \epsilon^2 {\bf B}_{41} +\epsilon^2{\bf B}_{42},
\end{align}
with
\begin{align} \label{B42}
{\bf B}_{41} =
 {\mathbf b}_1^2\Big[
\varpi_{1,x}
+\frac{2}{3}x\varpi_1
\Big],\qquad
{\bf B}_{42} =
 {\mathbf b}_1^2\Big[
\frac{2}{3}f\varpi_{2,x}
+ \frac{4}{9} f x\varpi_2
+\frac{2}{3}f\varpi_1
+ \frac{4}{9} f^2\varpi_2
\Big].
\end{align}

\smallskip\noindent $\clubsuit$
 From the expression of $B_5$ given in \eqref{new opreator B5}, we have
\begin{align} \label{expression of B5phi1}
B_5(\epsilon\phi_1) \, =\,&\epsilon^3
\Bigg\{\frac{2\pi}{{\tilde\ell}}\frac{\beta'} {\beta^2}\Big[- f' \phi_{1,x} +\phi_{1,\tau}\Big]
 +\frac{4\pi^2}{{\tilde\ell}^2}\Big[|f'|^2 \phi_{1,xx} -f''\phi_{1,x} -2f' \phi_{1,x\tau}\Big]
\nonumber\\[2mm]
&\qquad
+ \frac{4\pi}{{\tilde\ell}} \frac{\beta'}{\beta^2} (x+f)\big[- f' \phi_{1,xx} +\phi_{1,x\tau}\big]
+\Big[\frac{\beta''}{\beta^3} \,-\,\frac{\kappa^2}{\beta^2} \Big]
(x+f) \phi_{1,x}
\nonumber\\[2mm]
&\qquad
+ \frac{|\beta'|^2}{\beta^4}(x+f)^2 \phi_{1,xx}
-\frac{1}{2}\, \frac{\epsilon^2W_{tt}(0, \epsilon z)}{\beta^4} (x+f)^2\phi_{1}
\Bigg\}.
\end{align}

\smallskip\noindent $\clubsuit$
By the definition of $B_6$ given in \eqref{B2w-2222new}, we
 select the terms from $ B_6(\epsilon\phi_1)$ involving $f''$
\begin{align}
& \epsilon^4 \kappa(\epsilon z) a_2(\epsilon s, \epsilon z) \frac{ 4\pi^2} {{\tilde\ell}^2}\frac{1}{\beta} (x+f)\big[ -f''\phi_{1, x} +\phi_{1, \tau\tau}\big]
\nonumber\\[2mm]
& = \epsilon^4 \kappa(\epsilon z) a_2(\epsilon s, \epsilon z) \frac{ 4\pi^2} {{\tilde\ell}^2}\frac{1}{\beta} (x+f)\Bigg[ -f'' \big({\mathbf b}_1\varpi_{1,x} +\frac{2}{3}{\mathbf b}_1 f \varpi_{2,x} \big)
\nonumber\\[2mm]
& \qquad\qquad
+ \Big(\frac{{\tilde\ell}^2}{4\pi^2} {\mathbf b}_3 \varpi_{1}
+
\frac{2}{3}\frac{{\tilde\ell}^2}{4\pi^2}{\mathbf b}_3 f \varpi_{2}
+
\frac{4}{3}\frac{{\tilde\ell}}{2\pi}{\mathbf b}_2f' \varpi_{2}
+
\frac{2}{3}{\mathbf b}_1 f''\varpi_{2} \Big)\Bigg]
\nonumber\\[2mm]
& = \underbrace{\epsilon^4 \kappa(\epsilon z) a_2(\epsilon s, \epsilon z) \frac{ 4\pi^2} {{\tilde\ell}^2}\frac{1}{\beta} (x+f)\Bigg[ -f'' \big({\mathbf b}_1\varpi_{1,x} +\frac{2}{3}{\mathbf b}_1 f \varpi_{2,x} \big)
 +
\frac{2}{3}{\mathbf b}_1 f''\varpi_{2} \Bigg]}_{:= \frac{\epsilon^4}{\beta^2} \hat{\mathcal F}_2\left[\tau, x, \beta f, \frac{\beta f''}{\tilde\ell^2} \right] }
\nonumber\\[2mm]
& \quad + \epsilon^4 \kappa(\epsilon z) a_2(\epsilon s, \epsilon z) \frac{ 4\pi^2} {{\tilde\ell}^2}\frac{1}{\beta} (x+f) \Big(\frac{{\tilde\ell}^2}{4\pi^2} {\mathbf b}_3 \varpi_{1}
+
\frac{2}{3}\frac{{\tilde\ell}^2}{4\pi^2}{\mathbf b}_3 f \varpi_{2}
+
\frac{4}{3}\frac{{\tilde\ell}}{2\pi}{\mathbf b}_2f' \varpi_{2} \Big)
\nonumber\\[2mm]
& = \frac{\epsilon^4}{\beta^2} \hat{\mathcal F}_2\left[\tau, x, \beta f, \frac{\beta f''}{\tilde\ell^2} \right]
\nonumber\\[2mm]
& \quad+ \epsilon^4 \kappa(\epsilon z) a_2(\epsilon s, \epsilon z) \frac{ 4\pi^2} {{\tilde\ell}^2}\frac{1}{\beta} (x+f) \Big(\frac{{\tilde\ell}^2}{4\pi^2} {\mathbf b}_3 \varpi_{1}
+
\frac{2}{3}\frac{{\tilde\ell}^2}{4\pi^2}{\mathbf b}_3 f \varpi_{2}
+
\frac{4}{3}\frac{{\tilde\ell}}{2\pi}{\mathbf b}_2f' \varpi_{2} \Big). \label{hat mathcal F2}
\end{align}
And we can rewrite $ B_6(\epsilon\phi_1) $ as
\begin{align}\label{expression of B6phi1}
 B_6(\epsilon\phi_1) = \frac{\epsilon^4}{\beta^2} \mathcal F_2\left[\theta, x, \beta f, \frac{\beta f'}{\tilde\ell}\right]
 + \frac{\epsilon^4}{\beta^2} \hat{\mathcal F}_2\left[\tau, x, \beta f, \frac{\beta f''}{\tilde\ell^2} \right],
\end{align}
where
\begin{align} \label{mathcal F2}
& \frac{\epsilon^4}{\beta^2} \mathcal F_2\left[\theta, x, \beta f, \frac{\beta f'}{\tilde\ell}\right]
\nonumber\\[2mm]
=&
 - \frac{\epsilon^4}{6} \frac{\, \, \epsilon^2 W_{ttt}(0, \epsilon z)}{\beta^5} (x+f)^3 \phi_1
 -\frac{\epsilon^5}{24} \frac{\, \, \epsilon^2 W_{tttt}\big(\hbar(\epsilon z)\epsilon s, \, \epsilon z\big)}{\beta^6} (x+f)^4\phi_1
\nonumber\\[2mm]
&
 +\epsilon^4 \kappa^3(\epsilon z) a_3(\epsilon s, \epsilon z) \frac{1}{\, \beta^3\, }(x+f)^2\phi_{1, x}
\nonumber\\[2mm]
&\,+\,
\epsilon^4 \kappa(\epsilon z) a_2(\epsilon s, \epsilon z)
\Bigg\{
\frac{\beta''} {\beta^4} (x+f) \phi_1
+\frac{2|\beta'|^2}{\beta^5} (x+f)^2 \phi_{1, x}
+\frac{ 6\pi}{ {\tilde\ell}}\frac{\beta'} {\beta^3} (x+f)\big[-f'\phi_{1, x} +\phi_{1, \tau}\big]
\nonumber\\[2mm]
&\qquad\qquad\qquad\qquad\qquad
+\frac{ 4\pi}{ {\tilde\ell}}\frac{\beta'} {\beta^3} (x+f)^2[- f' w_{xx} +w_{x\tau}]
+ \frac{|\beta'|^2}{\beta^5} (x+f)^3\phi_{1, xx}
+ \frac{\beta''}{\beta^4} (x+f)^2 \phi_{1, x}
\nonumber\\[2mm]
&\qquad\qquad\qquad\qquad\qquad
+\frac{ 4\pi^2} {{\tilde\ell}^2}\frac{1}{\beta} (x+f)\big[ |f'|^2 \phi_{1, xx} -2f' \phi_{1, x\tau}\big]
\Bigg\}
\nonumber\\[2mm]
&\,+\,
\epsilon^4 \kappa'(\epsilon z)a_1(\epsilon s, \epsilon z)
\Bigg\{
\frac{\beta'}{\beta^4} (x+f)\phi_1
+\frac{\beta'}{\beta^4} (x+f)^2 \phi_{1, x}
+\frac{ 2\pi}{ {\tilde\ell}}\frac{1}{ \beta^2} (x+f)\big[-f'\phi_{1, x} +\phi_{1, \tau}\big]
\Bigg\}
\nonumber\\[2mm]
& + \epsilon^4 \kappa(\epsilon z) a_2(\epsilon s, \epsilon z) \frac{ 4\pi^2} {{\tilde\ell}^2}\frac{1}{\beta} (x+f) \Big(\frac{{\tilde\ell}^2}{4\pi^2} {\mathbf b}_3 \varpi_{1}
+
\frac{2}{3}\frac{{\tilde\ell}^2}{4\pi^2}{\mathbf b}_3 f \varpi_{2}
+
\frac{4}{3}\frac{{\tilde\ell}}{2\pi}{\mathbf b}_2f' \varpi_{2} \Big).
\end{align}

\smallskip\noindent $\clubsuit$
Similarly, for $\mathbb{N}_0(\epsilon\phi_1)$ as in \eqref{nonlinerarterm}, we also have the decomposition
\begin{align}\label{decompose N0}
 \mathbb{N}_0(\epsilon\phi_1)
\,=\,&
 \left[ \frac{\epsilon^26e U\mathcal Z \phi_1}{\beta}
 +\frac{\epsilon^3 3 e^2\mathcal Z^2 \phi_1}{\beta^2}
 +\epsilon^2 3 U \phi_{1}^2
 +\frac{\epsilon^3 3 e{\mathcal Z} \phi_{1}^2}{\beta}
 +\epsilon^3 \phi_{1}^3\right]
\nonumber\\[2mm]
\,=\,&
\frac{\epsilon^2 6 e U{\mathcal Z} {\mathbf b}_1}\beta\Big[\varpi_1 + \frac{2}{3}f\varpi_2\Big]
+
\epsilon^2 3 {\mathbf b}_1^2U\Big[\varpi_1 + \frac{2}{3} f\varpi_2\Big]^2
\nonumber\\[2mm]
&
+
\frac{\epsilon^3 3 e^2\mathcal Z^2 {\mathbf b}_1}{\beta^2}\Big[\varpi_1 + \frac{2}{3} f\varpi_2\Big]
+
\frac{\epsilon^3 3 e{\mathcal Z} {\mathbf b}_1^2}{\beta}\Big[\varpi_1 + \frac{2}{3} f\varpi_2\Big]^2
+
\epsilon^3 {\mathbf b}_1^3\Big[\varpi_1 + \frac{2}{3}f\varpi_2\Big]^3
\nonumber\\[2mm]
\,=\,&
\frac{\epsilon^2 6 e U{\mathcal Z} {\mathbf b}_1}\beta\Big[\varpi_1 + \frac{2}{3}f\varpi_2 \Big]
+ \epsilon^2 3 {\mathbf b}_1^2U\Big[\varpi_1^2+\frac{4}{3}f \varpi_1\varpi_2+\frac{4}{9}|f|^2 \varpi_2^2 \Big]
\nonumber\\[2mm]
&
+ \frac{\epsilon^3 3 e^2\mathcal Z^2 {\mathbf b}_1}{\beta^2}\Big[\varpi_1 + \frac{2}{3}f \varpi_2\Big]
 + \frac{\epsilon^3 3 e{\mathcal Z} {\mathbf b}_1^2}{\beta}\Big[\varpi_1 + \frac{2}{3} f \varpi_2\Big]^2
 +\epsilon^3 {\mathbf b}_1^3\Big[\varpi_1 + \frac{2}{3}f\varpi_2\Big]^3
\nonumber\\[2mm]
\,=\,&\epsilon^2 {\bf N}_{01}+\epsilon^2 {\bf N}_{02}+\epsilon^3{\bf N}_{03},
\end{align}
with
\begin{align}\label{N01}
{\bf N}_{01} = 3{\mathbf b}_1^2U\varpi_1^2,
\end{align}
\begin{align}\label{N02}
{\bf N}_{02} =
\frac{6 {\mathbf b}_1e U{\mathcal Z}\varpi_1
+4 {\mathbf b}_1fe U{\mathcal Z} \varpi_2}\beta
+\frac{4}{3}{\mathbf b}_1^2|f|^2 U\varpi_2^2
+ 4{\mathbf b}_1^2f U\varpi_1\varpi_2,
\end{align}
\begin{align}\label{N03}
{\bf N}_{03} =&
\frac{3 e^2\mathcal Z^2 {\mathbf b}_1}{\beta^2}\Big[\varpi_1 + \frac{2}{3}f \varpi_2\Big]
+ \frac{3 e{\mathcal Z} {\mathbf b}_1^2}{\beta}\Big[\varpi_1 + \frac{2}{3}f \varpi_2\Big]^2
+ {\mathbf b}_1^3\Big[\varpi_1 + \frac{2}{3}f \varpi_2\Big]^3.
\end{align}

\medskip
 Combining \eqref{sw1new}, \eqref{decompose S3+S4}, \eqref{B4phi1}, \eqref{decompose N0}, we can get
\begin{align*}
 S(w_1)
 \,=\, &\,\epsilon^2 S_{3}
 + \epsilon^2 S_{41} + \epsilon^3 S_{42}+ \epsilon^3 S_{5}
 +\epsilon^3 S_{6} + \epsilon^3 S_{7} +\epsilon^3S_{8} + \frac1\beta\frac{4\pi^2}{{\tilde\ell}^2} e''\mathcal Z
 + \frac1\beta\epsilon\lambda_0 e{\mathcal Z}
\nonumber\\[2mm]
 &
 \underbrace{ +\epsilon^4 \frac{4\pi^2} {{\tilde\ell}^2} \frac{\kappa a_2}{\, \beta^2\, } e'' (x+f)\mathcal Z
 -\epsilon^4  \frac{ 4\pi^2} {{\tilde\ell}^2} \frac{\kappa a_2}{\, \beta^2\, } (x+f) f'' e \mathcal Z' + \frac{\epsilon^4}{\beta^2} \hat{\mathcal F}_1\left[\tau, x, \beta f, \frac{\beta f''}{\tilde\ell^2} \right] }_{\text{terms involving}~f'', e''}
\nonumber\\[2mm]
 &
+\underbrace{\epsilon^2 {\bf B}_{41} +\epsilon^2{\bf B}_{42}}_{B_4(\epsilon\phi_1)}
+\underbrace{\epsilon^2 {\bf N}_{01}+\epsilon^2 {\bf N}_{02} +\epsilon^3{\bf N}_{03}}_{\mathbb{N}_0(\epsilon\phi_1)}
 + \frac{\epsilon^4}{\beta^2} \mathcal F_1\left[\tau, x, \beta f, \frac{\beta f'}{\tilde\ell}\right]
\nonumber\\[2mm]
 &
 +\frac{\epsilon^4} {\beta^2} \mathcal G_1\left[\tau, x, \beta f, \frac{\beta f'}{\tilde\ell}, e, \frac{e'}{\tilde\ell}\right]
 + \epsilon^3\frac{ 4\pi^2}{{\tilde\ell}^2} {\phi}_{1, \tau\tau}
 +B_{5}(\epsilon\phi_1)+B_{6}(\epsilon\phi_1).
\end{align*}
 Note that,  $B_{5}(\epsilon\phi_1), B_{6}(\epsilon\phi_1)$ as in \eqref{expression of B5phi1}, \eqref{expression of B6phi1}.

\subsection{ The third approximate solution via second correction}\label{Section5.3}\

In this part, we want to eliminate the terms
$\epsilon^2 S_{41},  \epsilon^2 {\bf B}_{41}, \epsilon^2 {\bf N}_{01} $
in $S(w_1)$ by finding a solution to the problem
\begin{align}
\label{new second croection}
 & \phi_{2, xx} - \phi_{2} + 3 U^2 \phi_{2}
 \,=\, - \left[ S_{41} + {\bf B}_{41} + {\bf N}_{01}\right],
\quad \phi_{2}(\pm \infty, \tau) \,=\, 0.
\end{align}
In order to solve \eqref{new second croection}, we need to consider the following linear problem
\begin{align*}
\varpi_{2, xx} - \varpi_2 + 3 U^2 \varpi_2 \,=\, -U, \qquad \varpi_2(\pm \infty) \,=\, 0,
\end{align*}
\begin{align*}
\varpi_{3, xx} - \varpi_{3} + 3 U^2 \varpi_{3} \,=\,  - x^2U,
\qquad \varpi_3(\pm \infty) \,=\, 0,
\end{align*}
\begin{align*}
\varpi_{4, xx} - \varpi_{4} + 3 U^2 \varpi_{4} \,=\,- xU',
\qquad \varpi_4(\pm \infty) \,=\, 0,
\end{align*}
\begin{align*}
\varpi_{5, xx} - \varpi_{5} + 3 U^2\varpi_{5}
\,=\, -x^2U'',
\qquad \varpi_5(\pm \infty) \,=\, 0,
\end{align*}
\begin{align*}
\varpi_{6, xx} - \varpi_{6} + 3 U^2\varpi_{6}
\,=\,-\varpi_{1,x} ,
\qquad \varpi_6(\pm \infty) \,=\, 0,
\end{align*}
\begin{align*}
\varpi_{7, xx} - \varpi_{7} + 3 U^2\varpi_{7}
\,=\, -x\varpi_1,
\qquad \varpi_7(\pm \infty) \,=\, 0,
\end{align*}
\begin{align*}
\varpi_{8, xx} - \varpi_{8} + 3 U^2\varpi_{8}
\,=\, -U\varpi_1^2,
\qquad \varpi_{8}(\pm \infty) \,=\, 0.
\end{align*}
Recall the expressions of $\epsilon^2 S_{41}$, $ \epsilon^2 {\bf B}_{41}$, $\epsilon^2 {\bf N}_{01} $ given in \eqref{barS41},  \eqref{B42}  and \eqref{N02} respectively.
Then we can give the solution of \eqref{new second croection} as
\begin{align}\label{definition of e2phi2}
\epsilon^2\phi_2 = \epsilon^2\sum_{j=1}^{7}\phi_{2j},
\end{align}
where
\begin{align}\label{phi21-phi210}
& \phi_{21} = \frac{\beta''} {\beta^3}\varpi_2,
\qquad
\phi_{22} =- \frac{1}{2}\, \frac{\epsilon^2W_{tt}(0, \epsilon z)}{\beta^4}\varpi_3,
\qquad
\phi_{23} = \Big[\frac{2|\beta'|^2}{\beta^4}
+\frac{\beta''} {\beta^3}\,-\,\frac{\kappa^2}{\beta^2}\Big]\varpi_4,
\nonumber\\[2mm]
&
\phi_{24} = \frac{|\beta'|^2}{\beta^4} \varpi_5,
\qquad \qquad
\phi_{25} = {\mathbf b}_1^2\varpi_6,
\qquad \qquad
\phi_{26} =\frac23 {\mathbf b}_1^2\varpi_7,
\qquad
\phi_{27} = 3 {\mathbf b}_1^2\varpi_{8}.
\end{align}
We observe that $\phi_2$ is an even function of $x$.

\medskip\medskip
As a conclusion, we take
\begin{align}\label{newthirdapproximatesolution}
w_2(x, \tau) \,=\, w_1(x, \tau)\,+\,\epsilon^2 \phi_2(x, \tau)
\end{align}
as the {\bf{third approximate solution}}.
Then we have the new error
\begin{align}\label{sw2new}
S (w_2)
\,=\,
&S(w_1)
+
\epsilon^2\Big[ \phi_{2, xx} - \phi_2 + 3 U^2 \phi_2 \Big]
+
\left[ \left( w_1+ \epsilon^2\phi_2\right)^3 -w_1^3- \epsilon^2 3U^2\phi_2 \right]
\nonumber\\[2mm]
& + \epsilon^4\frac{4\pi^2}{{\tilde\ell}^2}\phi_{2, \tau\tau}
+ B_4(\epsilon^2\phi_2)+ B_5(\epsilon^2\phi_2)+ B_6(\epsilon^2\phi_2)
\nonumber\\[2mm]
\,=\,&
\epsilon^2 S_{3} + \epsilon^2 S_{42} + \epsilon^3 S_{5} +\epsilon^3 S_{6}
 +\epsilon^3 S_{7} +\epsilon^3 S_{8}
 +\frac 1\beta\epsilon^3\frac{4\pi^2}{{\tilde\ell}^2} e''\mathcal Z
 +\frac 1\beta \epsilon\lambda_0 e{\mathcal Z}
 +\epsilon^2 {\bf N}_{02}
 +\epsilon^3{\bf N}_{03}
\nonumber\\[2mm]
& + \epsilon^2 {\bf B}_{42}
 \underbrace{ +\epsilon^4 \frac{4\pi^2} {{\tilde\ell}^2} \frac{\kappa a_2}{\, \beta^2\, } e'' (x+f)\mathcal Z
 -\epsilon^4  \frac{ 4\pi^2} {{\tilde\ell}^2} \frac{\kappa a_2}{\, \beta^2\, } (x+f) f'' e \mathcal Z' + \frac{\epsilon^4}{\beta^2} \hat{\mathcal F}_1\left[\tau, x, \beta f, \frac{\beta f''}{\tilde\ell^2} \right] }_{\text{terms involving}~f'', e''}
\nonumber\\[2mm]
&
 + \epsilon^3 \frac{4\pi^2}{{\tilde\ell}^2} {\phi}_{1, \tau\tau}
+ \epsilon^4 \frac{4\pi^2}{{\tilde\ell}^2}\phi_{2, \tau\tau}
+B_{5}(\epsilon\phi_1)+B_{6}(\epsilon\phi_1)
\nonumber\\[2mm]
&
 + \frac{\epsilon^4}{\beta^2} \mathcal F_1\left[\tau, x, \beta f, \frac{\beta f'}{\tilde\ell}\right]
 +\frac{\epsilon^4} {\beta^2} \mathcal G_1\left[\tau, x, \beta f, \frac{\beta f'}{\tilde\ell}, e, \frac{e'}{\tilde\ell}\right]
\nonumber\\[2mm]
&+ \mathbb{N}_1 (\epsilon^2 \phi_2)
+ B_4(\epsilon^2\phi_2)+ B_5 (\epsilon^2{\phi_2})+ B_6 (\epsilon^2{\phi_2}),
\end{align}
where the linear operators $ B_4, B_5, B_6 $ are given in \eqref{new opreator B4}, \eqref{new opreator B5}, \eqref{B2w-2222new}.
The term $\mathbb{N}_1 (\epsilon^2 \phi_2)$ is given by
\begin{align}
\mathbb{N}_1 (\epsilon^2\phi_2)
\,=\, & \left( w_1+ \epsilon^2\phi_2\right)^3 -w_1^3 - \epsilon^2 3 U^2 \phi_2
\nonumber\\[2mm]
\,=\,&\epsilon^3 \frac{6e}{\beta}U{\mathcal Z}\phi_2 + \epsilon^3 6 U\phi_1\phi_2
+\epsilon^4 \frac{3e^2}{\beta^2}{\mathcal Z}^2\phi_2 +\epsilon^4 3\phi_1^2\phi_2
+\epsilon^4 \frac{6e}{\beta}{\mathcal Z}\phi_1\phi_2
\nonumber\\[2mm]
&+ \epsilon^4 3U \phi_2^2
+\epsilon^5 \frac{3e}{\beta}{\mathcal Z}\phi_2^2 + \epsilon^5 3\phi_1 \phi_2^2
+\epsilon^6 \phi_2^3,
\label{definition of mathbbN1}
\end{align}
where $\phi_1$ and $\phi_2$ are given in \eqref{definition of phi1} and \eqref{definition of e2phi2}.

\smallskip\noindent $\clubsuit$
From the expression of $\epsilon^2\phi_2$ as in \eqref{definition of e2phi2}, \eqref{phi21-phi210}, and the definition of the operator $B_4$ as in \eqref{new opreator B4}, it is easy to notice that the terms in $B_4(\epsilon^2\phi_2)$ are independent of $f'', e''$.

\smallskip\noindent $\clubsuit$
By the definition of $B_5$ given in \eqref{new opreator B5}, we have
\begin{align} \label{decomposition of B5(phi2)}
B_5 (\epsilon^2{\phi_2})
\,=\,&\epsilon^4\frac{4\pi^2}{{\tilde\ell}^2}\Big[|f'|^2 \phi_{2, xx} -f''\phi_{2, x} -2f' \phi_{2, x\tau}\Big]
+
\epsilon^4\frac{6\pi}{{\tilde\ell}}\frac{\beta'} {\beta^2}\Big[- f' \phi_{2, x} +\phi_{2, \tau}\Big]
\nonumber\\[2mm]
&
+ \epsilon^4 \frac{4\pi}{{\tilde\ell}} \frac{\beta'}{\beta^2} (x+f)\big[- f' \phi_{2, xx} +\phi_{2, x\tau}\big]
+
\epsilon^4\Big[
\frac{\beta''}{\beta^3}\,+\,\frac{2|\beta'|^2}{\beta^4} \,-\,\frac{\kappa^2}{\beta^2} \Big] (x+f) \phi_{2, x}
\nonumber\\[2mm]
&+\epsilon^4\frac{\beta''}{\beta^3}\phi_2
+ \epsilon^4 \frac{|\beta'|^2}{\beta^4}(x+f)^2 \phi_{2, xx}
-\frac{1}{2}\, \epsilon^4\frac{\epsilon^2W_{tt}(0, \epsilon z)}{\beta^4} (x+f)^2 \phi_2.
\end{align}
Here we note that the less regular term is $$-\epsilon^4\frac{4\pi^2}{{\tilde\ell}^2} f''\phi_{2, x}. $$

\smallskip\noindent $\clubsuit$
We select the term in $B_6 (\epsilon^2{\phi_2})$ that involves $f''$, which is
\begin{align*}
-\epsilon^5 \kappa(\epsilon z) a_2(\epsilon s, \epsilon z) \frac{ 4\pi^2} {{\tilde\ell}^2}\frac{1}{\beta} (x+f)f'' \phi_{2, x}.
\end{align*}
Furthermore, we denote
\begin{align} \label{decomposition of B6phi2}
\frac{\epsilon^5}{\beta^3} \mathcal F_3\left[\tau, x, \beta f, \frac{\beta f'}{\tilde\ell}\right] \,:=\, B_6(\epsilon^2\phi_2) -\Bigg[-\epsilon^5 \kappa(\epsilon z) a_2(\epsilon s, \epsilon z) \frac{ 4\pi^2} {{\tilde\ell}^2}\frac{1}{\beta} (x+f)f'' \phi_{2, x}\Bigg].
\end{align}
Here the terms in $\frac{\epsilon^5}{\beta^3} \mathcal F_3\left[\tau, x, \beta f, \frac{\beta f'}{\tilde\ell}\right]$ are all independent of $f'', e''$ and satisfies
\begin{align} \label{estimate of mathcal F3-x}
\frac{\epsilon^5}{\beta^3} \mathcal F_3\left[\tau, x, \beta f, \frac{\beta f'}{\tilde\ell}\right] \le C e^{-c|x|} \left(1+ |\beta f| + \Big|\frac{\beta f'}{\tilde\ell}\Big| \right).
\end{align}

\medskip
\section{The gluing procedure to derive the projected form of equation  (\ref{eq3})}
\label{Section6}

\subsection{ The gluing procedure} \label{Section6.1}\

The gluing method from \cite{delPKowWei1} will be used to transform problem \eqref{eq3} in ${\mathbb R}^2$
into a projected version on the infinite strip ${\mathfrak S}$, in which we can use the local coordinates $(x, \tau)$ in \eqref{definition of x 2.14}.
We define a smooth cut-off function
$
\eta_{3\delta}^{\epsilon}(s) \,=\, \eta_{3\delta}(\epsilon |s|),
$
where $\eta_{\delta}(t)$ is given by
$$
\eta_{\delta}(t) \,=\, 1\ \mbox{ for }\ 0\leq t\leq \delta
\qquad \mbox{and}\qquad
\eta_{\delta}(t) \,=\, 0
\ \mbox{ for }\ t>2\delta,
$$
for any fixed $\delta\leq \frac{\delta_0}{100}$ with small $\delta_0=M_0$ given in (\ref{fermi}).
For problem \eqref{eq3}, the global approximation $\mathbf{U}(y)$ on the whole space ${\mathbb R}^2$ can be defined simply in the form
\begin{equation}\label{defofbfU}
\mathbf{U}(y) \,=\,
\eta_{3\delta}^\epsilon(s) \beta(\epsilon z)w_2(x, \tau), \quad \forall\,y \in {\mathbb R}^2,
\end{equation}
where the local approximate solution $w_2 (x, \tau)$ is defined in \eqref{newthirdapproximatesolution} and the relations between $y$, $(s, z)$ and $ (x,\tau)$
are given in \eqref{definitionofy}, \eqref{fermi}, \eqref{eqszfinal}, \eqref{coordinate-tildes}, \eqref{ztau}, and \eqref{definition of x 2.14}.

\medskip
For a perturbation term $\Phi(y) \,=\, \eta_{3\delta}^{\epsilon}(\epsilon |s|){\tilde\phi}(y) \,+\, {\tilde\psi}(y)$, the function $v(y) \,=\, \mathbf{U}(y)+\Phi(y)$
satisfies (\ref{eq3}) if the pair $({{\tilde\phi}}, {{ \tilde \psi}})$ satisfies the following coupled system:
\begin{align}
\eta_{3\delta}^{\epsilon}\mathbf{L}(\tilde \phi)
 \,=\,
\eta_{\delta}^{\epsilon}\Big[-{\mathcal E}
+ \mathbf{N}\big(\eta_{3\delta}^{\epsilon}{\tilde\phi}+{\tilde\psi}\big)
-3\, \mathbf{U}^2 \, {\tilde\psi}\Big],
\label{equivalent-system-1}
\end{align}
\begin{align} \label{equivalent-system-2}
 \Delta \tilde \psi- [1+ \epsilon^2 W] \tilde \psi + (1-\eta_{\delta}^{\epsilon}) 3\mathbf{U}^2 \, {\tilde\psi} & \,=\, -\epsilon^2 (\Delta \eta_{3\delta}^{\epsilon}){ \tilde \phi}-2\epsilon \nabla \eta_{3\delta}^{\epsilon} \cdot
\nabla{\tilde\phi}
-(1-\eta_{\delta}^{\epsilon}){\mathcal E}
\nonumber
\\[2mm]
&\quad +(1-\eta_{\delta}^{\epsilon})\mathbf{N}\big(\eta_{3\delta}^{\epsilon}{\tilde\phi}+{\tilde\psi}\big),
\end{align}
 with
\begin{align*}
 \mathbf{L}({\tilde\phi}) \,=\, \Delta{\tilde\phi} - [1+ \epsilon^2 W]{\tilde\phi} +3{\bf U}^2{\tilde\phi},
\end{align*}
\begin{align*}
 {\mathcal E} \,=\, \Delta {\bf U} -[1+\epsilon^2 W] {\bf U} + {\bf U}^3,
\end{align*}
 and
\begin{align*}
 \mathbf{N}\big(\eta_{3\delta}^{\epsilon}{\tilde\phi}+{\tilde\psi}\big)
 & \,=\,
 -\Big[{\bf U} + \big(\eta_{3\delta}^{\epsilon}{\tilde\phi}+{\tilde\psi}\big)\Big]^3 + {\bf U}^3 + 3{\bf U}^2\big(\eta_{3\delta}^{\epsilon}{\tilde\phi}+{\tilde\psi}\big)\nonumber
 \\[2mm]
 & =
 -3{\bf U}\big(\eta_{3\delta}^{\epsilon}{\tilde\phi}+{\tilde\psi}\big)^2 -\big(\eta_{3\delta}^{\epsilon}{\tilde\phi}+{\tilde\psi}\big)^3.
\end{align*}
The term ${ {\mathcal E}}$ can be locally recast in $(x, \tau)$ coordinate system by the relation
\begin{equation}\label{errorrelationinterior}
{{\mathcal E}} \,=\, \beta^3\, S(w_2), \quad \text{for $|\epsilon s|\leq 10\delta$},
\end{equation}
where for the explicit expression of $S(w_2)$, the reader can refer to \eqref{sw2new}.

 First, given a small $ \tilde {\phi}$ satisfying the following decay property
\begin{equation}\label{decay-assumption}
|\nabla {\tilde\phi}|\,+\, |{\tilde\phi}|\,\leq\, e^{- { \upsilon_1\, \delta}/{\epsilon}}, \, \, \, \, \mbox{for}\, \, \, |s|>\delta/\epsilon,
\end{equation}
where $\upsilon_1$ is a very small positive constant, we will solve (\ref{equivalent-system-2}).
Note that the nonlinear operator $\mathbf{N}$ has a power-like behaviour with power greater than one, and a direct application of contraction mapping principle yields
that (\ref{equivalent-system-2}) has a unique (small) solution ${\tilde\psi} \,=\, {\tilde\psi}({\tilde\phi})$ with
\begin{equation}\label{contraction-varphi}
\|{\tilde\psi}({ \tilde \phi})\|_{L^{\infty}}\,\leq \, C\, \epsilon \lf[\|{\tilde\phi}\|_{L^{\infty}(|s|>\frac{\delta}{\epsilon})}\,+\, \|\nabla {\tilde\phi}\|_{L^{\infty}(|s|>\frac{\delta}{\epsilon})}\ri]\,+\,
e^{-\frac{ \upsilon_2 \delta}{\epsilon}},
\end{equation}
where $|s|>\delta/\epsilon$ denotes the complement of $\delta/\epsilon$-neighbourhood of
$\Gamma_\epsilon$ and $\upsilon_2$ denotes a small positive constant.
Moreover, the nonlinear operator $\tilde {\psi}$ satisfies a Lipschitz condition of the form
\begin{equation}\label{varphi-lip}
\|{\tilde\psi}({\tilde\phi}_1)-{\tilde\psi}({ \tilde \phi}_2)\|_{L^{\infty}}
\leq
C\epsilon \lf[ \|{\tilde\phi}_1-{\tilde\phi}_2\|_{L^{\infty}(|s|>\delta/\epsilon)}+\|\nabla {\tilde\phi}_1-\nabla {\tilde\phi}_2\|_{L^{\infty}(|s|>
\delta/\epsilon)}\ri].
\end{equation}

Therefore, after solving (\ref{equivalent-system-2}),
we can concern (\ref{equivalent-system-1}) as a local nonlinear problem involving ${\tilde\psi} \,=\, { \tilde \psi}({\tilde\phi})$,
which can be expressed in local coordinates $(x, \tau)$.
Using Fermi coordinates \eqref{fermi}, \eqref{definitionofy} and the calculations in Section \ref{Section3}, especially \eqref{eqszfinal}, \eqref{coordinate-tildes}, \eqref{ztau} and \eqref{definition of x 2.14}, we give the following lemma.
\begin{lemma}
By the relation
\begin{align}
{\tilde\phi}(y) = \beta(\epsilon z) \phi^*(x, \tau), \quad \text{for $|\epsilon s|\leq 10\delta$},
\label{relation betwen tildephi and phi}
\end{align}
we have
\begin{align}
\mathcal L \left(\phi^*\right)
\,=\,
\frac{1}{\beta^3}\mathbf{L}({\tilde\phi}), \quad \text{for $|\epsilon s|\leq 10\delta$},
\end{align}
where
\begin{align}\label{expression of mathcal L}
\mathcal L (\phi^*) \,=\, L_2(\phi^*) \ +\ \chi(\epsilon|s|) \big[\, B_4(\phi^*)+ B_5(\phi^*)+ B_6(\phi^*)\,\big],
\end{align}
with
\begin{align}\label{definition of L2operator}
 L_2(\phi^*) = \frac{\epsilon^2 4\pi^2}{{\tilde\ell}^2} \phi^*_{\tau\tau} + \phi^*_{xx} - \phi^*+ 3w_2^2 \phi^*.
\end{align}
In the above, the smooth cut-off function $\chi (r)$ is given by
$$
\chi(t) \,=\, 1, \quad \text{if}~t<10\delta, \qquad \chi(t) \,=\, 0, \quad\text{if}\ t>20\delta,
$$
and the operators $B_4$, $B_5$ and $B_6$ have been given in \eqref{new opreator B4}-\eqref{B2w-2222new}.
On the other hand, for the local form of the nonlinear part, we denote by the notation
\begin{align}\label{Nmathcal0}
{\mathcal N}(\phi^*)
\,=\,
\frac{1}{\beta^3}\left[\, \mathbf{N}\big(\eta_{3\delta}^{\epsilon}{\tilde\phi}+{\tilde\psi}\big) -3\, \mathbf{U}^2 \, {\tilde\psi} \,\right],
 \quad \text{for $|\epsilon s|\leq 10\delta$}.
\end{align}
where ${\tilde\psi}$ is concerned as an operator of $\tilde\phi$.\qed
\end{lemma}

\medskip
Note that the approximate solution $\mathbf{U}$ has unknown parameter pair $(f, e)$, see \eqref{defofbfU}.
We will deal with the following projected problem:
for given $(f, e)$ with constraints in \eqref{constraint-f}-\eqref{constraint-e}, finding
functions $\phi^*\in H^2({\mathfrak S})$, \, and $c(\tau), d(\tau) \in L^2(0, 2\pi)$,
such that
\begin{align} \label{project 6.13}
\mathcal L (\phi^*) \,=\, \eta_\delta^\epsilon(s) \Big[-\frac{1}{\beta^3}{\mathcal E} + {\mathcal N}(\phi^*)\Big]
+ c(\tau) \chi(\epsilon|s|) U'(x)+ d(\tau) \chi(\epsilon|s|){\mathcal Z}(x), \quad \text{for}\,(x, \tau)\in \mathfrak S,
\end{align}

\begin{equation}
\phi^*(x, 0) \,=\, \phi^*(x, 2\pi), \qquad \phi^*_\tau(x, 0) \,=\, \phi^*_\tau(x, 2\pi),
\quad
-\infty<x<+\infty,
\label{periodic condition}
\end{equation}

\begin{equation}
\label{projectedproblem3}
 \int_{-\infty}^{+\infty} \phi^*(x, \tau){\mathcal Z}(x) \,{\mathrm d}x
 \,=\, \int_{-\infty}^{+\infty} \phi^*(x, \tau) U'(x)\,{\mathrm d}x
 \,=\, 0, \quad 0<\tau<2\pi,
\end{equation}
where the definition of $\mathfrak S$ has been given in \eqref{huaxies}.

\medskip
\subsection{ Solvability of the linear problem}\label{Section6.2}\

We shall mention that we here basically follow the method in \cite{delPKowWei1} to show the resolution theory for the linear problem.

We first consider the linear problem
\begin{align} \label{linearphi1}
 \frac{\epsilon^2 4\pi^2}{{\tilde\ell}^2} \phi_{\tau\tau} + \phi_{xx} - \phi+3U^2 \phi = h,
\quad \text{for}\,(x, \tau)\in \mathfrak S,
\end{align}
\begin{align}
 \phi(x, 0) = \phi(x, 2\pi), \quad \phi_\tau(x, 0) = \phi_\tau(x, 2\pi), \quad \text{for}~x\in {\mathbb R},
\label{linearphi2}
\end{align}
\begin{align} \label{linearphi3}
 \int_{\mathbb R}\phi(x, \tau)\,\mathcal Z(x) \,{\mathrm d}x = \int_{\mathbb R}\phi(x, \tau)\,U'(x)\,{\mathrm d}x \,=\, 0, \quad \text{for}~\tau\in [0, 2\pi),
\end{align}
and then give the following lemma.
\begin{lemma}\label{linear lemma6-1}
 Suppose $\phi$ solves \eqref{linearphi1}-\eqref{linearphi3}. Then there exists a constant $C>0$, independent of $\epsilon$, such that
\begin{align} \label{norm h2 mathfrak S estimate}
 \|\phi\|_{H^2_*(\mathfrak S)} \leq C \|h\|_{L^2(\mathfrak S)},
\end{align}
where the  norm  $\|\cdot\|_{_{H^2_*(\mathfrak S)}}$ is given by
\begin{align} \label{norm h2 mathfrak S}
\|\phi\|_{H^2_*(\mathfrak S)}
= \left({\displaystyle\int}_{\mathfrak S} \Bigg[\, |\phi|^2+ |\phi_{x}|^2+ |\phi_{xx}|^2 +\Big|\frac{\epsilon}{\tilde\ell}\phi_{x\tau}\Big|^2
+ \Big|\frac{\epsilon}{\tilde\ell}\phi_{\tau}\Big|^2 + \Big|\frac{\epsilon^2}{\tilde\ell^2}\phi_{\tau\tau}\Big|^2 \, \Bigg] \,{\mathrm d}x {\mathrm d}\tau \right)^{\frac12},
\end{align}
with ${\tilde\ell}$ given in \eqref{lambda*}.
Moreover,
we have
\begin{align} \label{betanorm h2 mathfrak S estimate}
 \|\beta\phi\|_{H^2_*(\mathfrak S)} \leq C \|\beta h\|_{L^2(\mathfrak S)}.
\end{align}

\end{lemma}

\begin{proof}
 We express $h$ and $\phi$ as
\begin{align} \label{fourier-h}
 h(x, \tau) = \sum_{k=0}^\infty \Big[h_{1k} (x)\cos k\tau + h_{2k} (x)\sin k\tau\Big],
 \\[2mm]
 \label{fourier-phi}\phi(x, \tau) = \sum_{k=0}^\infty \Big[\phi_{1k} (x)\cos k\tau + \phi_{2k} (x)\sin k\tau\Big].
\end{align}
 Then we can  get
\begin{align} \label{linear philk1}
 - \frac{\epsilon^2k^2 4\pi^2}{{\tilde\ell}^2}\phi_{jk}+ \phi_{jk, xx} - \phi_{jk}+3U^2\phi_{jk} = h_{jk}\quad x\in {\mathbb R},
\end{align}
and $\phi_{jk}$ satisfies the orthogonality conditions
\begin{align} \label{linear philk2}
 \int_{\mathbb R}\phi_{jk}{\mathcal Z} \,{\mathrm d}x= \int_{\mathbb R}\phi_{jk} U' \,{\mathrm d}x \,=\, 0,
\end{align}
where $j=1,2$ and $k=0,1,2,\cdots$.

We first derive the estimates for $\phi_{jk}$ with $j=1,2$ and $k=0,1,2,\cdots$.

\smallskip\noindent$\clubsuit$\
 Consider the bilinear form
\begin{align*}
 B(\psi, \psi) = \int_{\mathbb R}\Big[|\psi_x|^2+ (1-3U^2)|\psi|^2 \Big] \,{\mathrm d}x.
\end{align*}
Since \eqref{linear philk2}, we can get
\begin{align} \label{bilinear 1}
 \|\phi_{jk}\|^2_{L^2({\mathbb R})} + \|\phi_{jk, x}\|^2_{L^2({\mathbb R})} \leq C B(\phi_{jk}, \phi_{jk}).
\end{align}

\smallskip\noindent$\clubsuit$\
Using \eqref{linear philk1} and \eqref{bilinear 1}, we can  claim that
\begin{align} \label{lin7-12}
\Bigg[1+\frac{\epsilon^4k^4}{{\tilde\ell}^4}\Bigg] \|\phi_{jk}\|^2_{L^2({\mathbb R})}
+ \Bigg[1+\frac{\epsilon^2k^2}{{\tilde\ell}^2}\Bigg] \|\phi_{jk, x}\|^2_{L^2({\mathbb R})} \leq C\|h_{jk}\|^2_{L^2({\mathbb R})}.
\end{align}
In fact, from \eqref{linear philk1}, we have
\begin{align*}
 \Bigg[1+\frac{\epsilon^2k^2}{{\tilde\ell}^2}\Bigg] \|\phi_{jk}\|^2_{L^2({\mathbb R})} + \|\phi_{jk, x}\|^2_{L^2({\mathbb R})} \leq C\|h_{jk}\|^2_{L^2({\mathbb R})}.
\end{align*}
 If $ \frac{\epsilon^2k^2}{{\tilde\ell}^2} \leq 1$, then we have $ \frac{\epsilon^4k^4}{{\tilde\ell}^4} \leq \frac{\epsilon^2k^2}{{\tilde\ell}^2}$. We can obtain \eqref{lin7-12}  directly .

\noindent If $ \frac{\epsilon^2k^2}{{\tilde\ell}^2} \geq 1$, then from \eqref{linear philk1}, we have
\begin{align*}
 \frac{\epsilon^2k^2}{{\tilde\ell}^2} B(\phi_{jk}, \phi_{jk}) + \frac{\epsilon^4k^4}{{\tilde\ell}^4} \|\phi_{jk}\|^2_{L^2({\mathbb R})} \leq C\|h_{jk}\|^2_{L^2({\mathbb R})} + (1-c)\frac{\epsilon^4k^4}{{\tilde\ell}^4} \|\phi_{jk}\|^2_{L^2({\mathbb R})},
\end{align*}
 i.e.,
\begin{align*}
 \frac{\epsilon^2k^2}{{\tilde\ell}^2} B(\phi_{jk}, \phi_{jk}) + c\frac{\epsilon^4k^4}{{\tilde\ell}^4} \|\phi_{jk}\|^2_{L^2({\mathbb R})} \leq C\|h_{jk}\|^2_{L^2({\mathbb R})},
\end{align*}
with $c\in (0, 1)$.
 Then by \eqref{bilinear 1}, we get \eqref{lin7-12}.

\smallskip\noindent$\clubsuit$\
By \eqref{lin7-12} and \eqref{linear philk1}, we can directly get
\begin{align}\label{second order estiamte}
\|\phi_{jk, xx}\|^2_{L^2({\mathbb R})} \leq C\|h_{jk}\|^2_{L^2({\mathbb R})}.
\end{align}

 From the expression of $\phi$ as in \eqref{fourier-phi}, we have
\begin{align*}
 \phi_x(x, \tau) &= \sum_{k=0}^\infty \Big[\phi_{1k, x} (x)\cos k\tau + \phi_{2k, x} (x)\sin k\tau\Big],
 \\[2mm]
 \phi_{xx}(x, \tau) &= \sum_{k=0}^\infty \Big[\phi_{1k, xx} (x)\cos k\tau + \phi_{2k, xx} (x)\sin k\tau\Big],
 \\[2mm]
 \phi_\tau(x, \tau) & = k\sum_{k=0}^\infty \Big[-\phi_{1k} (x)\sin k\tau + \phi_{2k} (x)\cos k\tau\Big],
 \\[2mm]
 \phi_{\tau\tau}(x, \tau) & = -k^2\sum_{k=0}^\infty \Big[\phi_{1k} (x)\cos k\tau + \phi_{2k} (x)\sin k\tau\Big],
 \\[2mm]
 \phi_{x\tau}(x, \tau) & = k\sum_{k=0}^\infty \Big[-\phi_{1k, x} (x)\sin k\tau + \phi_{2k,x} (x)\cos k\tau\Big].
\end{align*}
Therefore, we can derive that
\begin{align*}
 \|\phi\|_{L^2(\mathfrak S)}^2 & =\pi \sum_{k=0}^\infty \left[\|\phi_{1k}\|^2_{L^2({\mathbb R})}+ \|\phi_{2k}\|^2_{L^2({\mathbb R})}\right],
 \\[2mm]
 \|\phi_x\|_{L^2(\mathfrak S)}^2 & = \pi\sum_{k=0}^\infty \left[\|\phi_{1k, x}\|^2_{L^2({\mathbb R})}+ \|\phi_{2k, x}\|^2_{L^2({\mathbb R})}\right],
 \\[2mm]
 \|\phi_{xx}\|_{L^2(\mathfrak S)}^2 &= \pi \sum_{k=0}^\infty \left[\|\phi_{1k, xx}\|^2_{L^2({\mathbb R})}+ \|\phi_{2k, xx}\|^2_{L^2({\mathbb R})}\right],
 \\[2mm]
 \|\phi_\tau\|_{L^2(\mathfrak S)}^2 &= \pi \sum_{k=0}^\infty \left[k^2\|\phi_{1k}\|^2_{L^2({\mathbb R})}+ k^2\|\phi_{2k}\|^2_{L^2({\mathbb R})}\right],
 \\[2mm]
 \|\phi_{\tau\tau}\|_{L^2(\mathfrak S)}^2 &=\pi \sum_{k=0}^\infty \left[k^4\|\phi_{1k}\|^2_{L^2({\mathbb R})}+ k^4\|\phi_{2k}\|^2_{L^2({\mathbb R})}\right],
 \\[2mm]
 \|\phi_{x\tau}\|_{L^2(\mathfrak S)}^2 &= \pi \sum_{k=0}^\infty \left[k^2\|\phi_{1k, x}\|^2_{L^2({\mathbb R})}+ k^2\|\phi_{2k,x}\|^2_{L^2({\mathbb R})}\right].
\end{align*}
 Then from \eqref{lin7-12}, \eqref{second order estiamte}, we have
\begin{align*}
 &\|\phi\|_{L^2(\mathfrak S)}^2 + \|\phi_x\|_{L^2(\mathfrak S)}^2 + \|\phi_{xx}\|_{L^2(\mathfrak S)}^2+ \frac{\epsilon^2}{\tilde\ell^2} \|\phi_\tau\|_{L^2(\mathfrak S)}^2 + \frac{\epsilon^4}{\tilde\ell^4} \|\phi_{\tau\tau}\|_{L^2(\mathfrak S)}^2 + \frac{\epsilon^2}{\tilde\ell^2} \|\phi_{x\tau}\|_{L^2(\mathfrak S)}^2
\nonumber\\[2mm]
 & \leq C\sum_{k=0}^\infty \|h_{jk}\|^2_{L^2({\mathbb R})} \leq  C\|h\|_{L^2(\mathfrak S)}^2.
\end{align*}
This finishes the proof of \eqref{norm h2 mathfrak S estimate}.

 Next, we consider the proof of \eqref{betanorm h2 mathfrak S estimate}.
By denoting \[ \hat\phi = \beta \phi, \] we have,
for $\,(x, \tau)\in \mathfrak S$,
\begin{align*}
 \frac{\epsilon^2 4\pi^2}{{\tilde\ell}^2} \Big(\frac{\hat\phi} {\beta} \Big)_{\tau\tau} + \frac1\beta\Big[\hat\phi_{xx} - \hat\phi+3U^2 \hat\phi \Big]= h,
\end{align*}
i.e.,
\begin{align*}
\frac1\beta\Big[ \frac{\epsilon^2 4\pi^2}{{\tilde\ell}^2} \hat\phi_{\tau\tau} + \hat\phi_{xx} - \hat\phi+3U^2 \hat\phi \Big] + \frac{\epsilon^2 4\pi^2}{{\tilde\ell}^2}\Big[2\hat\phi_\tau \Big(\frac1\beta \Big)_\tau + \hat\phi \Big(\frac1\beta \Big)_{\tau\tau} \Big] = h,
\end{align*}
i.e.,
\begin{align*}
\frac1\beta\Big[ \frac{\epsilon^2 4\pi^2}{{\tilde\ell}^2} \hat\phi_{\tau\tau} + \hat\phi_{xx} - \hat\phi+3U^2 \hat\phi \Big] + \frac{\epsilon^2 4\pi^2}{{\tilde\ell}^2}\Big[2\hat\phi_\tau \Big(-\frac{\beta'}{\beta^2} \frac{\tilde \ell} {2\pi \beta}\Big) + \hat\phi \Big(-\frac{\beta''} {\beta^3} +3\frac{|\beta'|^2} {\beta^4} \Big)\frac{\tilde \ell^2} {4\pi^2\beta} \Big] = h,
\end{align*}
with
\begin{align*}
\Big(\frac1\beta \Big)_\tau = -\frac{\beta'}{\beta^2} \frac{\tilde \ell} {2\pi \beta},
\qquad
\Big(\frac1\beta \Big)_{\tau\tau} = \Big(-\frac{\beta''} {\beta^3} +3\frac{|\beta'|^2} {\beta^4} \Big)\frac{\tilde \ell^2} {4\pi^2\beta}.
\end{align*}
Then, by setting
\begin{align*}
 \tilde h
& =\beta h - \epsilon^2\Big[2\hat\phi_\tau \Big(-\frac{\beta'}{\beta^2} \frac{2\pi}{\tilde \ell} \Big) + \hat\phi \Big(-\frac{\beta''} {\beta^3} +3\frac{|\beta'|^2} {\beta^4} \Big) \Big],
\end{align*}
 we get
\begin{align*}
\frac{\epsilon^2 4\pi^2}{{\tilde\ell}^2} \hat\phi_{\tau\tau} + \hat\phi_{xx} - \hat\phi+3U^2 \hat\phi = \tilde h,
\end{align*}
with the same conditions in \eqref{linearphi2}-\eqref{linearphi3} by replacing $\phi$ with $\hat\phi$.
By \eqref{beta-bound}, we get \eqref{betanorm h2 mathfrak S estimate} directly from \eqref{norm h2 mathfrak S estimate}.
We finish the proof.
\end{proof}

Next, we consider the following projected problem
\begin{align} \label{pro-linearphi1}
 \frac{\epsilon^2 4\pi^2}{{\tilde\ell}^2} \phi_{\tau\tau} + \phi_{xx} - \phi+3U^2 \phi = h+c(\tau) U'(x)+ d(\tau){\mathcal Z}(x),
 \quad \text{for}\,(x, \tau)\in \mathfrak S,
\end{align}
\begin{align}
 \phi(x, 0) = \phi(x, 2\pi), \quad \phi_\tau(x, 0) = \phi_\tau(x, 2\pi), \quad \text{for}~x\in {\mathbb R},
\end{align}
\begin{align} \label{pro-linearphi3}
 \int_{\mathbb R}\phi(x,\tau){\mathcal Z}(x) \,{\mathrm d}x = \int_{\mathbb R}\phi(x,\tau) U'(x)\,{\mathrm d}x \,=\, 0, \quad \text{for}~\tau\in [0, 2\pi).
\end{align}

\begin{lemma} \label{pro-linear lemma6-1}
Problem \eqref{pro-linearphi1}-\eqref{pro-linearphi3} has a unique solution with the estimates
\begin{align} \label{pro estimate H2-l2}
 \|\phi\|_{H^2_*(\mathfrak S)} \leq C \|h\|_{L^2(\mathfrak S)},
\end{align}
and
\begin{align} \label{pro estimate H2-l2-beta}
 \|\beta\phi\|_{H^2_*(\mathfrak S)} \leq C \|\beta h\|_{L^2(\mathfrak S)}.
\end{align}
\end{lemma}
\begin{proof}
Similar to the analysis in Lemma \ref{linear lemma6-1}, by the expansions \eqref{fourier-h}-\eqref{fourier-phi}, we transfer \eqref{pro-linearphi1}-\eqref{pro-linearphi3} to the following
\begin{align} \label{pro-linear philk1}
 - \frac{\epsilon^2k^2 4\pi^2}{{\tilde\ell}^2}\phi_{jk}+ \phi_{jk, xx} - \phi_{jk}+3U^2\phi_{jk} = h_{jk} +c_{jk}U'+ d_{jk}\mathcal{Z}
 \quad \mbox{in } {\mathbb R},
\end{align}
\begin{align} \label{pro-linear philk2}
 \int_{\mathbb R}\phi_{jk}{\mathcal Z} \,{\mathrm d}x= \int_{\mathbb R}\phi_{jk} U' \,{\mathrm d}x \,=\, 0,
\end{align}
where $j=1,2$ and $k=0,1,2,\cdots$.
By Fredholm's theory, we know that problem \eqref{pro-linear philk1}-\eqref{pro-linear philk2} is solvable if
\begin{align*}
c_{jk} = - \frac{\displaystyle\int_{\mathbb R} h_{jk} U' \,{\mathrm d}x} {\displaystyle\int_{\mathbb R} |U'|^2 \,{\mathrm d}x},
\qquad
d_{jk} = - \frac{\displaystyle\int_{\mathbb R} h_{jk} {\mathcal Z} \,{\mathrm d}x} {\displaystyle\int_{\mathbb R}{\mathcal Z}^2 \,{\mathrm d}x},
\end{align*}
where $j=1,2$ and $k=0,1,2,\cdots$.
Then define
\begin{align*}
\phi(x, \tau) = \sum_{k=0}^\infty \Big[\phi_{1k} (x)\cos k\tau + \phi_{2k} (x)\sin k\tau\Big],
\end{align*}
and correspondingly
\begin{align*}
c(\tau) & = \sum_{k=0}^\infty \Big[c_{1k}\cos k\tau + c_{2k}\sin k\tau\Big],
\\
d(\tau) & = \sum_{k=0}^\infty \Big[d_{1k}\cos k\tau + d_{2k}\sin k\tau\Big].
\end{align*}
 By Lemma \ref{linear lemma6-1}, we have
\begin{align} \label{es1}
 \|\phi\|_{H^2_*(\mathfrak S)} \leq C\ \Big( \|h\|_{L^2(\mathfrak S)} + \|c(\tau) U'\|_{L^2(\mathfrak S)} +\| d(\tau){\mathcal Z}\|_{L^2(\mathfrak S)} \Big).
\end{align}
There holds
\begin{align} \label{es2}
 \|c(\tau) U'(x)\|_{L^2(\mathfrak S)}^2
 & \leq C \sum_{k=0}^\infty \Big(c_{1k}^2+ c_{2k}^2\Big)
 \leq C \sum_{j=1}^2\sum_{k=0}^\infty \int_{\mathbb R}|h_{jk}|^2 \,{\mathrm d}x \leq C\|h\|^2_{L^2(\mathfrak S)}.
\end{align}
 Similarly, we have
\begin{align} \label{es3}
 \|d(\tau){\mathcal Z}(x)\|_{L^2(\mathfrak S)}^2 \leq C\|h\|^2_{L^2(\mathfrak S)}.
\end{align}
By \eqref{es1}, \eqref{es2}, \eqref{es3}, we get \eqref{pro estimate H2-l2}.

The proof of \eqref{pro estimate H2-l2-beta} is similar to \eqref{betanorm h2 mathfrak S estimate}, and we omit it for concise writing.
\end{proof}

\medskip
 Consider the projection problem
\begin{align} \label{L2-pro-linearphi1}
 \mathcal L (\phi^*) = h+c(\tau) U'(x)+ d(\tau){\mathcal Z}(x),
 \quad \text{for}\,(x, \tau)\in \mathfrak S,
\end{align}
\begin{align}
 \phi^*(x, 0) = \phi^*(x, 2\pi), \qquad \phi^*_\tau(x, 0) = \phi^*_\tau(x, 2\pi), \quad \text{for}~x\in {\mathbb R},
\end{align}
\begin{align} \label{L2-pro-linearphi3}
 \int_{\mathbb R}\phi^*(x,\tau){\mathcal Z}(x) \,{\mathrm d}x = \int_{\mathbb R}\phi^*(x,\tau) U'(x)\,{\mathrm d}x \,=\, 0, \quad \text{for}~\tau\in [0, 2\pi).
\end{align}
The operator $\mathcal L$ is given in \eqref{expression of mathcal L}.
The invertibility of $\mathcal L$ in the projection space can be easily obtained from Lemma \ref{pro-linear lemma6-1}
and the Contraction Mapping Theorem.

\begin{lemma} \label{pro-linear lemma6-1-mathcal L}
If assumptions in {\bf{(A1)}} and {\bf{(A2)}} hold, then problem \eqref{L2-pro-linearphi1}-\eqref{L2-pro-linearphi3} has a unique solution with the estimates
\begin{align} \label{pro estimate H2-l2-mathcal}
 \|\phi^*\|_{H^2_*(\mathfrak S)} \leq C \|h\|_{L^2(\mathfrak S)},
\end{align}
and
\begin{align} \label{pro estimate H2-l2-mathcal-beta}
 \|\beta\phi\|_{H^2_*(\mathfrak S)} \leq C \|\beta h\|_{L^2(\mathfrak S)}.
\end{align}
for a universal constant $C$.
\end{lemma}
\begin{proof}
By Lemma \ref{pro-linear lemma6-1}, combined with the Contraction Mapping Theorem, the proof can be completed.
\end{proof}

\subsection{ The $L^2$-estimates of the error terms}\label{Section6.3}\

We denote
\begin{align*}
\mathfrak S_\delta = \mathfrak S \cap \Big\{ |s| < \frac {10\delta}\epsilon \Big\}.
\end{align*}
Note that, if $|s|\leq {10\delta}/\epsilon$, there holds
\begin{align*}
 \frac{1}{\beta^3} {{\mathcal E}} \,=\, S(w_2).
\end{align*}
From the expression of $S(w_2)$ as in \eqref{sw2new}, we can decompose $S(w_2)$ as
\begin{align}\label{E11+E12}
S(w_2) = E_{11} + E_{12},
\end{align}
where
\begin{align} \label{E12}
 E_{12} = \epsilon^3\frac{4\pi^2}{{\tilde\ell}^2} e''\mathcal Z+ \epsilon\lambda_0 e{\mathcal Z},
\end{align}
and
\begin{align}\label{E11}
\beta E_{11} =&
\beta \Big\{\epsilon^2 S_{3} + \epsilon^2 S_{42} + \epsilon^3 S_{5} +\epsilon^3 S_{6}
 +\epsilon^3 S_{7} +\epsilon^3 S_{8}
 +\epsilon^2 {\bf N}_{02}
  +\epsilon^3 {\bf N}_{03}
 + \epsilon^2 {\bf B}_{42}\Big\}
\nonumber\\[2mm]
&
 +\epsilon^4 \frac{4\pi^2} {{\tilde\ell}^2} \frac{\kappa a_2}{\, \beta\, } e'' (x+f)\mathcal Z -\epsilon^4 \frac{ 4\pi^2} {{\tilde\ell}^2} \frac{\kappa a_2 }{\, \beta\, } (x+f) f'' e \mathcal Z' + \frac{\epsilon^4}{\beta} \hat{\mathcal F}_1\left[\tau, x, \beta f, \frac{\beta f''}{\tilde\ell^2} \right]
\nonumber\\[2mm]
&
 +\beta\Bigg\{ \epsilon^3 \frac{4\pi^2}{{\tilde\ell}^2} {\phi}_{1, \tau\tau}
+ \epsilon^4 \frac{4\pi^2}{{\tilde\ell}^2}\phi_{2, \tau\tau}
+B_{5}(\epsilon\phi_1)+B_{6}(\epsilon\phi_1)
\nonumber\\[2mm]
&
\qquad\quad
+ \frac{\epsilon^4}{\beta^2} \mathcal F_1\left[\tau, x, \beta f, \frac{\beta f'}{\tilde\ell}\right]
 +\frac{\epsilon^4} {\beta^2} \mathcal G_1\left[\tau, x, \beta f, \frac{\beta f'}{\tilde\ell}, e, \frac{e'}{\tilde\ell}\right]
\nonumber\\[2mm]
&\qquad\quad
+ \mathbb{N}_1 (\epsilon^2 \phi_2)
+ B_4(\epsilon^2\phi_2)+ B_5 (\epsilon^2{\phi_2})+ B_6 (\epsilon^2{\phi_2})\Bigg\}.
\end{align}

We give the following lemma.
\begin{lemma}
By the assumption $\bf{(B3)}$, the expression of $S(w_2)$ in \eqref{sw2new} and decomposition \eqref{E11+E12}, we easily get
\begin{align} \label{estimate of L2mathfrak S-E11}
\|\beta E_{11}\|_{L^2(\mathfrak S_\delta)} \leq C\left( \epsilon^3+ \epsilon^2\|f\|_* + \epsilon^2\|e\|_{**}\right) \le C\left( \epsilon^3+ \epsilon^{3-\alpha} + \epsilon^{\frac 32 \varrho + \frac 52- 2\alpha}\right) \le C \epsilon^{3-\alpha},
\end{align}
and
\begin{align} \label{Lipschitz of E11-E12 in mathfrak S}
\big\|\beta\left(E_{11} (f_1, e_1) - E_{11} (f_2, e_2)\right)\big\|_{L^2(\mathfrak S_\delta)} \leq C \epsilon^2 \Big[\|f_1- f_2\|_*+ \|e_1- e_2\|_{**}  \Big],
\end{align}
for any $(f_1, e_1)$ and $(f_2, e_2)$ satisfy \eqref{constraint-f}-\eqref{constraint-e}.
\end{lemma}

\begin{proof}
Here we only give the proof of \eqref{estimate of L2mathfrak S-E11}.
The proof of \eqref{Lipschitz of E11-E12 in mathfrak S} is similar to \eqref{estimate of L2mathfrak S-E11} and we omit it for concise.
We note that $(f, e)$ satisfies the constraints \eqref{constraint-f}-\eqref{constraint-e}. Moreover, recall \eqref{Linfty of f}, we have
\begin{align} \label{f infty and f' infty}
\|\beta f(\tau)\|_{L^\infty(0, 2\pi)} \le C\tilde\ell^{\frac12} \|f\|_* \le \epsilon^{\frac{1+\varrho-2\alpha} 2} \ll 1,
\qquad
\Big\|\frac{\beta f'(\tau)}{\tilde \ell}\Big\|_{L^\infty(0, 2\pi)} \le C\tilde\ell^{\frac12} \|f\|_*\le C\epsilon^{\frac{1+\varrho-2\alpha} 2} \ll 1,
\end{align}
 and
\begin{align} \label{e Linfty}
 \|e\|_{L^\infty(0, 2\pi) } \le \|e\|_{**} \le C\epsilon^{\frac 32 \varrho + \frac 12- 2\alpha}.
\end{align}
Next, we will give the estimates of $\|\beta E_{11}\|_{L^2(\mathfrak S_\delta)} $.

\smallskip\noindent$\clubsuit$
 Noting the expressions of $\epsilon^2 S_{3}, \epsilon^2 S_{42}, \epsilon^3 S_{5}, \epsilon^3 S_{6}, \epsilon^3 S_{7}, \epsilon^3 S_{8} $ are given in \eqref{expression of s3}, \eqref{barS42}, \eqref{S5}, \eqref{S6}, \eqref{S7}, \eqref{S8}, then we have
\begin{align} \label{estimate in mathcal S-L2-1}
 & \int_0^{2\pi} \int_{|s| \le \frac{10\delta} \epsilon} \Big|\beta \big(\epsilon^2 S_{3} + \epsilon^2 S_{42} + \epsilon^3 S_{5} +\epsilon^3 S_{6}
 +\epsilon^3 S_{7} +\epsilon^3 S_{8} \big)\Big|^2 \,{\mathrm d}x \,{\mathrm d}\tau
\nonumber\\[2mm]
& \le \int_0^{2\pi} \int_{{\mathbb R}} \Big[ \epsilon^4 |\beta\, S_{3}|^2 + \epsilon^4 |\beta\, S_{42}|^2 + \epsilon^6 |\beta\, S_{5}|^2 + \epsilon^6 |\beta\, S_{6}|^2 + \epsilon^6 |\beta\, S_{7}|^2 + \epsilon^6 |\beta\, S_{8}|^2 \Big] \,{\mathrm d}x \,{\mathrm d}\tau
\nonumber\\[2mm]
& \le C \big(\epsilon^6+ \epsilon^4 \|f\|_*^2+ \epsilon^4\|e\|_{**}^2 \big).
\end{align}

\smallskip\noindent$\clubsuit$
Next, by \eqref{N02}, we have
\begin{align} \label{estimate in mathcal S-L2-2}
 \int_0^{2\pi} \int_{|s| \le \frac{10\delta} \epsilon} \epsilon^4\beta^2| {\bf N}_{02}|^2 \,{\mathrm d}x \,{\mathrm d}\tau
& \le C\left( \epsilon^4 \|e\|^2_{L^2(0, 2\pi)} + \epsilon^4 \|\beta fe\|^2_{L^2(0, 2\pi)} + \epsilon^4 \|(\beta f)^2\|^2_{L^2(0, 2\pi)} + \epsilon^4 \|\beta f\|^2_{L^2(0, 2\pi)} \right)
\nonumber\\[2mm]
& \le C \big(\epsilon^4 \|f\|_*^2+ \epsilon^4\|e\|_{**}^2 \big).
\end{align}
The last inequality holds due to \eqref{f infty and f' infty}-\eqref{e Linfty}.

\smallskip\noindent$\clubsuit$
By \eqref{B42}, we have
\begin{align} \label{estimate in mathcal S-L2-3}
 \int_0^{2\pi} \int_{|s| \le \frac{10\delta} \epsilon} \epsilon^4\beta^2|{\bf B}_{42}|^2 \,{\mathrm d}x \,{\mathrm d}\tau
& \le C\left( \epsilon^4 \|(\beta f)^2\|^2_{L^2(0, 2\pi)} + \epsilon^4 \|\beta  f\|^2_{L^2(0, 2\pi)} \right)
 \le C \left( \epsilon^4 \|f\|_*^2+ \epsilon^4\|e\|_{**}^2 \right).
\end{align}

\smallskip\noindent$\clubsuit$
By the expression of $a_2(\epsilon s, \epsilon z)$ given in \eqref{a1a2a3}, we know that
\[|a_2(\epsilon s, \epsilon z)| \le 2, \quad \text{if}~|\epsilon s| \le 10\delta. \]
Then we have
\begin{align} \label{estimate in mathcal S-L2-3-1}
& \int_0^{2\pi} \int_{|s| \le \frac{10\delta} \epsilon}
\beta^2\left|\epsilon^4 \frac{4\pi^2} {{\tilde\ell}^2} a_2\frac{\kappa}{\, \beta\, } e'' (x+f)\mathcal Z \right|^2 \,{\mathrm d}x \,{\mathrm d}\tau
\nonumber\\[2mm]
&+ \int_0^{2\pi} \int_{|s| \le \frac{10\delta} \epsilon} \beta^2\left|\epsilon^4 \kappa a_2 \frac{ 4\pi^2} {{\tilde\ell}^2} \frac{1}{\, \beta\, } (x+f) f'' e \mathcal Z'\right|^2 \,{\mathrm d}x \,{\mathrm d}\tau
\nonumber\\[2mm]
& \le C \left(\epsilon^4 \|f\|_*^2+ \epsilon^4\|e\|_{**}^2 \right).
\end{align}

\smallskip\noindent$\clubsuit$
By the expression of $\hat{\mathcal F}_1\left[\tau, x, \beta f, \frac{\beta f''}{\tilde\ell^2} \right]$ given in \eqref{hat mathcal F1}, we directly get
\begin{align} \label{estimate in mathcal S-L2-4}
 \int_0^{2\pi} \int_{|s| \le \frac{10\delta} \epsilon} \beta^2\left| \frac{\epsilon^4}{\beta} \hat{\mathcal F}_1\left[\tau, x, \beta f, \frac{\beta f''}{\tilde\ell^2} \right] \right|^2\,{\mathrm d}x \,{\mathrm d}\tau \le
 C \left(\epsilon^6+ \epsilon^4 \|f\|_*^2\right).
\end{align}

\smallskip\noindent$\clubsuit$
By \eqref{expression of phi1tautau}, we have
\begin{align} \label{estimate in mathcal S-L2-5}
 & \epsilon^6\int_0^{2\pi} \int_{|s| \le \frac{10\delta} \epsilon}
 \beta^2 \Big|\frac{4\pi^2}{{\tilde\ell}^2} {\phi}_{1, \tau\tau} \Big|^2 \,{\mathrm d}x \,{\mathrm d}\tau
\nonumber\\[2mm]
& \le C \epsilon^6\int_0^{2\pi} \int_{|s| \le \frac{10\delta} \epsilon} \left[ \beta^2 |{\mathbf b}_3|^2\left( \varpi_{1}^2 + f^2\varpi_2^2\right)
+ \beta^2 |{\mathbf b}_2|^2\Big|\frac{f'} {\tilde\ell}\Big|^2\varpi_2^2
+ \beta^2 |{\mathbf b}_1|^2\Big|\frac{f''} {\tilde\ell^2}\Big|^2\varpi_2^2 \right]\,{\mathrm d}x \,{\mathrm d}\tau
\nonumber\\[2mm]
& \le C \big(\epsilon^6+ \epsilon^4 \|f\|_*^2\big).
\end{align}
 Here we note that we used the estimates
$
 |{\mathbf b}_\ell| \le C \frac1{\beta}, \quad \text{for}~\ell= 1, 2, 3.
$

\smallskip\noindent$\clubsuit$
We recall the condition \eqref{W_theta-control} and the definition of $\epsilon^2 \phi_2$ given in \eqref{definition of e2phi2}. By some tedious but straight calculations, we know that
\begin{align}
\Big|\frac{4\pi^2}{{\tilde\ell}^2} \phi_{2, \tau\tau}\Big| \le C \frac{1} {\beta^2} e^{-|x|}.
\end{align}
Then
\begin{align} \label{estimate in mathcal S-L2-6}
 & \epsilon^8\int_0^{2\pi} \int_{|s| \le \frac{10\delta} \epsilon}
 \beta^2 \Big|\frac{4\pi^2}{{\tilde\ell}^2} {\phi}_{2, \tau\tau} \Big|^2 \,{\mathrm d}x \,{\mathrm d}\tau \le C\epsilon^8.
\end{align}

\smallskip\noindent$\clubsuit$
By \eqref{expression of B5phi1}, we have
\begin{align} \label{estimate in mathcal S-L2-7}
\int_0^{2\pi} \int_{|s| \le \frac{10\delta} \epsilon}
\beta^2\left| B_{5}(\epsilon\phi_1) \right|^2 \,{\mathrm d}x \,{\mathrm d}\tau \le C \left(\epsilon^6+ \epsilon^4 \|f\|_*^2+ \epsilon^4\|e\|_{**}^2 \right).
\end{align}
 Moreover, we recall the decomposition
\begin{align*}
 B_6(\epsilon\phi_1) = \frac{\epsilon^4}{\beta^2} \mathcal F_2\left[\theta, x, f, \frac{\beta f'}{\tilde\ell}\right]+ \frac{\epsilon^4}{\beta^2} \hat{\mathcal F}_2\left[\tau, x, \beta f, \frac{\beta f''}{\tilde\ell^2} \right],
\end{align*}
where $\frac{\epsilon^4}{\beta^2} \mathcal F_2\left[\theta, x, f, \frac{\beta f'}{\tilde\ell}\right], \frac{\epsilon^4}{\beta^2} \hat{\mathcal F}_2\left[\tau, x, \beta f, \frac{\beta f''}{\tilde\ell^2} \right]$ are defined in \eqref{mathcal F2}, \eqref{hat mathcal F2}
and satisfies
\begin{align*}
& \left|\frac{\epsilon^4}{\beta^2} \mathcal F_2\left[\theta, x, f, \frac{\beta f'}{\tilde\ell}\right]\right| \le \frac{\epsilon^4}{\beta^2} e^{-c|x|} \left( 1+ \beta |f| +\Big|\frac{\beta f'}{\tilde\ell}\Big| \right),
\\[2mm]
 &\left| \frac{\epsilon^4}{\beta^2} \hat{\mathcal F}_2\left[\tau, x, \beta f, \frac{\beta f''}{\tilde\ell^2} \right]\right| \le \frac{\epsilon^4}{\beta^2} e^{-c|x|} \left( 1+ \beta |f| + \Big|\frac{\beta f''}{\tilde\ell^2}\Big| \right),
\end{align*}
if $|\epsilon s| \le 10\delta$.
Then we can directly get
\begin{align} \label{estimate in mathcal S-L2-8}
\int_0^{2\pi} \int_{|s| \le \frac{10\delta} \epsilon} \beta^2\left| B_{6}(\epsilon\phi_1) \right|^2 \,{\mathrm d}x \,{\mathrm d}\tau \le C \left(\epsilon^6+ \epsilon^4 \|f\|_*^2 \right).
\end{align}

 \smallskip\noindent$\clubsuit$
 By \eqref{mathcal F1}, \eqref{hat mathcal F1},  if $ |\epsilon s| \le 10\delta$, we have
\begin{align*}
 &\left|\frac{\epsilon^4}{\beta^2} \mathcal F_1\left[\tau, x, \beta f, \frac{\beta f'}{\tilde\ell}\right] \right| \le \frac{\epsilon^4}{\beta^2} e^{-c|x|} \left( 1+ \beta |f| +\Big|\frac{\beta f'}{\tilde\ell}\Big| \right),
 \\[2mm]
 &\left| \frac{\epsilon^4} {\beta^2} \mathcal G_1\left[\tau, x, \beta f, \frac{\beta f'}{\tilde\ell}, e, \frac{e'}{\tilde\ell}\right]\right| \le \frac{\epsilon^4}{\beta^2} e^{-c|x|} \left( 1+ \beta |f| +\Big|\frac{\beta f'}{\tilde\ell}\Big| + |e| + \Big|\frac{e'}{\tilde\ell}\Big| \right).
\end{align*}
Then
\begin{align} \label{estimate in mathcal S-L2-9}
& \int_0^{2\pi} \int_{|s| \le \frac{10\delta} \epsilon}\Bigg\{ \beta^2\left|\frac{\epsilon^4}{\beta^2} \mathcal F_1\left[\tau, x, \beta f, \frac{\beta f'}{\tilde\ell}\right] \right|^2 + \beta^2\left| \frac{\epsilon^4} {\beta^2} \mathcal G_1\left[\tau, x, \beta f, \frac{\beta f'}{\tilde\ell}, e, \frac{e'}{\tilde\ell}\right]\right|^2\Bigg\} \,{\mathrm d}x \,{\mathrm d}\tau
\nonumber\\[2mm]
&
\le C \left(\epsilon^6+ \epsilon^4 \|f\|_*^2+ \epsilon^4\|e\|_{**}^2 \right).
\end{align}

\smallskip\noindent$\clubsuit$
At last, by \eqref{definition of e2phi2} and recalling the definition of $B_4, B_5, B_6$ given in \eqref{new opreator B4}, \eqref{new opreator B5}, \eqref{B2w-2222new}, we have
\begin{align} \label{estimate in mathcal S-L2-10}
 & \int_0^{2\pi} \int_{|s| \le \frac{10\delta} \epsilon} \beta^2
 \left| \mathbb{N}_1 (\epsilon^2 \phi_2)
+ B_4(\epsilon^2\phi_2)+ B_5 (\epsilon^2{\phi_2})+ B_6 (\epsilon^2{\phi_2}) \right|^2 \,{\mathrm d}x \,{\mathrm d}\tau
\nonumber\\[2mm]
&
\le C \left(\epsilon^6+ \epsilon^4 \|f\|_*^2+ \epsilon^4\|e\|_{**}^2 \right).
\end{align}

 Combining all the estimates in \eqref{estimate in mathcal S-L2-1}-\eqref{estimate in mathcal S-L2-10},
we can get \eqref{estimate of L2mathfrak S-E11}.
\end{proof}

\subsection{ Solving the nonlinear intermediate problem}\label{Section6.4}\

In this part, by using the facts given in the above, we will solve problem \eqref{project 6.13}-\eqref{projectedproblem3}.
The term
\begin{equation*}
 E_{12} = \epsilon^3\frac{4\pi^2}{{\tilde\ell}^2} e''\mathcal Z+ \epsilon\lambda_0 e{\mathcal Z}
\end{equation*}
has precisely the form $d \chi{\mathcal Z}$ and can be absorbed in the term.
Therefore, problem \eqref{project 6.13}-\eqref{projectedproblem3} can be replaced by the following
\begin{align} \label{projectedproblem1-refined-e11}
\mathcal L (\phi^*) \,=\, \eta_\delta^\epsilon(s) \Big[- E_{11}+ {\mathcal N}(\phi^*)\Big]
+ c(\tau) \chi(\epsilon|x|) U'(x)
+ d(\tau) \chi(\epsilon|x|){\mathcal Z}(x), \quad \text{for}\,(x, \tau)\in \mathfrak S,
\end{align}

\begin{equation}
{\phi^* }(x, 0) \,=\, {\phi^* }(x, 2\pi), \quad \phi^*_\tau(x, 0) \,=\, \phi^* _\tau(x, 2\pi),
\quad
-\infty<x<+\infty,
\end{equation}

\begin{equation}
\label{projectedproblem3-refined-aa}
 \int_{-\infty}^{+\infty} \phi^*(x, \tau){\mathcal Z}(x) \,{\mathrm d}x \,=\, \int_{-\infty}^{+\infty} \phi^*(x, \tau) U'(x)\,{\mathrm d}x \,=\, 0, \quad 0<\tau<2\pi.
\end{equation}

\medskip
Let ${\mathcal T}$ be the operator defined by Lemma \ref{pro-linear lemma6-1-mathcal L},
then problem \eqref{projectedproblem1-refined-e11}-\eqref{projectedproblem3-refined-aa} is equivalent to the fixed-point problem
\begin{align} \label{mathbb-A}
 \phi^* ={\mathcal T}\big(-E_{11} +{\mathcal N}(\phi^*)\big) : =\mathbb A (\phi^*).
\end{align}
 By the effect of cut-off function $\chi(\epsilon|x|)$, using the maximum principle and Sobolev embedding, we can get
\begin{align*}
 |\left(\beta\phi^*\right)(x, \tau)| \,+\, |\nabla\left(\beta\phi^*\right)(x, \tau)| \,\leq\, \|\beta\phi^*\|_{H^2_*(\mathfrak S)} e^{-\frac{2\delta} \epsilon}, \quad \text{for}~|s|\geq \frac{40\delta}{\varepsilon},
\end{align*}
 where the norm $\|\cdot\|_{H^2_*(\mathfrak S)} $ is defined in \eqref{norm h2 mathfrak S}.
 In fact, from the expression of $\mathcal L (\phi^*)$ given in \eqref{expression of mathcal L}, we know that, if $|\epsilon s|>20 \delta$, $\phi^*$ satisfies an equation of the form
\begin{align} \label{asymptotic of phi equation}
\frac{\epsilon^2 4\pi^2}{{\tilde\ell}^2} \phi^*_{\tau\tau} + \phi^*_{xx} - \big(1+ o(1)\big)\phi^* = 0.
\end{align}
Then
we get
\begin{align*}
\frac1\beta\left[\frac{\epsilon^2 4\pi^2}{{\tilde\ell}^2} \left(\beta \phi^*\right)_{\tau\tau} +2\epsilon^2\Big(-\frac{\beta'}{\beta^2} \frac{2\pi}{\tilde \ell} \Big)\left(\beta \phi^*\right)_\tau+ \left(\beta \phi^*\right)_{xx} - \big(1+ o(1)\big) \left(\beta \phi^*\right) \right] = 0.
\end{align*}
For $|s|\geq \frac{20 \delta} {\epsilon}$, we use a barrier function of the form
\begin{align*}
\varphi(x, \tau) = C \|\beta \phi^*\|_{\infty} e^{-\frac12\big(x- \frac{20\beta\delta}{\epsilon}\big)}.
\end{align*}
It's easy to verify that $\varphi(x, \tau) $ is a sup-solution of \eqref{asymptotic of phi equation}
and satisfies
\begin{align*}
\varphi(x, \tau) > \left|\left(\beta \phi^*\right)(x, \tau)\right|,
\quad
\mbox{at}\ s= \frac{20\delta}{\epsilon},
\end{align*}
where the relation between $s$ and $x$ is given by \eqref{coordinate-tildes} and \eqref{definition of x 2.14}.
Then we have
\begin{align*}
 |\left(\beta\phi^*\right)(x, \tau)| \leq C \|\beta \phi^*\|_\infty e^{-\frac12\big(x- \frac{20\beta\delta}{\epsilon}\big)}
 \leq C \|\beta \phi^*\|_{H^2_*(\mathfrak S)} e^{-\frac{5 \beta \delta} \epsilon},
 \quad \text{for}~|s|\geq \frac{40\delta}{\varepsilon}.
\end{align*}
The estimate of $\nabla \phi^*$ follows by local elliptic estimate and we omit the proof for brevity.

\medskip
By defining the following closed and bounded set
\begin{align*}
 \mathfrak B = &\Bigg\{\phi(x, \tau) \, :\, \phi~\text{satisfies}~\eqref{periodic condition},
\qquad
\|\beta\phi\|_{H^2_*(\mathfrak S)}\leq D \left( \epsilon^3+ \epsilon^2\|f\|_* + \epsilon^2\|e\|_{**}\right),
\nonumber \\[1mm]
 & \qquad \qquad \qquad
 \Big\|\, |\left(\beta \phi\right)(x, \tau) |+ | \nabla \left(\beta \phi\right)(x, \tau)|\, \Big\|_{L^\infty(|s|\geq \frac{40\delta}{\varepsilon})}
\leq
\|\beta \phi\|_{H^2_*(\mathfrak S)} e^{-\frac{2\delta} \epsilon}\Bigg\},
\end{align*}
we claim that:
{\em the map $\mathbb A $ defined in \eqref{mathbb-A} is a contraction map from $ \mathfrak B$ to itself.
}

In fact, recall
\begin{align*}
{\mathcal N}(\phi^*) & \,=\,
\frac{1}{\beta^3}
\left[\mathbf{N}( \eta_{3\delta}^{\epsilon}{\tilde\phi} +{\tilde\psi}) -3\, \mathbf{U}^2 \, {\tilde\psi}(\tilde \phi) \right]
\nonumber
 \\[2mm]
 &= \frac{1}{\beta^3}
\left[
-3\mathbf{U} \left(\eta_{3\delta}^{\epsilon}{\tilde\phi} +{\tilde\psi} \right)^2
- \left(\eta_{3\delta}^{\epsilon}{\tilde\phi} +{\tilde\psi} \right)^3
-3\, \mathbf{U}^2 \, {\tilde\psi}(\tilde \phi)
\right],
\end{align*}
with $\mathbf{U}(y)$ and $\tilde\phi(y)$ as in \eqref{defofbfU} and \eqref{relation betwen tildephi and phi}.

 For any $\phi_1^*, \phi_2^*\in \mathfrak B$, we have
\begin{align}\label{absolute vale N1-N2}
 &\Big|{\mathcal N}(\phi_1^*) - {\mathcal N}(\phi_2^*) \Big|
\nonumber\\[2mm]
 &= \frac{C}{\beta^3}
\Bigg|
 -3 \mathbf{U} \left(\eta_{3\delta}^{\epsilon}{\tilde\phi}_1 +{\tilde\psi}_1 \right)^2
 +3\mathbf{U} \left(\eta_{3\delta}^{\epsilon}{\tilde\phi}_2 +{\tilde\psi}_2 \right)^2
 -\left(\eta_{3\delta}^{\epsilon}{\tilde\phi}_1 +{\tilde\psi}_1 \right)^3
 +\left(\eta_{3\delta}^{\epsilon}{\tilde\phi}_2 +{\tilde\psi}_2 \right)^3
 -3\, \mathbf{U}^2 (\tilde\psi_1-\tilde\psi_2)
\Bigg|
\nonumber\\[2mm]
 & \leq \frac{C}{\beta^3}
\Bigg[\,
 \eta_{3\delta}^{\epsilon} \mathbf{U}\big|{\tilde\phi}_1^2- {\tilde\phi}_2^2\big|
 +\mathbf{U}\big|{\tilde\psi}_1^2- {\tilde\psi}_2^2\big|
 +\mathbf{U} \big(|\tilde{\psi}_1|+ |\tilde{\psi}_2|\big) \big|{\tilde\phi}_1- {\tilde\phi}_2\big|
 +\mathbf{U} \big(|\tilde{\phi}_1|+ |\tilde{\phi}_2|\big) \big|{\tilde\psi}_1- {\tilde\psi}_2\big|
\Bigg]
\nonumber\\[2mm]
 & \quad
 \ +\ \frac{C}{\beta^3} \big(|\tilde{\phi}_1|^2+ |\tilde{\phi}_2|^2 + |\tilde{\psi}_1|^2+ |\tilde{\psi}_2|^2 \big)
 \big( |\tilde{\phi}_1-\tilde{\phi}_2| +|\tilde{\psi}_1-\tilde{\psi}_2| \big) \ +\ \frac{C}{\beta^3} \mathbf{U}^2 |\tilde\psi_1-\tilde\psi_2|
\nonumber\\[2mm]
& \leq
 C |w_2|\big(|{\phi}_1^*|+| {\phi}_2^*|\big) \, |\phi^*_1-\phi^*_2|
\ +\
\frac{C}{\beta^2} |w_2| \big(|{\tilde\psi}_1|+| {\tilde\psi}_2|\big) \, |\tilde\psi_1-\tilde\psi_2|
\nonumber\\[2mm]
&\quad
 \ +\ \frac{C} {\beta}
\Big[\, |w_2|\, \big(|{\phi}_1^*|+| {\phi}_2^*|\big) + |{\phi}_1^*|^2+| {\phi}_2^*|^2\,\Big] \,|\tilde\psi_1-\tilde\psi_2|
\ +\ \frac{C} {\beta} |w_2|\,\big(|\tilde\psi_1|+\tilde\psi_2| \big) |{\phi}_1^*- {\phi}_2^*|
\nonumber\\[2mm]
 &\quad
 \ +\ C \big(|{\phi}_1^*|^2+| {\phi}_2^*|^2\big) \, |\phi^*_1-\phi^*_2|
 \ +\ \frac{C}{\beta^2} \big(|{\tilde\psi}_1|^2+| {\tilde\psi}_2|^2\big) \, |\phi^*_1-\phi^*_2|
\nonumber\\[2mm]
 &\quad
 \ +\ \frac{1}{\beta^3} \big(|{\tilde\psi}_1|^2+| {\tilde\psi}_2|^2\big)\,|\tilde\psi_1-\tilde\psi_2| \ +\ \frac{C}{\beta}| w_2|^2 |\tilde\psi_1-\tilde\psi_2|.
\end{align}
 Here we have used the notation
\[ \tilde{\psi}_j = {\tilde\psi}({ \tilde \phi}_j),\quad j =1,2.\]
 By Lemma \ref{pro-linear lemma6-1-mathcal L}, we can get
\begin{align} \label{mathbb A1-mathbb A2}
\|\mathbb A (\phi^*_1) - \mathbb A (\phi^*_2)\|_{H^2_*(\mathfrak S)}&= \|{\mathcal T}\left( {\mathcal N}(\phi^*_1)- {\mathcal N}(\phi^*_2) \right) \|_{H^2_*(\mathfrak S)}
\leq
\|{\mathcal T}\|\, \|{\mathcal N}(\phi_1^*) - {\mathcal N}(\phi_2^*)\|_{L^2(\mathfrak S)},
\end{align} where $ {\mathcal T}$ is the linear operator given in Lemma \ref{pro-linear lemma6-1-mathcal L}.
Then from \eqref{absolute vale N1-N2}, we can get
\begin{align} \label{mathcal N1-mathcal N2}
& \|{\mathcal N}(\phi_1^*) - {\mathcal N}(\phi_2^*)\|_{L^2(\mathfrak S)}
\nonumber\\[2mm]
 & \leq \Big(\|\phi_1^*\|_{L^4(\mathfrak S)}+ \|\phi_2^*\|_{L^4(\mathfrak S)}
 + \Big\|\frac{{\tilde\psi}_1}{\beta} \Big\|_{L^4(\mathfrak S_\delta)}
 + \Big\|\frac{{\tilde\psi}_2}{\beta} \Big\|_{L^4(\mathfrak S_\delta)}\Big) \|\phi_1^*- \phi_2^*\|_{L^4(\mathfrak S)}
\nonumber
 \\[2mm]
 & \quad + \Big(\|\phi_1^*\|_{L^4(\mathfrak S)}+ \|\phi_2^*\|_{L^4(\mathfrak S)}
 + \Big\|\frac{{\tilde\psi}_1}{\beta} \Big\|_{L^4(\mathfrak S_\delta)}
 +  \Big\|\frac{{\tilde\psi}_2}{\beta} \Big\|_{L^4(\mathfrak S_\delta)} +1\Big)
  \Big\|\frac{{\tilde\psi}_1- {\tilde\psi}_2}{\beta} \Big\|_{L^4(\mathfrak S_\delta)}.
\end{align}
 We recall the estimate and Lipschitz property of $\tilde\psi$
\begin{equation} \label{estimate of psi-star}
\|{\tilde\psi}({\tilde\phi})\|_{L^{\infty}}
\,\leq \,
C\, \epsilon \lf[\|\tilde\phi\|_{L^{\infty}(|s|>\frac{20\delta}{\epsilon})}
\,+\, \|\nabla \tilde\phi\|_{L^{\infty}(|s|>\frac{20\delta}{\epsilon})}\ri]
\,+\,
e^{-\frac{ \upsilon_2 \delta}{\epsilon}},
\end{equation}
and
\begin{equation}\label{estimeta of psistar12}
\|{\tilde\psi}(\tilde\phi_1)-{\tilde\psi}(\tilde\phi_2)\|_{L^{\infty}}
\leq
C\epsilon \lf[ \|\tilde\phi_1-\tilde\phi_2\|_{L^{\infty}(|s|>{20\delta}/\epsilon)}+\|\nabla\tilde\phi_1-\nabla\tilde\phi_2\|_{L^{\infty}(|s|>
{20\delta}/\epsilon)}\ri].
\end{equation}
Then by \eqref{estimate of psi-star} and the relation \eqref{relation betwen tildephi and phi}}, we get
\begin{align}\label{estimate tilde psi1 tilde psi 2 L4}
\Big\|\frac{{\tilde\psi}_1}{\beta}\Big\|_{L^4(\mathfrak S_\delta)}
+\Big\|\frac{{\tilde\psi}_2}{\beta}\Big\|_{L^4(\mathfrak S_\delta)}
 & \leq
\left\{ \,\lf[ \sum_{j =1}^2\left(\| \tilde\phi_j\|_{L^{\infty}(|s|>\frac{20\delta}{\epsilon})}\,+\, \|\nabla \tilde\phi_j \|_{L^{\infty}(|s|>\frac{20\delta}{\epsilon})} \right)\ri]\,+\,
e^{-\frac{ \upsilon_2 \delta}{\epsilon}} \,\right\}
\left( \int_{\mathfrak S_\delta} \,{\mathrm d}x {\mathrm d}\tau\right)^{\frac14}
\nonumber
 \\[2mm]
 & \leq C\left\{ \,\lf[ \sum_{j =1}^2\left(\| \tilde\phi_j\|_{L^{\infty}(|s|>\frac{20\delta}{\epsilon})}\,+\, \|\nabla \tilde\phi_j \|_{L^{\infty}(|s|>\frac{20\delta}{\epsilon})} \right)\ri]\,+\,
e^{-\frac{ \upsilon_2 \delta}{\epsilon}} \,\right\}\left( \frac{\int_0^{2\pi} \beta {\mathrm d}\tau} {\epsilon} \right)^{\frac14}
\nonumber
 \\[2mm]
 & \leq C\left( \frac{\int_0^{2\pi} \beta {\mathrm d}\tau} {\epsilon} \right)^{\frac14}
 \,\sum_{j=1}^2\|\left(\beta\phi_j^*\right)\|_{H^2_*(\mathfrak S)} e^{-\frac{2\delta} \epsilon} + O(e^{-\frac{\delta}{\epsilon}})
 \leq C e^{-\frac{\delta}{\epsilon}}.
\end{align}
 In fact, by \eqref{length-constraint}, we have
\begin{align*}
 \int_0^{2\pi} \beta {\mathrm d}\tau = \frac{2\pi} {\tilde\ell}\int_0^{\ell} \beta^2 {\mathrm d}\theta \leq C \epsilon^{\varrho-1}.
\end{align*}
Therefore, we can obtain that
\begin{align*}
\left(  {\epsilon}^{-1}{\int_0^{2\pi} \beta {\mathrm d}\tau} \right)^{\frac14} e^{-\frac{2\delta} \epsilon} \leq e^{-\frac{\delta} \epsilon},
\end{align*}
 for $\epsilon$ small enough.

On the other hand, by \eqref{estimeta of psistar12}, we have
\begin{align}\label{tilde psi 1-tilde psi 2}
\|\tilde\psi_1- \tilde\psi_2\|_{L^4(\mathfrak S_\delta)}
 & \leq \|\tilde\psi_1- \tilde\psi_2\|_{L^{\infty} (\mathfrak S_\delta)} \left( \int_{\mathfrak S_\delta} \,{\mathrm d}x {\mathrm d}\tau\right)^{\frac14}
\nonumber
 \\[2mm]
 & \leq C
 \Big[\, \|\beta\left({\phi}^*_1-{ \phi}^*_2\right)\|_{L^{\infty}(|s|>{20\delta}/\epsilon)}
 \ +\
 \|\nabla \left(\beta{\phi}^*_1-\beta{\phi}^*_2 \right) \|_{L^{\infty}(|s|>{20\delta}/\epsilon)}\Big]  \left( \frac{\int_0^{2\pi} \beta {\mathrm d}\tau} {\epsilon} \right)^{\frac14}
\nonumber
 \\[2mm]
 & \leq o(1) \|\beta\left({\phi}^*_1-{ \phi}^*_2\right)\|_{H^2_*(\mathfrak S)}.
\end{align}
 The last inequality holds due to
\begin{align*}
 \big[\, \|\beta\left({\phi}^*_1-{ \phi}^*_2\right)\|_{L^{\infty}(|s|>{20\delta}/\epsilon)}+\|\nabla \left(\beta{\phi}^*_1-\beta{\phi}^*_2 \right) \|_{L^{\infty}(|s|>{20\delta}/\epsilon)}\big] \leq \|\beta\left({\phi}^*_1-{ \phi}^*_2\right)\|_{H^2_*(\mathfrak S)} e^{-\frac{2\delta} \epsilon},
\end{align*}
 for any ${\phi}^*_1, {\phi}^*_2 \in \mathfrak B.$
 By \eqref{mathbb A1-mathbb A2}, \eqref{mathcal N1-mathcal N2}, \eqref{estimate tilde psi1 tilde psi 2 L4}, \eqref{tilde psi 1-tilde psi 2} and Sobolev embedding, we prove that $\mathbb{A}$ is a contraction map.

\medskip
Next, we will show that $\mathbb{A}$ maps $\mathfrak B$ into itself.
For any $\phi^*\in \mathfrak B$, by Lemma \ref{pro-linear lemma6-1-mathcal L}, then
\begin{align} \label{estimate of mathbbAphi1}
\|\beta\mathbb A (\phi^*)\|_{H^2_*(\mathfrak S)}
 =\,& \|\beta{\mathcal T}\left( - E_{11}+ {\mathcal N}(\phi^*)\right) \|_{H^2_*(\mathfrak S)}
\nonumber
 \\[2mm]
\leq\, & \|{\mathcal T}\| \left\{\|\beta E_{11}\|_{L^2(\mathfrak S)}
\ +\
\Bigg\|\frac{\beta}{\beta^3} \left[-3\eta_{\delta}^{\epsilon}\mathbf{U} \left(\eta_{3\delta}^{\epsilon}
{\tilde\phi} +{\tilde\psi} \right)^2+ \left(\eta_{3\delta}^{\epsilon}
{\tilde\phi} +{\tilde\psi} \right)^3 -3\eta_{\delta}^{\epsilon}\, \mathbf{U}^2 \, {\tilde\psi}(\tilde \phi)\right] \Bigg\|_{L^2(\mathfrak S)} \right\}
\nonumber
 \\[2mm]
\leq\, & \|{\mathcal T}\|\Bigg\{\|\beta E_{11}\|_{L^2(\mathfrak S)}
 +C\|\beta\phi^*\|^2_{L^4(\mathfrak S)}
 +C \|\tilde\psi\|^2_{L^4(\mathfrak S_\delta)}
\nonumber
 \\[2mm] & \qquad \quad
 +C \|\beta\phi^*\|^3_{L^6(\mathfrak S)}
 + C \|\tilde\psi\|^3_{L^6(\mathfrak S_\delta)}
 +C \|\tilde\psi\|_{L^2(\mathfrak S_\delta)} \Bigg\}.
\end{align}
Using \eqref{estimate of psi-star} and similar to \eqref{estimate tilde psi1 tilde psi 2 L4}, we have
\begin{align} \label{estimate of tildepsi}
 \|\tilde\psi\|_{L^p(\mathfrak S_\delta)}
& \leq C\left( \frac{\int_0^{2\pi} \beta {\mathrm d}\tau} {\epsilon} \right)^{\frac1p}
\, \|\beta\phi^*\|_{H^2_*(\mathfrak S)} e^{-\frac{2\delta} \epsilon} + O(e^{-\frac{\delta}{\epsilon}})
\nonumber
 \\[2mm]
 &
 \leq Ce^{-\frac{\delta}{\epsilon}} \Big(\|\beta\phi^*\|_{H^2_*(\mathfrak S)}+1\Big),
\end{align}
 provided with
\begin{align*}
\left( \frac{\int_0^{2\pi} \beta {\mathrm d}\tau} {\epsilon} \right)^{\frac1p} e^{-\frac{2\delta} \epsilon} \leq e^{-\frac{\delta} \epsilon}.
\end{align*}
By \eqref{estimate of mathbbAphi1}, \eqref{estimate of tildepsi} and Sobolev embedding, we get
\begin{align}\label{estimate of mathbbAphi2}
\|\mathbb A (\phi^*)\|_{H^2_*(\mathfrak S)}
& \leq C\Big\{\|\beta E_{11}\|_{L^2(\mathfrak S)}
+ C\|\beta\phi^*\|^2_{H^2_*(\mathfrak S)}
+ C e^{-\frac{\delta}{\epsilon}} \Big(\|\beta\phi^*\|_{H^2_*(\mathfrak S)}+1\Big) \Big\}
\nonumber
 \\[2mm]
 & \leq C \left( \epsilon^3+ \epsilon^2\|f\|_* + \epsilon^2\|e\|_{**}\right).
\end{align}
Next, for any $\phi^*$, since $\varphi^* :=\mathbb A (\phi^*) $ satisfies the equation
\begin{align*}
\frac{\epsilon^2 4\pi^2}{{\tilde\ell}^2} \varphi^*_{\tau\tau} + \varphi^*_{xx} - \big(1+ o(1)\big)\varphi^* = h
\end{align*}
with $h$ compactly supported. Hence $ \varphi^* $ satisfies the estimates required for functions in the $\mathcal B$.
Then we know that $\mathbb{A}$ map $\mathfrak B$ to itself if $D$ large enough.

\medskip
By summarizing the above results, we present the following resolution theory.
\begin{proposition}\label{proposition6.5}
There is a number $D>0$ such that for all $\epsilon$ small enough and all parameters $(f, e)$ in $\hat{\mathcal F}$,
problem \eqref{project 6.13}-\eqref{projectedproblem3} has a unique solution $\phi^* \,=\, \phi^*(f, e)$, which satisfies the estimates
\begin{align*}
 \|\beta \phi^*\|_{H_*^2(\mathfrak S)} + \| \phi^*\|_{H_*^2(\mathfrak S)} & \,\leq\, D \left( \epsilon^3+ \epsilon^2\|f\|_* + \epsilon^2\|e\|_{**}\right)
\nonumber
 \\[2mm]
 &\le D \left( \epsilon^3+ \epsilon^{3-\alpha} + \epsilon^{\frac 32 \varrho + \frac 52- 2\alpha} \right) \le ( D+1) \epsilon^{3-\alpha},
\end{align*}
and
\begin{align*}
\big\| |\beta\phi^*|+|\nabla \left(\beta \phi^*\right)| \big\|_{L^{\infty}(|x|>40\delta/\epsilon)}\,\leq\, D \|\beta\phi^*\|_{H_*^2(\mathfrak S)}e^{-\delta/\epsilon}.
\end{align*}
 Moreover, $\phi^*$ depends Lipschitz-continuously on the parameters $f$ and $e$ in the sense of the estimate
\begin{equation}
\label{characteriztion}
\|\phi^*(f_1, e_1)\,-\, \phi^*(f_2, e_2)\|_{H_*^2(\mathfrak S)}
\,\leq\, D\epsilon^2\big [\, \|f_1-f_2\|_{*}\,+\, \|e_1-e_2\|_{**}\, \big].
\end{equation}
\qed
\end{proposition}


\medskip
\section{Deriving the reduced system of differential equations}\label{Section7}

After solving problem \eqref{project 6.13}-\eqref{projectedproblem3} in Proposition \ref{proposition6.5},
we shall choose the unknown parameters $f$ and $e$ to make the Lagrange multipliers $c$ and $d$ vanish.
In fact, $c\equiv 0$ and $d\equiv 0$ are equivalent to
\begin{align} \label{project 6.13-1}
\int_{\mathbb R}\Big\{\,\mathcal L (\phi^*) +\eta_\delta^\epsilon(s)\frac{1}{\beta^3}{\mathcal E} - \eta_\delta^\epsilon(s){\mathcal N}(\phi^*)\Big\}
U'(x) \,{\mathrm d}x\,=\, 0,
\end{align}
and
\begin{align} \label{project 6.13-2}
\int_{\mathbb R}\Big\{\,\mathcal L (\phi^*) +\eta_\delta^\epsilon(s)\frac{1}{\beta^3}{\mathcal E} - \eta_\delta^\epsilon(s){\mathcal N}(\phi^*)\Big\}
{\mathcal Z}(x) \,{\mathrm d}x\,=\, 0.
\end{align}
The calculations of the above terms will be given in Sections \ref{Section7.1}-\ref{Section7.2}.
These will consequently derive the reduced system with differential equations involving the
parameters $f$ and $e$, see Section \ref{Section7.3}.

\subsection{ Projections of the error terms}\label{Section7.1}\

We will denote by $\tilde{\mathcal D}_j(\theta)$ with $ j\in \mathbb{N}^+, $ the smooth functions that are uniformly bounded and independent of $f, f', f'', e, e', e''$.

\subsubsection{Computation for $\beta^{-3}\int_{{\mathbb R}}\, \eta_\delta^\epsilon(s)\, {\mathcal E}\, U' \,{\mathrm d}x$}\

We first concern
\begin{align*}
& \frac{1}{\beta^3}\int_{\mathbb R}\eta_\delta^\epsilon(s){\mathcal E}U'(x) \,{\mathrm d}x
\nonumber\\[2mm]
=&\frac{1}{\beta^3}\int_{\mathbb R} {\mathcal E}U'(x) \,{\mathrm d}x
\ +\
\frac{1}{\beta^3}\int_{\mathbb R} \big[\eta_\delta^\epsilon(s)-1\big]{\mathcal E}U'(x) \,{\mathrm d}x
\nonumber\\[2mm]
=& \int_{{\mathbb R}}\, S\big(w_2\big)\, U'(x) \,{\mathrm d}x
\ +\
\int_{\mathbb R} \big[\eta_\delta^\epsilon(s)-1\big]S\big(w_2\big)U'(x) \,{\mathrm d}x
\nonumber\\[2mm]
=&
\int_{\mathbb R} \Bigg[\epsilon^2 S_{3} + \epsilon^2 S_{42} + \epsilon^3 S_{5} +\epsilon^3 S_{6}
 +\epsilon^3 S_{7} +\epsilon^3 S_{8}
 +\frac 1\beta\epsilon^3\frac{4\pi^2}{{\tilde\ell}^2} e''\mathcal Z
 +\frac 1\beta \epsilon\lambda_0 e{\mathcal Z}
 +\epsilon^2 {\bf N}_{02}
  +\epsilon^2 {\bf N}_{03}
 + \epsilon^2 {\bf B}_{42}
\nonumber\\[2mm]
&\qquad
 + \epsilon^3 \frac{4\pi^2}{{\tilde\ell}^2} {\phi}_{1, \tau\tau}
+ \epsilon^4 \frac{4\pi^2}{{\tilde\ell}^2}\phi_{2, \tau\tau}
+B_{5}(\epsilon\phi_1) + \mathbb{N}_1 (\epsilon^2 \phi_2)
+ B_4(\epsilon^2\phi_2)+ B_5 (\epsilon^2{\phi_2}) \Bigg]U'(x) \,{\mathrm d}x + \mathfrak G(\theta)
\nonumber\\[2mm]
&
 +\int_{\mathbb R} \eta_\delta^\epsilon(s) \Bigg\{\epsilon^4 \frac{4\pi^2} {{\tilde\ell}^2} a_2\frac{\kappa}{\, \beta^2\, } e'' (x+f)\mathcal Z -\epsilon^4 \kappa a_2 \frac{ 4\pi^2} {{\tilde\ell}^2} \frac{1}{\, \beta^2\, } (x+f) f'' e \mathcal Z' + \frac{\epsilon^4}{\beta^2} \hat{\mathcal F}_1\left[\tau, x, \beta f, \frac{\beta f''}{\tilde\ell^2} \right]
\nonumber\\[2mm]
&\qquad\qquad\quad
+B_{6}(\epsilon\phi_1)
 + \frac{\epsilon^4}{\beta^2} \mathcal F_1\left[\tau, x, \beta f, \frac{\beta f'}{\tilde\ell}\right]
 +\frac{\epsilon^4} {\beta^2} \mathcal G_1\left[\tau, x, \beta f, \frac{\beta f'}{\tilde\ell}, e, \frac{e'}{\tilde\ell}\right]
 + B_6 (\epsilon^2{\phi_2}) \Bigg\}  U'(x) \,{\mathrm d}x,
\end{align*}
due to the expression of $\mathcal{E}$ given in \eqref{errorrelationinterior},
where
\begin{align*}
 \mathfrak G(\tau) & = \int_{\mathbb R}\big[\eta_\delta^\epsilon(s)-1\big] \Bigg[\epsilon^2 S_{3} + \epsilon^2 S_{42} + \epsilon^3 S_{5} +\epsilon^3 S_{6}
 +\epsilon^3 S_{7} +\epsilon^3 S_{8}
 +\frac 1\beta\epsilon^3\frac{4\pi^2}{{\tilde\ell}^2} e''\mathcal Z
 \nonumber\\[2mm]
&\qquad \qquad \qquad \qquad
 +\frac 1\beta \epsilon\lambda_0 e{\mathcal Z}
 +\epsilon^2 {\bf N}_{02}
 +\epsilon^3 {\bf N}_{03}
 + \epsilon^2 {\bf B}_{42}
 + \epsilon^3 \frac{4\pi^2}{{\tilde\ell}^2} {\phi}_{1, \tau\tau}
 + \epsilon^4 \frac{4\pi^2}{{\tilde\ell}^2}\phi_{2, \tau\tau}
\nonumber\\[2mm]
&\qquad \qquad \qquad \qquad
+B_{5}(\epsilon\phi_1) + \mathbb{N}_1 (\epsilon^2 \phi_2)
+ B_4(\epsilon^2\phi_2)+ B_5 (\epsilon^2{\phi_2}) \Bigg] U'(x) \,{\mathrm d}x.
\end{align*}
Note that the terms from
\begin{align*}
 & \epsilon^4 \frac{4\pi^2} {{\tilde\ell}^2} a_2\frac{\kappa}{\, \beta^2\, } e'' (x+f)\mathcal Z -\epsilon^4 \kappa a_2 \frac{ 4\pi^2} {{\tilde\ell}^2} \frac{1}{\, \beta^2\, } (x+f) f'' e \mathcal Z' + \frac{\epsilon^4}{\beta^2} \hat{\mathcal F}_1\left[\tau, x, \beta f, \frac{\beta f''}{\tilde\ell^2} \right]
\nonumber\\[2mm]
&
+B_{6}(\epsilon\phi_1)
 + \frac{\epsilon^4}{\beta^2} \mathcal F_1\left[\tau, x, \beta f, \frac{\beta f'}{\tilde\ell}\right]
 +\frac{\epsilon^4} {\beta^2} \mathcal G_1\left[\tau, x, \beta f, \frac{\beta f'}{\tilde\ell}, e, \frac{e'}{\tilde\ell}\right]
 + B_6 (\epsilon^2{\phi_2})
\end{align*}
 are involving $a_j(t, \theta), j =1,2,3$ given in \eqref{a1a2a3-1}.
By the exponential decay of the term in $S\big(w_2\big)$ with respect to $x$,
 we neglect the calculations of the terms in $ \mathfrak G(\theta)$.

Since the following terms in $S\big(w_2\big)$
$$
\epsilon^2 S_{42},\quad \epsilon^3 S_6,\quad \epsilon^3 S_8,\quad
\frac 1\beta\epsilon^3\frac{4\pi^2}{{\tilde\ell}^2} e''\mathcal Z +\frac 1\beta \epsilon\lambda_0 e{\mathcal Z},
\quad \epsilon^4 \frac{4\pi^2}{{\tilde\ell}^2}\phi_{2, \tau\tau}
$$
are even functions of $x$, then integration against $U'(x)$ just vanishes.
 And there holds
\begin{align*}
 \epsilon^4 \frac{4\pi^2}{{\tilde\ell}^2}\int_{\mathbb R}\phi_{1, \tau\tau} U'(x) \,{\mathrm d}x=0,
\end{align*}
 due to the fact
$$
\int_{\mathbb R}\phi_{1}(x,\tau) U'(x)\,{\mathrm d}x=0.
$$

Therefore, we can obtain that
\begin{align}
& \frac{1}{\beta^3}\int_{\mathbb R}\eta_\delta^\epsilon(s){\mathcal E}U'(x) \,{\mathrm d}x
\nonumber\\[2mm]
& = \int_{{\mathbb R}}\epsilon^2 S_{3} U'(x)\,{\mathrm d}x
+\int_{{\mathbb R}}\epsilon^3 S_{5}U'(x)\,{\mathrm d}x
+\int_{{\mathbb R}}\epsilon^3 S_{7}U'(x)\,{\mathrm d}x
+\int_{{\mathbb R}}\epsilon^2 {\bf N}_{02}U'(x)\,{\mathrm d}x
+\int_{{\mathbb R}}\epsilon^3 {\bf N}_{03}U'(x)\,{\mathrm d}x
\nonumber\\[2mm]
& \quad
+\int_{{\mathbb R}}\epsilon^2 {\bf B}_{42} U'(x)\,{\mathrm d}x
+\int_{{\mathbb R}}B_{5}(\epsilon\phi_1)U'(x)\,{\mathrm d}x
+\int_{{\mathbb R}}\mathbb{N}_1 (\epsilon^2 \phi_2)U'(x)\,{\mathrm d}x
+\int_{{\mathbb R}}B_4(\epsilon^2\phi_2)U'(x)\,{\mathrm d}x
\nonumber\\[2mm]
& \quad
+\int_{{\mathbb R}}B_5 (\epsilon^2{\phi_2}) U'(x)\,{\mathrm d}x
 + \int_{\mathbb R} \eta_\delta^\epsilon(s) \Bigg[ \epsilon^4 \frac{4\pi^2} {{\tilde\ell}^2} a_2\frac{\kappa}{\, \beta^2\, } e'' (x+f)\mathcal Z
 -\epsilon^4 \kappa a_2 \frac{ 4\pi^2} {{\tilde\ell}^2} \frac{1}{\, \beta^2\, } (x+f) f'' e \mathcal Z' \Bigg]U'(x)\,{\mathrm d}x
\nonumber\\[2mm]
& \quad
  + \int_{\mathbb R} \eta_\delta^\epsilon(s)\Bigg\{\frac{\epsilon^4}{\beta^2} \hat{\mathcal F}_1\left[\tau, x, \beta f, \frac{\beta f''}{\tilde\ell^2} \right]
  + \frac{\epsilon^4}{\beta^2} \mathcal F_1\left[\tau, x, \beta f, \frac{\beta f'}{\tilde\ell}\right]
 +\frac{\epsilon^4} {\beta^2} \mathcal G_1\left[\tau, x, \beta f, \frac{\beta f'}{\tilde\ell}, e, \frac{e'}{\tilde\ell}\right]
\Bigg\} U'(x)\,{\mathrm d}x
 \nonumber\\[2mm]
& \quad
 + \int_{\mathbb R} \eta_\delta^\epsilon(s)\, \Big[B_{6}(\epsilon\phi_1)+B_6 (\epsilon^2{\phi_2}) \Big]\,U'(x)\,{\mathrm d}x
\nonumber\\[2mm]
& = \sum_{j=0}^{13}\,I_j .
\end{align}
We calculate these terms as follows.

\smallskip\noindent$\clubsuit$ By the expression of $\epsilon^2 S_{3}$ as in \eqref{expression of s3}, we have
\begin{align} \label{bar S3intergralUx}
I_1
& = -\, \epsilon^2 \frac{\epsilon^2W_{tt}(0, \theta)}{\beta^4}f \int_{{\mathbb R}} xU U'(x)\,{\mathrm d}x
+\epsilon^2\Big[\frac{|\beta'|^2}{\beta^4} 2f
-\frac{4\pi}{ {\tilde\ell}}\frac{\beta'}{\beta^2}f'
\Big]\int_{{\mathbb R}} xU'' U'(x)\,{\mathrm d}x
\nonumber\\[2mm]
&\quad
+\epsilon^2 \Big[
-\frac{\pi^2}{{\tilde\ell}^2}f''
- \frac{6\pi}{\tilde\ell} \frac{\beta'}{\beta^2} f'
+\frac{2|\beta'|^2}{\beta^4}f
+\frac{\beta''} {\beta^3}f
\,-\,\frac{\kappa^2}{\beta^2}f
\Big]\int_{{\mathbb R}} |U'|^2 \,{\mathrm d}x
\nonumber\\[2mm]
& \quad +\epsilon^2\frac{\kappa}{\, \beta^2\, }\int_{{\mathbb R}}\left[ e{\mathcal Z}'+\frac 23 e x\mathcal Z \right]U'(x) \,{\mathrm d}x
\nonumber\\[2mm]
& = \epsilon^2 \varrho_1\Bigg\{ - \frac{4\pi^2}{{\tilde\ell}^2}f'' - \frac{4\pi}{\tilde\ell} \frac{\beta'}{\beta^2} f'
+ \Big[\frac{|\beta'|^2}{\beta^4}
+\frac{\beta''} {\beta^3}
\,-\,\frac{\kappa^2}{\beta^2} \Big] f + \, \frac 32\frac{\epsilon^2 W_{tt}(0, \theta)}{\beta^4} f\Bigg\}
+ \frac {\epsilon^2 {\mathbf b}_1d_1e}\beta,
\end{align}
where
\[\varrho_1= \int_{\mathbb R} |U'|^2{\mathrm d}x, \quad d_1 = \int_{\mathbb R}\left[ {\mathcal Z}'+\frac 23 x\mathcal Z \right]U'(x) \,{\mathrm d}x. \]

On the other hand, from the expressions of $\epsilon^2 {\bf B}_{42}$, $\epsilon^2{\bf N}_{02}$ as in \eqref{B42} and \eqref{N02}, we can get
\begin{align*}
I_6\,=\, \epsilon^2 {\mathbf b}_1^2\int_{\mathbb R} \Big[
\frac{2}{3}f\varpi_{2,x}
+ \frac{4}{9} f x\varpi_2
+\frac{2}{3}f\varpi_1
\Big]U'(x) \,{\mathrm d}x,
\end{align*}
and
\begin{align*}
I_4
& = \epsilon^2\int_{\mathbb R}\left[\frac{6 {\mathbf b}_1e}\beta U{\mathcal Z}\varpi_1
+\frac{4 {\mathbf b}_1fe}\beta U{\mathcal Z} \varpi_2
+\frac{4}{3}{\mathbf b}_1^2|f|^2 U\varpi_2^2
+ 4{\mathbf b}_1^2f U\varpi_1\varpi_2 \right]U'(x) \,{\mathrm d}x
\nonumber\\[2mm]
& = \epsilon^2\int_{\mathbb R}\left[\frac{6 {\mathbf b}_1e}\beta U{\mathcal Z}\varpi_1
+ 4{\mathbf b}_1^2f U\varpi_1\varpi_2 \right]U'(x) \,{\mathrm d}x
\nonumber\\[2mm]
& = \frac{\epsilon^2 6 {\mathbf b}_1e}\beta\int_{\mathbb R} U{\mathcal Z}\varpi_1 U'(x) \,{\mathrm d}x + 4\epsilon^2{\mathbf b}_1^2f \int_{\mathbb R} U\varpi_1\varpi_2 U'(x) \,{\mathrm d}x.
\end{align*}

By the identities
\begin{align*}
 \int_{{\mathbb R}} \varpi_1 \varpi_2 U U'(x)\,{\mathrm d}x =\frac1 6 \int_{\mathbb R}\left[U'(x) + \frac 23 x U \right] \varpi_{2, x} \,{\mathrm d}x
 - \frac 16\int_{\mathbb R} \varpi_1 U'(x)\,{\mathrm d}x,
\end{align*}
\begin{align*}
 & 2\int_{{\mathbb R}} \varpi_{2, x} U'(x) \,{\mathrm d}x = -\frac 32 \int_{\mathbb R} |U'|^2 \,{\mathrm d}x,
\qquad
\int_{{\mathbb R}} \varpi_{2} U \,{\mathrm d}x = - \frac 34\int_{\mathbb R} |U'|^2 \,{\mathrm d}x,
 \\[2mm]
&
 \int_{{\mathbb R}} x \varpi_{2} U'(x) \,{\mathrm d}x = - \int_{{\mathbb R}} \varpi_{2} U \,{\mathrm d}x - \int_{{\mathbb R}} x \varpi_{2, x} U \,{\mathrm d}x,
\end{align*}
we can get
\begin{align}
I_4+I_6
& =
\frac{\epsilon^2 6 {\mathbf b}_1d_2e}\beta
+\epsilon^2{\mathbf b}_1^2f 4 \int_{\mathbb R} U\varpi_1\varpi_2 U'(x) \,{\mathrm d}x\
+\epsilon^2{\mathbf b}_1^2f\int_{\mathbb R} \Big[\frac{2}{3}\varpi_{2,x} + \frac{4}{9} x\varpi_2 +\frac{2}{3}\varpi_1 \Big]U'(x) \,{\mathrm d}x\nonumber\\[2mm]
& = \frac{\epsilon^2 {\mathbf b}_1d_2e}\beta -\frac23\varrho_1\epsilon^2{\mathbf b}_1^2f,
\end{align}
where
\[d_2 =6 \int_{\mathbb R} U{\mathcal Z}\varpi_1 U'(x) \,{\mathrm d}x. \]
Therefore, we have
\begin{align} \label{S32+N03+B43intUx}
I_1+I_4+I_6
& = \epsilon^2 \varrho_1\Bigg\{ - \frac{4\pi^2}{{\tilde\ell}^2}f'' - \frac{4\pi}{\tilde\ell} \frac{\beta'}{\beta^2} f'
+ \Big[\frac{|\beta'|^2}{\beta^4}
+\frac{\beta''} {\beta^3}
\,-\,\frac53\frac{\kappa^2}{\beta^2} \Big] f + \, \frac 32\frac{\epsilon^2 W_{tt}(0, \theta)}{\beta^4} f\Bigg\}
\nonumber\\[2mm]
 & \quad+\frac {\epsilon^2 {\mathbf b}_1(d_1+d_2)e}\beta.
\end{align}

\smallskip\noindent$\clubsuit$
 On the other hand, recall the expression of $\epsilon^3 S_{5}$ given in \eqref{S5}, we can easily get
\begin{align} \label{bar S31+bar S51intergralUx}
I_2\,=\,
\epsilon^3\Bigg\{ \frac1{\beta}\tilde{\mathcal D}_1(\theta)
+\frac1{\beta}\tilde{\mathcal F}_1\left[\tau, \beta f, \frac{\beta {f}'}{\tilde\ell}\right]
- \frac{8\pi^2}{{\tilde\ell}^2}\frac{\kappa}{\beta} f''f \varrho_1 \Bigg\},
\end{align}
 where
\begin{align*}
\frac1{\beta}\tilde{\mathcal D}_1(\theta)
& = - \frac{1}{6} \frac{ \epsilon^2W_{ttt}(0, \theta)}{\beta^5}d_3 + \Big[ -2 \frac{\beta''\kappa}{\beta^4} -4\frac{|\beta'|^2\kappa}{\beta^5} +\frac{\beta' \kappa'} {\beta^4} \Big](-\frac 32)\varrho_1
\nonumber\\[2mm]
 & \quad
 + \Big[ \frac{\kappa^3}{\beta^3}-4\frac{|\beta'|^2\kappa}{\beta^5}-2\frac{\beta''\kappa}{\beta^4} + \frac{\beta' \kappa'} {\beta^4} \Big]d_4 -2 \frac{|\beta'|^2\kappa}{\beta^5}d_5,
\end{align*}
with
\[ d_3 = \int_{\mathbb R} x^3 U U'(x) \,{\mathrm d}x,\qquad d_4 = \int_{\mathbb R} x^2|U'|^2 \,{\mathrm d}x, \qquad d_5= \int_{\mathbb R} x^3 U'' U'(x) \,{\mathrm d}x, \]
and
\begin{align*}
 \frac{ \epsilon^3}{\beta}\tilde{\mathcal F}_1\left[\tau, \beta f, \frac{\beta {f}'}{\tilde\ell}\right]
 & = - \frac{1}{2} \frac{\epsilon^2 W_{ttt}(0, \theta)}{\beta^5}f^2 \int_{\mathbb R} x UU' \,{\mathrm d}x
\nonumber\\[2mm]
 & \quad +
 \Big[
 \frac{\kappa^3}{\beta^3}f^2
 -\frac{4\kappa|\beta'|^2}{\beta^5} f^2
 - \frac{2\kappa\beta''}{\beta^4}f^2
 +\frac{12\pi }{ {\tilde\ell}}\frac{\kappa\beta'}{\beta^3} f' f
 + \frac{\kappa'\beta'} {\beta^4}f^2
 - \frac{2\pi}{{\tilde\ell}} \frac{\kappa'}{\beta^2} f'f\Big]\int_{\mathbb R}|U'|^2 {\mathrm d}x\nonumber\\[2mm]
 & \quad +
 \Big[ -\frac{6\kappa|\beta'|^2}{\beta^5}f^2
 + \frac{16\pi}{ {\tilde\ell}} \frac{\kappa\beta'}{\beta^3} f' f
-\frac{8\pi^2} {{\tilde\ell}^2}\frac{\kappa}{\beta} |f'|^2\Big]\int_{\mathbb R} x U'' U'{\mathrm d}x.
\end{align*}

\smallskip\noindent $\clubsuit$
 From the expression of $\epsilon^3 S_7$ as in \eqref{S7}, we have
\begin{align} \label{intB4-new}
I_3
 = & \frac{\epsilon^3}\beta\Bigg\{
\Big[
-\frac{4\pi^2}{{\tilde\ell}^2}f''
-\frac{6\pi}{\tilde\ell} \frac{\beta'}{\beta^2} f'
+\frac{2|\beta'|^2}{\beta^4}f
+\frac{\beta''} {\beta^3}f
\,-\,\frac{\kappa^2}{\beta^2}f
\Big]e\int_{{\mathbb R}}{\mathcal Z}' U'(x)\,{\mathrm d}x
\nonumber\\[2mm]
&\qquad\quad+ \Big[2\frac{|\beta'|^2}{\beta^4} f
-\frac{4\pi}{ {\tilde\ell}}\frac{\beta'}{\beta^2}f'
\Big]e\int_{{\mathbb R}} x{\mathcal Z}''U'(x)\,{\mathrm d}x
\nonumber\\[2mm]
&\qquad\quad -\, \frac{\epsilon^2W_{tt}(0, \theta)}{\beta^4}f e \int_{{\mathbb R}}x{\mathcal Z}U'(x)\,{\mathrm d}x
+
\Big[ -\frac{8\pi^2}{{\tilde\ell}^2}f' +\frac{4\pi}{ {\tilde\ell}}\frac{\beta'}{\beta^2} f\Big]e' \int_{{\mathbb R}}{\mathcal Z}' U'(x)\,{\mathrm d}x
\Bigg\}
\nonumber\\[2mm]
 = &\frac{\epsilon^3}{\beta^3}\tilde{\mathcal F}_2\left[\tau, \beta f, \frac{\beta f'}{\tilde\ell} \right] e
 +\frac{\epsilon^3}{\beta^3}\Big[ -\frac{8\pi^2}{{\tilde\ell}^2}f' +\frac{4\pi}{ {\tilde\ell}}\frac{\beta'}{\beta^2} f\Big]e'd_8 - \frac{\epsilon^3}\beta\frac{4\pi^2}{{\tilde\ell}^2}f'' e d_8,
\end{align}
where
\begin{align*}
\frac{\epsilon^3}{\beta^3}\tilde{\mathcal F}_2\left[\tau, \beta f, \frac{\beta f'}{\tilde\ell} \right] e &=\frac{\epsilon^3}{\beta} \Bigg\{\Big[
-\frac{6\pi}{\tilde\ell} \frac{\beta'}{\beta^2} f'
+\frac{2|\beta'|^2}{\beta^4}f
+\frac{\beta''} {\beta^3}f
\,-\,\frac{\kappa^2}{\beta^2}f
\Big] \int_{{\mathbb R}}{\mathcal Z}' U'(x)\,{\mathrm d}x
\nonumber\\[2mm]
& \qquad \qquad
+\Big[2\frac{|\beta'|^2}{\beta^4} f
-\frac{4\pi}{ {\tilde\ell}}\frac{\beta'}{\beta^2}f'
\Big]\int_{{\mathbb R}} x{\mathcal Z}''U'(x)\,{\mathrm d}x
\nonumber\\[2mm]
& \qquad \qquad- \frac{\epsilon^2W_{tt}(0, \theta)}{\beta^4}f \int_{{\mathbb R}}x{\mathcal Z}U'(x)\,{\mathrm d}x \Bigg\}e,
\end{align*}
 with $$d_8 = \int_{{\mathbb R}}{\mathcal Z}' U'(x)\,{\mathrm d}x. $$

From the expression of ${\bf N}_{03}$ as in \eqref{N03}, we have
\begin{align}\label{I5}
I_5\,=\,&\int_{{\mathbb R}}\epsilon^3 {\bf N}_{03}U'(x)\,{\mathrm d}x
\nonumber\\[2mm]
 \, =\,&\epsilon^3
\int_{{\mathbb R}}\Bigg\{\frac{3 e^2\mathcal Z^2 {\mathbf b}_1}{\beta^2}\Big[\varpi_1 + \frac{2}{3}f \varpi_2\Big]
+ \frac{3 e{\mathcal Z} {\mathbf b}_1^2}{\beta}\Big[\varpi_1 + \frac{2}{3}f \varpi_2\Big]^2
+ {\mathbf b}_1^3\Big[\varpi_1 + \frac{2}{3}f \varpi_2\Big]^3
\Bigg\} \, U'(x) \,{\mathrm d}x
\nonumber\\[2mm]
 \, =\,&\epsilon^3
\int_{{\mathbb R}}\Bigg\{\frac{3 e^2 {\mathbf b}_1}{\beta^2}\mathcal Z^2\varpi_1
+ \frac{4 {\mathbf b}_1^2}{\beta}f e{\mathcal Z}\varpi_1  \varpi_2
+ {\mathbf b}_1^3\Big[\varpi_1^3+ 2f \varpi_1 \varpi_2\Big]
\Bigg\} \, U'(x) \,{\mathrm d}x
\nonumber\\[2mm]
 \, =\,&\epsilon^3
\Bigg\{\frac{3 e^2 {\mathbf b}_1}{\beta^2}\int_{{\mathbb R}}\mathcal Z^2\varpi_1 U' \,{\mathrm d}x
+\frac{4 {\mathbf b}_1^2}{\beta}f e\int_{{\mathbb R}} {\mathcal Z}\varpi_1  \varpi_2 U'\,{\mathrm d}x
+{\mathbf b}_1^3\int_{{\mathbb R}} \varpi_1^3U'{\mathrm d}x
+ 2{\mathbf b}_1^3f \int_{{\mathbb R}}\varpi_1 \varpi_2 U'{\mathrm d}x\,
\Bigg\}
\nonumber\\[2mm]
 \, =\,&\frac{\epsilon^3}{\beta}\tilde{\mathcal D}_2(\theta)
+\frac{\epsilon^3}{\beta}\tilde{\mathcal F}_{3}\left[\tau, \beta f, \frac{\beta {f}'}{\tilde\ell}, e\right],
\end{align}
where
\begin{equation*}
\tilde{\mathcal D}_2(\theta) \,=\, \frac{\kappa^3}{\beta^2}\,\int_{{\mathbb R}} \varpi_1^3U'{\mathrm d}x.
\end{equation*}

\smallskip\noindent$\clubsuit$
By the expression of $B_5(\epsilon \phi_1)$ given in \eqref{expression of B5phi1}, we have
\begin{align}\label{I7}
I_7
 \, =\,&\epsilon^3
\int_{{\mathbb R}}\Bigg\{\frac{2\pi}{{\tilde\ell}}\frac{\beta'} {\beta^2}\Big[- f' \phi_{1,x} +\phi_{1,\tau}\Big]
 +\frac{4\pi^2}{{\tilde\ell}^2}\Big[|f'|^2 \phi_{1,xx} -f''\phi_{1,x} -2f' \phi_{1,x\tau}\Big]
\nonumber\\[2mm]
&\qquad\qquad
+ \frac{4\pi}{{\tilde\ell}} \frac{\beta'}{\beta^2} (x+f)\big[- f' \phi_{1,xx} +\phi_{1,x\tau}\big]
+\Big[\frac{\beta''}{\beta^3} \,-\,\frac{\kappa^2}{\beta^2} \Big]
(x+f) \phi_{1,x}
\nonumber\\[2mm]
&\qquad\qquad
+ \frac{|\beta'|^2}{\beta^4}(x+f)^2 \phi_{1,xx}
-\frac{1}{2}\, \frac{\epsilon^2W_{tt}(0, \theta)}{\beta^4} (x+f)^2\phi_{1}
\Bigg\} \, U'(x) \,{\mathrm d}x
\nonumber\\[2mm]
=& - \epsilon^3 \frac{4\pi^2}{{\tilde\ell}^2} f'' \int_{{\mathbb R}} \phi_{1,x}\, U'(x) \,{\mathrm d}x
+\frac{\epsilon^3}{\beta}\tilde{\mathcal F}_4\left[\tau, \beta f, \frac{\beta {f}'}{\tilde\ell} \right]
\nonumber\\[2mm]
=& - \epsilon^3 \frac{4\pi^2}{{\tilde\ell}^2}  \frac{\kappa}{\beta}\,f''\,f\, d_9
+\frac{\epsilon^3}{\beta}\tilde{\mathcal F}_4\left[\tau, \beta f, \frac{\beta {f}'}{\tilde\ell} \right],
\end{align}
 where
\begin{equation*}
d_9\,=\,\frac{2}{3}\int_{{\mathbb R}}\varpi_2'(x)\,U'(x)\,{\mathrm d}x,
\end{equation*}
and
\begin{align}\label{B5phi1 intergral Ux}
 \frac{\epsilon^3}{\beta}\tilde{\mathcal F}_4\left[\tau, \beta f, \frac{\beta {f}'}{\tilde\ell} \right]
 &=
\epsilon^3
\int_{{\mathbb R}}\Bigg\{\frac{2\pi}{{\tilde\ell}}\frac{\beta'} {\beta^2}\Big[- f' \phi_{1,x} +\phi_{1,\tau}\Big]
 +\frac{4\pi^2}{{\tilde\ell}^2}\Big[|f'|^2 \phi_{1,xx} -2f' \phi_{1,x\tau}\Big]
\nonumber\\[2mm]
&\qquad\qquad\quad
+ \frac{4\pi}{{\tilde\ell}} \frac{\beta'}{\beta^2} (x+f)\big[- f' \phi_{1,xx} +\phi_{1,x\tau}\big]
+\Big[\frac{\beta''}{\beta^3} \,-\,\frac{\kappa^2}{\beta^2} \Big]
(x+f) \phi_{1,x}
\nonumber\\[2mm]
&\qquad\qquad\quad
+ \frac{|\beta'|^2}{\beta^4}(x+f)^2 \phi_{1,xx}
-\frac{1}{2}\, \frac{\epsilon^2W_{tt}(0, \theta)}{\beta^4} (x+f)^2\phi_{1}
\Bigg\} \, U'(x) \,{\mathrm d}x.
\end{align}

\smallskip\noindent$\clubsuit$
By the definition of $ \mathbb{N}_1(\epsilon^2 \phi_2)$ given in \eqref{definition of mathbbN1}, we have
\begin{align}\label{I8}
I_8=\int_{{\mathbb R}}\mathbb{N}_1 (\epsilon^2 \phi_2)U'(x)\,{\mathrm d}x= \frac{\epsilon^3}{\beta} \tilde{\mathcal D}_3 (\theta) e+\frac{\epsilon^4}\beta \tilde{\mathcal D}_4 (\theta).
\end{align}

\smallskip\noindent$\clubsuit$
Since $\phi_2$ is even of $x$, from the expression of $ B_4$ as in \eqref{new opreator B4}, then we have
\begin{align}\label{B4phi2 intergral Ux}
I_9
 \,=\,& \epsilon^3 \frac{\kappa}{\, \beta\, }\int_{{\mathbb R}} \Big[ \phi_{2,x} +\frac{2}{3} (x+f)\phi_2\Big] U'(x) \,{\mathrm d}x
\nonumber\\[2mm]
 \,=\,&\epsilon^3 \frac{\kappa}{\, \beta\, }\int_{{\mathbb R}} \Big[ \phi_{2,x} +\frac{2}{3} x\phi_2\Big] U'(x) \,{\mathrm d}x
 = \frac{\epsilon^3 }{\beta} \tilde{\mathcal D}_5(\theta).
\end{align}

\smallskip\noindent$\clubsuit$
By the expression of $ B_{5} (\epsilon^2 \phi_2)$ given in \eqref{decomposition of B5(phi2)}, we get
\begin{align}\label{B5phi2 intergral Ux}
I_{10}
 \,=\, & \frac{\epsilon^4}{\beta}\tilde{\mathcal F}_5\left[\tau, \beta f, \frac{\beta {f}'}{\tilde\ell} \right]-\epsilon^4\frac{4\pi^2}{{\tilde\ell}^2} f'' \int_{\mathbb R}\phi_{2, x}\, U'(x) \,{\mathrm d}x
\nonumber\\[2mm]
 \,=\,& \frac{\epsilon^4}{\beta}\tilde{\mathcal F}_5\left[\tau, \beta f, \frac{\beta {f}'}{\tilde\ell} \right]-\epsilon^4\frac{4\pi^2}{{\tilde\ell}^2} f'' \tilde{\mathcal D}_6(\theta),
\end{align}
where
\begin{align*}
\tilde{\mathcal D}_6 (\theta)= \int_{\mathbb R}\phi_{2, x}\, U'(x) \,{\mathrm d}x.
\end{align*}

\smallskip\noindent $\clubsuit$
By the expression of $a_2(\epsilon s, \theta) $ given in \eqref{a1a2a3}, then we have
\begin{align}\label{e-second order higher terms}
 I_{11}\,=\,& \int_{\mathbb R} \eta_\delta^\epsilon(s) \Bigg[ \epsilon^4 \frac{4\pi^2} {{\tilde\ell}^2} a_2\frac{\kappa}{\, \beta^2\, } e'' (x+f)\mathcal Z
 -\epsilon^4  \frac{ 4\pi^2} {{\tilde\ell}^2}  a_2\frac{\kappa}{\, \beta^2\, } (x+f) f'' e \mathcal Z' \Bigg]U'(x)\,{\mathrm d}x
\nonumber\\[2mm]
 \,=\,& -\epsilon^4 \frac{8\pi^2} {{\tilde\ell}^2} \frac{{\mathbf b}_1}{\, \beta\, } e''\int_{{\mathbb R}} x\mathcal Z U'(x) \,{\mathrm d}x
 \nonumber\\[2mm]
&+ \underbrace{ \epsilon^5 \int_{{\mathbb R}} \eta_\delta^\epsilon(s)\frac{4\pi^2} {{\tilde\ell}^2} \left[ \frac{3\kappa}{\beta}(x+f) + \epsilon {\tilde a}_2\right] \frac{\kappa}{\, \beta^2\, } (x+f)\mathcal Z U'(x) \,{\mathrm d}x\, e'' + O(e^{\frac{- \delta} \epsilon} ) \frac{\epsilon^2 e''}{\tilde\ell^2}}_{:= \frac{\epsilon^3}\beta \tilde{\mathcal F}_6\left[\tau, \beta f \right] \frac{\epsilon^2 e''}{\tilde\ell^2}}
\nonumber\\[2mm]
&-\epsilon^4 \frac{8\pi^2} {{\tilde\ell}^2} \frac{{\mathbf b}_1}{\, \beta\, } f f'' e\int_{{\mathbb R}} \mathcal Z' U'(x) \,{\mathrm d}x
 \nonumber\\[2mm]
&+ \underbrace{\epsilon^5 \int_{{\mathbb R}} \eta_\delta^\epsilon(s)\frac{4\pi^2} {{\tilde\ell}^2} \left[ \frac{3\kappa}{\beta}(x+f) + \epsilon {\tilde a}_2\right] \frac{\kappa}{\, \beta^2\, } (x+f) e\mathcal Z' U'(x) \,{\mathrm d}x\,  f''
 + O(e^{\frac{- \delta} \epsilon} ) \frac{f''}{\tilde\ell^2}}_{:= \frac{\epsilon^5}\beta \tilde{\mathcal F}_7\left[\theta, \beta f, e\right] \frac{\beta f''}{\tilde\ell^2}}
\nonumber\\[2mm]
 \,=\,&\epsilon^4 \frac{d_6} {\beta}{\mathbf b}_1 \frac{e''}{\tilde\ell^2} + \frac{\epsilon^3}\beta \tilde{\mathcal F}_6\left[\tau, \beta f \right] \frac{\epsilon^2 e''}{\tilde\ell^2}
 +\epsilon^4 \frac{d_7} {\beta}{\mathbf b}_1 f e\frac{f''}{\tilde\ell^2}+\frac{\epsilon^5}\beta \tilde{\mathcal F}_7\left[\tau, \beta f, e \right] \frac{\beta f''}{\tilde\ell^2},
\end{align}
where
\begin{align*}
d_6 = -8 \pi^2\int_{{\mathbb R}} x\mathcal Z U'(x) \,{\mathrm d}x,
\quad
d_7 = -8 \pi^2\int_{{\mathbb R}} \mathcal Z' U'(x) \,{\mathrm d}x.
\end{align*}

\smallskip
\noindent $\clubsuit$
We recall the expressions of
$$
\frac{\epsilon^4}{\beta^2} \mathcal F_1\left[\tau, x, \beta f, \frac{\beta f'}{\tilde\ell} \right],
\quad\,\frac{\epsilon^4}{\beta^2} \hat{\mathcal F}_1\left[\tau, x, \beta f, \frac{\beta f''}{\tilde\ell^2} \right],
\quad\,\frac{\epsilon^4} {\beta^2} \mathcal G_1\left[\tau, x, \beta f, \frac{\beta f'}{\tilde\ell}, e, \frac{e'}{\tilde\ell}\right]
$$
given in \eqref{mathcal F1}, \eqref{hat mathcal F1}, \eqref{mathcal G1}.
 Then we denote
\begin{align} \label{intergral Ux higher order}
I_{12}\,=\,
 \frac{\epsilon^3}{\beta^2}\tilde{\mathcal F}_{8}\left[\tau, \beta f, \frac{\beta {f}'}{\tilde\ell}, e, \frac{\epsilon e'}{\tilde\ell}\right]
 + \frac{\epsilon^4}{\beta}\tilde{\mathcal F}_{9}\left[\tau, \beta f \right] \frac{\beta f''} {\tilde\ell^2}.
\end{align}

 \smallskip\noindent $\clubsuit$
By recalling the decomposition of $B_6 (\epsilon \phi_1)$ given in \eqref{expression of B6phi1}, we have
\begin{align}\label{B6phi1 intergral Ux}
I_{13}
 &\,=\,
 \int_{{\mathbb R}} \eta_\delta^\epsilon(s) \frac{\epsilon^4}{\beta^2} \mathcal F_2\left[\tau, x, \beta f, \frac{\beta f'}{\tilde\ell}\right] \, U'(x) \,{\mathrm d}x
 + \int_{{\mathbb R}} \eta_\delta^\epsilon(s) \frac{\epsilon^4}{\beta^2} \hat{\mathcal F}_2\left[\tau, x, \beta f, \frac{\beta f''}{\tilde\ell^2} \right] \, U'(x) \,{\mathrm d}x
\nonumber\\[2mm]
 &=
 \frac{\epsilon^4}{\beta}\tilde{\mathcal F}_{10}\left[\tau, \beta f, \frac{\beta {f}'}{\tilde\ell}\right]
 + \frac{\epsilon^4}{\beta}\tilde{\mathcal F}_{11}\left[\tau, \beta f \right] \frac{\beta f''} {\tilde\ell^2}.
\end{align}
Here we note that the expressions of
$\frac{\epsilon^4}{\beta^2} \mathcal F_2\left[\tau, x, \beta f, \frac{\beta f'}{\tilde\ell}\right], \frac{\epsilon^4}{\beta^2} \hat{\mathcal F}_2\left[\tau, x, \beta f, \frac{\beta f''}{\tilde\ell^2} \right] $ are given in \eqref{mathcal F2}, \eqref{hat mathcal F2}.

\medskip
At last, combining the estimates in \eqref{S32+N03+B43intUx}-\eqref{B6phi1 intergral Ux},
we conclude that
\begin{align*}
\frac{1}{\beta^3}\int_{\mathbb R}\eta_\delta^\epsilon(s){\mathcal E}U'(x) \,{\mathrm d}x
= &
 \epsilon^2\varrho_1\Bigg[- \frac{4\pi^2}{{\tilde\ell}^2}{f}'' - \frac{4\pi}{\tilde\ell} \frac{\beta'}{\beta^2} f'
+ \Big(\frac{|\beta'|^2}{\beta^4}
+\frac{\beta''} {\beta^3}
\,-\,\frac{5\kappa^2}{3\beta^2} \Big) f
+ \, \frac 32\frac{\epsilon^2W_{tt}(0, \theta)}{\beta^4} f \Bigg]
\nonumber\\[2mm]
&\qquad + \frac {\epsilon^2 {\mathbf b}_1(d_1+d_2)e}\beta
+\frac{\epsilon^3}{\beta}{\mathcal D}_1(\theta)
+ \epsilon^4 \frac{d_6} {\beta}{\mathbf b}_1 \frac{e''}{\tilde\ell^2}
+ \frac{\epsilon^3}{\beta^3}\Big[ -\frac{8\pi^2}{{\tilde\ell}^2}f' +\frac{4\pi}{ {\tilde\ell}}\frac{\beta'}{\beta^2} f\Big]e'd_8
\nonumber\\[2mm]
&\qquad
+\frac{\epsilon^3}{\beta}\tilde{\mathcal F}^{*} \left[\tau, \beta f, \frac{\beta {f}'}{\tilde\ell}, e, \frac{\epsilon e'}{\tilde\ell}\right]
+ \frac{\epsilon^3}{\beta}\tilde{\mathcal F}^{**} \left[\tau, \beta f, \frac{\beta f''}{\tilde\ell^2} , e, \frac{\epsilon^2 e''}{\tilde\ell^2}\right],
\end{align*}
where
\begin{align*}
{\mathcal D}_1(\theta)\,=\,\tilde{\mathcal D}_1(\theta)\,+\,\tilde{\mathcal D}_2(\theta)\,+\,\epsilon \tilde{\mathcal D}_4 (\theta)
\,+\,\tilde{\mathcal D}_5(\theta),
\end{align*}
\begin{align*}
\frac{\epsilon^3}{\beta}\tilde{\mathcal F}^* \left[\tau, \beta f, \frac{\beta {f}'}{\tilde\ell}, e, \frac{\epsilon e'}{\tilde\ell}\right]
&= \frac{\epsilon^3}{\beta}
\Bigg\{\tilde{\mathcal F}_1\left[\tau, \beta f, \frac{\beta {f}'}{\tilde\ell}\right]
+\tilde{\mathcal F}_{3}\left[\tau, \beta f, \frac{\beta {f}'}{\tilde\ell}, e\right]
+\tilde{\mathcal F}_{4}\left[\tau, \beta f, \frac{\beta {f}'}{\tilde\ell} \right]
 \nonumber\\[2mm]
&\qquad \quad
+ \epsilon\tilde{\mathcal F}_{5}\left[\tau, \beta f, \frac{\beta {f}'}{\tilde\ell}\right]
 + \epsilon\tilde{\mathcal F}_{10}\left[\tau, \beta f, \frac{\beta {f}'}{\tilde\ell}\right]
\nonumber\\[2mm]
&\qquad \quad
+ \tilde{\mathcal D}_3 (\theta) e
+ \tilde{\mathcal F}_2\left[\tau, \beta f, \frac{\beta f'}{\tilde\ell} \right] e
+ \tilde{\mathcal F}_{8}\left[\tau, \beta f, \frac{\beta {f}'}{\tilde\ell}, e, \frac{\epsilon e'}{\tilde\ell}\right]
\Bigg\},
\end{align*}
and
\begin{align*}
 \frac{\epsilon^3}{\beta}\tilde{\mathcal F}^{**} \left[\tau, \beta f, \frac{\beta f''}{\tilde\ell^2}, e, \frac{\epsilon^2 e''}{\tilde\ell^2}\right]
& =\frac{\epsilon^3}{\beta}
\Bigg\{ - \frac{8\pi^2}{{\tilde\ell}^2} \kappa f''f \varrho_1
- \frac{4\pi^2}{{\tilde\ell}^2}f'' e d_8 + \epsilon d_7{\mathbf b}_1 f \beta e\frac{\beta f''}{\tilde\ell^2}
\nonumber\\[2mm]
& \qquad \quad
- \frac{4\pi^2}{{\tilde\ell}^2} \,\beta f''\,\beta f\,  d_9
+ \epsilon^2 \tilde{\mathcal F}_{7}\left[\tau, \beta f, e \right] \frac{\beta f''}{\tilde\ell^2}
-\epsilon \frac{4\pi^2}{{\tilde\ell}^2} \beta f'' \tilde{\mathcal D}_6(\theta)
 \nonumber\\[2mm]
&  \qquad \quad
 + \epsilon\tilde{\mathcal F}_{9}\left[\tau, \beta f \right] \frac{\beta f''} {\tilde\ell^2}
 + \epsilon\tilde{\mathcal F}_{11}\left[\tau, \beta f \right] \frac{\beta f''} {\tilde\ell^2}
 + \tilde{\mathcal F}_{6}\left[\tau, \beta f \right] \frac{\epsilon^2 e''}{\tilde\ell^2}
 \Bigg\}.
\end{align*}
Moreover, by the expressions of the following terms
$$
\frac{\epsilon^3}{\beta}\tilde{\mathcal F}^{*} \left[\tau, f, \frac{\beta {f}'}{\tilde\ell}, e, \frac{\epsilon e'}{\tilde\ell}\right],
\qquad
\frac{\epsilon^3}{\beta}\tilde{\mathcal F}^{**} \left[\tau, \beta f, \frac{\beta f''}{\tilde\ell^2}, e, \frac{\epsilon^2 e''}{\tilde\ell^2}\right],
$$
we can verify that
\begin{align*}
\left\|\beta \frac{\epsilon^3}{\beta}\tilde{\mathcal F}^{*} \left[\tau, \beta f, \frac{\beta {f}'}{\tilde\ell}, e, \frac{\epsilon e'}{\tilde\ell}\right] \right\|_{L^2(0, 2\pi)}
\le C \big(\epsilon^3+ \epsilon^3 \|f\|_* +\epsilon^3 \|e\|_{**} \big),
\end{align*}
\begin{align*}
\left\| \beta \frac{\epsilon^3}{\beta}\tilde{\mathcal F}^{**} \left[\tau, \beta f, \frac{\beta f''}{\tilde\ell^2}, e, \frac{\epsilon^2 e''}{\tilde\ell^2}\right] \right\|_{L^2(0, 2\pi)} \le C \big(\epsilon^3+ \epsilon^3 \|f\|_* + \epsilon^3 \|e\|_{**} \big),
\end{align*}
\begin{align*}
& \left\| \beta \frac{\epsilon^3}{\beta}\left(\tilde{\mathcal F}^{*} \left[\tau, \beta f_1, \frac{\beta {f_1}'}{\tilde\ell}, e_1, \frac{\epsilon e_1'}{\tilde\ell}\right]
- \tilde{\mathcal F}^{*} \left[\tau, \beta f_2, \frac{\beta {f_2}'}{\tilde\ell}, e_2, \frac{\epsilon e_2'}{\tilde\ell}\right] \right) \right\|_{L^2(0, 2\pi)}
\nonumber\\[2mm]
&
\le C \epsilon^3 \big(\|f_1- f_2\|_*+ \|e_1- e_2\|_{**}\big),
\end{align*}
and
\begin{align*}
& \left\| \beta \frac{\epsilon^3}{\beta}\left(\, \tilde{\mathcal F}^{**} \left[\tau, \beta f_1,\frac{\beta f_1''}{\tilde\ell^2} ,  e_1, \frac{\epsilon^2 e_1''}{\tilde\ell^2}\right]
-
\tilde{\mathcal F}^{**} \left[\tau, \beta f_2, \frac{\beta f_2''}{\tilde\ell^2}, e_2, \frac{\epsilon^2 e_2''}{\tilde\ell^2}\right] \right) \right\|_{L^2(0, 2\pi)}
\nonumber\\[2mm]
&
\le C \epsilon^3 \left(\|f_1- f_2\|_*+ \|e_1- e_2\|_{**}\right).
\end{align*}

\medskip
\medskip
\subsubsection{Computation for $\beta^{-3}\int_{{\mathbb R}}\, \eta_\delta^\epsilon(s)\, {\mathcal E}\,{\mathcal Z} \,{\mathrm d}x$}\

By \eqref{errorrelationinterior}, we have
\begin{align*}
& \frac{1}{\beta^3}\int_{\mathbb R}\eta_\delta^\epsilon(s){\mathcal E}\,{\mathcal Z} \,{\mathrm d}x
\nonumber\\[2mm]
=& \int_{{\mathbb R}}\, S\big(w_2\big)\,{\mathcal Z} \,{\mathrm d}x
\ +\
\int_{\mathbb R} \big[\eta_\delta^\epsilon(s)-1\big]S\big(w_2\big)\,{\mathcal Z} \,{\mathrm d}x
\nonumber\\[2mm]
=&
\int_{\mathbb R} \Bigg[\epsilon^2 S_{3} + \epsilon^2 S_{42} + \epsilon^3 S_{5} +\epsilon^3 S_{6}
 +\epsilon^3 S_{7} +\epsilon^3 S_{8}
 +\frac 1\beta\epsilon^3\frac{4\pi^2}{{\tilde\ell}^2} e''\mathcal Z
 +\frac 1\beta \epsilon\lambda_0 e{\mathcal Z}
 +\epsilon^2 {\bf N}_{02}
  +\epsilon^3 {\bf N}_{03}
\nonumber\\[2mm]
&\qquad+ \epsilon^2 {\bf B}_{42}
 + \epsilon^3 \frac{4\pi^2}{{\tilde\ell}^2} {\phi}_{1, \tau\tau}
+ \epsilon^4 \frac{4\pi^2}{{\tilde\ell}^2}\phi_{2, \tau\tau}
+B_{5}(\epsilon\phi_1) + \mathbb{N}_1 (\epsilon^2 \phi_2)
+ B_4(\epsilon^2\phi_2)+ B_5 (\epsilon^2{\phi_2}) \Bigg] \,{\mathcal Z} \,{\mathrm d}x + \mathfrak P(\theta)
\nonumber\\[2mm]
&
 +\int_{\mathbb R} \eta_\delta^\epsilon(s) \Bigg\{\epsilon^4 \frac{4\pi^2} {{\tilde\ell}^2} a_2\frac{\kappa}{\, \beta^2\, } e'' (x+f)\mathcal Z -\epsilon^4 \kappa a_2 \frac{ 4\pi^2} {{\tilde\ell}^2} \frac{1}{\, \beta^2\, } (x+f) f'' e \mathcal Z' + \frac{\epsilon^4}{\beta^2} \hat{\mathcal F}_1\left[\tau, x, \beta f, \frac{\beta f''}{\tilde\ell^2} \right]
\nonumber\\[2mm]
&\qquad\qquad\qquad
+B_{6}(\epsilon\phi_1)
 + \frac{\epsilon^4}{\beta^2} \mathcal F_1\left[\tau, x, \beta f, \frac{\beta f'}{\tilde\ell}\right]
 +\frac{\epsilon^4} {\beta^2} \mathcal G_1\left[\tau, x, \beta f, \frac{\beta f'}{\tilde\ell}, e, \frac{e'}{\tilde\ell}\right]
 + B_6 (\epsilon^2{\phi_2}) \Bigg\} \,{\mathcal Z} \,{\mathrm d}x,
\end{align*}
 where
\begin{align*}
 \mathfrak P(\theta) & = \int_{\mathbb R}\big[\eta_\delta^\epsilon(s)-1\big] \Bigg[\epsilon^2 S_{3} + \epsilon^2 S_{42} + \epsilon^3 S_{5} +\epsilon^3 S_{6}
 +\epsilon^3 S_{7} +\epsilon^3 S_{8}
 +\frac 1\beta\epsilon^3\frac{4\pi^2}{{\tilde\ell}^2} e''\mathcal Z
\nonumber\\[2mm]
&\qquad \qquad \qquad \qquad
 +\frac 1\beta \epsilon\lambda_0 e{\mathcal Z}
 +\epsilon^2 {\bf N}_{02}
  +\epsilon^3 {\bf N}_{03}+ \epsilon^2 {\bf B}_{42}
 + \epsilon^3 \frac{4\pi^2}{{\tilde\ell}^2} {\phi}_{1, \tau\tau}
\nonumber\\[2mm]
&\qquad \qquad \qquad \qquad
+ \epsilon^4 \frac{4\pi^2}{{\tilde\ell}^2}\phi_{2, \tau\tau}
 +B_{5}(\epsilon\phi_1) + \mathbb{N}_1 (\epsilon^2 \phi_2)
+ B_4(\epsilon^2\phi_2)+ B_5 (\epsilon^2{\phi_2}) \Bigg] \,{\mathcal Z} \,{\mathrm d}x.
\end{align*}
By the exponential decay of the terms in $S\big(w_2\big)$ with respect to $x$,
 we neglect the calculations of the terms in $ \mathfrak P(\theta). $ Since $S_{3}, S_5, S_7$ are odd functions of $x$, then integration against $\mathcal Z$ just vanishes.
This gives that
\begin{align}\label{projected eqs for enew}
& \frac{1}{\beta^3}\int_{\mathbb R}\eta_\delta^\epsilon(s){\mathcal E}\,{\mathcal Z} \,{\mathrm d}x
\nonumber\\[2mm]
=&
 \int_{\mathbb R} \epsilon^2 S_{42} \,{\mathcal Z} \,{\mathrm d}x
+\int_{\mathbb R} \epsilon^3 S_{6}\,{\mathcal Z} \,{\mathrm d}x
 +\int_{\mathbb R} \epsilon^3 S_{8}\,{\mathcal Z} \,{\mathrm d}x
 +\int_{\mathbb R} \frac 1\beta\epsilon^3\frac{4\pi^2}{{\tilde\ell}^2} e''\mathcal Z^2 \,{\mathrm d}x
 +\int_{\mathbb R} \frac 1\beta \epsilon\lambda_0 e{\mathcal Z}^2 \,{\mathrm d}x
\nonumber\\[2mm]
& +\int_{\mathbb R} \epsilon^2 {\bf N}_{02}\,{\mathcal Z} \,{\mathrm d}x
+\int_{\mathbb R} \epsilon^3 {\bf N}_{03}\,{\mathcal Z} \,{\mathrm d}x
+ \int_{\mathbb R} \epsilon^2 {\bf B}_{42}\,{\mathcal Z} \,{\mathrm d}x
+\int_{\mathbb R}  \epsilon^3 \frac{4\pi^2}{{\tilde\ell}^2} {\phi}_{1, \tau\tau}\,{\mathcal Z} \,{\mathrm d}x
+\int_{\mathbb R} \epsilon^4 \frac{4\pi^2}{{\tilde\ell}^2}\phi_{2, \tau\tau}\,{\mathcal Z} \,{\mathrm d}x
\nonumber\\[2mm]
&
+\int_{\mathbb R} B_{5}(\epsilon\phi_1) \,{\mathcal Z} \,{\mathrm d}x
+\int_{\mathbb R} \mathbb{N}_1 (\epsilon^2 \phi_2)\,{\mathcal Z} \,{\mathrm d}x
+\int_{\mathbb R} B_4(\epsilon^2\phi_2)\,{\mathcal Z} \,{\mathrm d}x
+\int_{\mathbb R} B_5 (\epsilon^2{\phi_2})\,{\mathcal Z} \,{\mathrm d}x
\nonumber\\[2mm]
&
 +\int_{\mathbb R} \eta_\delta^\epsilon(s) \Bigg\{ \epsilon^4 \frac{4\pi^2} {{\tilde\ell}^2} a_2\frac{\kappa}{\, \beta^2\, } e'' (x+f)\mathcal Z -\epsilon^4 \kappa a_2 \frac{ 4\pi^2} {{\tilde\ell}^2} \frac{1}{\, \beta^2\, } (x+f) f'' e \mathcal Z'  \Bigg\} \,{\mathcal Z} \,{\mathrm d}x
 \nonumber\\[2mm]
&
 +\int_{\mathbb R} \eta_\delta^\epsilon(s) \Bigg\{  \frac{\epsilon^4}{\beta^2} \hat{\mathcal F}_1\left[\tau, x, \beta f, \frac{\beta f''}{\tilde\ell^2} \right]
 + \frac{\epsilon^4}{\beta^2} \mathcal F_1\left[\tau, x, \beta f, \frac{\beta f'}{\tilde\ell}\right]
 +\frac{\epsilon^4} {\beta^2} \mathcal G_1\left[\tau, x, \beta f, \frac{\beta f'}{\tilde\ell}, e, \frac{e'}{\tilde\ell}\right] \Bigg\} \,{\mathcal Z} \,{\mathrm d}x
  \nonumber\\[2mm]
&
 +\int_{\mathbb R} \eta_\delta^\epsilon(s) \Bigg\{ B_{6}(\epsilon\phi_1)
 + B_6 (\epsilon^2{\phi_2}) \Bigg\} \,{\mathcal Z} \,{\mathrm d}x
 \nonumber\\[2mm]
& = \sum_{j=0}^{17}\,II_j.
\end{align}
We calculate these terms as follows.

\smallskip\noindent$\clubsuit$
 From the expression of $\epsilon^3 S_{42}$ as in \eqref{barS41}, we have
\begin{align}\label{ESTIMATEII1}
II_1
\,=\,&\epsilon^2 \int_{{\mathbb R}} \Bigg\{
-\frac{1}{2}\, \frac{ \epsilon^2W_{tt}(0, \theta)}{\beta^4}|f|^2 U
+
\Big[\frac{4\pi^2}{{\tilde\ell}^2}|f'|^2
-\frac{4\pi}{{\tilde\ell}}\frac{\beta'}{\beta^2}ff'
+\frac{|\beta'|^2}{\beta^4} |f|^2
\Big]U''
\nonumber\\[2mm]
 & \qquad\qquad
+\frac 23 \frac{\kappa}{\, \beta^2\, } fe{\mathcal Z}
+\frac3{\beta^2}e^2U{\mathcal Z}^2
\Bigg\}\mathcal Z \,{\mathrm d}x
\nonumber\\[2mm]
\,=\,& -\frac{1}{2}\, \epsilon^2 \frac{ \epsilon^2W_{tt}(0, \theta)}{\beta^4}|f|^2 b_5 + \epsilon^2\Big[\frac{4\pi^2}{{\tilde\ell}^2}|f'|^2
-\frac{4\pi}{{\tilde\ell}}\frac{\beta'}{\beta^2}ff'
+\frac{|\beta'|^2}{\beta^4} |f|^2
\Big] b_6
+\frac 23\epsilon^2 \frac{\kappa}{\, \beta^2\, } fe + 3\epsilon^2\frac{1}{\beta^2}e^2b_1 ,
\end{align}
where
\[
b_0 = \int_{{\mathbb R}}{\mathcal Z}^2 \,{\mathrm d}x=1,
\qquad
b_1= \int_{{\mathbb R}}U{\mathcal Z}^3 \,{\mathrm d}x,
\qquad
b_5 = \int_{{\mathbb R}} U\mathcal Z \,{\mathrm d}x,
\qquad
b_6 = \int_{{\mathbb R}} U''{\mathcal Z} \,{\mathrm d}x.
\]

\smallskip\noindent $\clubsuit$
 From the expression of $\epsilon^3 S_{6}$ as in \eqref{S6}, we have
\begin{align}\label{ESTIMATEII2}
II_{2}
\,=\,&\frac{\epsilon^3}{\beta}\tilde{\mathcal K}_1\left[\tau, {\beta f}, \frac{\beta{ f}'}{\tilde\ell} \right] -\epsilon^3\frac{8\pi^2}{{\tilde\ell}^2}\frac{\kappa}{\beta} f''b_4,
\end{align}
where
$$
b_4 = \int_{\mathbb R} xU' \mathcal Z \,{\mathrm d}x,
$$
and
\begin{align*}
 \frac{\epsilon^3}{\beta}\tilde{\mathcal K}_1\left[\tau, {\beta f}, \frac{\beta{ f}'}{\tilde\ell} \right]
& = \epsilon^3\Bigg\{
 - \frac{1}{2} \frac{\,\epsilon^2 W_{ttt}(0, \theta)}{\beta^5}f \int_{\mathbb R} x^2 U \mathcal Z\,{\mathrm d}x
 \nonumber\\[2mm]
 & \qquad \quad +
\Big[- \frac{1}{6} \frac{\epsilon^2W_{ttt}(0, \theta)}{\beta^5}f^3
 -\frac{2\kappa\beta''}{\beta^4} f +\frac{\kappa'\beta'} {\beta^4} f\Big]\int_{\mathbb R} U\mathcal Z\,{\mathrm d}x
\nonumber\\[2mm]
 &\qquad \quad + \Big[
 \frac{2\kappa^3}{\beta^3}f
 -\frac{8\kappa|\beta'|^2}{\beta^5} f
 -\frac{4\kappa\beta''}{\beta^4}f
+\frac{12\pi }{ {\tilde\ell}}\frac{\kappa\beta'}{\beta^3} f'
+\frac{2\kappa'\beta'}{\beta^4}f
- \frac{2\pi}{{\tilde\ell}} \frac{\kappa'}{\beta^2} f'\Big] \int_{\mathbb R} xU'\mathcal Z\,{\mathrm d}x
\nonumber\\[2mm]
 & \qquad \quad + \Big[ -\frac{6\kappa|\beta'|^2}{\beta^5}f
 +\frac{8\pi}{ {\tilde\ell}}\frac{\kappa\beta'}{\beta^3} f'\Big] \int_{\mathbb R} x^2 U''\mathcal Z\,{\mathrm d}x
\nonumber\\[2mm]
&\qquad \quad+ \Big[ -\frac{2\kappa|\beta'|^2}{\beta^5}f^3
 +\frac{8\pi}{ {\tilde\ell}} \frac{\kappa\beta'}{\beta^3} f' f^2
 -\frac{8\pi^2} {{\tilde\ell}^2}\frac{\kappa}{\beta} |f'|^2 f \Big]\int_{\mathbb R} U''\mathcal Z\,{\mathrm d}x
 \Bigg\}.
\end{align*}

\smallskip\noindent $\clubsuit$
 From the expression of $\epsilon^3 S_8$ as in \eqref{S8}, we have
\begin{align}\label{ESTIMATEII3}
II_{3}
\,=\,&\frac{\epsilon^3}\beta\int_{{\mathbb R}} \Bigg\{
 \frac{ 2\pi}{\tilde\ell} \frac{\beta'}{\beta^2} e'{\mathcal Z}
+\frac{4\pi}{{\tilde\ell}}\frac{\beta'}{\beta^2} e' x {\mathcal Z}'
+\Big[
-\frac{1}{2}\, \frac{\epsilon^2W_{tt}(0, \theta)}{\beta^4}f^2 \Big]e {\mathcal Z}
\nonumber\\[2mm]
&\qquad\qquad
+\Big[\frac{4\pi^2}{{\tilde\ell}^2}|f'|^2
-\frac{4\pi}{ {\tilde\ell}}\frac{\beta'}{\beta^2}ff'
+\frac{|\beta'|^2}{\beta^4} f^2
\Big]e{\mathcal Z}''
-\frac{1}{2}\, \frac{\epsilon^2W_{tt}(0, \theta)}{\beta^4} e x^2{\mathcal Z}
\nonumber\\[2mm]
&\qquad\qquad
+\Big[
\frac{\beta''} {\beta^3}
\,-\,\frac{\kappa^2}{\beta^2}
\Big]e x{\mathcal Z}'
+\frac{|\beta'|^2}{\beta^4} e x^2{\mathcal Z}''
+ \frac{e^3}{\beta^2} {\mathcal Z}^3
\Bigg\} {\mathcal Z} \,{\mathrm d}x
\nonumber\\[2mm]
\,=\,& \epsilon^3
\frac{1}{\beta} \tilde{\mathcal K}_{2}\left[\tau, {\beta f}, \frac{\beta { f}'}{\tilde\ell}\right] e + \epsilon^3\frac{e^3}{\beta^3}b_2 ,
\end{align}
where
\begin{align*}
b_2\,=\, \int_{\mathbb R} |\mathcal Z|^4 \,{\mathrm d}x,
\end{align*}
and
\begin{align*}
\frac{\epsilon^3}{\beta} \tilde{\mathcal K}_2\left[\tau, {\beta f}, \frac{\beta { f}'}{\tilde\ell}\right]e
\,=\,&\frac{\epsilon^3}\beta e\int_{{\mathbb R}} \Bigg\{
 \Big[
-\frac{1}{2}\, \frac{\epsilon^2W_{tt}(0, \theta)}{\beta^4}f^2 \Big] {\mathcal Z}
+\Big[\frac{4\pi^2}{{\tilde\ell}^2}|f'|^2
-\frac{4\pi}{ {\tilde\ell}}\frac{\beta'}{\beta^2}ff'
+\frac{|\beta'|^2}{\beta^4} f^2
\Big]{\mathcal Z}''
\nonumber\\[2mm]
&\qquad \qquad-\frac{1}{2}\, \frac{\epsilon^2W_{tt}(0, \theta)}{\beta^4} x^2{\mathcal Z}
+\Big[
\frac{\beta''} {\beta^3}
\,-\,\frac{\kappa^2}{\beta^2}
\Big] x{\mathcal Z}'
+\frac{|\beta'|^2}{\beta^4} x^2{\mathcal Z}''
\Bigg\} {\mathcal Z} \,{\mathrm d}x.
\end{align*}
Here we have used the identity
\begin{align*}
\int_{\mathbb R} x{\mathcal Z}'{\mathcal Z} \,{\mathrm d}x= - \frac12\int_{\mathbb R} |\mathcal Z|^2 \,{\mathrm d}x.
\end{align*}

\smallskip\noindent $\clubsuit$
Next, we have
\begin{align} \label{main term mathccal Z}
 II_{4}+II_{5}\,=\,& \int_{{\mathbb R}} \Bigg[\frac 1\beta\epsilon^3\frac{4\pi^2}{{\tilde\ell}^2} e''\mathcal Z
 +\frac 1\beta \epsilon\lambda_0 e{\mathcal Z} \Bigg] {\mathcal Z} \,{\mathrm d}x
 = \frac {1}\beta \Bigg[\epsilon^3\frac{4\pi^2}{{\tilde\ell}^2} e'' + \epsilon\lambda_0 e \Bigg].
\end{align}

\smallskip\noindent $\clubsuit$
 From the expression of ${\bf N}_{02}$ as in \eqref{N02}, we have
\begin{align}\label{ESTIMATEII6}
II_{6}
 \,=\,
&\epsilon^2\int_{{\mathbb R}}\Bigg\{\frac{6 {\mathbf b}_1e U{\mathcal Z}\varpi_1
+4 {\mathbf b}_1fe U{\mathcal Z} \varpi_2}\beta
+\frac{4}{3}{\mathbf b}_1^2|f|^2 U\varpi_2^2
+ 4{\mathbf b}_1^2f U\varpi_1\varpi_2\Bigg\}{\mathcal Z} \,{\mathrm d}x
\nonumber\\[2mm]
 \,=\,& \epsilon^2\int_{{\mathbb R}} \left\{ \frac{4 {\mathbf b}_1fe U{\mathcal Z} \varpi_2}\beta +\frac{4}{3}{\mathbf b}_1^2|f|^2 U\varpi_2^2\right\} {\mathcal Z} \,{\mathrm d}x
\nonumber\\[2mm]
 : =\,&
\frac{\epsilon^2}{\beta}\tilde{\mathcal K}_3\left[\tau, e, {\beta f}\right].
\end{align}

\smallskip\noindent $\clubsuit$
 From the expression of ${\bf N}_{03}$ as in \eqref{N03}, we have
\begin{align}\label{nonlinear term w1 emathcal Z}
II_{7}
 \,=\,
&\epsilon^3\int_{{\mathbb R}}\Bigg\{\frac{3 e^2\mathcal Z^2 {\mathbf b}_1}{\beta^2}\Big[\varpi_1 + \frac{2}{3}f \varpi_2\Big]
+ \frac{3 e{\mathcal Z} {\mathbf b}_1^2}{\beta}\Big[\varpi_1 + \frac{2}{3}f \varpi_2\Big]^2
+ {\mathbf b}_1^3\Big[\varpi_1 + \frac{2}{3}f \varpi_2\Big]^3 \Bigg\}{\mathcal Z} \,{\mathrm d}x
\nonumber\\[2mm]
 \,=\,& \epsilon^3\int_{{\mathbb R}} \left\{ \frac{2 e^2\mathcal Z^2 {\mathbf b}_1}{\beta^2} f \varpi_2 + \frac{3 e{\mathcal Z} {\mathbf b}_1^2}{\beta}\Big[\varpi_1^2 + \frac{4}{9}f^2 \varpi_2^2\Big] + {\mathbf b}_1^3\Big[ 2 \varpi_1^2 \varpi_2 f + \frac8{27} f^3 \varpi_2^3\Big] \right\} {\mathcal Z} \,{\mathrm d}x
\nonumber\\[2mm]
:=\,&
\frac{\epsilon^3}{\beta}\tilde{\mathcal K}_4\left[\tau, e, {\beta f}\right].
\end{align}

\smallskip\noindent $\clubsuit$
 From the expression of $ \epsilon^2 {\bf B}_{42}$ as in \eqref{B42}, we have
\begin{align}\label{B4emathcal Z2}
II_8
 =\epsilon^2\int_{{\mathbb R}} {\mathbf b}_1^2\Big[
\frac{2}{3}f\varpi_{2,x}
+ \frac{4}{9} f x\varpi_2
+\frac{2}{3}f\varpi_1
+ \frac{4}{9} f^2\varpi_2
\Big]\mathcal Z \,{\mathrm d}x
= \epsilon^2 \frac{4}{9}  {\mathbf b}_1^2 b_3\,f^2
=\epsilon^2 \frac{4}{9} \frac{\kappa^2}{\beta^2} b_3\,f^2 ,
\end{align}
where
$$
b_3= \int_{\mathbb R} \varpi_2\mathcal Z \,{\mathrm d}x,
\qquad
 {\mathbf b}_1\,=\,\frac{\kappa}{\beta}.
$$

\smallskip\noindent $\clubsuit$
 Let's review the expression for $\phi_{1, \tau\tau} $ in \eqref{expression of phi1tautau}, then we have
\begin{align} \label{phitautau mathcal Z}
II_9
 \,=\, &\epsilon^3 \int_{\mathbb R} \Big[{\mathbf b}_3 \varpi_{1}
+\frac{2}{3}{\mathbf b}_3 f \varpi_{2}
+\frac{4}{3}\frac{2\pi}{{\tilde\ell}}{\mathbf b}_2f' \varpi_{2}
+\frac{2}{3}\frac{4\pi^2}{{\tilde\ell}^2} {\mathbf b}_1 f''\varpi_{2} \Big]
\mathcal Z \,{\mathrm d}x
\nonumber\\[2mm]
:  \,=\,&\frac{\epsilon^3}{\beta}\tilde{\mathcal K}_5\left[\tau, \beta{ f}, \frac{\beta{f}'}{\tilde\ell}\right]
+ \frac{\epsilon^3}{\beta}\frac{8\pi^2}{3} \kappa\frac{ f''}{\tilde \ell^2}b_3,
\end{align}
where
\begin{align*}
\frac{\epsilon^3}{\beta}\tilde{\mathcal K}_5\left[\tau, \beta{ f}, \frac{\beta{f}'}{\tilde\ell}\right] = \epsilon^3 \int_{\mathbb R} \Big[
\frac{2}{3}{\mathbf b}_3 f \varpi_{2}
+
\frac{4}{3}\frac{2\pi}{{\tilde\ell}}{\mathbf b}_2f' \varpi_{2}
 \Big]
\mathcal Z \,{\mathrm d}x.
\end{align*}

\smallskip\noindent $\clubsuit$
By the definition of $\epsilon^2\phi_2$ given in \eqref{definition of e2phi2}, we directly get
\begin{align} \label{phi2tautau mathcal Z}
II_{10}\,=\,\epsilon^4 \frac{4\pi^2}{{\tilde\ell}^2}\int_{\mathbb R}\phi_{2, \tau\tau} {\mathcal Z} \,{\mathrm d}x =\frac{\epsilon^4}{\beta} \tilde{\mathcal D}_7(\theta).
\end{align}

\smallskip\noindent
$\clubsuit$ From the expression of $B_{5}(\epsilon \phi_1)$ as in \eqref{expression of B5phi1}, we have
\begin{align}\label{B5phi1 intergral Z}
II_{11}
& =
\epsilon^3
\int_{\mathbb R}
\Bigg\{\frac{2\pi}{{\tilde\ell}}\frac{\beta'} {\beta^2}\Big[- f' \phi_{1,x} +\phi_{1,\tau}\Big]
 +\frac{4\pi^2}{{\tilde\ell}^2}\Big[|f'|^2 \phi_{1,xx} -f''\phi_{1,x} -2f' \phi_{1,x\tau}\Big]
\nonumber\\[2mm]
&
\qquad\qquad + \frac{4\pi}{{\tilde\ell}} \frac{\beta'}{\beta^2} (x+f)\big[- f' \phi_{1,xx} +\phi_{1,x\tau}\big]
+\Big[\frac{\beta''}{\beta^3} \,-\,\frac{\kappa^2}{\beta^2} \Big]
(x+f) \phi_{1,x}
\nonumber\\[2mm]
&
\qquad\qquad + \frac{|\beta'|^2}{\beta^4}(x+f)^2 \phi_{1,xx}
-\frac{1}{2}\, \frac{\epsilon^2W_{tt}(0, \theta)}{\beta^4} (x+f)^2\phi_{1}
\Bigg\} {\mathcal Z} \,{\mathrm d}x
\nonumber\\[2mm]
& = - \epsilon^3 \frac{4\pi^2}{{\tilde\ell}^2} f'' \int_{{\mathbb R}} \left[{\mathbf b}_1\varpi_{1,x} +\frac{2}{3}{\mathbf b}_1 f \varpi_{2,x} \right]\, \mathcal Z \,{\mathrm d}x + \frac{\epsilon^3}{\beta}\tilde{\mathcal K}_6\left[\tau, \beta{ f}, \frac{\beta{f}'}{\tilde\ell}\right]
\nonumber\\[2mm]
& = - \epsilon^3 {\mathbf b}_1\frac{4\pi^2}{{\tilde\ell}^2} f'' \int_{{\mathbb R}} \varpi_{1,x} \, \mathcal Z \,{\mathrm d}x
+ \frac{\epsilon^3}{\beta}\tilde{\mathcal K}_6\left[\tau, \beta{ f}, \frac{\beta{f}'}{\tilde\ell}\right]
\nonumber\\[2mm]
& = - \epsilon^3 \frac{\kappa }{\beta}\,\frac{4\pi^2}{{\tilde\ell}^2} b_7\,f''
+ \frac{\epsilon^3}{\beta}\tilde{\mathcal K}_6\left[\tau, \beta{ f}, \frac{\beta{f}'}{\tilde\ell}\right],
\end{align}
where
\begin{equation*}
b_7\,=\, \int_{{\mathbb R}} \varpi_{1,x} \, \mathcal Z \,{\mathrm d}x,
\end{equation*}
and
\begin{align*}
\frac{\epsilon^3}{\beta}\tilde{\mathcal K}_6\left[\tau, \beta{ f}, \frac{\beta{f}'}{\tilde\ell}\right] & =
\epsilon^3
\int_{\mathbb R}
\Bigg\{\frac{2\pi}{{\tilde\ell}}\frac{\beta'} {\beta^2}\Big[- f' \phi_{1,x} +\phi_{1,\tau}\Big]
 +\frac{4\pi^2}{{\tilde\ell}^2}\Big[|f'|^2 \phi_{1,xx} -2f' \phi_{1,x\tau}\Big]
\nonumber\\[2mm]
&
\qquad\qquad + \frac{4\pi}{{\tilde\ell}} \frac{\beta'}{\beta^2} (x+f)\big[- f' \phi_{1,xx} +\phi_{1,x\tau}\big]
+\Big[\frac{\beta''}{\beta^3} \,-\,\frac{\kappa^2}{\beta^2} \Big]
(x+f) \phi_{1,x}
\nonumber\\[2mm]
&
\qquad\qquad + \frac{|\beta'|^2}{\beta^4}(x+f)^2 \phi_{1,xx}
-\frac{1}{2}\, \frac{\epsilon^2W_{tt}(0, \theta)}{\beta^4} (x+f)^2\phi_{1}
\Bigg\} {\mathcal Z} \,{\mathrm d}x.
\end{align*}
\smallskip\noindent
$\clubsuit$
By \eqref{definition of mathbbN1}, we have
\begin{align} \label{N1phi2 intergral Z}
II_{12}
=\int_{\mathbb R}\, \mathbb{N}_1 (\epsilon^2 \phi_2) \mathcal Z \,{\mathrm d}x
= & \int_{\mathbb R} \Bigg\{
 \epsilon^3 \frac{6e}{\beta}U{\mathcal Z}\phi_2 + \epsilon^3 6 U\phi_1\phi_2
+\epsilon^4 \frac{3e^2}{\beta^2}{\mathcal Z}^2\phi_2 +\epsilon^4 3\phi_1^2\phi_2
\nonumber\\[2mm]
&\qquad
+\epsilon^4 \frac{6e}{\beta}{\mathcal Z}\phi_1\phi_2 + \epsilon^4 3U \phi_2^2
+\epsilon^5 \frac{3e}{\beta}{\mathcal Z}\phi_2^2 + \epsilon^5 3\phi_1 \phi_2^2
+\epsilon^6 \phi_2^3 \Bigg\} \mathcal Z \,{\mathrm d}x
\nonumber\\[2mm]
 =& \int_{\mathbb R} \Bigg\{
 \epsilon^3 \frac{6e}{\beta}U{\mathcal Z}\phi_2 + \epsilon^3 4 {\mathbf b}_1 f\varpi_2\phi_2 U
+\epsilon^4 \frac{3e^2}{\beta^2}{\mathcal Z}^2\phi_2 +\epsilon^4 3 \big({\mathbf b}_1^2 \varpi_1^2 + \frac 49{\mathbf b}_1^2 f^2 \varpi_2^2 \big) \phi_2
\nonumber\\[2mm]
&\quad
\qquad +\epsilon^4 \frac{4e}{\beta}{\mathcal Z}{\mathbf b}_1 f \varpi_2\phi_2 + \epsilon^4 3U \phi_2^2 +\epsilon^5 \frac{3e}{\beta}{\mathcal Z}\phi_2^2 + \epsilon^5 2 {\mathbf b}_1 f \varpi_2 \phi_2^2
+\epsilon^6 \phi_2^3 \Bigg\} \mathcal Z \,{\mathrm d}x
\nonumber\\[2mm]
: = & \frac{\epsilon^3}{\beta}\tilde{\mathcal K}_8\left[\tau, e, {\beta f}\right] +  \frac{\epsilon^4}{\beta}\tilde {D}_8(\theta),
\end{align}
where
\begin{align*}
\tilde {D}_8 (\theta) = \int_{\mathbb R} \Big( 3 {\mathbf b}_1^2 \varpi_1^2 +  3U \phi_2^2 +\epsilon^2 \phi_2^3 \Big) \mathcal Z \,{\mathrm d}x,
\end{align*}
and
\begin{align*}
\frac{\epsilon^3}{\beta}\tilde{\mathcal K}_8\left[\tau, e, { f}\right] & = \int_{\mathbb R} \Bigg\{
 \epsilon^3 \frac{6e}{\beta}U{\mathcal Z}\phi_2 + \epsilon^3 4 {\mathbf b}_1 f\varpi_2\phi_2 U
+\epsilon^4 \frac{3e^2}{\beta^2}{\mathcal Z}^2\phi_2 + \frac 43 \epsilon^4 {\mathbf b}_1^2 f^2 \varpi_2^2 \phi_2
\nonumber\\[2mm]
&\qquad
\qquad +\epsilon^4 \frac{4e}{\beta}{\mathcal Z}{\mathbf b}_1 f \varpi_2\phi_2 +\epsilon^5 \frac{3e}{\beta}{\mathcal Z}\phi_2^2 + \epsilon^5 2 {\mathbf b}_1 f \varpi_2 \phi_2^2
 \Bigg\} \mathcal Z \,{\mathrm d}x.
\end{align*}

\smallskip\noindent
$\clubsuit$
Since $\phi_2$ is even of $x$, from the expression of $ B_4$ as in \eqref{new opreator B4}, then we have
\begin{align}\label{B4phi2 intergral Z}
 II_{13}
 =\int_{{\mathbb R}} B_4(\epsilon^2\phi_2)\,{\mathcal Z} \,{\mathrm d}x
 \,=\,& \epsilon^3 \frac{\kappa}{\, \beta\, }\int_{{\mathbb R}} \Big[ \phi_{2,x} +\frac{2}{3} (x+f)\phi_2\Big]\mathcal Z \,{\mathrm d}x
\nonumber\\[2mm]
 \,=\,&\epsilon^3 \frac{\kappa}{\, \beta\, }\int_{{\mathbb R}} \frac{2}{3} f\phi_2\mathcal Z \,{\mathrm d}x=
 \frac{2}{3} \epsilon^3 \frac{\kappa}{\, \beta\, } \tilde {D}_9 (\theta) f,
\end{align}
 where
\begin{align*}
 \tilde {D}_9 (\theta) = \int_{{\mathbb R}} \phi_2\mathcal Z \,{\mathrm d}x.
\end{align*}

\smallskip\noindent
$\clubsuit$
By the expression of the operator $B_5$ given in \eqref{new opreator B5} and $\phi_2$ is even of $x$, we have
\begin{align} \label{B5phi2 intergral Z}
II_{14}
\,=\,&\epsilon^4\int_{\mathbb R}\Bigg\{\frac{4\pi^2}{{\tilde\ell}^2}\Big[|f'|^2 \phi_{2, xx} -f''\phi_{2, x} -2f' \phi_{2, x\tau}\Big]
+\frac{6\pi}{{\tilde\ell}}\frac{\beta'} {\beta^2}\Big[- f' \phi_{2, x} +\phi_{2, \tau}\Big]
\nonumber\\[2mm]
&\qquad
+ \frac{4\pi}{{\tilde\ell}} \frac{\beta'}{\beta^2} (x+f)\big[- f' w_{xx} +w_{x\tau}\big]
+\Big[
\frac{\beta''}{\beta^3}\,+\,\frac{2|\beta'|^2}{\beta^4} \,-\,\frac{\kappa^2}{\beta^2} \Big] (x+f) \phi_{2, x}
\nonumber\\[2mm]
&\qquad
+\frac{\beta''}{\beta^3}\phi_2
+ \frac{|\beta'|^2}{\beta^4}(x+f)^2 \phi_{2, xx}
-\frac{1}{2}\, \frac{\epsilon^2 W_{tt}(0, \theta)}{\beta^4} (x+f)^2\phi_2 \Bigg\} \mathcal Z \,{\mathrm d}x
\nonumber\\[2mm]
\,=\,&\epsilon^4 \int_{\mathbb R}\Bigg\{\frac{4\pi^2}{{\tilde\ell}^2}\Big[|f'|^2 \phi_{2, xx} \Big]
+\frac{6\pi}{{\tilde\ell}}\frac{\beta'} {\beta^2} \phi_{2, \tau}
- \frac{4\pi}{{\tilde\ell}} \frac{\beta'}{\beta^2} (x+f) f' \phi_{2, xx}
+\Big[
\frac{\beta''}{\beta^3}\,+\,\frac{2|\beta'|^2}{\beta^4} \,-\,\frac{\kappa^2}{\beta^2} \Big] x\phi_{2, x}
\nonumber\\[2mm]
&\qquad\quad
+ \frac{\beta''}{\beta^3}\phi_2
+ \frac{|\beta'|^2}{\beta^4} (x^2+f^2) \phi_{2, xx}
-\frac{1}{2}\,\frac{ \epsilon^2W_{tt}(0, \theta)}{\beta^4} (x^2+f^2)\phi_2 \Bigg\} \mathcal Z \,{\mathrm d}x
\nonumber\\[2mm]
 : =\,& \frac{\epsilon^4}{\beta}\tilde{\mathcal K}_7\left[\tau, \beta{ f}, \frac{\beta{f}'}{\tilde\ell}\right].
\end{align}

\smallskip\noindent
$\clubsuit$
By the expression of $a_2(\epsilon s, \theta) $ given in \eqref{a1a2a3}, then we have
\begin{align} \label{second term of e intergral Z}
II_{15}
=\,& \int_{\mathbb R} \eta_\delta^\epsilon(s) \Bigg\{ \epsilon^4 \frac{4\pi^2} {{\tilde\ell}^2} a_2\frac{\kappa}{\, \beta^2\, } e'' (x+f)\mathcal Z -\epsilon^4 \kappa a_2 \frac{ 4\pi^2} {{\tilde\ell}^2} \frac{1}{\, \beta^2\, } (x+f) f'' e \mathcal Z'  \Bigg\} \,{\mathcal Z} \,{\mathrm d}x
\nonumber\\[2mm]
 =& -\int_{\mathbb R} \ \epsilon^4 \frac{8\pi^2} {{\tilde\ell}^2} \frac{\kappa}{\, \beta^2\, } e'' f\mathcal Z \mathcal Z \,{\mathrm d}x
\nonumber\\[2mm]
&+\underbrace{ \int_{\mathbb R} \eta_\delta^\epsilon(s) \Bigg[ \epsilon^4 \frac{4\pi^2} {{\tilde\ell}^2} \Big(\epsilon \frac{3\kappa}{\beta}(x+f) + \epsilon^2 {\tilde a}_2 \Big) \frac{\kappa}{\, \beta^2\, } e'' (x+f)\mathcal Z \Bigg] \mathcal Z \,{\mathrm d}x +O(e^{-c\frac{|\beta \delta|} \epsilon} )e''}_{:= \frac{\epsilon^3}{\beta}\tilde{\mathcal K}_9 \left[\tau, \beta f, \frac{\epsilon^2 e''}{\tilde \ell^2}\right]}
\nonumber\\[2mm]
&+\int_{\mathbb R}  \epsilon^4 \frac{ 8\pi^2} {{\tilde\ell}^2} \frac{\kappa }{\, \beta^2\, } f'' e x \mathcal Z'  \mathcal Z \,{\mathrm d}x
 \nonumber\\[2mm]
 &+ \underbrace{\int_{\mathbb R} \eta_\delta^\epsilon(s) \Bigg[ -\epsilon^4 \Big(\epsilon \frac{3\kappa}{\beta}(x+f) + \epsilon^2 {\tilde a}_2 \Big) \frac{ 4\pi^2} {{\tilde\ell}^2} \frac{\kappa }{ \beta^2} (x+f) f'' e \mathcal Z' \Bigg] \mathcal Z \,{\mathrm d}x + O(e^{-c\frac{|\beta \delta|} \epsilon} )f''}_{: = \frac{\epsilon^5}{\beta}\tilde{\mathcal K}_{10} \left[\tau, \beta f, \frac{f''}{\tilde \ell^2}\right]}
\nonumber\\[2mm]
 =&
- \epsilon^4 \frac{8\pi^2} {{\tilde\ell}^2} \frac{\kappa}{\, \beta^2\, } e'' f + \frac{\epsilon^3}{\beta}\tilde{\mathcal K}_9 \left[\tau, \beta f, \frac{\epsilon^2 e''}{\tilde \ell^2}\right]
- \epsilon^4\frac{4\pi^2} {{\tilde\ell}^2} \frac{\kappa }{\, \beta^2\, } f'' e+ \frac{\epsilon^5}{\beta}\tilde{\mathcal K}_{10} \left[\tau, \beta f, \frac{\beta f''}{\tilde \ell^2}\right],
\end{align}
where we have used the fact
\begin{equation*}
\int_{\mathbb R}{\mathcal Z}^2 \,{\mathrm d}x\,=\,1,
\qquad
\int_{\mathbb R} x{\mathcal Z}\,{\mathcal Z}'\,{\mathrm d}x\,=\,-\frac{1}{2}.
\end{equation*}

\smallskip\noindent$\clubsuit$
 We note that the expressions of the following terms
$$
 \frac{\epsilon^4}{\beta^2} \mathcal F_1\left[\tau, x, \beta f, \frac{\beta f'}{\tilde\ell} \right],
  \qquad
\frac{\epsilon^4}{\beta^2} \hat{\mathcal F}_1\left[\tau, x, \beta f, \frac{\beta f''}{\tilde\ell^2} \right],
 \qquad
 \frac{\epsilon^4} {\beta^2} \mathcal G_1\left[\tau, x, \beta f, \frac{\beta f'}{\tilde\ell}, e, \frac{e'}{\tilde\ell}\right],
$$
have been given in \eqref{mathcal F1}, \eqref{hat mathcal F1}, \eqref{mathcal G1}. We denote
\begin{align*}
& \int_{\mathbb R} \eta_\delta^\epsilon(s) \frac{\epsilon^4} {\beta^2} \mathcal G_1\left[\tau, x, \beta f, \frac{\beta f'}{\tilde\ell}, e, \frac{e'}{\tilde\ell}\right] \mathcal Z \,{\mathrm d}x
 = \frac{\epsilon^3}{\beta}\tilde{\mathcal K}_{11}\left[\tau, \beta f, \frac{\beta {f}'}{\tilde\ell}\right],
\end{align*}
\begin{align*}
 \int_{\mathbb R} \eta_\delta^\epsilon(s) \frac{\epsilon^4}{\beta^2} \mathcal F_1\left[\tau, x, \beta f, \frac{\beta f'}{\tilde\ell} \right] \mathcal Z \,{\mathrm d}x
 = \frac{\epsilon^4}{\beta}\tilde{\mathcal K}_{15}\left[\tau, \beta f, \frac{\beta {f}'}{\tilde\ell}\right],
\end{align*}
 and
\begin{align*}
 & \int_{\mathbb R} \eta_\delta^\epsilon(s) \Bigg\{ \frac{\epsilon^4}{\beta^2} \hat{\mathcal F}_1\left[\tau, x, \beta f, \frac{\beta f''}{\tilde\ell^2} \right] \Bigg\} \mathcal Z \,{\mathrm d}x
 = \frac{\epsilon^4}{\beta}\tilde{\mathcal K}_{12}\left[\tau, \beta f, \frac{\beta f''}{\tilde\ell^2}\right].
\end{align*}
Therefore, we can derive that
\begin{align}\label{ESTIMATEII16}
II_{16}\,=\,\frac{\epsilon^3}{\beta}\tilde{\mathcal K}_{11}\left[\tau, \beta f, \frac{\beta {f}'}{\tilde\ell}\right]\,+\, \frac{\epsilon^4}{\beta}\tilde{\mathcal K}_{15}\left[\tau, \beta f, \frac{\beta {f}'}{\tilde\ell}\right]
\,+\,\frac{\epsilon^4}{\beta}\tilde{\mathcal K}_{12}\left[\tau, \beta f, \frac{\beta f''}{\tilde\ell^2}\right].
\end{align}

\smallskip\noindent$\clubsuit$
At last, by the decompositions \eqref{expression of B6phi1}, \eqref{decomposition of B6phi2} and the estimate \eqref{estimate of mathcal F3-x}, we have
\begin{align} \label{B6phi1 + B6phi2-intergral Z}
II_{17}
=&\int_{\mathbb R} \eta_\delta^\epsilon(s) \Bigg\{ B_{6}(\epsilon\phi_1)
 + B_6 (\epsilon^2{\phi_2}) \Bigg\} \,{\mathcal Z} \,{\mathrm d}x
 \nonumber\\[2mm]
& = \int_{\mathbb R} \eta_\delta^\epsilon(s) \Bigg\{ \frac{\epsilon^4}{\beta^2} \mathcal F_2\left[\tau, x, f, \frac{\beta f'}{\tilde\ell}\right]+ \frac{\epsilon^4}{\beta^2} \hat{\mathcal F}_2\left[\tau, x, \beta f, \frac{\beta f''}{\tilde\ell^2} \right] \Bigg\} \mathcal Z \,{\mathrm d}x
\nonumber\\[2mm]
& \quad+ \int_{\mathbb R} \eta_\delta^\epsilon(s) \Bigg\{ -\epsilon^5 \kappa a_2(\epsilon s, \theta) \frac{ 4\pi^2} {{\tilde\ell}^2}\frac{1}{\beta} (x+f)f'' \phi_{2, x}+ \frac{\epsilon^5}{\beta^3} \mathcal F_3\left[\tau, x, \beta f, \frac{\beta f'}{\tilde\ell}\right] \Bigg\} \mathcal Z \,{\mathrm d}x
\nonumber\\[2mm]
& = \frac{\epsilon^4}{\beta}\tilde{\mathcal K}_{13}\left[\tau, \beta f, \frac{\beta f''}{\tilde\ell^2}\right] + \frac{\epsilon^4}{\beta}\tilde{\mathcal K}_{14}\left[\tau, \beta f, \frac{\beta f'}{\tilde\ell}\right],
\end{align}
 where
\begin{align*}
\frac{\epsilon^4}{\beta}\tilde{\mathcal K}_{13}\left[\tau, \beta f, \frac{\beta f''}{\tilde\ell^2}\right]
= \int_{\mathbb R} \eta_\delta^\epsilon(s) \Bigg\{ -\epsilon^5 \kappa a_2(\epsilon s, \theta) \frac{ 4\pi^2} {{\tilde\ell}^2}\frac{1}{\beta} (x+f)f'' \phi_{2, x}+ \frac{\epsilon^4}{\beta^2} \hat{\mathcal F}_2\left[\tau, x, \beta f, \frac{\beta f''}{\tilde\ell^2} \right] \Bigg\} \mathcal Z \,{\mathrm d}x,
\end{align*}
 and
\begin{align*}
\frac{\epsilon^4}{\beta}\tilde{\mathcal K}_{14}\left[\tau, \beta f, \frac{\beta f'}{\tilde\ell}\right]
= \int_{\mathbb R} \eta_\delta^\epsilon(s) \Bigg\{ \frac{\epsilon^4}{\beta^2} \mathcal F_2\left[\tau, x, f, \frac{\beta f'}{\tilde\ell}\right] + \frac{\epsilon^5}{\beta^3} \mathcal F_3\left[\tau, x, \beta f, \frac{\beta f'}{\tilde\ell}\right] \Bigg\} \mathcal Z \,{\mathrm d}x.
\end{align*}

\medskip
Finally, we combine the estimates in \eqref{projected eqs for enew}-\eqref{B6phi1 + B6phi2-intergral Z}
to conclude that
\begin{align} \label{Sw2 integral 2}
 \frac{1}{\beta^3}\int_{\mathbb R}\eta_\delta^\epsilon(s){\mathcal E}\,{\mathcal Z} \,{\mathrm d}x
 & = \frac {1}\beta \Bigg\{\epsilon^3\frac{4\pi^2}{{\tilde\ell}^2} e'' + \epsilon\lambda_0 e
 -\frac{1}{2}\, \epsilon^2 \frac{ \epsilon^2W_{tt}(0, \theta)}{\beta^3}|f|^2 b_5
\nonumber\\[2mm]
 & \qquad \quad+ \epsilon^2\beta\Bigg[\frac{4\pi^2}{{\tilde\ell}^2}|f'|^2
-\frac{4\pi}{{\tilde\ell}}\frac{\beta'}{\beta^2}ff'
+\frac{|\beta'|^2}{\beta^4} |f|^2
\Bigg] b_6
  +\frac 23\epsilon^2 \frac{\kappa}{\beta} fe
 + 3\epsilon^2 \frac{1}{\beta}e^2b_1
 \nonumber\\[2mm]
 & \qquad \quad
  + \beta\frac{4}{9} \epsilon^2 f^2 \frac{\kappa^2}{\beta^2} b_3
 - \epsilon^4 \frac{8\pi^2} {{\tilde\ell}^2} \frac{\kappa}{\beta} e'' f
 + \epsilon^3 \tilde{\mathcal K}^* \left[\tau, \beta f, \frac{\beta {f}'}{\tilde\ell}, e, \frac{\epsilon e'}{\tilde\ell}\right]
  \nonumber\\[2mm]
 & \qquad \quad
 + \epsilon^3 \tilde{\mathcal K}^{**} \left[\tau, \beta f, \frac{\beta f''}{\tilde\ell^2}, \frac{\epsilon^2 e''}{\tilde \ell^2}\right]
 + \epsilon^4 \tilde{\mathcal K}^{***} \left[\tau, \beta f, \frac{\beta {f}'}{\tilde\ell}, e, \frac{\epsilon e'}{\tilde\ell}\right] \Bigg\},
\end{align}
where
\begin{align*}
\epsilon^3 \tilde{\mathcal K}^* \left[\tau, \beta f, \frac{\beta {f}'}{\tilde\ell}, e, \frac{\epsilon e'}{\tilde\ell}\right]
\,=\, & \epsilon^3\Bigg\{\tilde{\mathcal K}_1\left[\tau, {\beta f}, \frac{\beta{ f}'}{\tilde\ell} \right]
+ \tilde{\mathcal K}_2\left[\tau, {\beta f}, \frac{\beta { f}'}{\tilde\ell}\right] e
\nonumber\\[2mm]
 & \qquad
 + \frac{e^3}{\beta^2}b_2
 + \tilde{\mathcal K}_3\left[\tau, e, {\beta f}\right]
 + \tilde{\mathcal K}_4\left[\tau, e, {\beta f}\right]
\nonumber\\[2mm]
 &  \qquad
 + \tilde{\mathcal K}_5\left[\tau, \beta{ f}, \frac{\beta{f}'}{\tilde\ell}\right]
 + \tilde{\mathcal K}_6\left[\tau, \beta{ f}, \frac{\beta{f}'}{\tilde\ell}\right]
 + \frac{2}{3} \kappa \tilde {D}_9 (\theta) f
\nonumber\\[2mm]
 &  \qquad
 + \tilde{\mathcal K}_8\left[\tau, e, {\beta f}\right]
 +\tilde{\mathcal K}_{11}\left[\tau, \beta f, \frac{\beta {f}'}{\tilde\ell}, e, \frac{\epsilon e'}{\tilde\ell}\right]
 \Bigg\},
\end{align*}
\begin{align*}
\epsilon^3 \tilde{\mathcal K}^{**} \left[\tau, \beta f, \frac{\beta f''}{\tilde\ell^2}, e,  \frac{\epsilon^2 e''}{\tilde \ell^2}\right]
\,=\, &
\epsilon^3\Bigg\{
  -\frac{8\pi^2}{{\tilde\ell}^2} \kappa b_4  f''
+ \frac{8\pi^2}{3\tilde \ell^2} \kappa b_3f''
 - \,\frac{4\pi^2}{{\tilde\ell}^2}\kappa b_7\,f''
- \epsilon \frac{ 4\pi^2} {{\tilde\ell}^2} \kappa  f'' e
\nonumber\\[2mm]
 & \qquad
  + {\epsilon^3} \tilde{\mathcal K}_9 \left[\tau, \beta f, \frac{\epsilon^2 e''}{\tilde \ell^2}\right]+ \epsilon^5 \tilde{\mathcal K}_{10} \left[\tau, \beta f, \frac{f''}{\tilde \ell^2}\right]
  \nonumber\\[2mm]
 & \qquad + \epsilon^3\tilde{\mathcal K}_{12}\left[\tau, \beta f, \frac{\beta f''}{\tilde\ell^2}\right]
 +\epsilon^4\tilde{\mathcal K}_{13}\left[\tau, \beta f, \frac{\beta f''}{\tilde\ell^2}\right]
 \Bigg\},
\end{align*}
and
\begin{align*}
\epsilon^4 \tilde{\mathcal K}^{***} \left[\tau, \beta f, \frac{\beta {f}'}{\tilde\ell}, e, \frac{\epsilon e'}{\tilde\ell}\right]
 = \epsilon^4 \left\{\tilde{\mathcal D}_7(\theta)
 +\tilde {D}_8 (\theta)
 + \tilde{\mathcal K}_7\left[\tau, \beta{ f}, \frac{\beta{f}'}{\tilde\ell}\right]
 + \tilde{\mathcal K}_{14}\left[\tau, \beta f, \frac{\beta f'}{\tilde\ell}\right]
 + \tilde{\mathcal K}_{15}\left[\tau, \beta f, \frac{\beta {f}'}{\tilde\ell}\right] \right\}.
\end{align*}

Moreover, we have
\begin{align*}
\left\|\epsilon^3 \tilde{\mathcal K}^* \left[\tau, \beta f, \frac{\beta {f}'}{\tilde\ell}, e, \frac{\epsilon e'}{\tilde\ell}\right] \right\|_{L^2(0, 2\pi)} \le C \epsilon^3 \big( \|f\|_* + \|e\|_{**} \big),
\end{align*}
\begin{align*}
\left\| \epsilon^3 \tilde{\mathcal K}^{**} \left[\tau, \beta f, \frac{\beta f''}{\tilde\ell^2}, e, \frac{\epsilon^2 e''}{\tilde \ell^2}\right] \right\|_{L^2(0, 2\pi)} \le C \epsilon^3 \big( \|f\|_* + \|e\|_{**} \big),
\end{align*}
\begin{align*}
\left\|\epsilon^4\tilde{\mathcal K}^{***} \left[\tau, \beta f, \frac{\beta {f}'}{\tilde\ell}, e, \frac{\epsilon e'}{\tilde\ell}\right] \right\|_{L^2(0, 2\pi)} \le C \big( \epsilon^4+ \epsilon^4\|f\|_* + \epsilon^4 \|e\|_{**} \big),
\end{align*}
\begin{align*}
& \left\|\epsilon^3 \tilde{\mathcal K}^* \left[\tau, \beta f_1, \frac{\beta {f_1}'}{\tilde\ell}, e_1, \frac{\epsilon e_1'}{\tilde\ell}\right]
\ -\ \epsilon^3\tilde{\mathcal K}^* \left[\tau, \beta f_2, \frac{\beta {f_2}'}{\tilde\ell}, e_2, \frac{\epsilon e_2'}{\tilde\ell}\right]\right\|_{L^2(0, 2\pi)}
\nonumber\\[2mm]
 &
 \le C \epsilon^3 \big(\|f_1- f_2\|_* + \|e_1- e_2\|_{**} \big),
\end{align*}
and
\begin{align*}
\left\|
\epsilon^3 \tilde{\mathcal K}^{**} \left[\tau, \beta f_1, \frac{\beta f_1''}{\tilde\ell^2}, e_1, \frac{\epsilon^2 e_1''}{\tilde \ell^2}\right]
\ -\
\epsilon^3 \tilde{\mathcal K}^{**} \left[\tau, \beta f_2, \frac{\beta f_2''}{\tilde\ell^2}, e_2, \frac{\epsilon^2 e_2''}{\tilde \ell^2}\right]
\right\|_{L^2(0, 2\pi)}
\le
C \epsilon^3 \big(\|f_1- f_2\|_* + \|e_1- e_2\|_{**} \big).
\end{align*}

\medskip
\subsection{ Projections of terms involving the perturbation}\label{Section7.2}\

%
%

\medskip
Recall the norm $\|\cdot\|_{H^2_*(\mathfrak S)} $ given in \eqref{norm h2 mathfrak S}.
Let
\begin{align} \label{Pi1}
\Pi_1(\tau, f, e) = \int_{{\mathbb R}} L_2(\phi^*) U'(x)\,{\mathrm d}x,
\end{align}
\begin{align} \label{Pi2}
\Pi_2(\tau, f, e) = \int_{{\mathbb R}} \chi(\epsilon|s|) \Big[B_4(\phi^*)+ B_5(\phi^*)+ B_6(\phi^*)\Big] U'(x)\,{\mathrm d}x,
\end{align}
and
\begin{align} \label{Pi3}
\Pi_3(\tau, f, e) = -\int_{{\mathbb R}}\chi(\epsilon|s|) {\mathcal N}(\phi^*) U'(x)\,{\mathrm d}x.
\end{align}

 We recall that
\begin{align*}
\epsilon\phi_1(x, \tau)
\,=\, \epsilon{\phi_{11}}(x, \tau) \ +\ \epsilon{\phi_{12}}(x, \tau)
= \epsilon{\mathbf b}_1(\theta)\varpi_1(x)
\ +\
\epsilon\frac{2}{3} {\mathbf b}_1(\theta)f(\tau) \varpi_2(x),
\end{align*}
the expression of $\epsilon^2\phi_2$ given in \eqref{definition of e2phi2}-\eqref{phi21-phi210},
 and there hold
\begin{align*}
 |\epsilon\phi_1| \leq \epsilon \frac{C}{\beta} e^{-c|x|}, \quad |\epsilon^2\phi_2| \leq \epsilon^2 \frac{C}{\beta^2} e^{-c|x|}.
\end{align*}
 For $\Pi_1(\tau, f, e) $, we have
\begin{align*}
 \Pi_1 (\tau, f, e) & = \int_{{\mathbb R}} L_2(\phi^*) U'(x)\,{\mathrm d}x
\nonumber\\[2mm]
 & =
\int_{{\mathbb R}}\left[ \frac{\epsilon^2 4\pi^2}{{\tilde\ell}^2} \phi^*_{\tau\tau} + \phi^*_{xx} - \phi+ 3w_2^2 \phi^* \right] U'(x)\,{\mathrm d}x
\nonumber\\[2mm]
 & =\int_{{\mathbb R}}\left[\phi^*_{xx} - \phi+ 3U^2 \phi^* + 3\left(w_2^2 - U^2 \right) \phi^* \right] U'(x)\,{\mathrm d}x
\nonumber\\[2mm]
 & =\int_{{\mathbb R}} 3\left(w_2^2 - U^2 \right) \phi^* U'(x)\,{\mathrm d}x
\nonumber\\[2mm]
 & =\int_{{\mathbb R}} 3\left[2\, U \Big( \frac1\beta\epsilon e \mathcal{Z}+\epsilon \phi_1+ \epsilon^2\phi_2 \Big)
\ +\
\Big(\frac1\beta\epsilon e \mathcal{Z}+\epsilon \phi_1+ \epsilon^2\phi_2 \Big)^2\right] \phi^* U'(x)\,{\mathrm d}x
\nonumber\\[2mm]
 & = \underbrace{ \frac6\beta \epsilon e \int_{{\mathbb R}} U\mathcal Z \phi^* U'(x)\,{\mathrm d}x + \frac3{\beta^2}\epsilon^2 e^2 \int_{{\mathbb R}}{\mathcal Z}^2 \phi^* U'(x)\,{\mathrm d}x}_{\le C\frac\epsilon\beta \left(|e| + |f| +1\right) \|\phi^*(\cdot, \tau)\|_{L^2({\mathbb R})} }
\nonumber\\[2mm]
 & \quad+\underbrace{ \int_{{\mathbb R}} \left[\,2\, U \Big(\epsilon \phi_1+ \epsilon^2\phi_2 \Big)
 \ +\ \frac{2 \epsilon e \mathcal{Z}} \beta\Big(\epsilon \phi_1+ \epsilon^2\phi_2 \Big)
 \ +\ \Big(\epsilon \phi_1+ \epsilon^2\phi_2 \Big)^2 \,\right] \phi^* U'(x)\,{\mathrm d}x}_{\le C\frac{\epsilon}{\beta} \left( |e| + |f| + 1 \right) \|\phi^*(\cdot, \tau)\|_{L^2({\mathbb R})}}.
\end{align*}

 Next, we consider $\Pi_2(\tau, f, e)$.
 By the expression of operators $B_4, B_5, B_6$ given in \eqref{new opreator B4}, \eqref{new opreator B5}, \eqref{B2w-2222new},
 we have the following formula
\begin{align*}
&\Pi_2(\tau, f, e)
\nonumber\\[2mm]
&\,=\, \int_{{\mathbb R}} \chi(\epsilon|s|) \big[B_4(\phi^*)+ B_5(\phi^*)+ B_6(\phi^*)\big] U'(x)\,{\mathrm d}x
\nonumber\\[2mm]
&\,=\,\epsilon \frac{\kappa}{\, \beta\, }\int_{\mathbb R} \chi(\epsilon|s|)\Big[ \phi^*_x +\frac{2}{3} (x+f)\phi^*\Big] U'(x)\,{\mathrm d}x
\nonumber\\[2mm]
&\quad
+ \epsilon^2\frac{4\pi^2}{{\tilde\ell}^2}\int_{\mathbb R} \chi(\epsilon|s|)\Big[|f'|^2 \phi^*_{xx} -f''\phi^*_x -2f' \phi^*_{x\tau}\Big]U'(x)\,{\mathrm d}x
\ +\
\epsilon^2\frac{6\pi}{{\tilde\ell}}\frac{\beta'} {\beta^2} \int_{\mathbb R} \chi(\epsilon|s|)\Big[- f' \phi^*_x +\phi^*_\tau\Big]U'(x)\,{\mathrm d}x
\nonumber\\[2mm]
&\quad
+ \epsilon^2 \frac{4\pi}{{\tilde\ell}} \frac{\beta'}{\beta^2} \int_{\mathbb R} \chi(\epsilon|s|)(x+f)\big[- f' \phi^*_{xx} +\phi^*_{x\tau}\big]U'(x)\,{\mathrm d}x
\nonumber\\[2mm]
&\quad\ +\
\epsilon^2\Big[\frac{\beta''}{\beta^3} +2\frac{|\beta'|^2}{\beta^4}\,-\,\frac{\kappa^2}{\beta^2} \Big]
\int_{\mathbb R} \chi(\epsilon|s|)(x+f) \phi^*_{x} U'(x)\,{\mathrm d}x
\nonumber\\[2mm]
&\quad
+ \epsilon^2 \frac{|\beta'|^2}{\beta^4} \int_{\mathbb R} \chi(\epsilon|s|)(x+f)^2 \phi^*_{xx} U'(x)\,{\mathrm d}x
\ -\
\frac{1}{2}\, \epsilon^2\frac{\epsilon^2W_{tt}(0, \theta)}{\beta^4} \int_{\mathbb R} \chi(\epsilon|s|)(x+f)^2 \phi^* U'(x)\,{\mathrm d}x
\nonumber\\[2mm]
&\quad
- \frac{\epsilon^3}{6} \frac{\, \, \epsilon^2 W_{ttt}(0, \theta)}{\beta^5} \int_{\mathbb R} \chi(\epsilon|s|)(x+f)^3 \phi^*U'(x)\,{\mathrm d}x
\ -\
\frac{\epsilon^4}{24} \frac{1}{\beta^6}\, \int_{\mathbb R} \chi(\epsilon|s|) \epsilon^2 W_{tttt}\big(\hbar(\theta)\epsilon s, \, \theta\big)(x+f)^4\phi^* U'(x)\,{\mathrm d}x
\nonumber\\[2mm]
&\quad
 +\epsilon^3 \kappa^3  \frac{1}{\, \beta^3\, }\int_{\mathbb R} \chi(\epsilon|s|)a_3(\epsilon s, \theta)(x+f)^2\phi^*_{x} U'(x)\,{\mathrm d}x
\nonumber\\[2mm]
&\quad
\,+\,
\epsilon^3 \kappa
\Bigg\{
\frac{\beta''} {\beta^4} \int_{\mathbb R} \chi(\epsilon|s|)a_2(\epsilon s, \theta)(x+f) \phi^*U'(x)\,{\mathrm d}x
+\frac{2|\beta'|^2}{\beta^5} \int_{\mathbb R} \chi(\epsilon|s|)a_2(\epsilon s, \theta)(x+f)^2 \phi^*_xU'(x)\,{\mathrm d}x
\nonumber\\[2mm]
& \qquad\qquad
+\frac{ 6\pi}{ {\tilde\ell}}\frac{\beta'} {\beta^3}
\int_{\mathbb R} \chi(\epsilon|s|)a_2(\epsilon s, \theta)(x+f)\big[-f'\phi^*_x +\phi^*_\tau\big]U'(x)\,{\mathrm d}x
\nonumber\\[2mm]
&\qquad\qquad
+\frac{ 4\pi}{ {\tilde\ell}}\frac{\beta'} {\beta^3}
\int_{\mathbb R} \chi(\epsilon|s|)a_2(\epsilon s, \theta)(x+f)^2[- f' \phi^*_{xx} + \phi^*_{x\tau}]U'(x)\,{\mathrm d}x
\nonumber\\[2mm]
& \qquad\qquad
+ \frac{|\beta'|^2}{\beta^5} \int_{\mathbb R} \chi(\epsilon|s|)a_2(\epsilon s, \theta)(x+f)^3 \phi^*_{xx}U'(x)\,{\mathrm d}x
+ \frac{\beta''}{\beta^4}\int_{\mathbb R} \chi(\epsilon|s|)a_2(\epsilon s, \theta)(x+f)^2 \phi^*_{x}U'(x)\,{\mathrm d}x
\nonumber\\[2mm]
&\qquad\qquad
+\frac{ 4\pi^2} {{\tilde\ell}^2}\frac{1}{\beta}
\int_{\mathbb R} \chi(\epsilon|s|)a_2(\epsilon s, \theta)(x+f)\big[ |f'|^2 \phi^*_{xx} -f''\phi^*_x +\phi^*_{\tau\tau} -2f' \phi^*_{x\tau}\big]U'(x)\,{\mathrm d}x
\Bigg\}
\nonumber\\[2mm]
&\quad
\,+\,
\epsilon^3 \kappa'
\Bigg\{
\frac{\beta'}{\beta^4} \int_{\mathbb R} \chi(\epsilon|s|)a_1(\epsilon s, \theta)(x+f)\phi^*U'(x)\,{\mathrm d}x
+\frac{\beta'}{\beta^4} \int_{\mathbb R} \chi(\epsilon|s|)a_1(\epsilon s, \theta)(x+f)^2 \phi^*_{x}U'(x)\,{\mathrm d}x
\nonumber\\[2mm]
&\qquad\qquad
+\frac{ 2\pi}{ {\tilde\ell}}\frac{1}{ \beta^2}\int_{\mathbb R} \chi(\epsilon|s|)a_1(\epsilon s, \theta)(x+f)\big[-f'\phi^*_x +\phi^*_\tau \big]U'(x)\,{\mathrm d}x
\Bigg\}.
\end{align*}
We select the less regular terms from $\Pi_{2}(\tau, f, e) $
 as $\Pi_{2*}(\tau, f, e) $, where
\begin{align}\label{Pi2star}
\Pi_{2*}(\tau, f, e)
=& -f''\epsilon^2\frac{4\pi^2}{{\tilde\ell}^2}\int_{\mathbb R}\, \phi^*_x U'(x)\,{\mathrm d}x
+ \epsilon^3 \kappa  \frac{4\pi^2}{{\tilde\ell}^2}\frac{1}{\beta} \int_{\mathbb R} a_2(\epsilon s, \theta)(x+f)\Big[ -f''\phi^*_x +\phi^*_{\tau\tau} \Big] U'(x)\,{\mathrm d}x.
\end{align}
Moreover, we have
\begin{align*}
\Big|\Pi_{2}(\tau, f, e) - \Pi_{2*}(\tau, f, e)\Big|
& \le \epsilon \frac1{\beta} \left\{ 1+ |\beta f| + \Big|\frac{\beta  f'} {\tilde \ell}\Big| \right\}
 \times
\Bigg[\|\phi^*(\cdot, \tau)\|_{L^2({\mathbb R})} + \|\phi^*_x(\cdot, \tau)\|_{L^2({\mathbb R})}
+ \|\phi^*_{xx}(\cdot, \tau)\|_{L^2({\mathbb R})}
\nonumber\\[2mm]
& \qquad\qquad \qquad\qquad \qquad\qquad \qquad
+ \frac{\epsilon}{\tilde\ell}\|\phi^*_\tau(\cdot, \tau)\|_{L^2({\mathbb R})}+ \frac{\epsilon}{\tilde\ell}\|\phi^*_{x\tau}(\cdot, \tau)\|_{L^2({\mathbb R})} \Bigg],
\end{align*}
and
\begin{align} \label{minus of Pi2star}
&\Pi_{2*} (\tau, f_1, e_1) - \Pi_{2*} (\tau, f_2, e_2)
\nonumber\\[2mm]
& =
 - \epsilon^2 f_1''\frac{4\pi^2}{{\tilde\ell}^2}\int_{\mathbb R}\,\chi(\epsilon|s|) \phi^*_x(f_1, e_1) U'(x)\,{\mathrm d}x
\ +\
\epsilon^2 f_2''\frac{4\pi^2}{{\tilde\ell}^2}\int_{\mathbb R}\, \chi(\epsilon|s|) \phi^*_x(f_2, e_2) U'(x)\,{\mathrm d}x
\nonumber\\[2mm]
&\quad + \epsilon^3 \kappa  \frac{4\pi^2}{{\tilde\ell}^2}\frac{1}{\beta} \int_{\mathbb R}\, \chi(\epsilon|s|) a_2(\epsilon s, \theta)(x+f_1)\Big[ -f_1''\phi^*_x(f_1, e_1) +\phi^*_{\tau\tau}(f_1, e_1) \Big] U'(x)\,{\mathrm d}x
\nonumber\\[2mm]
&\quad  -\epsilon^3 \kappa  \frac{4\pi^2}{{\tilde\ell}^2}\frac{1}{\beta} \int_{\mathbb R}\, \chi(\epsilon|s|) a_2(\epsilon s, \theta)(x+f_2)\Big[ -f_2''\phi^*_x(f_2, e_2)+\phi^*_{\tau\tau}(f_2, e_2) \Big] U'(x)\,{\mathrm d}x,
\end{align}
which implies
\begin{align}\label{minus of Pi2star1}
&\Big|\Pi_{2*} (\tau, f_1, e_1) - \Pi_{2*} (\tau, f_2, e_2)\Big|
\nonumber\\[2mm]
& \le \epsilon^2\big|f_1''-f_2''\big|\frac{4\pi^2}{{\tilde\ell}^2}
\int_{\mathbb R}\,  \big|\phi^*_x(f_1, e_1)\big| |U'(x)|\,{\mathrm d}x
\ +\
\epsilon^2|f_2''| \frac{4\pi^2}{{\tilde\ell}^2}
\int_{\mathbb R}\, \Big|\phi^*_x(f_1, e_1) - \phi^*_x(f_2, e_2)\Big| \big|U'(x)\big|\,{\mathrm d}x
\nonumber\\[2mm]
&\quad + \epsilon^3 \frac{4\pi^2}{{\tilde\ell}^2}\frac{1}{\beta}
\big|f_1''-f_2''\big|
\int_{\mathbb R}\, x\big|\phi^*_x(f_1, e_1)\big|U'(x)\,{\mathrm d}x
\ +\
\epsilon^3 \frac{4\pi^2}{{\tilde\ell}^2}\frac{1}{\beta}
\big|f_2''\big| \int_{\mathbb R}\, x\big|\phi^*_x(f_1, e_1)- \phi^*_x(f_2, e_2)\big||U'(x)|\,{\mathrm d}x
\nonumber\\[2mm]
&\quad + \epsilon^3 \frac{4\pi^2}{{\tilde\ell}^2}\frac{1}{\beta}
 \int_{\mathbb R}\, x\big|\phi^*_{\tau\tau}(f_1, e_1)- \phi^*_{\tau\tau}(f_2, e_2)\big| |U'(x)|\,{\mathrm d}x
\nonumber\\[2mm]
&\quad + \epsilon^3 \frac{4\pi^2}{{\tilde\ell}^2}\frac{1}{\beta} \Big[\big|f_1- f_2\big|| f_1''| +|f_2| |f_1''- f_2''|\Big]\int_{\mathbb R}\, \big|\phi^*_x(f_1, e_1)\big||U'(x)|\,{\mathrm d}x
\nonumber\\[2mm]
&\quad + \epsilon^3 \frac{4\pi^2}{{\tilde\ell}^2}\frac{1}{\beta}  |f_2||f_2''| \int_{\mathbb R}\, \Big(\big|\phi^*_x(f_1, e_1)-\phi^*_x(f_2, e_2)\big| \Big)|U'(x)|\,{\mathrm d}x
\nonumber\\[2mm]
&\quad + \epsilon^3 \frac{4\pi^2}{{\tilde\ell}^2}\frac{1}{\beta}\big|f_1-f_2\big|
 \int_{\mathbb R}\,\big|\phi^*_{\tau\tau}(f_1, e_1)\big| |U'(x)|\,{\mathrm d}x
\nonumber\\[2mm]
&\quad + \epsilon^3 \frac{4\pi^2}{{\tilde\ell}^2}\frac{1}{\beta}\big|f_2\big|
 \int_{\mathbb R}\, \big|\phi^*_{\tau\tau}(f_1, e_1)- \phi^*_{\tau\tau}(f_2, e_2)\big| |U'(x)|\,{\mathrm d}x.
\end{align}

\medskip
 For  $\Pi_3(\tau, f, e)$, we have
\begin{align}
\Big|\Pi_3(\tau, f, e) \Big|\,=\,&\Big|-\int_{{\mathbb R}}\chi(\epsilon|s|){\mathcal N}(\phi^*) U'(x)\,{\mathrm d}x\Big|
\nonumber\\[2mm]
\,\leq\,& \frac{C}{\beta^3} \displaystyle\int_{{\mathbb R}}\chi(\epsilon|s|)
\left[3\,|\mathbf{U}| \left(
|{\tilde\phi}|^2+|{\tilde\psi}|^2\right)+\left(
|{\tilde\phi}|^3 +|{\tilde\psi}|^3 \right)
+3\, \mathbf{U}^2 \,| {\tilde\psi}(\tilde \phi)|\right]\,|U'(x)| \,{\mathrm d}x,
\end{align}
with $\mathbf{U}(y)$ and $\tilde\phi(y)$ as in \eqref{defofbfU} and \eqref{relation betwen tildephi and phi}.

\medskip
We will derive the estimates of $\Pi_1$, $\Pi_2$ and $\Pi_3$.
 We note
\begin{align} \label{inequalities for f}
\|\beta f\|_{L^\infty(0, 2\pi)} \le\, C\, \tilde\ell^{\frac12} \|f\|_*,
\qquad
\Big\|\frac{\beta f'}{\tilde \ell}\Big\|_{L^\infty(0, 2\pi)} \le \, C\,\tilde\ell^{\frac12} \|f\|_*,
\end{align}
and
\begin{align} \label{7-89}
\sup_{\tau \in (0, 2\pi)} \|\phi^*(\cdot, \tau)\|^2_{L^2({\mathbb R})} \le \frac{\tilde \ell} {\epsilon} \|\phi^*\|_{H^2_*(\mathfrak S)}^2, \quad
\sup_{\tau \in (0, 2\pi)} \|\phi_x^*(\cdot, \tau)\|^2_{L^2({\mathbb R})} \le \frac{\tilde \ell} {\epsilon} \|\phi^*\|_{H^2_*(\mathfrak S)}^2.
\end{align}
The proof of \eqref{7-89} can be found in Appendix \ref{appendixa1}.

\smallskip\noindent$\clubsuit$
We have
\begin{align*}
\int_0^{2\pi} \beta^2\Pi_1^2 \,{\mathrm d}\tau
& \leq C \epsilon^2 \int_0^{2\pi} \left(\frac{1}{\beta^2} \left|\beta f\right|^2+|e|^2 + 1\right) \|\phi^*(\cdot, \tau)\|^2_{L^2({\mathbb R})}{\mathrm d}\tau
\nonumber\\[2mm]
& \leq
C \epsilon^2\left(1+\|e\|_{**}^2+ {\tilde\ell} \|f\|_*^2\right)\|\phi^*\|_{H^2_*(\mathfrak S)}^2.
\end{align*}

\smallskip\noindent$\clubsuit$
Similarly, we have
\begin{align*}
 \int_0^{2\pi} \beta^2\Pi_2^2 \,{\mathrm d}\tau
& \leq C \epsilon^2 \int_0^{2\pi}
 \Bigg[\|\phi^*(\cdot, \tau)\|^2_{L^2({\mathbb R})} + \|\phi_x^*(\cdot, \tau)\|^2_{L^2({\mathbb R})}+\|\phi^*_{xx}(\cdot, \tau)\|^2_{L^2({\mathbb R})}
\nonumber\\[2mm]
& \qquad\qquad\qquad
 + \frac{\epsilon^2}{\tilde\ell^2}\|\phi^*_{\tau}(\cdot, \tau)\|^2_{L^2({\mathbb R})}+ \frac{\epsilon^2}{\tilde\ell^2}\|\phi^*_{x\tau}(\cdot, \tau)\|^2_{L^2({\mathbb R})} \Bigg] \left(\left|f\right|^2 + \left|\frac{{\beta f}'}{\tilde\ell}\right|^2 + 1\right) {\mathrm d}\tau
\nonumber\\[2mm]
& \quad+ \int_0^{2\pi} \beta^2\Pi_{2*}^2 \,{\mathrm d}\tau
\nonumber\\[2mm]
& \leq C
 \epsilon^2\left(1 + {\tilde\ell} \|f\|_*^2\right)\|\phi^*\|_{H^2_*(\mathfrak S)}^2+ \int_0^{2\pi} \beta^2\Pi_{2*}^2 \,{\mathrm d}\tau.
\end{align*}
On the other hand, by \eqref{Pi2star} and \eqref{7-89}, there holds
\begin{align*}
\int_0^{2\pi} \beta^2\Pi_{2*}^2 \,{\mathrm d}\tau
&
\leq \epsilon^4 \int_0^{2\pi} \left|\frac{{\beta f}''}{\tilde\ell^2}\right|^2 \|\phi^*_x(\cdot, \tau)\|^2_{L^2({\mathbb R})}\,{\mathrm d}\tau
\ +\
\epsilon^2 \left(1 + {\tilde\ell} \|f\|_*^2\right) \int_0^{2\pi} \frac{\epsilon^4}{\tilde\ell^4}\|\phi^*_{\tau\tau}\|^2_{L^2({\mathbb R})} \,{\mathrm d}\tau
\nonumber\\[2mm]
&
\leq \epsilon^2\left(1 + {\tilde\ell} \|f\|_*^2\right)\|\phi^*\|_{H^2_*(\mathfrak S)}^2.
\end{align*}
Therefore, we have
\begin{align*}
& \int_0^{2\pi} \beta^2\Pi_2^2 \,{\mathrm d}\tau \leq \epsilon^2\left(1 + {\tilde\ell} \|f\|_*^2\right)\|\phi^*\|_{H^2_*(\mathfrak S)}^2.
\end{align*}
Moreover, by \eqref{minus of Pi2star1} and using the results of Proposition \ref{proposition6.5}, \eqref{inequalities for f} and \eqref{7-89}, we can verify that
\begin{align} \label{minus of Pi2star-in-L2-norm}
&\Big\|\beta\Big(\Pi_{2*} (\tau, f_1, e_1) - \Pi_{2*} (\tau, f_2, e_2)\Big)\Big\|_{L^2(0, 2\pi)}
\nonumber\\[2mm]
& \le C \epsilon^2 \sqrt{{\tilde\ell}/\epsilon\,}\, \big\|\phi^*_{f_1, e_1}\big\|_{H^2_*(\mathfrak S)}\big\|f_1- f_2\big\|_*
+ \epsilon^2\sqrt{{\tilde\ell}/\epsilon\,}\, \big\|f_2\big\|_* \,\big\|\phi^*_{f_1, e_1}- \phi^*_{f_2, e_2} \big\|_{H^2_*(\mathfrak S)}
\nonumber\\[2mm]
& \quad
+\epsilon^3 \sqrt{{\tilde\ell}/\epsilon\,}\, \big\|\phi^*_{f_1, e_1}\big\|_{H^2_*(\mathfrak S)}\big\|f_1- f_2\big\|_*
+ \epsilon^3\sqrt{{\tilde\ell}/\epsilon\,}\, \big\|f_2\big\|_* \,\big\|\phi^*_{f_1, e_1}- \phi^*_{f_2, e_2} \big\|_{H^2_*(\mathfrak S)}
\nonumber\\[2mm]
& \quad
+
\epsilon^3 \sqrt{{\tilde\ell}/\epsilon\,}\, {\tilde \ell}^{\frac 12}\big\|f_1\big\|_* \big\|f_1- f_2\big\|_* \big\|\phi^*_{f_1, e_1}\big\|_{H^2_*(\mathfrak S)}
+
\epsilon^3 \sqrt{{\tilde\ell}/\epsilon\,}\, {\tilde \ell}^{\frac 12}\big\|f_2\big\|_* \big\|f_1- f_2\big\|_* \big\|\phi^*_{f_1, e_1}\big\|_{H^2_*(\mathfrak S)}
\nonumber\\[2mm]
& \quad+ \epsilon^3\sqrt{{\tilde\ell}/\epsilon\,}\, {\tilde \ell}^{\frac 12} \big\|f_2\big\|_*^2\,\big\|\phi^*_{f_1, e_1}- \phi^*_{f_2, e_2} \big\|_{H^2_*(\mathfrak S)}
+\epsilon\big\|\phi^*_{f_1, e_1}- \phi^*_{f_2, e_2} \big\|_{H^2_*(\mathfrak S)}
\nonumber\\[2mm]
&\quad+ \epsilon \tilde\ell^{\frac{1}{2}} \, \big\|\phi^*_{f_1, e_1}\big\|_{H^2_*(\mathfrak S)}\big\|f_1- f_2\big\|_*
+\epsilon \tilde\ell^{\frac{1}{2}}  \big\|f_2\big\|_*\big\|\phi^*_{f_1, e_1}- \phi^*_{f_2, e_2} \big\|_{H^2_*(\mathfrak S)}
\nonumber\\[2mm]
& \le C \epsilon^{3} \left(\big\|f_1- f_2\big\|_*+ \big\|e_1- e_2\big\|_* \right),
\end{align}
where we have used the following facts
\begin{align*}
 {\tilde\ell}^{\frac 12} \|f_j\|_* \ll 1,  \quad j=1,2.
\end{align*}

\smallskip\noindent$\clubsuit$
Next, from the expressions $\mathbf{U}(y)$ and $\tilde\phi(y)$ as in \eqref{defofbfU} and \eqref{relation betwen tildephi and phi}, we have
\begin{align*}
 \int_0^{2\pi} \beta^2\Pi_3^2 \,{\mathrm d}\tau
& \leq
\int_0^{2\pi} \Bigg[\int_{{\mathbb R}} U |\beta\phi^*|^2 |U'(x)| \,{\mathrm d}x\Bigg]^2 \,{\mathrm d}\tau
\ +\
\int_0^{2\pi} \Bigg[\int_{{\mathbb R}} |\beta\phi^*|^3 |U'(x)| \,{\mathrm d}x \Bigg]^2 \,{\mathrm d}\tau
\nonumber\\[2mm]
& \quad + \int_0^{2\pi}\left|\frac{\beta}{\beta^3}\right|^2\left\{\int_{{\mathbb R}}\chi(\epsilon|s|)\left[3\mathbf{U} |{\tilde\psi}|^2+|{\tilde\psi}|^3 +3\, \mathbf{U}^2 \, \big|{\tilde\psi}\big|\right] |U'(x)| \,{\mathrm d}x \right\}^2\,{\mathrm d}\tau.
\end{align*}
By the estimate
\begin{align*}
\|{\tilde\psi}({\tilde\phi})\|_{L^{\infty}}\, & \leq \, C\, \epsilon \lf[\|\tilde\phi\|_{L^{\infty}(|s|>\frac{20\delta}{\epsilon})}\,+\, \|\nabla \tilde\phi\|_{L^{\infty}(|s|>\frac{20\delta}{\epsilon})}\ri]\,+\,
e^{-\frac{ \upsilon_2 \delta}{\epsilon}}
\nonumber\\[2mm]
& \leq
C\left( 1 + \|\beta \phi^*\|_{H^2_*(\mathfrak S)}^2 \right) e^{-\frac{ \upsilon_3 \delta}{\epsilon}},
\end{align*}
we get
\begin{align*}
 \int_0^{2\pi} \left|\frac{\beta}{\beta^3}\right|^2 \left\{\int_{{\mathbb R}}\chi(\epsilon|s|)\left[3\mathbf{U} |{\tilde\psi}|^2+|{\tilde\psi}|^3 +3\, \mathbf{U}^2 \, \big|{\tilde\psi}\big|\right] |U'(x)| \,{\mathrm d}x \right\}^2\,{\mathrm d}\tau
 \leq
 C\left( 1 + \|\beta \phi^*\|_{H^2_*(\mathfrak S)}^2 \right) e^{-\frac{2\upsilon_3 \delta}{\epsilon}}.
\end{align*}
By \eqref{7-89} and  Sobolev inequality, we can directly get
\begin{align} \label{phisuqareestimate}
\int_0^{2\pi} \left|\int_{{\mathbb R}} U |\beta\phi^*|^2 |U'(x)| \,{\mathrm d}x\right|^2 \,{\mathrm d}\tau
&\leq \int_0^{2\pi} \|\beta\phi^* \|_{L^{\infty}(\mathbb R)}^4 \,{\mathrm d}\tau
\nonumber\\[2mm]
&\leq \int_0^{2\pi} \left(\|\beta\phi^* \|_{L^{2}(\mathbb R)}^4+\|(\beta\phi^*)_{x} \|_{L^{2}(\mathbb R)}^4\right) \,{\mathrm d}\tau
\nonumber\\[2mm]
&\leq C \frac{{\tilde\ell}^2}{\epsilon^2}\|\beta\phi^*\|^4_{H^2_*(\mathfrak S)}.
\end{align}
 Similarly, we can prove
\begin{align*}
\int_0^{2\pi} \Bigg[\int_{{\mathbb R}} \left(\beta\phi^* \right)^3U'(x)\,{\mathrm d}x\Bigg]^2 \,{\mathrm d}\tau
&\leq \int_0^{2\pi} \|\beta\phi^* \|_{L^{\infty}(\mathbb R)}^6 \,{\mathrm d}\tau
\nonumber\\[2mm]
&\leq \int_0^{2\pi} \left(\|\beta\phi^* \|_{L^{2}(\mathbb R)}^6+\|(\beta\phi^*)_{x} \|_{L^{2}(\mathbb R)}^6\right) \,{\mathrm d}\tau
\nonumber\\[2mm]
&\leq C \frac{{\tilde\ell}^3}{\epsilon^3} \|\beta\phi^*\|^6_{H^2_*(\mathfrak S)}.
\end{align*}
 Therefore, we have
\begin{align*}
 \int_0^{2\pi} \beta^2\Pi_3^2 \,{\mathrm d}\tau \leq C \frac{{\tilde\ell}^2}{\epsilon^2} \|\beta\phi^* \|^4_{H^2_*(\mathfrak S)} + \frac{{\tilde\ell}^3}{\epsilon^3} \|\beta\phi^*\|^6_{H^2_*(\mathfrak S)}.
\end{align*}

\medskip
Finally, by denoting
\begin{align} \label{hatPi1}
\hat\Pi_1(\tau, f, e) = \int_{{\mathbb R}} L_2(\phi^*) {\mathcal Z} \,{\mathrm d}x,
\end{align}
\begin{align} \label{hatPi2}
\hat\Pi_2(\tau, f, e) = \int_{{\mathbb R}} \chi(\epsilon|s|) \big[B_4(\phi^*)+ B_5(\phi^*)+ B_6(\phi^*)\big]{\mathcal Z} \,{\mathrm d}x,
\end{align}
\begin{align} \label{hatPi3}
\hat\Pi_3(\tau, f, e) = -\int_{{\mathbb R}}\chi(\epsilon|x|) {\mathcal N}(\phi^*)\mathcal Z \,{\mathrm d}x,
\end{align}
we have
\begin{align*}
 \int_0^{2\pi} \beta^2 \hat\Pi_1^2 \,{\mathrm d}\tau \leq \epsilon^2\left(1+\|e\|_{**}^2+ {\tilde\ell} \|f\|_*^2\right)\|\phi^*\|_{H^2_*(\mathfrak S)}^2,
\end{align*}
\begin{align*}
 \int_0^{2\pi} \beta^2 \hat\Pi_2^2 \,{\mathrm d}\tau \leq \epsilon^2\left(1 + {\tilde\ell} \|f\|_*^2\right)\|\phi^*\|_{H^2_*(\mathfrak S)}^2,
\end{align*}
\begin{align*}
 \int_0^{2\pi} \beta^2 \hat\Pi_3^2 \,{\mathrm d}\tau
 \leq
C \frac{{\tilde\ell}^2}{\epsilon^2} \|\beta\phi^*\|^4_{H^2_*(\mathfrak S)}
+ \frac{{\tilde\ell}^3}{\epsilon^3} \|\beta\phi^*\|^6_{H^2_*(\mathfrak S)}.
\end{align*}
Moreover, we select the less regular terms from $\hat\Pi_1, \hat\Pi_2, \hat\Pi_3$ as $\hat\Pi_{2*}$, where
\begin{align}\label{hat-Pi2star}
\hat \Pi_{2*}(\tau, f, e)
=& -f''\epsilon^2\frac{4\pi^2}{{\tilde\ell}^2}\int_{\mathbb R}\, \phi^*_x \mathcal Z\,{\mathrm d}x
+ \epsilon^3 \kappa \frac{4\pi^2}{{\tilde\ell}^2}\frac{1}{\beta} \int_{\mathbb R} a_2(\epsilon s, \theta)(x+f)\big[ -f''\phi^*_x +\phi^*_{\tau\tau} \big] \mathcal Z\,{\mathrm d}x.
\end{align}

We note that for $f$ and $e$ with constraints in \eqref{constraint-f}-\eqref{constraint-e} and satisfy
\begin{align*}
 {\tilde\ell}^{\frac 12} \|f\|_* \ll 1, \quad \|e\|_{**} \ll \epsilon^{1-\alpha}.
\end{align*}
Now we summarize the above calculations and present the following lemma.
\begin{lemma}
 Recall the definitions of $\Pi_j (\tau, f, e)$ and $\hat\Pi_j (\tau, f, e)$ with $j =1,2,3$ given in \eqref{Pi1}-\eqref{Pi3}, \eqref{hatPi1}-\eqref{hatPi3}. We have
\begin{align*}
 \int_0^{2\pi} \beta^2 \Pi_1^2 \,{\mathrm d}\tau &\leq \epsilon^2\left(1+\|e\|_{**}^2+ {\tilde\ell} \|f\|_*^2\right)\|\phi^*\|_{H^2_*(\mathfrak S)}^2
 \le C \epsilon^{8-2\alpha},
\end{align*}
\begin{align*}
 \int_0^{2\pi} \beta^2\Pi_2^2 \,{\mathrm d}\tau & \leq C\epsilon^2\left(1 + {\tilde\ell} \|f\|_*\right)\|\phi^*\|_{H^2_*(\mathfrak S)}^2 \le C \epsilon^{8-2\alpha},
\end{align*}
\begin{align*}
 \int_0^{2\pi} \beta^2 \Pi_3^2 \,{\mathrm d}\tau & \leq C \frac{{\tilde\ell}^2}{\epsilon^2} \|\beta\phi^*\|^4_{H^2_*(\mathfrak S)} + \frac{{\tilde\ell}^3}{\epsilon^3} \|\beta\phi^*\|^6_{H^2_*(\mathfrak S)}
\nonumber\\[2mm]
&
 \le C \epsilon^{2\varrho -4} \epsilon^{12-4\alpha} +C\epsilon^{3\varrho -6} \epsilon^{18-6\alpha}
 \le C \epsilon^{8+2\varrho-4\alpha},
\end{align*}
\begin{align*}
 \int_0^{2\pi} \beta^2 \hat\Pi_1^2 \,{\mathrm d}\tau &\leq \epsilon^2\left(1+\|e\|_{**}^2+ {\tilde\ell} \|f\|_*^2\right)\|\phi^*\|_{H^2_*(\mathfrak S)}^2
 \le C \epsilon^{8-2\alpha},
\end{align*}
\begin{align*}
 \int_0^{2\pi} \beta^2 \hat\Pi_2^2 \,{\mathrm d}\tau &\leq C\epsilon^2\left(1 + {\tilde\ell} \|f\|_*\right)\|\phi^*\|_{H^2_*(\mathfrak S)}^2
 \le \epsilon^{8-2\alpha},
\end{align*}
\begin{align*}
 \int_0^{2\pi} \beta^2 \hat\Pi_3^2 \,{\mathrm d}\tau
\leq
C \frac{{\tilde\ell}^2}{\epsilon^2} \|\beta\phi^*\|^4_{H^2_*(\mathfrak S)}
+ \frac{{\tilde\ell}^3}{\epsilon^3} \|\beta\phi^*\|^6_{H^2_*(\mathfrak S)}
 \le \epsilon^{8+2\varrho-4\alpha}.
\end{align*}
Moreover, the terms $ \Pi(\tau, f, e) = \sum_{j =1}^3\Pi_j(\tau, f, e)$ and $\hat\Pi(\tau, f, e) = \sum_{j=1}^3\hat \Pi_j (\tau, f, e)$ can be decomposed in the following way
\begin{align*}
 \Pi(\tau, f, e) = \Pi^*(\tau, f, e) + \Pi_{2*}(\tau, f, e), \qquad \hat\Pi(\tau, f, e) = \hat\Pi^*(\tau, f, e) + \hat\Pi_{2*}(\tau, f, e).
\end{align*}
In the above, the terms $\Pi^*(\tau, f, e)$ and $\hat\Pi^*(\tau, f, e)$ define compact operators of pair $(f,e)$ into $L^2(0, 2\pi)$, while the terms $\Pi_{2*}(\tau, f, e) $ and $ \hat\Pi_{2*}(\tau, f, e)$ satisfy
\begin{align*}
\Big\|\beta\Big(\Pi_{2*} (\tau, f_1, e_1) - \Pi_{2*} (\tau, f_2, e_2)\Big)\Big\|_{L^2(0, 2\pi)}
&\leq C \epsilon^{3} \big[\|f_1- f_2\|_*+ \|e_1- e_2\|_{**}\big],
\end{align*}
\begin{align*}
\Big\|\beta\left(\hat\Pi_{2*} (\tau, f_1, e_1) - \hat\Pi_{2*} (\tau, f_2, e_2) \right)\|_{L^2(0, 2\pi)}
&
\leq C \epsilon^{3} \big[\|f_1- f_2\|_*+ \|e_1- e_2\|_{**}\big].
\end{align*}
where $\beta$ has been given in \eqref{beta}.
\qed
\end{lemma}

\medskip
\subsection{ The reduced system of differential equations}\label{Section7.3}\

By the calculations in the above, we conclude that \eqref{project 6.13-1} and \eqref{project 6.13-2} will imply the following system
\begin{align*}
0\,=\,&
 \epsilon^2\varrho_1\Bigg[- \frac{4\pi^2}{{\tilde\ell}^2}{f}'' - \frac{4\pi}{\tilde\ell} \frac{\beta'}{\beta^2} f'
+ \Big(\frac{|\beta'|^2}{\beta^4}
+\frac{\beta''} {\beta^3}
\,-\,\frac{5\kappa^2}{3\beta^2} \Big) f
+ \, \frac 32\frac{\epsilon^2W_{tt}(0, \theta)}{\beta^4} f \Bigg]
\nonumber\\[2mm]
&\qquad + \frac {\epsilon^2 {\mathbf b}_1(d_1+d_2)e}\beta
+\frac{\epsilon^3}{|\beta|^3}{\mathcal D}_1(\theta)
+ \epsilon^4 \frac{d_6} {\beta}{\mathbf b}_1 \frac{e''}{\tilde\ell^2}
+ \frac{\epsilon^3}{\beta^3}\Big[ -\frac{8\pi^2}{{\tilde\ell}^2}f' +\frac{4\pi}{ {\tilde\ell}}\frac{\beta'}{\beta^2} f\Big]e'd_8
\nonumber\\[2mm]
&\qquad
+\frac{\epsilon^3}{\beta}\tilde{\mathcal F}^{*} \left[\tau, f, \frac{\beta {f}'}{\tilde\ell}, e, \frac{e'}{\tilde\ell}\right]+ \Pi(\tau, f, e) + \tilde M(f, e),
\end{align*}
and
\begin{align*}
 0&\,=\,\epsilon^3\frac{4\pi^2}{{\tilde\ell}^2} e'' + \epsilon\lambda_0 e
 -\frac{1}{2}\, \epsilon^2 \frac{ \epsilon^2W_{tt}(0, \theta)}{\beta^3}|f|^2 b_5 +\frac{4}{9} \beta\epsilon^2 f^2 {\mathbf b}_1^2 b_3
\nonumber\\[2mm]
 & \quad
 + \epsilon^2\beta\Big[\frac{4\pi^2}{{\tilde\ell}^2}|f'|^2
-\frac{4\pi}{{\tilde\ell}}\frac{\beta'}{\beta^2}ff'
+\frac{|\beta'|^2}{\beta^4} |f|^2
\Big] b_6
 +\frac 23\epsilon^2 {\mathbf b}_1 fe + \frac3{\beta}e^2b_1 \epsilon^2 - \epsilon^4 \frac{8\pi^2} {{\tilde\ell}^2} {\mathbf b}_1 e'' f
\nonumber\\[2mm]
 & \quad
  + \epsilon^3 \tilde{\mathcal K}^* \left[\tau, \beta f, \frac{\beta {f}'}{\tilde\ell}, e, \frac{\epsilon e'}{\tilde\ell}\right] + \epsilon^4 \tilde{\mathcal K}^{***} \left[\tau, \beta f, \frac{\beta {f}'}{\tilde\ell}, e, \frac{\epsilon e'}{\tilde\ell}\right]
+ \beta \hat\Pi(\tau, f, e)+ \tilde N(f, e),
\end{align*}
 where
\begin{align*}
 \tilde M(f, e) = \frac{\epsilon^3}{\beta}\tilde{\mathcal F}^{**} \left[\tau, \beta f, e, \frac{\epsilon^2 e''}{\tilde\ell^2}, \frac{\beta f''}{\tilde\ell^2}\right] +\Pi_{2*}(\tau, f, e),
\end{align*}
\begin{align*}
 \tilde N(f, e) = \epsilon^3 \tilde{\mathcal K}^{**} \left[\tau, \beta f, \frac{\beta f''}{\tilde\ell^2}, \frac{\epsilon^2 e''}{\tilde \ell^2}\right]+\beta \hat\Pi_{2*}(\tau, f, e).
\end{align*}
After some trivial algebras, we then obtain
\begin{align} \label{last reduced equation on U'}
 & - \frac{4\pi^2}{{\tilde\ell}^2}{f}'' - \frac{4\pi}{\tilde\ell} \frac{\beta'}{\beta^2} f'
+ \Big(\frac{|\beta'|^2}{\beta^4}
+\frac{\beta''} {\beta^3}
\,-\,\frac{5\kappa^2}{3\beta^2} \Big) f
+ \, \frac 32\frac{\epsilon^2W_{tt}(0, \theta)}{\beta^4} f
\nonumber\\[2mm]
 &
= -\left\{ \frac {{\mathbf b}_1(d_1+d_2)e}{\varrho_1\beta}
+\frac{\epsilon}{|\beta|^3 \varrho_1}{\mathcal D}_1(\theta)
+ \frac{d_6} {\beta\, \varrho_1}{\mathbf b}_1 \frac{\epsilon^2 e''}{\tilde\ell^2}
+ \frac{\epsilon^3}{\beta^3}\Big[ -\frac{8\pi^2}{{\tilde\ell}^2}f' +\frac{4\pi}{ {\tilde\ell}}\frac{\beta'}{\beta^2} f\Big]e'd_8 \right\}
\nonumber\\[2mm]
 &\quad -\mathcal A_{1\epsilon} (f, e) - \mathcal K_{1\epsilon} (f, e),
\end{align}
and
\begin{align} \label{last reduced equation on mathcal Z}
 \epsilon^2\frac{4\pi^2}{{\tilde\ell}^2} e'' + \lambda_0 e & =
 - \Bigg\{-\frac{1}{2}\, \epsilon \frac{ \epsilon^2W_{tt}(0, \theta)}{\beta^3}|f|^2 b_5+\frac{4}{9} \beta\epsilon f^2 {\mathbf b}_1^2 b_3 + \epsilon\beta\Big[\frac{4\pi^2}{{\tilde\ell}^2}|f'|^2
-\frac{4\pi}{{\tilde\ell}}\frac{\beta'}{\beta^2}ff'
+\frac{|\beta'|^2}{\beta^4} |f|^2
\Big] b_6
\nonumber\\[2mm]
 & \qquad\quad
 +\frac 23\epsilon {\mathbf b}_1 fe + \epsilon \frac3{\beta}e^2b_1 - \epsilon^3 \frac{8\pi^2} {{\tilde\ell}^2} {\mathbf b}_1 e'' f \Bigg\}- \mathcal A_{2\epsilon} (f, e) - \mathcal K_{2\epsilon} (f, e),
\end{align}
where
\begin{align*}
\mathcal K_{1\epsilon} (f, e)
& =
\frac{\epsilon}{\beta\, \varrho_1}\tilde{\mathcal F}^{*} \left[\tau, f, \frac{\beta {f}'}{\tilde\ell}, e, \frac{e'}{\tilde\ell}\right]
\ +\
\frac{ \Pi(\tau, f, e) }{\epsilon^2\varrho_1},
\end{align*}
\begin{align*}
\mathcal K_{2\epsilon} (f, e)
&= \epsilon^2 \tilde{\mathcal K}^* \left[\tau, \beta f, \frac{\beta {f}'}{\tilde\ell}, e, \frac{\epsilon e'}{\tilde\ell}\right]
\ +\
\epsilon^3 \tilde{\mathcal K}^{***} \left[\tau, \beta f, \frac{\beta {f}'}{\tilde\ell}, e, \frac{\epsilon e'}{\tilde\ell}\right]
\ +\
\frac{\beta  \hat\Pi(\tau, f, e) }\epsilon,
\end{align*}
 and
\begin{align*}
 \mathcal A_{1\epsilon} (f, e) = \frac{\tilde M(f, e)} {\epsilon^2\, \varrho_1}, \quad \mathcal A_{2\epsilon} (f, e) =\frac{\tilde N(f, e)} {\epsilon}.
\end{align*}

 \medskip
Based on the above results, we state the following important lemma.
\begin{lemma} \label{Lemma7-2}
To achieve the validity of the relations in \eqref{project 6.13-1}-\eqref{project 6.13-2}, we need to find $f$ and $e$ with constraints \eqref{constraint-f}-\eqref{constraint-e} in such a way that \eqref{last reduced equation on U'}- \eqref{last reduced equation on mathcal Z} hold.
Here we know that
 $\mathcal K_{1\epsilon} (f, e)$ and $\mathcal K_{2\epsilon} (f, e)$ define compact operators of pair $(f,e)$ into $L^2(0, 2\pi)$, while $\mathcal A_{1\epsilon} (f, e)$ and $\mathcal A_{2\epsilon} (f, e)$
 satisfy the following Lipschitz continuity
\begin{align*}
 \| \mathcal A_{1\epsilon} (f, e) - \mathcal A_{1\epsilon} (f, e) \|_{L^2(0, 2\pi)} \le \epsilon (\|f_1- f_2\|_* + \|e_1- e_2\|_{**} ),
\end{align*}
\begin{align*}
 \| \mathcal A_{2\epsilon} (f, e) - \mathcal A_{2\epsilon} (f, e) \|_{L^2(0, 2\pi)} \le \epsilon^2 (\|f_1- f_2\|_* + \|e_1- e_2\|_{**} ).
\end{align*}
 Moreover, we have
\begin{align*}
 \|\mathcal K_{1 \epsilon} (f, e) \|_{L^2(0, 2\pi)} + \|\mathcal A_{1 \epsilon} (f, e) \|_{L^2(0, 2\pi)} \le C \left(\epsilon + \epsilon \|f\|_* + \epsilon \|e\|_{**} \right),
\end{align*}
\begin{align*}
 \|\mathcal K_{2 \epsilon} (f, e) \|_{L^2(0, 2\pi)} + \|\mathcal A_{2 \epsilon} (f, e) \|_{L^2(0, 2\pi)} \le C \left(\epsilon^3 + \epsilon^2 \|f\|_* + \epsilon^2 \|e\|_{**} \right),
\end{align*}
where the norms $\|\cdot\|_*$ and $ \|\cdot\|_{**}$ have been given in \eqref{constraint-f}-\eqref{constraint-e}.
\qed
\end{lemma}


\medskip
\section{Solving the reduced system} \label{Section8}

 \smallskip
\subsection{ Auxiliary linear problems} \label{Section8.1}\

For the sake of later use, the resolution results for two linear problems will be provided
in Proposition \ref{reduced tilde f boundarycondition-lemma} and Proposition \ref{linear theory for e}.

\subsubsection{}
Consider
\begin{align}
- \frac{4\pi^2}{{\tilde\ell}^2}{\mathfrak f}_{\tau\tau} - \frac{4\pi}{\tilde\ell} \frac{\beta'}{\beta^2}{\mathfrak f}_{\tau}
+ \Big[\frac{|\beta'|^2}{\beta^4}
+\frac{\beta''} {\beta^3}
\,-\,\frac{5\kappa^2}{3\beta^2} \Big]{\mathfrak f} + \, \frac 32\frac{\epsilon^2 W_{tt}(0, \theta)}{\beta^4} {\mathfrak f} = {\mathfrak h}(\tau),
\label{problem-auxil}
\end{align}
with
\begin{align}
{\mathfrak f}(0) = {\mathfrak f}(2\pi), \quad {\mathfrak f}'(0) = {\mathfrak f}'(2\pi),
\label{problem-auxil-boundary}
\end{align}
where the relation between $\tau$ and $\theta=\epsilon z$ is given in \eqref{ztau}, $\beta$ and $\kappa$ are smooth functions on the variable $\theta$, see \eqref{beta} and \eqref{curvature bounded}.

\begin{proposition} \label{reduced tilde f boundarycondition-lemma}
For any given $\epsilon$ small with the validity of \eqref{resonace-f}, we have a constant $C>0$ such that
\begin{align} \label{reduced mathfrak f boundarycondition}
\left\|\beta {\mathfrak f}\right\|_{L^2(0, 2\pi)}^2
+\left\|\frac{\beta}{\tilde\ell} {\mathfrak f}_\tau\right\|_{L^2(0, 2\pi)}^2
+\left\|\frac{\beta}{{\tilde\ell}^2} {\mathfrak f}_{\tau\tau}\right\|_{L^2(0, 2\pi)}^2
\leq
\frac{C}{{\mathrm d}_\epsilon^2}\left\|\beta {\mathfrak h}\right\|_{L^2(0, 2\pi)}^2,
\end{align}
where ${\mathrm d}_\epsilon$ is given in \eqref{resonace-f}.
\qed
\end{proposition}

\medskip
\begin{proof} The proof will be divided into several steps.

{\bf{Step 1.}}
By denoting
\begin{align}
{\mathfrak f}(\tau) = \beta(\theta) {\tilde f}(\tau), \qquad {\mathfrak h}(\tau) = \beta(\theta){\tilde h}(\tau),
\label{setting1}
\end{align}
we have
\begin{align}
{\mathfrak f}_\tau = \beta {\tilde f}_\tau + \frac{\tilde\ell}{2\pi}\frac{\beta'}{\beta} {\tilde f},
\label{setting2}
\end{align}
\begin{align}
{\mathfrak f}_{\tau\tau}
& = \beta {\tilde f}_{\tau\tau}
+ \frac{\tilde\ell}{\pi}\frac{\beta'}{\beta} {\tilde f}_\tau
+ \frac{\tilde\ell^2} {4\pi^2} \frac{\beta\beta''- |\beta'|^2}{\beta^3} {\tilde f}.
\label{setting3}
\end{align}
The above setting will give a new form of \eqref{problem-auxil}
\begin{align*}
- \frac{4\pi^2}{{\tilde\ell}^2}\Big[ {\tilde f}_{\tau\tau} + \beta' {\tilde f}_\tau \frac{\tilde\ell} {\pi \beta^2} \Big]
- \frac{4\pi}{\tilde\ell} \frac{\beta'}{\beta^2} \Big[{\tilde f}_\tau + \beta' {\tilde f} \frac{\tilde\ell} {2\pi \beta^2}\Big]
+\Big[2\frac{|\beta'|^2}{\beta^4}\,-\,\frac{5\kappa^2}{3\beta^2} \Big]{\tilde f} + \, \frac 32\frac{\epsilon^2 W_{tt}(0, \theta)}{\beta^4}{\tilde f}
= {\tilde h},
\end{align*}
i.e.,
\begin{align*}
- \frac{4\pi^2}{{\tilde\ell}^2}\Big[ {\tilde f}_{\tau\tau} +2 \beta' {\tilde f}_\tau \frac{\tilde\ell} {\pi \beta^2} \Big]
+
\Big[\frac 32\frac{\epsilon^2 W_{tt}(0, \theta)}{\beta^4}\,-\,\frac{5\kappa^2}{3\beta^2} \Big]{\tilde f}
= {\tilde h},
\end{align*}
i.e.,
\begin{align*}
 - \frac{4\pi^2}{{\tilde\ell}^2} \frac 1{ \beta^4} \p_\tau \left( \beta^4 {\tilde f}_\tau \right)
+
\Big[\frac 32\frac{\epsilon^2 W_{tt}(0, \theta)}{\beta^4}\,-\,\frac{5\kappa^2}{3\beta^2} \Big]{\tilde f}
= {\tilde h}.
\end{align*}

\medskip
{\bf{Step 2.}}
We consider the weighted eigenvalue problem
\begin{align} \label{weighted-reduced tilde f}
- \frac{4\pi^2}{{\tilde\ell}^2} \p_\tau \left( \beta^4 {\tilde f}_\tau \right)
+
\Big[\frac 32\frac{\epsilon^2 W_{tt}(0, \theta)}{\beta^4}\,-\,\frac{5\kappa^2}{3\beta^2} \Big] \beta^4{\tilde f} = \Lambda_\epsilon \beta^4 {\tilde f},
\end{align}
subjecting to periodic boundary conditions
\begin{align} \label{newreduced tilde f boundarycondition}
\tilde f(0) = \tilde f(2\pi), \quad \tilde f'(0) = \tilde f'(2\pi).
\end{align}
Simple calculations will show that problem \eqref{weighted-reduced tilde f}-\eqref{newreduced tilde f boundarycondition} is
equivalent to the eigenvalue problem \eqref{mathbbF-boundary}.
By introducing the weighted inner product $\langle \cdot, \cdot\rangle$ as
\begin{align*}
 \langle {\tilde f}, {\tilde h}\rangle_{W} = \int_0^{2\pi} \beta^4 {\tilde f} {\tilde h}{\mathrm d}\tau,
\end{align*}
we first give the following claim, which is well-known.

\smallskip
\noindent{\bf{Claim 1.}} \label{weighted eigenvalue problem lemma}
{\em
There exists a sequence $\{\Lambda_{\epsilon, k}\}_{k=1}^\infty$ of eigenvalues to the weighted eigenvalue problem \eqref{weighted-reduced tilde f}-\eqref{newreduced tilde f boundarycondition}.
Moreover, the corresponding eigenfunctions $\{\omega_{\epsilon, k}\}_{k=1}^\infty$ can be selected to form an orthonormal basis of $L^2(0, 2\pi)$ with respect to the weighted inner product $\langle \cdot, \cdot\rangle_W$.
}

\medskip
{\bf{Step 3.}}
We then concern the problem
\begin{align}
- \frac{4\pi^2}{{\tilde\ell}^2} \p_\tau \left( \beta^4 {\tilde f}_\tau \right)
+
\Big[\frac 32\frac{\epsilon^2 W_{tt}(0, \theta)}{\beta^4}\,-\,\frac{5\kappa^2}{3\beta^2} \Big] \beta^4{\tilde f} = \beta^4{\tilde h},
\label{weighted-reduced tilde f222222}
\end{align}
with
\begin{align}
\tilde f(0) = \tilde f(2\pi), \quad \tilde f'(0) = \tilde f'(2\pi),
\label{newreduced tilde f boundarycondition222222}
\end{align}
and obtain the following claim.

\smallskip
\noindent
{\bf{Claim 2.}} \label{reduced tilde f boundarycondition-lemma22222}
{\em
Consider \eqref{weighted-reduced tilde f222222}-\eqref{newreduced tilde f boundarycondition222222},
and for any given $\epsilon$ small, we have a constant $C>0$ such that
\begin{align} \label{reduced tilde f boundarycondition}
\left\| \beta^2{\tilde f} \right\|_{L^2(0, 2\pi)}^2
+\left\| \frac{\beta^2}{\tilde\ell}{\tilde f}_\tau \right\|_{L^2(0, 2\pi)}^2
+\left\| \frac{\beta^2}{{\tilde\ell}^2}{\tilde f}_{\tau\tau} \right\|_{L^2(0, 2\pi)}^2
\leq
\frac{C}{{\mathrm d}_\epsilon^2}\left\| \beta^2{\tilde h} \right\|_{L^2(0, 2\pi)}^2,
\end{align}
where ${\mathrm d}_\epsilon$ is defined in \eqref{resonace-f}.
}

\noindent
Indeed, by Claim 1, 
we can express $\tilde f$ as
\[
\tilde f= \sum_{k=1}^\infty {\tilde f}_k \omega_{\epsilon, k}, \quad \text{where}~ {\tilde f}_k =\int_0^{2\pi} \beta^4 \tilde f \omega_{\epsilon, k}{\mathrm d}\tau,
\]
 and then substitute it into \eqref{weighted-reduced tilde f222222} to derive
\[
 \sum_{k=1}^\infty {\tilde f}_k\, \Lambda_{\epsilon, k}\,\beta^2\omega_{\epsilon, k}\,=\,\beta^2{\tilde h}.
\]
Then by Claim 1,
and the assumption in \eqref{resonace-f}, for any $\tilde f\in L^2(0, 2\pi)$, there holds
\begin{align*}
\left\| \beta^2{\tilde f} \right\|_{L^2(0, 2\pi)}^2
=\sum_{k=1}^\infty |{\tilde f}_k|^2
\leq \frac1{{\mathrm d}_\epsilon^2} \sum_{k=1}^\infty \Lambda_{\epsilon, k}^2 {\tilde f}_k^2
= \frac1{{\mathrm d}_\epsilon^2}\left\| \beta^2{\tilde h} \right\|_{L^2(0, 2\pi)}^2.
\end{align*}
Moreover, by equations in \eqref{weighted-reduced tilde f222222}-\eqref{newreduced tilde f boundarycondition222222}
and the above estimate as well as the assumptions in \eqref{curvature bounded} and \eqref{W_t-control},
it is easy to derive \eqref{reduced tilde f boundarycondition} due to \eqref{beta-bound}.

\medskip
{\bf{Step 4.}}
A conclusive remark is that the proof of Proposition \ref{reduced tilde f boundarycondition-lemma} will be easily derived from Claim 2
and the identities \eqref{setting1}-\eqref{setting3}.
In fact, it is easy to obvious that
\begin{align*}
\left\| \beta{\mathfrak f} \right\|_{L^2(0, 2\pi)}^2
= \big\| \beta^2{\tilde f} \big\|_{L^2(0, 2\pi)}^2,
\end{align*}
\begin{align*}
\left\| \frac{\beta}{\tilde\ell} {\mathfrak f}_\tau \right\|_{L^2(0, 2\pi)}^2
\leq
C\left\| \frac{\beta^2}{\tilde\ell}{\tilde f}_\tau \right\|_{L^2(0, 2\pi)}^2
+
C\big\| \beta{\tilde f} \big\|_{L^2(0, 2\pi)}^2,
\end{align*}
and
\begin{align*}
\left\| \frac{\beta}{{\tilde\ell}^2} {\mathfrak f}_{\tau\tau} \right\|_{L^2(0, 2\pi)}^2
\leq
C\left\| \frac{\beta^2}{{\tilde\ell}^2}{\tilde f}_{\tau\tau} \right\|_{L^2(0, 2\pi)}^2
+
C\left\| \frac{\beta}{{\tilde\ell}}{\tilde f}_\tau \right\|_{L^2(0, 2\pi)}^2
+
C\big\| {\tilde f} \big\|_{L^2(0, 2\pi)}^2,
\end{align*}
due to \eqref{beta-bound}.
The combination of the above estimates and \eqref{reduced tilde f boundarycondition}
will finish the proof of \eqref{reduced mathfrak f boundarycondition}.
\end{proof}

\medskip
\subsubsection{}
For the linear problem
\begin{align} \label{reduced for e2 new}
\epsilon^2\frac{4\pi^2}{{\tilde\ell}^2} {\mathfrak e}'' + \lambda_0 {\mathfrak e} \,=\, d\quad\mbox{in }\ (0,2\pi),
\quad {\mathfrak e}(0) \,=\, {\mathfrak e}(2\pi),
\quad {\mathfrak e}'(0) \,=\, {\mathfrak e}' (2\pi),
\end{align}
where the constants ${\tilde\ell}$, $\lambda_0$ are given in \eqref{ztau} and \eqref{lambda0}, we now present the following linear theory.

\begin{proposition} \label{linear theory for e}
Suppose that \eqref{gap condition1} holds.
If $d\in L^2(0, 2\pi)$, there exits a unique solution ${\mathfrak e}\in H^2 (0,2\pi)$ to the linear problem \eqref{reduced for e2 new}, which satisfies the estimate
\begin{align} \label{estimate norm of e1}
\|{\mathfrak e}\|^2_{**}
:\,=\,
\|{\mathfrak e}\|^2_{L^\infty(0, 2\pi)}
+ \frac{\epsilon^2}{{\tilde\ell}^2} \|{\mathfrak e}'\|^2_{L^2 (0, 2\pi)}
+ \frac{\epsilon^4} {{\tilde\ell}^4}\|{\mathfrak e}''\|^2_{L^2(0, 2\pi)}
\leq C \frac{{\tilde\ell}^2} {\epsilon^2} \| d\|^2_{L^2(0, 2\pi)}.
\end{align}

\end{proposition}

\begin{proof}
We follow the method in \cite{delPKowWei1} and consider the auxiliary eigenvalue problem
\begin{align} \label{eigenvalue1 new}
 \Theta'' + \Lambda\Theta \,=\, 0,
 \qquad
 \Theta(0) \,=\, \Theta(2\pi), \quad \Theta'(0) \,=\, \Theta' (2\pi).
\end{align}
 It is well-known that problem \eqref{eigenvalue1 new} has a sequence of eigenvalues $\big\{\Lambda_0, \Lambda_1, \cdots, \Lambda_j,\cdots\big\}$
 with an associated orthogonal basis in $L^2(0, 2\pi)$, constituted by eigenfunctions $\big\{\Theta_0, \Theta_1, \cdots, \Theta_j,\cdots\big\} $.
 The asymptotic behaviour of eigenvalue is
\begin{align} \label{Lambda_k asymptotic behavior}
\sqrt{ \Lambda_j } =j+O(1/j^3), \quad \forall\,j\in {\mathbb N}^{+} .
\end{align}
 We note that
\[
|\Theta_j|\leq C, \quad \forall\,j \,=\, 0, 1, 2, \cdots.
\]
Suppose
\[
{\mathfrak e} = \sum_{j=0}^\infty {\mathfrak e}_j \Theta_j,
\]
and
\[
d = \sum_{j=0}^\infty d_j \Theta_j, \quad \text{with}~~ d_j = \int_0^{2\pi} d\, \Theta_j \,{\mathrm d}\tau.
\]
We can rewrite \eqref{reduced for e2 new} as
\begin{align*}
\epsilon^2\frac{4\pi^2}{{\tilde\ell}^2} {\mathfrak e}_j \Theta_j'' + \lambda_0 {\mathfrak e}_j \Theta_j \,=\, d_j \Theta_j,
\end{align*}
i.e.,
\begin{align*}
-\epsilon^2\frac{4\pi^2}{{\tilde\ell}^2} {\mathfrak e}_j \Lambda_j\Theta_j + \lambda_0 {\mathfrak e}_j \Theta_j \,=\, d_j \Theta_j,
\end{align*}
i.e.,
\begin{align*}
-\epsilon^2\frac{4\pi^2}{{\tilde\ell}^2} {\mathfrak e}_j \Lambda_j + \lambda_0 {\mathfrak e}_j \,=\, d_j,
\end{align*}
for all $j=0,1,\cdots$.
Then
\begin{align*}
{\mathfrak e}_j = \frac{d_j} {\lambda_0 - \epsilon^2\frac{4\pi^2}{{\tilde\ell}^2} \Lambda_j},
\quad \text{i.e.,}\quad
{\mathfrak e}= \sum_{j \,=\, 0}^\infty \frac{d_j} {\lambda_0 - \epsilon^2\frac{4\pi^2}{{\tilde\ell}^2}\Lambda_j} \Theta_j.
\end{align*}
 Then we know that problem \eqref{reduced for e2 new} is solvable if and only if
\begin{align*}
\lambda_0 - \epsilon^2\frac{4\pi^2}{{\tilde\ell}^2}\Lambda_j \ne 0, \quad \text{for all $j=0, 1, 2, \cdots$}.
\end{align*}
By the asymptotic behavior of $\Lambda_j$ in \eqref{Lambda_k asymptotic behavior}, we can choose $\epsilon$ such that
\begin{align*}
\left|\lambda_0 - \epsilon^2\frac{4\pi^2}{{\tilde\ell}^2}\Lambda_j \right| >c \frac\epsilon {\tilde\ell},
\end{align*}
for a small $c>0$ depending only on $\lambda_0$.
This is the gap condition given in \eqref{gap condition1}.
Then we have
\begin{align*}
\|{\mathfrak e}\|_{L^2(0, 2\pi)}^2 \leq \sum_{j \,=\, 0}^\infty \frac{d_j^2} {\left|\lambda_0 - \epsilon^2\frac{4\pi^2}{{\tilde\ell}^2}\Lambda_j\right|^2} \leq \frac {{\tilde\ell}^2}{\epsilon^2} \|d\|_{L^2(0, 2\pi)}^2,
\end{align*}
\begin{align*}
\|{\mathfrak e}'\|_{L^2(0, 2\pi)}^2
= - \int_0^{2\pi} {\mathfrak e}'' {\mathfrak e} \,{\mathrm d}\tau
= \sum_{j \,=\, 0}^\infty
\frac{\Lambda_j d_j^2}{\left|\lambda_0 -\epsilon^2\frac{4\pi^2}{{\tilde\ell}^2}\Lambda_j\right|^2}
\leq
\frac {{\tilde\ell}^4}{\epsilon^4} \|d\|_{L^2(0, 2\pi)}^2.
\end{align*}
By using the equation in \eqref{reduced for e2 new}, we get
\begin{align*}
\frac{\epsilon^2}{{\tilde\ell}^2} \|{\mathfrak e}''\|_{L^2(0, 2\pi)}
\le
\|{\mathfrak e}\|_{L^2(0, 2\pi)} + \|d\|_{L^2(0, 2\pi)}
\le
\frac {{\tilde\ell}}{\epsilon} \|d\|_{L^2(0, 2\pi)}.
\end{align*}
Next, we consider the $L^\infty$ estimate of ${\mathfrak e}$.
There holds
\begin{align*}
|{\mathfrak e}(\tau)|^2
= \left|\sum_{j \,=\, 0}^\infty \frac{d_j} {\lambda_0 - \epsilon^2\frac{4\pi^2}{{\tilde\ell}^2}\Lambda_j} \Theta_j\right|^2
\le
\Bigg(\sum_{j \,=\, 0}^\infty d_j^2 \Bigg)\Bigg(\sum_{j \,=\, 0}^\infty \left|\frac{ \Theta_j} {\lambda_0 - \epsilon^2\frac{4\pi^2}{{\tilde\ell}^2}\Lambda_j} \right|^2\Bigg).
\end{align*}

We claim that
\begin{align} \label{linfty claim}
\sum_{j \,=\, 0}^\infty \left|\frac{ \Theta_j} {\lambda_0 - \epsilon^2\frac{4\pi^2}{{\tilde\ell}^2}\Lambda_j} \right|^2 \le C \frac {{\tilde\ell^2}}{\epsilon^2}.
\end{align}
In fact, we have
\begin{align*}
\sum_{j \,=\, 0}^\infty \left|\frac{ \Theta_j} {\lambda_0 - \epsilon^2\frac{4\pi^2}{{\tilde\ell}^2}\Lambda_j} \right|^2
&
 \le C \sum_{j \,=\, 0}^\infty \left|\frac{1} {\lambda_0 - \epsilon^2\frac{4\pi^2}{{\tilde\ell}^2}\Lambda_j} \right|^2
\nonumber
\\ &
 = C \sum_{\{j \,:\,  \epsilon^2\frac{4\pi^2}{{\tilde\ell}^2}\Lambda_j < \frac{\lambda_0}2 \}} \left|\frac{1} {\lambda_0 - \epsilon^2\frac{4\pi^2}{{\tilde\ell}^2}\Lambda_j} \right|^2
  + C \sum_{\{j \,:\,  \frac{\lambda_0}2 \le \epsilon^2\frac{4\pi^2}{{\tilde\ell}^2}\Lambda_j \le \lambda_0 \}} \left|\frac{1} {\lambda_0 - \epsilon^2\frac{4\pi^2}{{\tilde\ell}^2}\Lambda_j} \right|^2
\nonumber
\\ & \quad
 +C \sum_{\{j \,:\, \epsilon^2\frac{4\pi^2}{{\tilde\ell}^2}\Lambda_j > \lambda_0\}} \left|\frac{1} {\lambda_0 - \epsilon^2\frac{4\pi^2}{{\tilde\ell}^2}\Lambda_j} \right|^2.
\end{align*}
By the asymptotic expansion of $\Lambda_j$, it's easy to get that
\begin{align*}
 &\sum_{\{j \,:\,  \epsilon^2\frac{4\pi^2}{{\tilde\ell}^2}\Lambda_j < \frac{\lambda_0}2 \}} \left|\frac{1} {\lambda_0 - \epsilon^2\frac{4\pi^2}{{\tilde\ell}^2}\Lambda_j} \right|^2
 \le
C \sum_{i =0}^{\frac{\tilde\ell} \epsilon} \frac 1{\lambda_0} \le C \frac{{\tilde\ell}^2} {\epsilon^2},
\end{align*}
and
\begin{align*}
\sum_{\{j \,:\,  \epsilon^2\frac{4\pi^2}{{\tilde\ell}^2}\Lambda_j > \lambda_0\}} \left|\frac{1} {\lambda_0 - \epsilon^2\frac{4\pi^2}{{\tilde\ell}^2}\Lambda_j} \right|^2
& \le C \left|\frac{1} {\lambda_0 - \epsilon^2\frac{4\pi^2}{{\tilde\ell}^2}\Lambda_{j_1}} \right|^2
+
C \sum_{j =j_1+1} \left|\frac{1} {\epsilon^2\frac{4\pi^2}{{\tilde\ell}^2}\Lambda_{j_1} - \epsilon^2\frac{4\pi^2}{{\tilde\ell}^2}\Lambda_j} \right|^2 \le C \frac{\tilde\ell^2} {\epsilon^2}.
\end{align*}
Here we note that $j_1$ is smallest number such that $\epsilon^2\frac{4\pi^2}{{\tilde\ell}^2}\Lambda_{j_1} > \lambda_0.$
 Moreover, we have
\begin{align*}
 & \sum_{\{j \,:\,  \frac{\lambda_0}2 \le \epsilon^2\frac{4\pi^2}{{\tilde\ell}^2}\Lambda_j \le \lambda_0 \}} \left|\frac{1} {\lambda_0 - \epsilon^2\frac{4\pi^2}{{\tilde\ell}^2}\Lambda_j} \right|^2
\nonumber
\\ & \le C \left|\frac{1} {\lambda_0 - \epsilon^2\frac{4\pi^2}{{\tilde\ell}^2}\Lambda_{j_2}} \right|^2
+ C \sum_{j = j_3}^{j_2-1} \left|\frac{1} {\epsilon^2\frac{4\pi^2}{{\tilde\ell}^2}\Lambda_{j_2} - \epsilon^2\frac{4\pi^2}{{\tilde\ell}^2}\Lambda_j} \right|^2 \le C \frac{\tilde\ell^2} {\epsilon^2}.
\end{align*}
 We note $j_2$ is the largest number such that
$$
\epsilon^2\frac{4\pi^2}{{\tilde\ell}^2}\Lambda_{j_2} < \lambda_0,
$$
and $j_3$ is the smallest number such that
$$
\epsilon^2\frac{4\pi^2}{{\tilde\ell}^2}\Lambda_{j_3} \ge \frac{\lambda_0}2.
$$
Therefore, we prove \eqref{linfty claim}.
The validity of \eqref{estimate norm of e1} can be easily derived.
\end{proof}


\subsection{ Solving the reduced nonlinear system}\label{Section8.2}\

 Note that for any pair $(f, e)$ with the constraints in \eqref{constraint-f}-\eqref{constraint-e}, there hold
$$
 \|f \|_*\leq \epsilon^{1-\alpha}, \quad \|e \|_{**}\leq \epsilon^{\frac 32 \varrho + \frac 12- 2\alpha}.
$$
 We denote $\hat{\mathcal L} (f, e)= \left( \hat {\mathcal L}_1, \hat {\mathcal L}_2 \right) (f, e), $ with
\begin{align*}
 \hat {\mathcal L}_1 (f, e):
\,=\, & - \frac{4\pi^2}{{\tilde\ell}^2}{f}'' - \frac{4\pi}{\tilde\ell} \frac{\beta'}{\beta^2} f'
+ \Big(\frac{|\beta'|^2}{\beta^4}
+\frac{\beta''} {\beta^3}
\,-\,\frac{5\kappa^2}{3\beta^2} \Big) f
+ \, \frac 32\frac{\epsilon^2W_{tt}(0, \theta)}{\beta^4} f,
\nonumber\\[2mm]
\hat {\mathcal L}_2 (f, e):
\,=\, & \epsilon^2\frac{4\pi^2}{{\tilde\ell}^2} e'' + \lambda_0 e.
\end{align*}
By Proposition \ref{reduced tilde f boundarycondition-lemma} and Proposition \ref{linear theory for e} provided in Section \ref{Section8.1}, we know that the problem
\[ \hat {\mathcal L} (f, e) = \left({\mathfrak h}, d\right) \]
with the boundary conditions
\begin{align}
f(0) \,=\, f(2\pi), \quad f'(0) \,=\, f' (2\pi), \quad e(0) \,=\, e(2\pi), \quad e'(0) \,=\, e' (2\pi),
\label{boundaryhatf-e}
\end{align}
is invertible in the sense that
\begin{align}
 \|f\|_* \le \frac{C} {\epsilon^\alpha}\|\beta {\mathfrak h}\|_{L^2(0, 2\pi)}, \quad \|e\|_{**} \le \frac{C}{\epsilon^{2-\varrho}} \|d\|_{L^2(0, 2\pi)}.
\end{align}

\medskip
Next, we denote
\begin{align*}
 \mathcal Q_{1\epsilon} (f, e)
=  -\left\{ \frac {{\mathbf b}_1(d_1+d_2)e}{\varrho_1\beta}
+\frac{\epsilon}{|\beta|^3 \varrho_1}{\mathcal D}_1(\theta)
+ \frac{d_6} {\beta\, \varrho_1}{\mathbf b}_1 \frac{\epsilon^2 e''}{\tilde\ell^2}
+ \frac{\epsilon^3}{\beta^3}\Big[ -\frac{8\pi^2}{{\tilde\ell}^2}f' +\frac{4\pi}{ {\tilde\ell}}\frac{\beta'}{\beta^2} f\Big]e'd_8 \right\},
\end{align*}
and
\begin{align*}
\mathcal Q_{2\epsilon} (f, e)
& = - \Bigg(-\frac{1}{2}\, \epsilon \frac{ \epsilon^2W_{tt}(0, \theta)}{\beta^3}|f|^2 b_5
+\frac{4}{9} \beta\epsilon f^2 {\mathbf b}_1^2 b_3
\nonumber\\[2mm]
 & \qquad +\epsilon\beta\Big[\frac{4\pi^2}{{\tilde\ell}^2}|f'|^2
-\frac{4\pi}{{\tilde\ell}}\frac{\beta'}{\beta^2}ff'
+\frac{|\beta'|^2}{\beta^4} |f|^2
\Big] b_6
 +\frac 23\epsilon {\mathbf b}_1 fe + \epsilon \frac3{\beta}e^2b_1
 - \epsilon^3 \frac{8\pi^2} {{\tilde\ell}^2} {\mathbf b}_1 e'' f \Bigg).
\end{align*}
For the nonlinear coupling problem
\begin{align} \label{7-138}
 \hat {\mathcal L} (f, e) = \Big[\left(\mathcal Q_{1\epsilon}, \mathcal Q_{2\epsilon}\right) - \left(\mathcal A_{1\epsilon}, \mathcal A_{2\epsilon} \right) \Big](f, e)
\,+\, \left({\mathfrak h}, d\right),
\end{align}
with boundary conditions in \eqref{boundaryhatf-e},
we can deduce from the contraction mapping principle that the problem is uniquely solvable if
$$
\|\beta {\mathfrak h}\|_{L^2(0, 2\pi)} \le {C} \epsilon, \quad \|d\|_{L^2(0, 2\pi)} \le C \epsilon^{\frac32-\frac\varrho 2- \alpha}.
$$
In fact, define the set
\[
\tilde{\mathcal B} = \big\{ (f, e)\,:\, (f, e)~\text{satisfies} ~\eqref{boundaryhatf-e}~\text{and}~\|f\|_* \le \bar{C} \epsilon^{1-\alpha}, \|e\|_{**} \le \bar{C} \epsilon^{\frac 32 \varrho+ \frac12-2\alpha} \big\}.
\]
with $\bar C$ large enough.
Problem \eqref{7-138} can be transformed into the following fixed point problem
\begin{align} \label{fixed point problem}
 (f,e) =\, & \hat {\mathcal L}^{-1}\Bigg( \Big[ \left(\mathcal Q_{1\epsilon}, \mathcal Q_{2\epsilon}\right) - \left(\mathcal A_{1\epsilon}, \mathcal A_{2\epsilon}\right) \Big] (f, e)
 \,+\,
 \left({\mathfrak h}, d\right)\Bigg)
\nonumber\\[2mm]
: = \, & \mathbb{T} (f, e),
\end{align}
where
\begin{align*}
\hat {\mathcal L}^{-1} = \left( \hat {\mathcal L}_1^{-1}, \hat {\mathcal L}_2^{-1}\right).
\end{align*}
 By the results provided in Lemma \ref{Lemma7-2}, and combining \eqref{constraint-alpha-varrho}, \eqref{Linfty of f},  we can show that
$\mathbb{T} $ is a contradiction map from $\tilde{\mathcal B}$ to itself. From the contraction mapping principle, problem \eqref{fixed point problem} has a fixed point in $\tilde{\mathcal B}$.

\medskip
Finally, we consider the problem
\begin{align} \label{7-}
 \hat {\mathcal L} (f, e)
 = \Big[\left(\mathcal Q_{1\epsilon}, \mathcal Q_{2\epsilon}\right)
 - \left(\mathcal A_{1\epsilon}, \mathcal A_{2\epsilon}\right)
 - \left(\mathcal K_{1\epsilon}, \mathcal K_{2\epsilon}\right)  \Big](f, e),
\end{align}
 under the boundary condition \eqref{boundaryhatf-e}.
 Denote by
\begin{align*}
 \tilde {\mathcal L} (e,f)
 = \hat {\mathcal L} (f, e)
 - \Big[\left(\mathcal Q_{1\epsilon}, \mathcal Q_{2\epsilon}\right) - \left(\mathcal A_{1\epsilon}, \mathcal A_{2\epsilon}\right)\Big] (f,e).
\end{align*}
 Then we transfer problem \eqref{7-} into the following fixed point problem
\begin{align*}
 (f,e) = \underbrace{- \tilde {\mathcal L}^{-1} \left(\mathcal K_{1\epsilon}, \mathcal K_{2\epsilon}\right)}_{\text{compact operator}} (f,e).
\end{align*}
 By Schauder's fixed point theorem and the estimates provided in Lemma \ref{Lemma7-2}, we can get the existence of $(f, e)$ for \eqref{7-}, which satisfies \eqref{constraint-f}-\eqref{constraint-e}.
 \qed

\medskip
 \section{Choosing the parameter $\epsilon$ from mass constraint (\ref{constraint2})}\label{Section9}

After finding a solution $v$ of \eqref{eq3} in Sections \ref{Section5}-\ref{Section8}, we now consider the constraint \eqref{constraint2},
which can be formulated as the following
\begin{align} \label{constraint6}
a & \,=\, \int_{ {\mathfrak S}_\delta} |v|^2\,{\mathrm d}y
\ +\
\int_{{\mathbb R}^2\setminus{\mathfrak S}_\delta} |v|^2\,{\mathrm d}y
\nonumber\\[2mm]
 & \,=\, \int_{{\mathfrak S}_\delta } |v|^2 \Big[1+ \epsilon\kappa(\epsilon z) s \Big] {\mathrm d} s\,{\mathrm d} z
 \ +\
\int_{{\mathbb R}^2\setminus {\mathfrak S}_\delta } |v|^2\,{\mathrm d}y
\nonumber\\[2mm]
 & \,=\,
\frac1\epsilon\int_{{\mathfrak S}_\delta} |v|^2 \Bigg[1+ \kappa \frac{\epsilon (x+f)} {\beta} \Bigg] \frac {\tilde \ell} {2\pi \beta}{\mathrm d} \left(\frac{ x+f} {\beta}\right)\,{\mathrm d}\tau
\ +\
\int_{{\mathbb R}^2\setminus{\mathfrak S}_\delta} |v|^2\,{\mathrm d}y,
\end{align}
where
\begin{align*}
{\mathfrak S}_\delta=\Bigg\{ (x, \tau): \quad |x| \le \frac {\beta \delta} \epsilon, \quad 0<\tau< 2\pi\Bigg\}.
\end{align*}
 Since $|x| \le  \frac {\beta \delta} \epsilon$, then we have
\begin{align} \label{out of delta}
 -\frac\delta \epsilon + \frac{\epsilon f} {\beta} < s= \frac{(x+f)} {\beta} \le \frac \delta \epsilon+ \frac{\epsilon f} {\beta}.
\end{align}
 Following the analysis in Section \ref{Section6}, we know that
\begin{align*}
 v(y) & = \mathbf{U}(y) \,+\, \eta_{3\delta}^{\epsilon}(\epsilon |s|){\tilde\phi}(y)
 \,+\, {\tilde\psi}(y)
\nonumber\\[2mm]
 &
 = \eta_{3\delta}^\epsilon(s) \beta(\epsilon z)w_2(x, \tau)
 \,+\, \eta_{3\delta}^{\epsilon}(\epsilon |s|){\tilde\phi}(y)
 \,+\, {\tilde\psi}(y).
\end{align*}
In the above, $ \eta_{3\delta}^\epsilon(s) $ is the cut-off function given in Section \ref{Section6.1}, and the function $w_2(x, \tau) $ is given in \eqref{newthirdapproximatesolution} with the form
\begin{align*}
w_2(x, \tau) = U(x) \,+\, \epsilon \frac{e(\tau)}{\beta(\epsilon z)} \mathcal Z(x) \,+\, \epsilon \phi_1(x, \tau) \,+\, \epsilon^2\phi_2(x, \tau).
\end{align*}
where $\phi_1, \phi_2$ are given in \eqref{definition of phi1}, \eqref{definition of e2phi2}.

In the region ${\mathbb R}^2\setminus{\mathfrak S}_\delta$, which means far from the curve $\G_\epsilon$, the functions
\[ w_2(x, \tau), \quad \eta_{3\delta}^{\epsilon}(\epsilon |s|){\tilde\phi}(y), \quad {\tilde\psi}(y), \]
exhibit exponential decay, and by \eqref{out of delta}, thus the integral
\begin{align}
 \int_{{\mathbb R}^2\setminus{\mathfrak S}_\delta} |v|^2\,{\mathrm d}y \in O\left(e^{-\frac{\tilde \delta}\epsilon} \right).
\label{estimate1}
\end{align}
where $\tilde \delta$ is small positive constant.
On the other hand, we have
\begin{align}
 & \frac1\epsilon\int_{{\mathfrak S}_\delta} |v|^2 \Bigg[1+ \kappa \frac{\epsilon (x+f)} {\beta} \Bigg] \frac {\tilde \ell} {2\pi \beta}{\mathrm d} \left(\frac{ x+f} {\beta}\right)\,{\mathrm d}\tau
\nonumber\\[2mm]
 &
 = \frac {\tilde \ell} {2\pi \epsilon} \int_{{\mathfrak S}_\delta} |v|^2 \frac {1} {\beta^2}\, {\mathrm d} x \,{\mathrm d}\tau
\nonumber\\[2mm]
 &
 = \frac {\tilde \ell} {2\pi \epsilon} \int_{{\mathfrak S}_\delta} \frac {1} {\beta^2} \left[ |\beta|^2 w_2^2+ |{\tilde\phi}|^2+ |{\tilde\psi}|^2 + 2\beta w_2 \big( {\tilde\phi} + {\tilde\psi} \big) + 2 \tilde \phi\tilde \psi \right] \, {\mathrm d} x \,{\mathrm d}\tau
\nonumber\\[2mm]
 &
 = \frac {\tilde \ell} {2\pi \epsilon} \int_{{\mathfrak S}_\delta} w_2^2 \, {\mathrm d} x \,{\mathrm d}\tau
\ +\
\frac {\tilde \ell} {2\pi \epsilon} \int_{{\mathfrak S}_\delta} \frac {1} {\beta^2} \left[ |{\tilde\phi}|^2+ |{\tilde\psi}|^2 + 2\beta w_2 \big( {\tilde\phi} + {\tilde\psi} \big) + 2 \tilde \phi\tilde \psi \right] \, {\mathrm d} x \,{\mathrm d}\tau.
\label{estimate2}
\end{align}
The terms in the above formula will be estimated in the following way.

\smallskip
\noindent $\spadesuit$\
For the first term, we have
\begin{align*}
 & \frac {\tilde \ell} {2\pi \epsilon} \int_{{\mathfrak S}_\delta} w_2^2 \, {\mathrm d} x \,{\mathrm d}\tau
\nonumber\\[2mm]
 &
 = \frac {\tilde \ell} {2\pi \epsilon} \int_{{\mathfrak S}_\delta} \big|U(x) + \epsilon \frac e\beta \mathcal Z + \epsilon \phi_1 + \epsilon^2\phi_2 \big|^2\, {\mathrm d} x \,{\mathrm d}\tau
 \nonumber\\[2mm]
 &
 = \frac {\tilde \ell} {2\pi \epsilon} \int_{{\mathfrak S}} \big|U(x) + \epsilon \frac e\beta \mathcal Z + \epsilon \phi_1 + \epsilon^2\phi_2 \big|^2\, {\mathrm d} x \,{\mathrm d}\tau
  +O\left(e^{-\frac{\tilde \delta}\epsilon} \right)
\nonumber\\[2mm]
 &
 = \frac {\tilde \ell} {2\pi \epsilon} \int_{{\mathfrak S}}
\Big[ U^2 + \epsilon^2\frac{e^2}{\beta^2}|\mathcal Z|^2 + \epsilon^2|\phi_1|^2+ \epsilon^4|\phi_2|^2\Big]\, {\mathrm d} x \,{\mathrm d}\tau
 + \frac {\tilde \ell} {\pi \epsilon} \int_{{\mathfrak S}} U \left(\epsilon \frac e\beta \mathcal Z + \epsilon \phi_1 + \epsilon^2\phi_2 \right)\, {\mathrm d} x \,{\mathrm d}\tau
\nonumber\\[2mm]
 &
 \quad
 +\epsilon\frac {\tilde \ell} {\pi} \int_{{\mathfrak S}} \frac e\beta \mathcal Z \phi_1\, {\mathrm d} x \,{\mathrm d}\tau
 + \epsilon^2\frac {\tilde \ell} {\pi} \int_{{\mathfrak S}}\left(\frac e\beta \mathcal Z \phi_2 + \phi_1 \phi_2 \right) \, {\mathrm d} x \,{\mathrm d}\tau
 +O\left(e^{-\frac{\tilde \delta}\epsilon} \right).
\end{align*}
 We have
\begin{align*}
 & \frac {\tilde \ell} {2\pi \epsilon} \int_{{\mathfrak S}}
\Big[U^2 + \epsilon^2 \frac{e^2}{\beta^2}|\mathcal Z|^2 + \epsilon^2|\phi_1|^2+ \epsilon^4|\phi_2|^2\Big]\, {\mathrm d} x \,{\mathrm d}\tau
\nonumber\\[2mm]
 & = \frac {\tilde \ell} {2\pi \epsilon} \left[ \int_0^{2\pi} \int_{\mathbb R} U^2 \, {\mathrm d} x \,{\mathrm d}\tau
 + \epsilon^2 \int_0^{2\pi} \int_{\mathbb R} \frac{e^2}{\beta^2} |\mathcal Z|^2 \, {\mathrm d} x \,{\mathrm d}\tau
 + \epsilon^2 \int_0^{2\pi} \int_{\mathbb R} |\phi_1|^2 \, {\mathrm d} x \,{\mathrm d}\tau + \epsilon^4 \int_0^{2\pi} \int_{\mathbb R} |\phi_2|^2 \, {\mathrm d} x \,{\mathrm d}\tau \right]
\nonumber\\[2mm]
& =
 \frac {\tilde \ell} {2\pi \epsilon} \left[ 2\pi \varrho_0 + \epsilon^2 \Big\|\frac e\beta \Big\|^2_{L^2(0, 2\pi)} + \epsilon^2 \int_0^{2\pi} \int_{\mathbb R} \left( {\mathbf b}_1^2\varpi_1^2+ \frac49f^2 {\mathbf b}_1^2\varpi_2^2 \right) \, {\mathrm d} x \,{\mathrm d}\tau + \epsilon^4 \bar{D}_{1\epsilon} \right]
\nonumber\\[2mm]
& =
 \frac {\tilde \ell} {2\pi \epsilon} \left[ 2\pi \varrho_0 + \epsilon^2 \Big\|\frac e\beta \Big\|^2_{L^2(0, 2\pi)} + \epsilon^2\big( \bar{D}_{0\epsilon} + O\left(\|f\|_*^2 \right) \big) + \epsilon^4 \bar{D}_{1\epsilon} \right],
\end{align*}
 where
\begin{align*}
 \bar{D}_{0\epsilon} := \int_0^{2\pi} \int_{\mathbb R} {\mathbf b}_1^2\varpi_1^2 \, {\mathrm d} x \,{\mathrm d}\tau,
\qquad
\bar{D}_{1\epsilon} := \int_0^{2\pi} \int_{\mathbb R} |\phi_2|^2 \, {\mathrm d} x \,{\mathrm d}\tau
\end{align*}
 are uniformly bounded, and
\[ \int_0^{2\pi} \int_{\mathbb R} \frac49f^2 {\mathbf b}_1^2\varpi_2^2 \, {\mathrm d} x \,{\mathrm d}\tau \in O\left(\|f\|_*^2 \right). \]
 And
\begin{align*}
\frac {\tilde \ell} {\pi \epsilon} \int_{{\mathfrak S}} U \left(\epsilon \frac e\beta \mathcal Z + \epsilon \phi_1 + \epsilon^2\phi_2 \right)\, {\mathrm d} x \,{\mathrm d}\tau
& = \frac {\tilde \ell} {\pi} O\left(\|f\|_* + \|e\|_{**} \right) +\epsilon \frac {\tilde \ell} {\pi} \bar{D}_{2\epsilon},
\end{align*}
\begin{align*}
& \epsilon\frac {\tilde \ell} {\pi} \int_{{\mathfrak S}} \frac e\beta \mathcal Z \phi_1\, {\mathrm d} x \,{\mathrm d}\tau
 + \epsilon^2\frac {\tilde \ell} {\pi} \int_{{\mathfrak S}}\left( \frac e\beta \mathcal Z \phi_2 + \phi_1 \phi_2 \right) \, {\mathrm d} x \,{\mathrm d}\tau
= \epsilon \frac {\tilde \ell} {\pi} O\left(\|e\|_{**} \right)
+\epsilon^2\frac {\tilde \ell} {\pi} \bar{D}_{3\epsilon},
\end{align*}
 where the terms
\begin{align*}
 \bar{D}_{2\epsilon} = \int_{{\mathfrak S}} U \phi_2 \, {\mathrm d} x \,{\mathrm d}\tau,
\qquad
\bar{D}_{3\epsilon} = \int_{{\mathfrak S}} \phi_1 \phi_2 \, {\mathrm d} x \,{\mathrm d}\tau,
\end{align*}
 are bounded.
 Then we conclude that
\begin{align} \label{w2integral}
\frac {\tilde \ell} {2\pi \epsilon} \int_{{\mathfrak S}_\delta} w_2^2 \, {\mathrm d} x \,{\mathrm d}\tau
& =
 \frac {\tilde \ell} {2\pi \epsilon} \left[ 2\pi \varrho_0 + \epsilon^2\|\frac e\beta\|^2_{L^2(0, 2\pi)} + \epsilon^2\left( \bar{D}_{0\epsilon} + O\left(\|f\|_*^2 \right) \right) + \epsilon^4 \bar{D}_{1\epsilon} \right]
\nonumber\\[2mm]
& \quad +
 \frac {\tilde \ell} {\pi} O\left(\|f\|_* + \|e\|_{**} \right) +\epsilon \frac {\tilde \ell} {\pi} \bar{D}_{2\epsilon} +\epsilon \frac {\tilde \ell} {\pi} O\left(\|e\|_{**} \right) +\epsilon^2\frac {\tilde \ell} {\pi} \bar{D}_{3\epsilon}
\nonumber\\[2mm]
& = \frac {\tilde \ell} {\epsilon} \varrho_0 +
 \frac {\tilde \ell} {\pi} O\left(\|f\|_* + \|e\|_{**} +\epsilon \right).
\end{align}

\smallskip
\noindent $\spadesuit$\
 By the relation $\tilde \phi = \beta \phi^* $, and \eqref{contraction-varphi},  we have
\begin{align} \label{tilde phi+tilde psi}
& \frac {\tilde \ell} {2\pi \epsilon} \int_{{\mathfrak S}_\delta} \frac {1} {\beta^2}
\left[
|{\tilde\phi}|^2+ |{\tilde\psi}|^2
+ 2\beta w_2 \big( {\tilde\phi} + {\tilde\psi} \big)
+ 2 \tilde \phi\tilde \psi
\right] \, {\mathrm d} x \,{\mathrm d}\tau
\nonumber\\[2mm]
& = \frac {\tilde \ell} {2\pi \epsilon} \|\phi^*\|_{L^2(\mathfrak S_\delta)}
+
O\left(
\epsilon \lf[\|{\tilde\phi}\|_{L^{\infty}(|s|>\frac{\delta}{\epsilon})}
\,+\, \|\nabla {\tilde\phi}\|_{L^{\infty}(|s|>\frac{\delta}{\epsilon})}\ri]
\,+\,
e^{-\frac{ \upsilon_2 \delta}{\epsilon}}
\right)
\frac {\tilde \ell} {2\pi \epsilon} \int_{{\mathfrak S}_\delta} \frac {1} {\beta^2} \, {\mathrm d} x \,{\mathrm d}\tau
\nonumber\\[2mm]
& = \frac {\tilde \ell} {2\pi \epsilon} \|\phi^*\|_{L^2(\mathfrak S_\delta)}
+
\frac{\tilde \ell} \epsilon O\left(e^{-\frac{\tilde \delta}\epsilon} \right).
\end{align}
 Here we note that we can show
\begin{align*}
\epsilon \lf[\|{\tilde\phi}\|_{L^{\infty}(|s|>\frac{\delta}{\epsilon})}
 \,+\, \|\nabla {\tilde\phi}\|_{L^{\infty}(|s|>\frac{\delta}{\epsilon})}\ri]
\in O\left(e^{-\frac{\tilde \delta}\epsilon} \right),
\end{align*}
 by a barrier argument.

\medskip
 Combining \eqref{constraint6}, \eqref{estimate1} and \eqref{estimate2} together with \eqref{w2integral}-\eqref{tilde phi+tilde psi}, we get
\begin{align} \label{a-equality}
a &= \frac {\tilde \ell} {\epsilon} \varrho_0
+\frac {\tilde \ell} {\epsilon\pi} O\Big(\epsilon\|f\|_* +\epsilon \|e\|_{**} +\epsilon^2 +\|\phi^*\|_{L^2(\mathfrak S_\delta)} \Big)
+ \frac{\tilde \ell} \epsilon O\left(e^{-\frac{\tilde \delta}\epsilon} \right)
\nonumber\\[2mm]
& = \frac {\tilde \ell} {\epsilon} \varrho_0 \left(1+ O\left(\epsilon^{2-\alpha}\right) \right),
\end{align}
due to \eqref{constraint-f}, \eqref{constraint-e} and \eqref{constraint-alpha-varrho}.
 By \eqref{a-equality}, we conclude that
\begin{align} \label{realtion for a eps}
 \frac\epsilon{\tilde \ell} = \frac{\varrho_0}{a} \left(1+ O\left(\epsilon^{2-\alpha}\right) \right),
\end{align}
where the estimate of ${\tilde\ell}$ is given in \eqref{boundneess of length}.
Moreover, for any $\epsilon$ small, we need the gap condition \eqref{gap condition1}
to ensure the validity of estimate \eqref{estimate norm of e1} in Lemma \ref{linear theory for e}.
By \eqref{gap condition1.22} and \eqref{realtion for a eps},
we need the following holds
\begin{align*}
\Bigg|\, \frac{\varrho_0}{a} \left(1+ O\left(\epsilon^{2-\alpha}\right) \right) - \frac{\sqrt{\lambda_*}} {j}\, \Bigg|
\,>\,
\frac{c} {j^2},
\quad \forall\,j \in {\mathbb N}.
\end{align*}
In fact, if the validity of \eqref{condition for a} holds, we then have
\begin{align*}
\Bigg|\, \frac{\varrho_0}{a} \left(1+ O\left(\epsilon^{2-\alpha}\right) \right)- \frac{\sqrt{\lambda_*}} {j}\, \Bigg|
&\,>\,
\Bigg|\, \frac{\varrho_0} a - \frac{\sqrt{\lambda_*}} {j}\, \Bigg|
\,-\,
C \epsilon^{\varrho-\alpha} \frac {\epsilon^{2-\varrho}} {a}
\nonumber\\[2mm]
&
\,>\, \frac{2c} {j^2}-C\epsilon^{\varrho-\alpha} \frac {1} {a^2} >\frac{c} {j^2},
\quad \forall\,j \in {\mathbb N},
\end{align*}
due to \eqref{realtion for a eps}, \eqref{boundneess of length},  and the assumption in \eqref{constraint-alpha-varrho}.

\medskip
\begin{appendix}

\medskip
\section{Proof of (\ref{7-89}) } \label{appendixa1}

By the norm defined in \eqref{norm h2 mathfrak S}, we will prove \eqref{7-89}, i.e.
\begin{align}
\sup_{\tau \in (0, 2\pi)} \|\phi^*(\cdot, \tau)\|^2_{L^2({\mathbb R})} \le C\frac{\tilde \ell} {\epsilon} \|\phi^*\|_{H^2_*(\mathfrak S)}^2,
\qquad
\sup_{\tau \in (0, 2\pi)} \|\phi_x^*(\cdot, \tau)\|^2_{L^2({\mathbb R})} \le C \frac{\tilde \ell} {\epsilon} \|\phi^*\|_{H^2_*(\mathfrak S)}^2.
\label{7-89-1}
\end{align}
In fact, we consider the following two cases.

\smallskip
\noindent $\clubsuit$\
If
$$
\sup_{\tau \in (0, 2\pi)} \|\phi^*(\cdot, \tau)\|^2_{L^2({\mathbb R})}
\leq 2\inf_{\tau \in (0, 2\pi)} \|\phi^*(\cdot, \tau)\|^2_{L^2({\mathbb R})} ,
$$
then
$$
\sup_{\tau \in (0, 2\pi)} \|\phi^*(\cdot, \tau)\|^2_{L^2({\mathbb R})}
\leq
2\frac{1}{2\pi}\int_0^{2\pi}\|\phi^*(\cdot, \tau)\|^2_{L^2({\mathbb R})} {\mathrm d}\tau
\leq C\frac{\tilde \ell} {\epsilon} \|\phi^*\|_{H^2_*(\mathfrak S)}^2,
$$
where ${\tilde\ell}$ has a uniformly positive lower bound in \eqref{tildeell-lowerbound}.

\smallskip
\noindent $\clubsuit$\
We consider the case
$$
\sup_{\tau\in (0, 2\pi)} \|\phi^*(\cdot, \tau)\|^2_{L^2({\mathbb R})} > 2\inf_{\tau\in (0, 2\pi)} \|\phi^*(\cdot, \tau)\|^2_{L^2({\mathbb R})} .
$$
By denoting $\tau_0$ the  point  such that
\begin{equation*}
 \|\phi^*(\cdot, \tau_0)\|^2_{L^2({\mathbb R})} \leq \frac{5}{4}\inf_{\tau\in (0, 2\pi)} \|\phi^*(\cdot, \tau)\|^2_{L^2({\mathbb R})},
\end{equation*}
we get
\begin{align*}
\sup_{\tau\in (0, 2\pi)} \|\phi^*(\cdot, \tau)\|^2_{L^2({\mathbb R})}
\,\leq\,&
\|\phi^*(\cdot, \tau_0)\|^2_{L^2({\mathbb R})}
\,+\,2\int^{2\pi}_0 \Bigg[\Big(\int_{\R} |\phi^*(x, \tau)|^2  {\mathrm d} x \Big)^{\frac{1}{2}}
\frac{{\mathrm d}}{{\mathrm d} \tau}\Big(\int_{\R} |\phi^*(x, \tau)|^2  {\mathrm d} x \Big)^{\frac{1}{2}} \Bigg] {\mathrm d} \tau
\\[2mm]
\,\leq\,&
\frac{5}{4}\inf_{\tau\in (0, 2\pi)} \|\phi^*(\cdot, \tau)\|^2_{L^2({\mathbb R})}
\,+\,2\int^{2\pi}_0\Bigg[ \Big(\int_{\R} |\phi^*(x, \tau)|^2  {\mathrm d} x \Big)^{\frac{1}{2}}
\frac{{\mathrm d}}{{\mathrm d} \tau}\Big(\int_{\R} |\phi^*(x, \tau)|^2  {\mathrm d} x \Big)^{\frac{1}{2}} \Bigg] {\mathrm d} \tau
\\[2mm]
\,\leq\,&
\frac{5}{8}\sup_{\tau\in (0, 2\pi)} \|\phi^*(\cdot, \tau)\|^2_{L^2({\mathbb R})}
\,+\,2\int^{2\pi}_0 \Bigg[\Big(\int_{\R} |\phi^*(x, \tau)|^2  {\mathrm d} x \Big)^{\frac{1}{2}}
\frac{{\mathrm d}}{{\mathrm d} \tau}\Big(\int_{\R} |\phi^*(x, \tau)|^2  {\mathrm d} x \Big)^{\frac{1}{2}}\Bigg]  {\mathrm d} \tau.
\end{align*}
Hence,
\begin{align*}
\sup_{\tau\in (0, 2\pi)} \|\phi^*(\cdot, \tau)\|^2_{L^2({\mathbb R})}
\leq
C\,\int^{2\pi}_0\Bigg[ \Big(\int_{\R} |\phi^*(x, \tau)|^2  {\mathrm d} x \Big)^{\frac{1}{2}}
\frac{{\mathrm d}}{{\mathrm d} \tau}\Big(\int_{\R} |\phi^*(x, \tau)|^2  {\mathrm d} x \Big)^{\frac{1}{2}}\Bigg]  {\mathrm d} \tau.
\end{align*}
The term in the right hand side can be estimated as follows
\begin{align*}
&\int^{2\pi}_0\Bigg[ \Big(\int_{\R} |\phi^*(x, \tau)|^2  {\mathrm d} x \Big)^{\frac{1}{2}}
\frac{{\mathrm d}}{{\mathrm d} \tau}\Big(\int_{\R} |\phi^*(x, \tau)|^2  {\mathrm d} x \Big)^{\frac{1}{2}}\Bigg]  {\mathrm d} \tau
\nonumber\\[2mm]
&\le C \Bigg\{\frac{\tilde\ell}{\epsilon} \int_0^{2\pi} \int_{\R} |\phi^*(x, \tau)|^2  {\mathrm d} x {\mathrm d}\tau
+ \frac{\epsilon}{\tilde\ell} \int_0^{2\pi} \Bigg[\frac{{\mathrm d}}{{\mathrm d} \tau}\Big(\int_{\R} |\phi^*(x, \tau)|^2  {\mathrm d} x \Big)^{\frac{1}{2}} \Bigg]^2{\mathrm d}\tau \Bigg\}
\nonumber\\[2mm]
 &\leq C \Bigg\{\frac{\tilde\ell}{\epsilon} \int_0^{2\pi} \int_{\R} |\phi^*(x, \tau)|^2  {\mathrm d} x {\mathrm d}\tau
+\frac{\tilde\ell}{\epsilon}\int_0^{2\pi}\int_{\R} \Big| \frac{\epsilon}{\tilde\ell}\phi^*_{\tau}(x, \tau)\Big|^2 {\mathrm d} x{\mathrm d}\tau \Bigg\}
\nonumber\\[2mm]
 &\le C\frac{\tilde \ell} {\epsilon} \|\phi^*\|_{H^2_*(\mathfrak S)}^2,
\end{align*}
where we have used the fact
\begin{align*}
 \int_0^{2\pi} \Bigg[\frac{{\mathrm d}}{{\mathrm d} \tau}\Big(\int_{\R} |\phi^*(x, \tau)|^2  {\mathrm d} x \Big)^{\frac{1}{2}} \Bigg]^2{\mathrm d}\tau
 \,=\,&\int_0^{2\pi} \Bigg[\frac{1}{2}\Big(\int_{\R} |\phi^*(x, \tau)|^2  {\mathrm d} x \Big)^{-\frac{1}{2}} \frac{{\mathrm d}}{{\mathrm d} \tau}\Big(\int_{\R} |\phi^*(x, \tau)|^2  {\mathrm d} x \Big) \Bigg]^2{\mathrm d}\tau
\nonumber\\[2mm]
  \,=\,&\frac{1}{4}\int_0^{2\pi}\frac{\Big(\int_{\R} 2\phi^*(x, \tau)\phi^*_{\tau}(x, \tau) {\mathrm d} x \Big)^2}{\int_{\R} |\phi^*(x, \tau)|^2  {\mathrm d} x }{\mathrm d}\tau
  \nonumber\\[2mm]
  \,\leq\,&C\int_0^{2\pi}\frac{\Big(\int_{\R} | \frac{\epsilon}{\tilde\ell}\phi^*_{\tau}(x, \tau)|^2 {\mathrm d} x \Big)\Big(\int_{\R} |\frac{\tilde\ell}{\epsilon}\phi^*(x, \tau)|^2 {\mathrm d} x \Big)}{\int_{\R} |\phi^*(x, \tau)|^2  {\mathrm d} x }{\mathrm d}\tau
    \nonumber\\[2mm]
  \,\leq\,&C \frac{\tilde\ell^2}{\epsilon^2}\int_0^{2\pi}\int_{\R} \Big| \frac{\epsilon}{\tilde\ell}\phi^*_{\tau}(x, \tau)\Big|^2 {\mathrm d} x{\mathrm d}\tau.
\end{align*}
Therefore, the first estimate in \eqref{7-89-1} has been verified, and the second one can be shown in a similar way.

\end{appendix}

\medskip
\medskip
{\bf Acknowledgements:}
L. Duan  is  supported by NSFC (No. 12201140), Guangdong Basic and Applied Basic Research Foundation (No. 2022A1515111131) and Guangzhou Basic and Applied Basic Research Foundation (No. 2025A04J0029).
S. Wei was supported by  Guangzhou Basic and Applied Basic Research Foundation (No. 2024A04J4907).
J. Yang was supported by NSFC (No. 12171109 and No. 12431003).



\begin{thebibliography}{99}

\bibitem{aftalionriviere}
A. Aftalion and T. Rivi\'{e}re,
\emph{Vortex energy and vortex bending for a rotating BEC},
Phys. Rev. A
64 (2001), 043611.

\bibitem{aftalionbook}
A. Aftalion,
\emph{Vortices in Bose Einstein Condensates},
Progr, Nonlinear Differential Equations Appl., vol. 67, Birkh\"{a}user
Boston, Inc., Boston, MA, 2006.

\bibitem{AMNconjecture1}
A. Ambrosetti, A. Malchiodi and W.-M. Ni,
\emph{Singularly perturbed elliptic equations with symmetry: existence of solutions concentrating on spheres I},
Comm. Math. Phys.
235 (2003), no. 3, 427-466.




\bibitem{AMNconjecture2}
A. Ambrosetti, A. Malchiodi and W.-M. Ni,
\emph{Singularly perturbed elliptic equations with symmetry: existence of solutions concentrating on spheres II},
Indiana Univ. Math. J.
53 (2004), no. 2, 297-329.


\bibitem{Ambrosetti}
A. Ambrosetti, M. Badiale and S. Cingolani,
\emph{Semiclassical states of nonlinear Schr\"odinger equations},
Arch. Rat. Mech. Anal.
{140} (1997), 285-300.


\bibitem{AmbMalSecc2001}
A. Ambrosetti, A. Malchiodi and S. Secchi,
\emph{Multiplicity results for some nonlinear Schr\"odinger equations with potentials},
Arch. Rat. Mech. Anal.
159 (2001), 253-271.


\bibitem{Bao}
W. Bao and Y. Cai,
\emph{Mathmatical theory and numerical methods for Bose-Einstein condensation},
Kinetic and Related Models
{6} (2013), 1-135.

\bibitem{numerics}
W. Bao and Q. Du,
\emph{Computing the ground state solution of BECs by a normalized gradient flow},
SIAM J. Sci. Comput.
25 (2004), 1674-1679.


\bibitem{baoREVIEW}
W. Bao,
\emph{Ground states and dynamics of rotating BECs}, in Transport phenomena and kinetic
theory, C. Cercignani and E. Gabetta (eds), Birkha\"{u}ser, 2007.

\bibitem{BP}
T. Bartsch and S. Peng,
\emph{Semiclassical symmetric Schr\"odinger equations: existence of solutions concentrating simultaneously on several spheres},
Z. Angew. Math. Phys.
58 (2007), 778-804.

\bibitem{BP1}
T. Bartsch and S. Peng,
\emph{Existence of solutions concentrating on higher dimensional subsets for singularly perturbed elliptic equations I},
Indiana Univ. Math. J.
57 (2008), 1599-1632.

\bibitem{BP2}
T. Bartsch and S. Peng,
\emph{Solutions concentrating on higher dimensional subsets for singularly perturbed elliptic equations II},
J. Differential Equations
248 (2010), 2746-2767.

\bibitem{Bloch}
I. Bloch, J. Dalibard and W. Zwerger,
\emph{Many-body physics with ultracold gases},
Rev. Mod. Phys.
{80} (2008), 885.


\bibitem{CPY2019}
D. Cao, S. Peng and S. Yan,
\emph{Singularly perturbed methods for nonlinear elliptic problems},
Cambridge Studies in Advanced Mathematics, 191.
Cambridge University Press, Cambridge, 2021.

\bibitem{Cornell}
E. Cornell and C. Wieman,
\emph{Nobel Lecture: Bose-Einstein condensation in a dilute gas, the first 70 years and some recent experiments},
Rev. Mod. Phys.
74 (2002), 875-893.



\bibitem{DY10}
E. Dancer and S. Yan,
\emph{A new type of concentrating solutions for a singularly perturbed elliptic problem},
Trans. Amer. Math. Soc.
359 (2007), 1765-1790.

\bibitem{Davis}
K. Davis, M. Mewes, M. Andrews, N. van Druten, D. Durfee, D. Kurn and W. Ketterle,
\emph{Bose-Einstein condensation in a gas of sodium atoms},
Phys. Rev. Lett.
75 (1995), 3969-3973.



\bibitem{delPFelmer1996}
M. del Pino and P. Felmer,
\emph{Local mountain passes for semilinear elliptic problems in unbounded domains},
Calc. Var. Partial Differential Equations
4 (1996), no. 2, 121-137.




\bibitem{delPFelmer1997}
M. del Pino and P. Felmer,
\emph{Semi-classical states for nonlinear Schr\"odinger equations},
J. Funct. Anal.
149 (1997), no. 1, 245-265.




\bibitem{delPKowWei0} 
M. del Pino, M. Kowalczyk and J. Wei,
\emph{Resonance and interior layers in an inhomogeneous phase transition model},
SIAM J. Math. Anal. 38 (2006/07), no. 5, 1542-1564.



\bibitem{delPKowWei1} 
M. del Pino, M. Kowalczyk and J. Wei,
\emph{Concentration on curves for nonlinear Schr\"odinger equations},
Comm. Pure Appl. Math. {60} (2007), no. 1, 113-146.


\bibitem{delPKowWeiYang} 
M. del Pino, M. Kowalczyk and J. Wei, J. Yang,
\emph{Interface foliation near minimal submanifolds in Riemannian manifolds with positive Ricci curvature},
Geom. Funct. Anal. 20 (2010), no. 4, 918-957.


\bibitem{docarmo}
M. P. do Carmo, Differential geometry of curves and surfaces,
Translated from the Portuguese. Prentice-Hall, Inc., Englewood Cliffs, N.J., 1976. viii+503 pp.


%












\bibitem{FloerWein1986}
 A. Floer and A. Weinstein,
 \emph{Nonspreading wave packets for the cubic Schr\"odinger equation with a bounded potential},
 J. Funct. Anal.
 69 (1986), no. 3, 397-408.

\bibitem{Gilbarg}
D. Gilbarg and N. Trudinger,
\emph{Elliptic partial differential equations of second order}, Second edition. Springer-Verlag, Berlin, 1983.



\bibitem{Gierer-Meinhardt}
A. Gierer and H. Meinhardt,
\emph{A theory of biological pattern formation},
newblock Kybernetik 12(1):30-39, 1972.



\bibitem{Gross}
E. Gross,
\emph{Structure of a quantized vortex in boson systems},
Nuovo Cimento
20 (1961), 454-466.


\bibitem {Guoqing-chunhua}
 Q. Guo, P. Luo, C. Wang, J. Yang,
\emph{ Existence and local uniqueness of normalized peak solutions for a Schrodinger-Newton system},
https://arxiv.org/pdf/2008.01557.pdf, to appear in Ann. Sc. Norm. Super. Pisa Cl. Sci.


\bibitem {Zhouyang}
Q. Guo, S. Tian and Y. Zhou,
\emph{Curve-like concentration for Bose-Einstein condensates},
Calc. Var. Partial Differential Equations
61 (2022), no. 2, Paper No. 63, 20 pp.

\bibitem{GuoYang}
Y. Guo and J. Yang,
\emph{Concentration on surfaces for a singularly perturbed Neumann problem in three-dimensional domains},
J. Differential Equations 255 (2013), no. 8, 2220-2266.

\bibitem{Guo JMPA}
Y. Guo,
\emph{ The nonexistence of vortices for rotating Bose-Einstein condensates in non-radially symmetric traps},
 J. Math. Pures Appl. (9) 170 (2023), 1-32.

\bibitem{Guo-luo-peng}
Y. Guo, Y. Luo and S. Peng,
\emph{Local uniqueness of ground states for rotating Bose-Einstein condensates with attractive interactions},
 Calc. Var. Partial Differential Equations 60 (2021), no. 6, Paper No. 237, 27 pp.

\bibitem{GS}
Y. Guo and R. Seiringer,
\emph{On the mass concentration for Bose-Einstein condensates with attractive interactions},
Lett. Math. Phys.
104 (2014), 141-156.




\bibitem{Guo}
Y. Guo, Z. Wang, X. Zeng and H. Zhou,
\emph{Properties of ground states of attractive Gross-Pitaevskii equations with multi-well potentials},
Nonlinearity
31 (2018), 957-979.


\bibitem{Guo1}
Y. Guo, X. Zeng and H. Zhou,
\emph{Energy estimates and symmetry breaking in attractive Bose-Einstein condensates with ring-shaped potentials},
Ann. Inst. H. Poincar\'e Anal. Non Lin\'eaire
33 (2016), 809-828.




\bibitem{KG}
G. Karali and C. Sourdis,
\emph{The ground state of a Gross-Pitaevskii energy with general potential in the Thomas-Fermi limit},
Arch. Ration. Mech. Anal.
217 (2015), no. 2, 439-523.

\bibitem{LS}
F. Lederer, G.I. Stegeman, D.N. Christodoulides, G. Assanto, M. Segev and Y. Silberberg,
\emph{Discrete solitons in optics},
Phys. Rep.
463 (2008), 1-126.

\bibitem{LevitanSargsjan}
B. M. Levitan and I. S. Sargsjan,
Sturm-Liouville and Dirac operators. Mathematics and Its Applications
(Soviet Series), 59. Kluwer Academic Publishers Group, Dordrecht, 1991.



\bibitem{Lieb2}
E. Lieb, R. Seiringer and J. Yngvason,
\emph{A rigorous derivation of the Gross-Pitaevskii energy functional for a two-dimensional Bose gas},
Commun. Math. Phys.
224 (2001), 17-31.



\bibitem{LinNiTak1988}
C.-S. Lin, W.-M. Ni and I. Takagi,
\emph{Large amplitude stationary solutions to a chemotaxis system},
J. Differential Equations
72 (1988), no. 1, 1-27.

\bibitem{Luo}
P. Luo, S. Peng, J. Wei and S. Yan,
\emph{Excited states of Bose-Einstein condensates with degenerate attractive interactions},
Calc. Var. Partial Differential Equations
60, 155 (2021).




\bibitem{machiodi}
 F. Mahmoudi, A. Malchiodi and M. Montenegro,
\emph{Solutions to the nonlinear Schr\"{o}dinger equation carrying momentum along a curve},
Comm. Pure Appl. Math.
62 (2009), no. 9, 1155-1264.

\bibitem{Malchodi1}
A. Malchiodi and M. Montenegro,
\emph{Boundary concentration phenomena for a singularly perturbed elliptic problem},
Comm. Pure Appl. Math.
55 (2002), no. 12, 1507-1568.





\bibitem{Pistoia}
 B. Pellacci, A. Pistoia, G. Vaira and Verzini,
\emph{Normalized concentrating solutions to nonlinear elliptic problems},
 J. Differential Equations 275 (2021), 882-919.


\bibitem{Pethick-Smith}
C. Pethick and H. Smith,
\emph{Bose-Einstein Condensation in Dilute Gases},
 Cambridge University Press, Cambridge, UK, 2002.


\bibitem{Pitaevskii}
L. Pitaevskii,
\emph{Vortex lines in an imperfect Bose gas},
Sov. Phys. JETP
13 (1961), 451-454.

\bibitem{WangWeiYang}
L. Wang, J. Wei and J. Yang,
\emph{On Ambrosetti-Malchiodi-Ni conjecture for general hypersurfaces},
Comm. Partial Differential Equations
36 (2011), no. 12, 2117-2161.


\bibitem{Wangzhao}
L. Wang and C. Zhao,
\emph{Concentration on curves for a nonlinear Schr\"{o}dinger problem with electromagnetic potential},
J. Differential Equations
266 (2019), no. 8, 4800-4834.

\bibitem{WeiYang}
J. Wei and J. Yang,
\emph{Concentration on lines for a singularly perturbed Neumann problem in two-dimensional domains},
Indiana Univ. Math. J. 56 (2007), no. 6, 3025-3073.


\bibitem{Wei-xu-yang}
S. Wei, B. Xu and J. Yang,
\emph{On Ambrosetti-Malchiodi-Ni conjecture on two-dimensional smooth bounded domains},
Calc. Var. Partial Differential Equations
57 (2018), no. 3, Paper No. 87, 45 pp.

\bibitem{SWei-Yang}
S. Wei and J. Yang,
\emph{On Ambrosetti-Malchiodi-Ni conjecture on two-dimensional smooth bounded domains: clustering concentration layers},
J. Funct. Anal.
281 (2021), no. 10, Paper No. 109220, 102 pp.






\end{thebibliography}
\end{document}